\documentclass[12pt]{article}
\usepackage{amsmath,amssymb,latexsym, amsfonts, amscd, amsthm}
\usepackage{graphicx,epsfig}

 \DeclareMathOperator{\gal}{Gal}

\theoremstyle{plain}
\newtheorem{theorem}{Theorem}[section]

\newtheorem{proposition}[theorem]{Proposition}
\newtheorem{corollary}[theorem]{Corollary}

\newtheorem{lemma}[theorem]{Lemma}

\theoremstyle{definition}
\newtheorem{definition}[theorem]{Definition}
\newtheorem{remark}[theorem]{Remark}

\newcommand{\ds}{\displaystyle}

\newcommand{\lra}{\longrightarrow}
\newcommand{\ra}{\rightarrow}

\newcommand{\ket}{\rm ket}
\newcommand{\fket}{\rm fket}
\newcommand{\form}{\rm form}
\newcommand{\et}{\rm et}

\newcommand{\tensor}{{\otimes}}
\newcommand{\Hom}{\mbox{Hom}}

\newcommand{\Spec}{\mbox{Spec}}
\newcommand{\Spf}{\mbox{Spf}}
\newcommand{\Spm}{\mbox{Spm}}

\newcommand{\Ker}{\mbox{Ker}}
\newcommand{\Fil}{\mbox{\rm Fil}}

\newcommand{\GL}{{\rm \mathbf{GL}}}

\newcommand{\Z}{{\mathbb Z}}
\newcommand{\Q}{{\mathbb Q}}
\newcommand{\C}{{\mathbb C}}
\newcommand{\R}{{\mathbb R}}

\newcommand{\N}{{\mathbb N}}

\newcommand{\Kbar}{{\overline{K}}}

\newcommand{\bareta}{{\overline{\eta}}}
\newcommand{\Rbar}{{\overline{R}}}
\newcommand{\Nbar}{{\overline{N}}}

\newcommand{\hatcO}{\widehat{{\overline{\cO}}}}

\newcommand{\bDcrisgeo}{\bD_{\rm log}^{\rm geo}}
\newcommand{\bDcrisM}{\bD_{\rm log,L}}
\newcommand{\bDcrisar}{\bD_{\rm log}^{\rm ar}}

\newcommand{\cL}{{\mathbb L}}

\newcommand{\cA}{{\cal A}}
\newcommand{\cH}{{\cal H}}
\newcommand{\cI}{{\cal I}}
\newcommand{\cJ}{{\cal J}}
\newcommand{\cU}{{\cal U}}
\newcommand{\cP}{{\cal P}}
\newcommand{\cQ}{{\cal Q}}

\newcommand{\cF}{{\cal F}}
\newcommand{\cG}{{\cal G}}
\newcommand{\cE}{{\cal E}}

\newcommand{\bD}{{\mathbb D}}

\newcommand{\bV}{{\mathbb V}}
\newcommand{\bH}{{\mathbb H}}
\newcommand{\bA}{{\mathbb A}}
\newcommand{\bB}{{\mathbb B}}

\newcommand{\cS}{{\cal S}}

\newcommand{\cO}{{\cal O}}

\newcommand{\Sh}{{\mbox{Sh}}}
\newcommand{\Mod}{{\mbox{Mod}}}

\newcommand{\Crys}{{\mbox{Crys}}}
\newcommand{\Isoc}{{\mbox{Isoc}}}

\newcommand{\Rep}{{\mbox{Rep}}}

\newcommand{\bL}{{\mathbb L}}
\newcommand{\WW}{{\mathbb W}}

\newcommand{\cW}{{\cal W}}

\newcommand{\fX}{{\mathfrak{X}}}
\newcommand{\Xtilde}{{\widetilde{X}}}
\newcommand{\Stilde}{{\widetilde{S}}}
\newcommand{\Ntilde}{{\widetilde{N}}}
\newcommand{\Mtilde}{{\widetilde{M}}}

\newcommand{\cC}{{\cal C}}

\newcommand{\cM}{{\cal M}}

\begin{document}

\title{Semistable sheaves and comparison isomorphisms in the semistable case}
\author{Fabrizio Andreatta \\ Adrian Iovita\\} \maketitle

\rightline{\em Dedicato a Francesco Baldassarri, con affetto}

\tableofcontents \pagebreak

\section{Introduction}
\label{sec:intro}

Let $K$ be a finite extension of $\Q_p$ with ring of integers $\cO_K$ and fix for the rest of this article
a uniformizing parameter $\pi$ of $\cO_K$. We denote by $S:={\rm Spec}(\cO_K)$ and by $M$ the log structure
on $S$ associated to the prelog structure $\N\lra \cO_K$ sending $n\in \N$ to $\pi^n\in \cO_K$. We denote by
$(S,M)$ the associated log scheme.

Let $X\lra S$ be a morphism of schemes of finite type (or a morphism of formal schemes topologically of finite type) with semistable reduction, by which we mean
that there exists a log structure $N$ on $X$ and a morphism of log schemes (or log formal schemes) $f\colon (X,N) \lra (S,M)$ satisfying the assumptions of section
\S 2.1.2. In particular $f$ is log smooth.

Let now $\WW:=\WW(\cO_K/\pi\cO_K)$ and we denote by $\cO:=\WW[\![Z]\!]$ and by $\cO\lra \cO_K$ the natural $\WW$-algebra homomorphism sending $Z$ to $\pi$. Write
$P_\pi(Z)\in \WW[Z]$ for the monic irreducible polynomial of $\pi$ over $\WW$. It is a generator of $\Ker\bigl(\cO\lra \cO_K  \bigr)$. We denote by
$\widetilde{S}:=\Spf(\cO)$ and by $\widetilde{M}$ the log structure on $\widetilde{S}$ associated to the prelog structure $\N\lra \cO$ sending $n\in \N$ to $Z^n\in
\cO$. Let us consider the natural diagram of log formal schemes
$$
\begin{array}{cccccccccc}
(X,N)\\
f\downarrow&&\\
(S,M)&\lra&(\widetilde{S},\widetilde{M}).
\end{array}
$$
We assume that there exists a GLOBAL deformation $\tilde{f}\colon  \big(\widetilde{X},\widetilde{N}\big)\lra \big(\widetilde{S},\widetilde{M}\big)$ of $f$. Such
deformations exist for example if $X$ is affine or if the relative dimension of $X$ over $S$ is $1$, but not in general.

\noindent
Our main concern in this article is to:

\bigskip
\noindent 1) Define Faltings's logarithmic sites $\fX_{K}$ and $\fX_{\Kbar}$ associated to $f\colon  (X,N)\lra (S,M)$ and Fontaine (ind continuous) sheaves on it associated to the
deformation $\tilde{f}\colon (\widetilde{X},\widetilde{N}) \lra (\widetilde{S}, \widetilde{M})$: $\bB_{\rm cris}^\nabla$, $\bB_{\rm
log}^\nabla$, $\overline{\bB}_{\rm log}^\nabla$, $\bB_{\rm log}$ and $\overline{\bB}_{\rm log}$.

\bigskip
\noindent 2) Define the category ${\rm Sh}(\fX_K)_{\rm ss}$ of semistable (in fact arithmetically semistable) \'etale local systems on $\fX_K$ and study its
properties; see \S\ref{sec:defsemsitablesheaf} and \S\ref{sec:defsemsitablesheafglobal}.

\bigskip
\noindent 3) Define in \S\ref{sec:defsemsitablesheafglobal} a Fontaine functor $\bD_{\rm log}^{\rm ar}$ from the category of semistable \'etale local systems on
$\fX_K$ to the category of log filtered $F$-isocrystals on $X$ relative to $\cO$. More precisely these are Frobenius isocrystals (considering  the Kummer \'etale
site on  $(X, N)$ modulo $p$)   relatively to the $p$-adic completion of the divided power envelope of $\cO$ with respect to the ideal generated by $p$ and
$P_\pi(Z)$, with filtration on their base change via $\cO\to \cO_K$ defined by mapping $Z$ to $\pi$; see \S\ref{sec:isocrystals}.

\bigskip
 \noindent 4) We prove the following {\em comparison isomorphism theorem}, see \ref{thm:mainthm}. Suppose that $\bL$ is a $p$-adic Kummer \'etale local system on
$X_K$, which when viewed as an \'etale local system on $\fX_\Kbar$ is semistable. Assume that $X$ is a proper and geometrically connected scheme over $\cO_K$. We have, see \ref{thm:mainthm},

\begin{theorem}
\label{thm:mainresult} a) The $p$-adic representation ${\rm H}^i\bigl(X_\Kbar^{\ket}, \bL\bigr)$ of\/ $G_K:={\rm Gal}(\Kbar/K)$ is semistable for all $i\ge 0$.

b) There are natural isomorphisms respecting all additional structures (i.e. the filtrations, after extending the scalars to $K$, the Frobenii and the monodromy
operators)
$$
D_{\rm st}\bigl({\rm H}^i(X_\Kbar^{\ket}, \bL) \bigr)\cong {\rm H}^i\bigl(\bigl(X_k/\WW(k)^+\bigr)^{\rm cris}_{\rm log}, \bDcrisar(\bL)^+\bigr).
$$
\end{theorem}
\smallskip Here, $\bDcrisar(\bL)^+$ is the Frobenius log isocrystal on $X_k$ relative to $\WW(k)^+$ obtained from $\bDcrisar(\bL)$ by base change via the map $\cO\to \WW(k)$ sending $Z$ to $0$. Here, $\WW(k)^+$ is $\WW(k)$ with log structure defined by $\N\to \WW(k)$ given by sending every $n\in \N$ to $0$. In particular, ${\rm H}^i\bigl(\bigl(X_k/\WW(k)^+\bigr)^{\rm cris}_{\rm log}, \bDcrisar(\bL)^+\bigr)$ is a finite dimensional $K_0$-vector space endowed with a Frobenius linear automorphism and a monodromy operator. Its base change to $K$ coincides with the cohomology of the filtered log isocrystal  $\bDcrisar(\bL)_{X_K}$ given by base change of $\bDcrisar(\bL)$ via the map $\cO\to \cO_K$, sending $Z$ to $\pi$. Thus these cohomology groups are endowed with filtrations coming from the filtration on $\bDcrisar(\bL)_{X_K}$.\medskip

For the constructions in (1)--(3) the existence of local deformations of $X$ to $\cO$ would suffice; namely the notion of semistable \'etale local systems and the
functor $\bD_{\rm log}^{\rm ar}$ can be defined  locally and then glued. On the contrary, it is in (4) that we definitely need the existence of a global deformation
$\widetilde{X}$ in order to guarantee the finiteness of the cohomology of Frobenius isocrystals on the reduction of $(X,N)$ modulo $p$ relatively to $\cO_{\rm
cris}$, a key ingredient to prove the theorem. We hope to be able to  remove this assumption in the future.

\bigskip
\noindent All these constructions are generalizations to the semistable case of the analogue results in the smooth case. The comparison isomorphisms in the smooth
case were recently proved in \cite{andreatta_iovita_comparison} (after having been proved before in different ways and various degrees of generality by G.~Faltings,
T.~Tsuji, W.~Niziol etc. see the introduction of \cite{andreatta_iovita_comparison} for an account on the history of the problem to date.)

The proof of the comparison isomorphisms in the smooth case presented in \cite{andreatta_iovita_comparison}
was in fact a result cumulating three sources:

i) \cite{andreatta_iovita_comparison} in which Faltings' site associated to a smooth scheme (or formal scheme)
was defined (in that article $K$ was supposed unramified over $\Q_p$ and so no deformation was required)
and the global theory of Fontaine sheaves on the site was developed.

ii) \cite{brinon} where the local Fontaine theory in the relative smooth case was worked out. In particular, if $R$ is an $\cO_K$-algebra, ``small'' (in Faltings'
sense) and smooth over $\cO_K$ it was proved in \cite{brinon} the following fundamental result: the inclusion $R[1/p]\hookrightarrow B_{\rm cris}(R)$ is faithfully
flat.

iii) \cite{andreatta_brinon} where (in the notations of ii) above)
the geometric Galois cohomology of $B_{\rm cris}(R)$ was calculated.

\bigskip
\noindent The present article generalizes to the semistable case all three articles quoted above as follows: in chapter $2$ we develop the global theory, i.e., we
define Faltings' logarithmic sites $\fX_K$ and $\fX_\Kbar$ and the Fontaine sheaves on it. In chapter $3$ we work out the local Fontaine theory in the relative
semistable case generalizing \cite{brinon}: we define semistable representations and prove their main properties. The situation is more complicated than in the
smooth case, namely let $\mathcal{U}=\Spf (R)$ be a small log affine open of $(X,N)$ and $\widetilde{\mathcal{U}}=\Spf(\widetilde{R})$ a deformation of it to
$(\widetilde{S},\widetilde{M})$. We define relative Fontaine rings $B_{\rm log}^{\rm cris}(\widetilde{R})$ and $B_{\rm log}^{\rm max}(\widetilde{R})$, which are
both $\widetilde{R}[1/p]$-algebras and {\bf together} generalize $B_{\rm cris}(\widetilde{R})$ to the semistable case. More precisely:

\bigskip
\noindent a) Let $\widetilde{R}_{\rm max}$ be the $p$-adic completion of the ring $ \widetilde{R}\Bigl[\frac{P_{\pi}(Z)}{p}\Bigr]$ as a subring of
$\widetilde{R}[1/p]$. We prove that the inclusion $ \widetilde{R}_{\rm max}\bigl[p^{-1}\bigr]\hookrightarrow B_{\rm log}^{\rm max}(\widetilde{R})$ is close to being faithfully flat; see \ref{thm:Blogmaxff}. More precisely we show:

i) If $\alpha=1$, see the assumptions on \S 3.1  (i.e. we are in the semistable reduction case) then in fact $ \widetilde{R}_{\rm max}\bigl[p^{-1}\bigr]\hookrightarrow B_{\rm log}^{\rm max}(\widetilde{R})$ is faithfully flat.

ii) If $\alpha>1$ then the situation is more complicated, namely there exists an algebra $A$ (denoted $A^{\rm +,log}_{\rm \widetilde{R}^o, max}$ in the proof of
theorem \ref{thm:Blogmaxff}) such that $ \widetilde{R}_{\rm max}\bigl[p^{-1}\bigr]\hookrightarrow A\bigl[p^{-1}\bigr]\hookrightarrow B_{\rm log}^{\rm
max}(\widetilde{R})$ and having the properties that a faithfully flat $ \widetilde{R}_{\rm max}\bigl[p^{-1}\bigr]$-algebra $C$ is a direct summand of $A\bigl[
p^{-1}\bigr]$ as $C$-module and the extension $A\bigl[p^{-1}\bigr]\hookrightarrow B_{\rm log}^{\rm max}(\widetilde{R})$ is faithfully flat. It follows that if a
sequence of $ \widetilde{R}_{\rm max}\bigl[p^{-1}\bigr]$-modules
$$0\longrightarrow M'\longrightarrow M\longrightarrow M"\longrightarrow 0$$
becomes exact after base changing it to $ B_{\rm log}^{\rm max}(\widetilde{R})$ then it was exact to start with and that an $ \widetilde{R}_{\rm
max}\bigl[p^{-1}\bigr]$-module  is finite and projective if it is so after base changing  to $ B_{\rm log}^{\rm max}(\widetilde{R})$. These properties are what we
call ``close to faithful flatness" and allow to prove that $\bD_{\rm log}^{\rm ar}$ of a semistable sheaf is an $F$-isocrystal.

\bigskip
\noindent
b) If we denote by $G_R$ the (algebraic) fundamental group of
${\rm Spm}(R_\Kbar)$ for a geometric base point, we compute
the continuous $G_R$-cohomology of $B_{\rm log}^{\rm cris}(\widetilde{R})$
with results similar to those in \cite{andreatta_brinon}.

\bigskip
\noindent
c) Finally, if $\cG_R$ is the (algebraic) fundamental group of
$R[1/p]$ for the same choice of geometric base point as at b) above
and if $V$ is a $p$-adic representation of $\cG_R$ then we prove:
$V$ is $B_{\rm log}^{\rm cris}(\widetilde{R})$-admissible if and only if
$V$ is $B_{\rm log}^{\rm max}(\widetilde{R})$-admissible if and only
if the \'etale local system $\bL$ attached to the representation $V$
is semistable, in which case $V$ itself is called a semistable representation.

Moreover if $V$ is a semistable representation then $D_{\rm log}^{\rm cris}(V)$ and $D_{\rm log}^{\rm max}(V)$ determine one another and $D_{\rm log}^{\rm cris}(V)$
provides $\bD_{\rm log}^{\rm ar}(\bL)$.

Using all these results in the second part of chapter $2$ we prove
the semistable comparison isomorphism theorem \ref{thm:mainresult}
stated above.

\bigskip
\noindent We'd like to point out that T.~Tsuji has a preprint \cite{tsuji} where the theory of semistable \'etale sheaves on a semistable proper scheme over
$\cO_K$ is developed. On the one hand his work is more general than ours as he has less restrictive assumptions on the logarithmic structures allowed and on the
existence of a global deformation over $(\widetilde{S},\widetilde{M})$. On the other hand neither does the author prove in that article any faithful flatness result
nor does he derive comparison isomorphisms for the cohomology of the semistable \'etale local systems defined there.

Finally, recent work of P.~Scholze \cite{Scholze} might lead in the future to results in the direction of proving that de Rham  \'etale sheaves are potentially semistable.

\bigskip
{\bf Acknowledgements} We are very grateful to the referee of this article  for the careful reading of the text and for pointing out some errors in an earlier
version.

\section{Fontaine's sheaves on Faltings' site}
\subsection{Notations}\label{sec:notation}

Let  $p>0$ denote a prime integer and $K$ a complete discrete
valuation field of characteristic~$0$ and perfect residue field
$k$ of characteristic~$p$. Let $K_0$ be the field of fractions of
$\WW(k)$. Let $\cO_K$ be the ring of integers of $K$ and choose a
uniformizer $\pi\in \cO_K$. Fix an algebraic closure~$\Kbar$
of~$K$ and write $G_K$ for the Galois group of $K\subset \Kbar$.
In~$\Kbar$ choose:\smallskip

(a)   a compatible systems of $n!$--roots $\pi^{\frac{1}{n!}}$
of~$\pi$;\smallskip

(b)   a compatible systems of  primitive $n$--roots $\epsilon_{n}$ of~$1$ for varying~$n\in\N$. \smallskip

Define $K_n':=K[\pi^{\frac{1}{n!}}]$ and $K'_\infty:=\cup_n K_n'$.
Since $T^m-\pi$ is an Eisenstein polynomial over $\cO_K$, then
$\cO_{K_n'}:=\cO_K[\pi^{\frac{1}{
n!}}]=\cO_K[T]/\bigl(T^{n!}-\pi\bigr)$ is a complete dvr with
fraction field precisely~$K_n'$.\smallskip

Let $M$ be the log structure on $S:=\Spec(\cO_K)$ associated to the prelog structure $\psi\colon \N \to \cO_K$ given by $1 \mapsto \pi$. Let $\psi_K\colon \cO_K[\N]
\to \cO_K$ be the associated map of $\cO_K$--algebras. For every $n\in\N$ we write $\bigl(S_n,M_n\bigr)$ for the compatible system of log schemes given by
$S_n:=\Spec\bigl(\cO_K/\pi^n \cO_K\bigr)$ and log structure $M_n$ associated to the prelog structure $\N \to \cO_K/\pi^n \cO_K$, $1\mapsto \pi$. We refer to
\cite{katolog} for generalities on logarithmic geometry. \smallskip

Write $\cO:=\WW(k)[\![Z]\!]$ for the power series ring in the variable $Z$ and let $N_\cO$ be the log structure associated to the prelog structure $\psi_\cO\colon
\N \to \cO$ defined by $1\mapsto Z$. We define Frobenius on $\cO$ to be the homomorphism given by the usual Frobenius on $\WW(k)$ and by $Z\mapsto Z^p$. It extends
to a morphism of log schemes inducing multiplication by $p$ on $\N$. Let $P_\pi(Z)$ be the minimal polynomial of $\pi$ over $\WW(k)$. It is an Eisenstein polynomial
and $\theta_\cO\colon \cO \lra\cO_K$, defined by $Z \mapsto \pi$, induces an isomorphism, compatibly with the log structures, $\cO/\bigl(P_\pi(Z)\bigr) \lra\cO_K$.

\smallskip

\subsubsection{The classical  period rings}\label{sec:classical} Write $A_{\rm cris}$
for the classical ring of periods constructed by Fontaine \cite[\S 2.3]{Fontaineperiodes} and $A_{\rm log}$ the classical ring of periods constructed by Kato
\cite[\S 3]{katoperiodes}. More precisely, let $\widetilde{\bf E}_{\cO_\Kbar}^+:=\ds \lim_{\leftarrow} \widehat{\cO}_{\Kbar}$ where the transition maps are given by raising to the $p$-th
power. Consider the elements $\overline{p}:=\bigl(p,p^{\frac{1}{p}},\ldots\bigr)$, $\overline{\pi}:=\bigl(\pi,\pi^{\frac{1}{p}}, \cdots \bigr)$ and
$\varepsilon:=\bigl(1,\epsilon_{p}, \cdots \bigr)$. The set $\widetilde{\bf E}_{\cO_\Kbar}^+$  has a natural ring structure \cite[\S 1.2.2]{Fontaineperiodes} in which $p\equiv 0$ and a
log structure associated to the morphism of monoids $\N \to \widetilde{\bf E}_{\cO_\Kbar}^+$ given by $1\mapsto \overline{\pi}$. Write $A_{\rm inf}\left(\cO_{\Kbar}\right)$, or simply
$A_{\rm inf}$, for the Witt ring $\WW\bigl(\widetilde{\bf E}_{\cO_\Kbar}^+\bigr)$. It is endowed with the log structure associated to the morphism of monoids $\N \to \WW\bigl(\widetilde{\bf E}_{\cO_\Kbar}^+\bigr)$ given by $1\mapsto \bigl[\overline{\pi}\bigr]$. There is a natural ring homomorphism $\theta\colon \WW\bigl(\widetilde{\bf E}_{\cO_\Kbar}^+\bigr) \lra \widehat{\cO}_\Kbar$
\cite[\S 1.2.2]{Fontaineperiodes} such that $\theta\bigl(\bigl[\overline{\pi}\bigr]\bigr)=\pi$. In particular, it is surjective and strict considering on
$\widehat{\cO}_\Kbar$ the log structure associated to $\N \to \widehat{\cO}_\Kbar$ given by $1\mapsto \pi$. Its kernel is principal and generated by
$P_\pi\bigl(\bigl[\overline{\pi}\bigr]\bigr)$ or by the element $\xi:=\bigl[\overline{p}\bigr]-p$. Write ${\mathcal I}$ for the ideal of $\WW\bigl(\widetilde{\bf E}_{\cO_\Kbar}^+\bigr)$
generated by $[\varepsilon]^{\frac{1}{p^n}}-1$ for $n\in\N$ and by the Teichm\"uller lifts $[x]$ for $x\in \widetilde{\bf E}_{\cO_\Kbar}^+$ such that $x^{(0)}$ lies in the maximal ideal
of $\widehat{\cO}_{\Kbar}$.

We recall that $A_{\rm cris}$  is the $p$-adic completion of the DP envelope  of $\WW\bigl(\widetilde{\bf E}_{\cO_\Kbar}^+\bigr)$ with respect to the ideal
generated by $p$ and the kernel of $\theta$. Similarly, $A_{\rm log}$ is the $p$-adic completion of the log DP envelope of the morphism $\WW\bigl(\widetilde{\bf
E}_{\cO_\Kbar}^+\bigr)\otimes_{\WW(k)} \cO$ with respect to the morphism $\theta\otimes \theta_\cO\colon \WW\bigl(\widetilde{\bf
E}_{\cO_\Kbar}^+\bigr)\otimes_{\WW(k)} \cO \lra \widehat{\cO}_{\Kbar}$. In particular, $$A_{\rm log}\cong A_{\rm cris}\left\{\langle u-1 \rangle\right\},$$by which
we mean that there exists an isomorphism of $A_{\rm cris}$-algebras from the $p$-adic completion $A_{\rm cris}\left\{\langle V \rangle\right\}$ of the DP polynomial
ring over $A_{\rm cris}$ in the variable $V$  and $A_{\rm log}$ sending $V$ to $ u-1$ with $u:= \frac{\bigl[\overline{\pi}\bigr]}{Z}$;
cf.~\cite[Prop.~3.3]{katoperiodes} and \cite[\S 2]{breuil} where the ring is denoted $\widehat{A}_{\rm st}$. We endow $A_{\rm cris}$ and $A_{\rm log}$ with the
$p$-adic topology and the divided power filtration. We write $B_{\rm cris}:=A_{\rm cris}\bigl[t^{-1}\bigr]$ and $B_{\rm log}:=A_{\rm log}\bigl[t^{-1}\bigr]$, where
$t:=\log\bigl([\varepsilon]\bigr)$, with the inductive limit topology and the filtration $\Fil^n B_{\rm cris}:=\sum_{m\in \N} \Fil^{n+m} A_{\rm cris} t^{-m}$ and
$\Fil^n B_{\rm log}:=\sum_{m\in \N} \Fil^{n+m} A_{\rm log} t^{-m}$.

Let $B_{\rm dR}^+$ be the classical ring of Fontaine defined as the completion of  $\WW\bigl(\widetilde{\bf E}_{\cO_\Kbar}^+\bigr)[p^{-1}]$ with respect to the ideal generated by $\ker
\theta$ with the filtration defined by this ideal.  Similarly, we construct $B_{\rm dR}^+(\cO)$ as follows. Define $A_{\rm inf}(\cO)$ as the completion of
$\WW\bigl(\widetilde{\bf E}_{\cO_\Kbar}^+\bigr)\otimes_{\WW(k)} \cO$ with respect to the ideal $\bigl(\theta \otimes \theta_\cO \bigr)^{-1}\bigl(p \widehat{\cO}_\Kbar\bigr)$ and simply
denote $\theta \otimes \theta_\cO\colon A_{\rm inf}(\cO)\to  \widehat{\cO}_\Kbar$ the map extending $\theta \otimes \theta_\cO$. Then, we set $B_{\rm dR}^+(\cO)$ to
be the completion of $A_{\rm inf}(\cO)\big[p^{-1}\big]$ with respect to the ideal generated by $\ker \theta \otimes \theta_\cO$, with the filtration defined by this
ideal. Define $B_{\rm dR}:=B_{\rm dR}^+\bigl[t^{-1}\bigr]$ and $B_{\rm dR}(\cO):=B_{\rm dR}^+(\cO)\bigl[t^{-1}\bigr]$. We extend the filtrations to $B_{\rm dR}$ and
$B_{\rm dR}(\cO)$ as before. Note that $B_{\rm dR}^+(\cO)\cong B_{\rm dR}^+[\![u-1]\!]\cong B_{\rm dR}^+[\![Z-\pi]\!] $ where the filtration is the composite of the
filtration on $B_{\rm dR}^+$ and the $(u-1)$-adic or $(Z-\pi)$-adic filtration; cf \ref{cor:BdRstr}(4). We have an inclusion $B_{\rm log}\subset B_{\rm dR}(\cO)$,
strict with respect to the filtrations. We also have the classical subrings $B_{\rm cris,K}:=B_{\rm cris}\otimes_{K_0} K$ and $B_{\rm st,K}:=B_{\rm st}\otimes_{K_0}
K$ of $B_{\rm dR}$ introduced by Fontaine; see \cite[\S 3.1.6]{Fontaineperiodes}  and \cite[Thm. 4.2.4]{Fontaineperiodes}. Define $\overline{B}_{\rm log}$ to be the
image of the composite map $f_\pi\colon B_{\rm log}\to B_{\rm dR}(\cO) \to B_{\rm dR}$ defined in \cite[\S 7]{breuil}, and given by $Z\mapsto \pi$. We consider the
image filtration which is the filtration inherited by $B_{\rm dR}$. For later use we remark

\begin{lemma}\label{lemma:BlogmodPpi} We have natural  morphisms $B_{\rm cris,K}\subset
B_{\rm st,K}\subset \overline{B}_{\rm log}  \subset B_{\rm dR}$, which are $G_K$-equivariant,
are strictly compatible with the filtrations and induce isomorphisms
on the associated graded rings.

\end{lemma}
\begin{proof} The map $f_\pi$ is clearly  compatible with $G_K$-action and the filtrations.
It sends $P_\pi(Z)$ to $0$. In particular,  $A_{\rm log}/\bigl(P_\pi(Z)\bigr) A_{\rm log}$ is an $ A_{\rm cris}\otimes_{\WW(k)} \cO_K$ algebra and contains the
divided powers of the element ${\bigl[\overline{\pi}\bigr]}/{\pi}-1$. In particular, $A_{\rm log}/\bigl(P_\pi(Z)\bigr)$ contains the element
$\log\bigl([\overline{\pi}]/\pi\bigr)$ which generates $B_{\rm st}$ as $B_{\rm cris}$-algebra by \cite[lemma 7.1]{breuil}. See also \cite[\S
3.1.6]{Fontaineperiodes}. This provides the claimed inclusions. As the map $B_{\rm cris}\to B_{\rm dR}$ induces an isomorphism on the associated graded rings, the
claim follows.
\end{proof}

The rings $A_{\rm cris}$ and $A_{\rm log}$, and hence $B_{\rm cris}$ and $B_{\rm log}$, are endowed with a
Frobenius having the property that $\varphi(u)=u^p$ and
$\varphi(t)=pt$ and a continuous action of the Galois group $G_K$. Moreover, there is a derivation
$$d\colon A_{\rm log} \lra A_{\rm log} \frac{d Z}{Z}$$which
is $A_{\rm cris}$ linear and satisfies $d\bigl((u-1)^{[n]}\bigr)= (u-1)^{[n-1]} u \frac{d Z}{Z}$; see \cite[Prop.~3.3]{katoperiodes} and \cite[Lemma 7.1]{breuil}.
Its kernel is $A_{\rm cris}$ and the inclusion $A_{\rm cris}\subset A_{\rm log}$ is split injective where the left inverse is defined by setting $(u-1)^{[n]}\mapsto
0$ for every $n\in\N$. We let $N$ be the $A_{\rm cris}$-linear operator on $A_{\rm log}$ such that $d(f)=N(f) \frac{d Z}{Z}$. In particular, $d$ and $N$ extend to
$B_{\rm log}$. It is proven in \cite[Thm.~3.7]{katoperiodes} that Fontaine's period ring $B_{\rm st}$, see \cite[\S 3.1.6]{Fontaineperiodes}, is isomorphic to the
subring of $B_{\rm log}$ where $N$ acts nilpotently.

\

{\em $B_{\rm log}$-admissible representations} According to
\cite[Def. 3.2]{breuil} a $\Q_p$-adic representation $V$ of $G_K$
is called $B_{\rm log}$-admissible if\smallskip

(1) $D_{\rm log}(V):=\bigl(B_{\rm log}\otimes_{\Q_p}
V\bigr)^{G_K}$ is a free $B_{\rm log}^{G_K}$-module;\smallskip

(2) the morphism  $B_{\rm log}\otimes_{B_{\rm log}^{G_K}} D_{\rm
log}(V)\lra B_{\rm log}\otimes_{\Q_p} V$ is an isomorphism,
strictly compatible with the filtrations.\smallskip

In this case $D_{\rm log}(V)$ is an object in the category ${\cal
MF}_{B_{\rm log}^{G_K}}(\varphi,N)$ of finite and free $B_{\rm
log}^{G_K}$-modules $M$, endowed with (i) a monodromy operator
$N_M$ compatible via Leibniz rule with the one on $B_{\rm
log}^{G_K}$, (ii) a decreasing exhaustive filtration $\Fil^n M$
which satisfies Griffiths' transversality with respect to $N_M$
and such that the multiplication map $B_{\rm log}^{G_K}\times M
\to M$ is compatible with the filtrations, (iii) a semilinear
Frobenius morphism $\varphi_M\colon M \to M$ such that $N_M \circ
\varphi_M= p \varphi_M\circ N_M$ and $\det \varphi_M$ is
invertible in $B_{\rm log}^{G_K}$. See \cite[\S 6.1]{breuil}.

\

{\em Comparison with semistable representations} Consider the category ${\cal MF}_{K}(\varphi,N)$ of finite dimensional
$K_0$-vector spaces $D$ endowed with (i) a
monodromy operator $N_D$, (ii) a descending and exhaustive filtration $\Fil^n D_K$ on $D_K:=D\otimes_{K_0} K$, (iii) a
Frobenius $\varphi_D$ such that ${\rm det}
\varphi_D\neq 0$ and $N_D \circ \varphi_D= p \varphi_D\circ N_D$; see \cite{colmez_fontaine}. Such a module
is called $B_{\rm st}$-admissible if there exists a
$\Q_p$-representation $V$ of $G_K$ such that $D_{\rm st}(V):=\bigl(V\otimes_{\Q_p}
B_{\rm st}\bigr)^{G_K}$ is isomorphic to $D$ compatibly with monodromy operator,
Frobenius and filtration after extending scalars to $K$.  Consider the functor $$T\colon {\cal MF}_{K}(\varphi,N)\lra
{\cal MF}_{B_{\rm
log}^{G_K}}(\varphi,N)$$sending $D \mapsto T(D):=D\otimes_{K_0} B_{\rm log}^{G_K}$ with  monodromy operator
$N_D\otimes 1 + 1\otimes N$, Frobenius $\varphi_D\otimes
\varphi$ and filtration defined on \cite[p. 201]{breuil} using the filtration on $D_K$ and
the monodromy operator. More precisely,   the map $f_\pi\colon B_{\rm
log}\to B_{\rm dR}$ defined in \ref{lemma:BlogmodPpi} by sending  $Z$ to $\pi$ induces a map
$B_{\rm log}^{G_K}\lra B_{\rm dR}^{G_K}=K$. This provides a
morphism $\rho\colon T(D)\to D_K$. Then, $\Fil^n T(D)$ is defined inductively on $n$ by setting $\Fil^n T(D):=\left\{x\in T(D)\vert \rho(x)\in \Fil^n D_K, \quad
N(x)\in \Fil^{n-1} T(D) \right\} $.

\begin{proposition}\label{prop:T} {\rm \cite{breuil}} (1) The functor $T$ is an
equivalence of categories.\smallskip

(2) The notions of $B_{\rm log}$-admissible representations of
$G_K$ and of $B_{\rm st}$-admissible representations are
equivalent. For any such, we have $T\bigl(D_{\rm st}(V)\bigr)\cong
D_{\rm log}(V)$.

\end{proposition}
\begin{proof}
(1) is proven in \cite[Thm.~6.1.1]{breuil}. (2) is proven in
\cite[Thm.~3.3]{breuil}. \end{proof}

{\em An admissibility criterion.} We prove  a criterion of admissibility very similar to the ones in
\cite{colmez_fontaine}. Let $M$ be an object of ${\cal MF}_{B_{\rm log}^{G_K}}(\varphi,N)$.
The map $B_{\rm log} \to B_{\rm dR}$ sending $Z$ to $\pi$ has image $\overline{B}_{\rm log} $ by
\ref{lemma:BlogmodPpi}. Define
$$V_{\rm log}^0(M):=\left(B_{\rm log}\otimes_{B_{\rm log}^{G_K}} M
\right)^{N=0,\varphi=1},\quad V_{\rm log}^1(M):= \bigl(\overline{B}_{\rm log}
\otimes_{B_{\rm log}^{G_K}} M\bigr)/\Fil^0\bigl(\overline{B}_{\rm log}  \otimes_{B_{\rm
log}^{G_K}} M\bigr).$$Let
$$
\delta(M)\colon V_{\rm log}^0(M) \lra V_{\rm log}^1(M)
$$
be the map given by the composite of the inclusion $V_{\rm log}^0(M)\subset B_{\rm log}  \otimes_{B_{\rm log}^{G_K}} M$
and the projection $B_{\rm log}
\otimes_{B_{\rm log}^{G_K}} M \to \overline{B}_{\rm log}  \otimes_{B_{\rm log}^{G_K}} M$. We simply write
$V_{\rm log}(M)$ for the kernel of $\delta(M)$.  Then,

\begin{proposition}
\label{prop:admis} (1) A filtered $(\varphi,N)$-module $M$ over
$B_{\rm log}^{G_K}$ is admissible  if and only if \enspace {\rm
(a)} $V_{\rm log}(M)$ is a finite dimensional $\Q_p$-vector space
and \enspace {\rm (b)} $\delta(M)$ is surjective.

\smallskip Moreover, if $V=V_{\rm log}(M)$ is finite dimensional
as $\Q_p$-vector space then it is a semistable representation of
$G_K$ and $D_{\rm log}(V)\subseteq M$. The latter is an equality
if and only if $M$ is admissible.

\smallskip (2) The functors  $V^0_{\rm log}$ and $V^1_{\rm log}$ on the
category ${\cal MF}_{B_{\rm log}^{G_K}}(\varphi,N)$ are exact and
the morphism $\delta(M)$ is not an isomorphism if $M\neq 0$.
\end{proposition}

\begin{proof} (1) Let $\bigl(D,\varphi, N, \Fil^\bullet
D_K\bigr)$ be a filtered $(\varphi,N)$-module over $K$,
cf.~\ref{sec:notation}. As in \cite[\S 5.1 \&
5.2]{colmez_fontaine} we define $V_{\rm st}^0(D):=
\left(B_{\rm st}\otimes_{K_0} D \right)^{N=0,\varphi=1}$ and
$V_{\rm st}^1(D):=B_{\rm dR}\otimes_K D_K/\Fil^0\bigl(B_{\rm
dR}\otimes_K D_K \bigr)$. We let $\delta(D)\colon V_{\rm
st}^0(D)\lra V_{\rm st}^1(D)$ be the natural map.

First of all we claim that the proposition holds replacing the category  ${\cal MF}_{B_{\rm log}^{G_K}}(\varphi,N)$ with the category of  filtered
$(\varphi,N)$-modules over $K$ and $V_{\rm log}^i(M)$, $i= {\rm nothing},0,1$ with $V_{\rm st}^i(D)$. Indeed, it is proven in \cite[Prop. 4.5]{colmez_fontaine} that
the $\Q_p$-vector space $V_{\rm st}(D)$ is finite dimensional if and only if for every subobject $D'\subset D$ we have $t_H(D')\leq t_N(D')$ (these are the Hodge
and Newton numbers attached to $D'$, respectively). Moreover, it is also shown in loc.~cit.~that in this case $V_{\rm st}(D)$ is a semistable representation of
$G_K$ whose associated filtered $(\varphi,N)$-module is contained in $D$. It coincides with $D$ if and only if $\dim_{\Q_p} V_{\rm st}(D)=\dim_{K_0} D$. It follows
from the proof of \cite[Prop. 5.7]{colmez_fontaine} that, if $V_{\rm st}(D)$ is finite dimensional, then $\dim_{\Q_p} V_{\rm st}(D)=\dim_{K_0} D$ if and only if
$\delta(D)$ is surjective. The claim follows for filtered $(\varphi,N)$-modules over $K$.\smallskip

Since $B_{\rm log}^{N=0}=B_{\rm st}^{N=0}=B_{\rm cris}$, it follows that $V_{\rm st}^0(D)\cong V_{\rm log}^0\bigl(T(D)\bigr)$. Since  $V_{\rm
log}^1\bigl(T(D)\bigr)\cong \bigl(\overline{B}_{\rm log} \otimes_{K} D_K\bigr)/\Fil^0\bigl(\overline{B}_{\rm log} \otimes_{K} D_K\bigr)$ and ${\rm Gr}^\bullet
\overline{B}_{\rm log}={\rm Gr}^\bullet B_{\rm dR}$ by \ref{lemma:BlogmodPpi}, we deduce that the complexes $\delta(D)\colon V_{\rm st}^0(D)\lra V_{\rm st}^1(D)$
and $\delta\bigl(T(D)\bigr)\colon V_{\rm log}^1\bigl(T(D)\bigr) \lra V_{\rm log}^1\bigl(T(D)\bigr)$ are identified. Thus,  via the equivalence of categories $T$ of
\ref{prop:T}, claim (1) follows from its analogue for filtered $(\varphi,N)$-modules over $K$ discussed above. This concludes the proof of (1).\smallskip

To prove (2) it suffices to show the exactness of $V_{\rm st}^0$
and $V_{\rm st}^1$ and the fact that $\delta$ is not an
isomorphism for non zero objects on the category of  filtered
$(\varphi,N)$-modules over $K$. This is proven in \cite[Prop.~5.1
\& Prop.~5.2]{colmez_fontaine}.
\end{proof}

\subsubsection{Assumptions}\label{sec:assumptions}
Fix a positive integer $\alpha$. We assume that we are in one of the following two situations:

\smallskip

({\it ALG}) $\bigl(X,N\bigr)$ is a log scheme and $f\colon \bigl(X,N\bigr) \to \bigl(S,M\bigr)$ is a morphism of log schemes of finite type admitting a covering by
\'etale open subschemes $\Spec(R)\subset X$, by which we mean that $\Spec(R)\to X$ is an \'etale morphism, of the form:

$$\begin{array}{ccc}
\cO_K[P] & \stackrel{\psi_R}{\longrightarrow} & R \cr \uparrow & &
\uparrow\cr \cO_K[\N] &\stackrel{\psi_\alpha}{\longrightarrow} &
\cO_K,\cr
\end{array}$$where\enspace (i) $P:=P_a\times P_b$ with $P_a:=\N^a$ and
$P_b:=\N^b$, \enspace (ii) the left vertical map is the morphism of $\cO_K$-algebras
defined by the map on monoids $\N\to P=P_{a}\times P_b$ given by $n\mapsto
\bigl((n,\ldots,n),(0,\ldots,0)\bigr)$, \enspace (iii) $\psi_\alpha$
is the map of $\cO_K$-algebras  with $\N\ni 1\mapsto \pi^\alpha$.\smallskip

We require that the morphism $\cO_K[P]\otimes_{\cO_K[\N]} \cO_K\to R$ on associated spectra is \'etale, in the classical sense, and that the log structure on
$\Spec(R)$ induced by $(X,N)$ is the pullback of the fibred product log structure on $\Spec\bigl(\cO_K[P]\otimes_{\cO_K[\N]} \cO_K\bigr)$. We further assume that
for every subset $J_a\subset \{1,\ldots,a\}$ and every subset $J_b\subset \{1,\ldots,b\}$ the ideal in $R$ generated by $\psi_R\bigl(\N^{J_a}\times \N^{J_b}\bigr)$
defines an irreducible closed subscheme of $\Spec(R)$, that  the ideal of $R$ generated by $\psi_R\bigl(P_a\bigr)$ is not the unit ideal and that the image of the
monoid $\cO_\Kbar^\ast \cdot \psi_R\bigl(P_b\bigr )$ is saturated in $R\otimes_{\cO_K} \cO_\Kbar$.

\smallskip

({\it FORM}) for every $n\in\N$ we have a log scheme $\bigl(X_n,N_n\bigr)$  and a morphism of log schemes of finite type $f_n\colon \bigl(X_n,N_n\bigr) \to
\bigl(S_n,M_n\bigr)$ such that $\bigl(X_n,N_n\bigr)$ is isomorphic as log scheme over $(S_n,M_n)$ to the fibred product of $\bigl(X_{n+1},N_{n+1}\bigr)$ and
$\bigl(S_n,M_n\bigr)$ over $\bigl(S_{n+1},M_{n+1}\bigr)$. Write $X_{\rm form}$ for the formal scheme associated to the $X_n$'s. We require that \'etale locally on
$X_1$ the formal scheme $X_{\rm form}\to \Spf(\cO_K)$ is of the form

$$\begin{array}{ccc}
\cO_K/\pi^n \cO_K [P] & \stackrel{\psi_{R,n}}{\longrightarrow} &
R/\pi^n R \cr \uparrow & & \uparrow\cr \cO_K/\pi^n\cO_K[\N]
&\stackrel{\psi_\alpha}{\longrightarrow} & \cO_K/\pi^n\cO_K,\cr
\end{array}$$where the left vertical map and $\psi_\alpha$
are defined as in the algebraic case and $\psi_{R,n}$ induces a morphism $\cO_K[P]\otimes_{\cO_K[\N]} \cO_K/\pi^n \cO_K \to R/\pi^n R$ which is \'etale and the log
structure on $\Spec(R/\pi^nR)$ induced from $(X_n,N_n)$ is the pullback of the fibred product log structure on $\Spec\bigl(\cO_K[P]\otimes_{\cO_K[\N]} \cO_K/\pi^n
\cO_K\bigr)$. As in the algebraic case we require that for every subset $J_a\subset \{1,\ldots,a\}$ and every subset $J_b\subset \{1,\ldots,b\}$ the ideal of $R/\pi
R$ generated by $\psi_{R,1}\bigl(\N^{J_a}\times \N^{J_b}\bigr)$ defines an irreducible closed subscheme of $\Spec(R/\pi R)$, that the ideal of $R$ generated by
$\psi_R\bigl(P_a\bigr)$ is not the unit ideal and that the image of the monoid $\cO_\Kbar^\ast \cdot \psi_R\bigl(P_b\bigr )$ is saturated in $R\otimes_{\cO_K}
\cO_\Kbar$.
\smallskip

We deduce from \ref{lemma:structureofP'} that\smallskip

(i) in the algebraic, respectively ~in the formal setting,
$\bigl(X,N\bigr)$ (respectively ~$\bigl(X_n,N_n\bigr)$) is a fine and
saturated log scheme;

(ii)  $f$ (resp.~$f_n$) is a log smooth morphism. \smallskip

In the algebraic case, by abuse of notation we write $X$ for $(X,M)$. An object $U=\Spec(R)\in X^{\et}$ with induced log structure satisfying the requirements above
will be called {\em small}.
\smallskip

In the formal case we write  $X_{\rm rig}$ for the rigid analytic fibre of $X_{\form}$.
The inverse limit of the log structures $N_n$ defines a morphism of sheaves
of monoids from the inverse limit $N_{\form}={\ds \lim_{\infty\leftarrow n}}
N_n$ to $\cO_{X_{\form}}$. It coincides with the inverse image of $N_1$ via the
canonical map $\cO_{X_{\rm form}}\to \cO_{X_1}$. We call it the
{\it formal log structure} on $X_{\form}$. We also write $X$ or $(X,N)$ for the
inductive
system~$\left\{\bigl(X_n,N_n\bigr)\right\}_{n\in\N}$. It follows from our assumptions that
$X_{\rm form}$ is a noetherian and $\pi$-adic formal scheme. An \'etale open
$\Spf(R) \to X_{\rm form}$ satisfying the requirements above is called{\em  small}.
By assumption we have a covering of $X_{\rm form}$ by small objects. For any
such small affine $\Spf(R)$ of $X_{\rm form}$
we also have $\pi$-adic formally \'etale morphisms
$$\Spf\bigl(\cO_K[P]\widehat{\otimes}_{\cO_K[\N]} \cO_K\bigr) \stackrel{\hat{\psi}_R}{\longleftarrow}
\Spf(R) \longrightarrow X_{\form},$$where $\widehat{\otimes}$ stands for the $\pi$-adic completion of the tensor product, with the property that the formal log
structure $N_{\form}$ on $\Spf(R)$ is induced by the formal log structure  on the fibred product $\Spf\bigl(\cO_K[P]\widehat{\otimes}_{\cO_K[\N]} \cO_K\bigr)$. We
call any such diagram a {\it formal chart} of $\bigl(X_{\form},N_{\form}\bigr)$.

\medskip

{\em Example:} Assume that $X$ is a regular scheme with a normal crossing divisor $D\subset X$. Then \'etale locally on $X$ we can choose local charts $\psi_R$
satisfying the conditions above. For example, for every closed point of $X$ one can take $P$ and $\psi_R$ \'etale locally to be defined by a regular sequence of
elements generating the maximal ideal  at $x$ so that $D$ is defined by part of such a sequence. In this case also the ideal generated by $\psi_R(P_b)$ in $R$ is
not the unit ideal and the conditions above are satisfied.

\subsubsection{Continuous sheaves}\label{section:continuoussheaves}
Given an abelian category $\cA$ admitting enough injectves we
consider the category $\cA^\N$ of inverse systems of objects of
$\cA$ indexed by $\N$. It is also abelian with enough injectives.
Given a left exact functor $F$ from $\cA$ to an abelian
category ${\mathcal B}$ we  have a left exact functor $F^\N\colon
\cA^\N \to {\mathcal B}^\N$ sending $(C_n)_{n\in\N}\mapsto
\big(F(C_n)\big)_{n\in\N}$ and its $i$-th derived functor ${\rm
R}^i \bigl(F^\N\bigr)$ is canonically $\bigl({\rm R}^i
F\bigr)^\N$. If projective limits exist in ${\mathcal B}$, one can
derive the functor $F^{\rm cont}\colon \cA^\N \to {\mathcal B}$
sending $(C_n)_{n\in\N} \mapsto {\ds \lim_{\infty\leftarrow n}}
F(C_n)$. We refer \cite[\S5.1]{andreatta_iovita} for details.

We also consider the category ${\rm Ind}\bigl(\cA\bigr)$ of
inductive systems of objects in $\cA$ indexed by $\Z$,
i.e.~$\bigl(A_h,\gamma_h\bigr)_{h\in\Z}$ with $\gamma_h\colon A_h
\to A_{h+1}$. Consider a non decreasing function $\alpha\colon
\Z\to \Z$. Given objects
$\underline{A}:=\bigl(A_i,\gamma_i\bigr)_{i\in\Z}$ and $
\underline{B}:=\bigl(B_j,\delta_j\bigr)_{j\in\Z}$ we define a
morphism $f\colon \underline{A}\to \underline{B}$ of type $\alpha$
to be a collection of morphisms $f_i\colon A_i\to B_{\alpha(i)}$
such that $f_{i+1}\circ \gamma_i= \prod_{\alpha(i) \leq j <
\alpha(i+1)} \delta_j \circ f_i $. We denote by
$\Hom^\alpha\bigl(\underline{A},\underline{B}\bigr)  $ the group
of homomorphisms of type $\alpha$. We say that two morphisms $f$
and $g$ of type $\alpha$ (resp.~$\beta$) are equivalent if there
exists $N\in\N$ such that $f_i$  composed with $B_{\alpha(i)}\to
B_{\max(\alpha(i),\beta(i))+N}$ and $g_i$ composed with
$B_{\beta(i)}\to B_{\max(\alpha(i),\beta(i))+N}$ coincide. One
checks that this defines an equivalence relation. We define a
morphism $\underline{A}\to \underline{B} $ in ${\rm
Ind}\bigl(\cA\bigr)$ to be a class of morphisms with respect to
this equivalence relation.

One can prove that ${\rm Ind}\bigl(\cA\bigr)$ is an abelian category.
If ${\mathcal B}$ admits inductive limits and $F\colon \cA \to {\mathcal B}$ is a left exact
functor, we define ${\rm R}^i F^{\rm cont}\colon {\rm Ind}\bigl(\cA^\N\bigr) \to
{\mathcal B}$ by ${\rm R}^i F^{\rm cont}\bigl((A_h,\gamma_h)_h\bigr):=\ds
\lim_{h\to \infty} {\rm R}^i F^{\rm cont}(A_h) $.
Then the family $\{{\rm R}^n F^{\rm cont}\}_n$ defines a cohomological
$\delta$-functor on ${\rm Ind}\bigl(\cA\bigr)$.

\smallskip

\subsection{Faltings' topos}

\subsubsection{The Kummer \'etale site of $X$}\label{sec:ketonX}
The notations are as in the previous section. Both in the
algebraic and in the formal case we write $X^{\ket}$ for the
Kummer \'etale site of $(X,N)$.

In the algebraic case the category is the full subcategory of
log schemes endowed with a  Kummer \'etale morphism $(Y,N_Y)\to
(X,N)$ in the sense of \cite[\S 2.1]{IllusieKummer},
i.e.~morphisms which are log \'etale and Kummer or equivalently
log \'etale and exact. The coverings are collections of Kummer
\'etale morphism $(Y_i,N_i)\to (X,N)$ such that $X$ is set
theoretically the union of the images of the $Y_i$'s. One verifies
that this defines a site; see loc.~cit.

In the formal case the objects are  Kummer \'etale morphisms $\left\{g_n\colon \bigl(Y_n,N_{Y_n}\bigr)
\to (X_n,N_n)\right\}_{n\in\N}$ such that $g_n$ is the
base change of $g_{n+1}$ via $(X_n,N_n)\to (X_{n+1},N_{n+1})$ for every $n\in\N$.
We simply write $g\colon (Y,N_Y)\to (X,N_X)$ for such inductive system of morphisms.
The morphisms from an object $(Y,N_Y)\to (X,N)$ to an object $(Z,N_Z):=
\left\{h_n\colon \bigl(Z_n,N_{Z_n}\bigr)\to (X_n,N_n)\right\}_{n\in\N}$  are collections of
morphisms $\left\{t_n\colon \bigl(Y_n,N_{Y_n}\bigr) \to \bigl(Z_n,N_{Z_n}\bigr)
\right\}_{n\in\N}$ as log schemes over $(X_n,N_n)$ such that $t_n$ is the base change
of $t_{n+1}$ via $(X_n,N_n)\to (X_{n+1},N_{n+1})$ for every $n\in\N$.
We simply write $t\colon (Y,N_Y)\to (Z,N_Z)$ for such an inductive system of morphisms. The
coverings are collections of Kummer \'etale morphisms $\left\{(Y_i,N_i)\to (X,N)\right\}_i$
such that $X_1$ is the  set theoretic  union of the images of the
$Y_{i,1}$'s. This defines a site. Due to the characterization of log \'etale morphisms in
\cite[prop 3.14]{katolog} the natural forgetful morphism of sites
$X^{\ket} \lra X_1^{\ket}$, sending $g\colon (Y,N_Y)\to (X,N)$ to $g_1\colon \bigl(Y_1,N_{Y_1}\bigr)
\to (X_1,N_1)$, is an equivalence of categories.

\medskip

\begin{lemma}\label{lemma:Ukfetisfppf} Let $(Y,H)\in X^{\ket}$.
Then,\smallskip

(1) $Y$ (resp.~$\Spec(R)$ if $Y_{\rm form}=\Spf(R)$ in the formal
case) are Cohen-Macaulay and normal schemes;\smallskip

(2) $(Y,H)$ (resp.~$\bigl(\Spec(R),H_{\rm form}\bigr)$ if $Y_{\rm form}=\Spf(R)$ in the formal case) are log regular in the sense of\/ \cite[Def.~2.1]{katotoric}.
\end{lemma}
\begin{proof} We provide a proof in the algebraic case.  Since $f\colon (Y,H)\to (X,N)$ is Kummer
\'etale, in particular it is log \'etale. Since $f\colon \bigl(X,N\bigr) \to \bigl(S,M\bigr)$ is log smooth the composite $(Y,H)\to \bigl(S,M\bigr)$ is log smooth.
Recall that $\bigl(S,M\bigr)$ is $\Spec(\cO_K)$ with the log structure defined by its maximal ideal. In particular it is log regular. Arguing as in \cite[Lemma
1.5.1]{tsujiinventiones} we deduce from \cite[Thm.~8.2]{katotoric} that also $(Y,H)$ is log regular. Due to \cite[Thm.~4.1]{katotoric} the scheme $Y$ is then
Cohen-Macaulay and normal. This proves the claims in the algebraic case.

For the proof in the formal case we make some preliminary remarks in the algebraic case. Let $y\in Y$ and set $x$ to be its image in $X$. Write $H_y$ and $N_x$ for
the stalk of the sheaves of monoids $H$ and $N$ and put $\overline{H}_y:=H_y/\cO_{Y,y}^\ast$ and $\overline{N}_x:=H_y/\cO_{X,x}^\ast$.  Since the log structures are
fine, $\overline{H}_y$ and $\overline{N}_x$ are finitely generated and we have inclusions $\overline{H}_y\subset \overline{H}_y^{\rm gp}$ and $\overline{N}_x\subset
\overline{N}_x^{\rm gp}$. The morphism $(Y,H)\to (X,N)$ being Kummer \'etale, the induced map $\iota\colon \overline{N}_x \to \overline{H}_y$ is injective and there
exists an integer $n$ invertible in $\cO_{Y,y}$ such that $n \overline{H}_y\subset \overline{N}_x$.  Since $\overline{N}_x^{\rm gp}$ is a finite and free
$\Z$-module we can find a splitting of the group homomorphism $N_x^{\rm gp} \to \overline{N}_x^{\rm gp}$ which composed with the inclusion $\overline{N}_x \subset
\overline{N}_x^{\rm gp}$ provides a chart $P \to N$ in a neighborhood $U_x$ of $x$ cf.~\cite[Lemma 2.10]{katolog}.   Proceeding similarly with $\overline{H}_y^{\rm
gp} $ we can find a splitting of $H_y^{\rm gp} \lra \overline{H}_y^{\rm gp}$. Since the  local ring $\cO_{Y,y}$ is taken with respect to the \'etale topology and
$n$ is invertible in $\cO_{Y,y}$, the group $\cO_{Y,y}^\ast$ is $n$-divisible and we can take the splitting compatible with the first splitting of $N_x^{\rm gp} \to
\overline{N}_x^{\rm gp}$. Composing with the inclusion $\overline{H}_y \subset \overline{H}_y^{\rm gp}$ we get a chart $Q\to H$ in a neighborhood $V_y$ of $y$
compatible with $P \to N$ via the map of sheaves $f^{-1}(N) \to H$.

To check that $R$ is Cohen-Macaulay in the formal case it suffices to prove that the complete local ring of $R$ at every maximal ideal $y$ is Cohen-Macaulay at the
image $x$ of $y$ in $X$. To prove that it is normal it further suffices to show that $R$ is regular in codimension $1$. Due to the assumptions and the proof in the
algebraic case, (1) and (2) hold if $\Spf(R)$ is a formal chart of $(X,N)$, i.e.~$f$ is the identity map. In the general case, using the considerations above,
we have
$$\widehat{\cO}_{Y,y}\cong \widehat{\cO}_{X,x}\widehat{\otimes}_{\Z[P]} \Z[Q]$$where $P\to Q$ is a morphism of monoids as above. By the construction of the chart
$P$, we have that $P^\ast=\{1\}$. We conclude from \cite[Thm 3.2]{katotoric} that $\widehat{\cO}_{X,x}\cong R[\![P]\!][\![T_1,\ldots,T_r]\!]/(\theta)$ for
$R=\WW\bigl(k(x)\bigr)$ and $\theta\equiv p$ modulo the ideal $\bigl(P\backslash\{1\},T_1, \ldots,T_r\bigr)$. Then, $\widehat{\cO}_{Y,y}\cong
R[\![Q]\!][\![T_1,\ldots,T_r]\!]/(\theta) $. Since  $Q$ is saturated and $Q^{\ast}=\{1\}$ by construction, also $\widehat{\cO}_{Y,y}$ is of the same form. The proof
of \cite[Thm.~4.1]{katotoric} applies to deduce that $\widehat{\cO}_{Y,y}$ is Cohen-Macaulay and regular in codimension $1$. This concludes the proof of (1) and (2)
in the formal case as well.

\end{proof}

In the algebraic case consider the presheaves $\cO_{X^{\ket}}$ and
$N_{X^{\ket}}$ respectively defined by $$X^{\ket}\ni (U,N_U) \lra
\Gamma\bigl(U,\cO_U\bigr), \qquad X^{\ket}\ni (U,N_U) \lra
\Gamma\bigl(U,N_U\bigr).$$Similarly, in the formal case for every
$h\in\N$ define the presheaves $\cO_{X^{\ket}_h}$ and
$N_{X^{\ket}_h}$
$$X^{\ket}\ni \bigl(U_n,N_{U_n}\bigr)_n \lra
\Gamma\bigl(U_h,\cO_{U_h}\bigr), \qquad X^{\ket}\ni
\bigl(U_n,N_{U_n}\bigr)_n \lra \Gamma\bigl(U_h,N_{U_h}\bigr).$$We
write $\cO_{X^{\ket}_{\form}}$  and $N_{X^{\ket}_{\form}}$ for the
presheaves defined as $\ds \lim_{\infty\leftarrow n}
\cO_{X^{\ket}_{\form}}$ and $\ds \lim_{\infty\leftarrow n}
N_{X^{\ket}_{\form}}$ respectively.

\begin{proposition}\label{Prop:OXisasheafonket} (1) In the algebraic case
the presheaves $\cO_{X^{\ket}}$, $\cO_{X^{\ket}}^\ast$   and
$N_{X^{\ket}}$ are sheaves and $N_{X^{\ket}} \to \cO_{X^{\ket}}$
is a morphism of sheaves of multiplicative monoids such that the
inverse image of $\cO_{X^{\ket}}^\ast$ is identified with
$\cO_{X^{\ket}}^\ast$.
\smallskip

(2) In the formal case the presheaves $\cO_{X^{\ket}_h}$,
$\cO_{X^{\ket}_h}^\ast$ and $N_{X^{\ket}_h}$ for every $h\in\N$
and the presheaves $\cO_{X^{\ket}_{\form}}$,
$\cO_{X^{\ket}_{\form}}^\ast$ and $N_{X^{\ket}_{\form}}$ are
sheaves. Moreover, $N_{X^{\ket}_h} \to \cO_{X^{\ket}_h}$ and
$N_{X^{\ket}_{\form}}\to \cO_{X^{\ket}_{\form}}$ is a morphism of
sheaves of multiplicative monoids such that the inverse image of
$\cO_{X^{\ket}_h}^\ast$ (resp.~$\cO_{X^{\ket}_{\form}}^\ast$) is
identified with $\cO_{X^{\ket}_h}^\ast$
(resp.~$\cO_{X^{\ket}_{\form}}^\ast$).
\end{proposition}
\begin{proof}
An unpublished result of K.~Kato implies that the Kummer \'etale
topology is coarser than the canonical topology. This implies the
claims that the given presheaves are sheaves, see \cite[\S
2.7(a)\&(b)]{IllusieKummer}. The other properties are clear.

\end{proof}

\subsubsection{The finite Kummer \'etale sites $U_L^{\fket}$}

Let $U\in X^{\ket}$ and let $K\subset L \subset \Kbar$. In the algebraic case we let $U_L^{\fket}$ be the site of {\it finite Kummer \'etale covers} of $U_L$
endowed with the log structure defined by $N$; see \cite[Def.~3.1]{IllusieKummer}. As remarked in \cite[Rmk.~3.11]{IllusieKummer}  a Kummer \'etale map $Y\to U_L$,
inducing a finite and surjective morphism at the level of underlying schemes, is a Kummer \'etale cover. Viceversa \cite[Cor.~3.10 \& Prop~3.12]{IllusieKummer}
implies that any Kummer \'etale cover $Y\to U_L$ is Kummer \'etale and induces a finite and surjective morphism on the underlying schemes.  \smallskip

In the formal case we proceed differently. If $K\subset L$ is a finite extension, let $U_L$ be the rigid analytic space associated to
$U_{\form}\widehat{\otimes}_{\cO_K} \cO_L$ and let $U_L^{\fket}$ be the site whose objects consist of finite surjective morphisms $W \to U_L$ of $L$--rigid analytic
spaces such that\smallskip

(1) $W$ is smooth over $L$; \smallskip

(2) for every formal chart
$$\Spf\bigl(\cO_K[P]\widehat\otimes_{\cO_K[\N]} \cO_K\bigr)
\stackrel{\hat{\psi}_R}{\longleftarrow} \Spf(R) \longrightarrow U_{\form},$$
the induced morphism $W\times_{U_L} \Spm\bigl(R\otimes_{\cO_K} L\bigr) \to
\Spm\bigl(\cO_K[P]\widehat\otimes_{\cO_K[\N]} L\bigr)$ defines a finite and
\'etale morphism of rigid analytic spaces over the open subspace of
$\Spm\bigl(L\{P\}\bigr)$ given by $\Spm\bigl(L\{P^{\rm gp}\}\bigr)$.\smallskip

The morphisms are morphisms as rigid analytic spaces over $U_L$.
The coverings are collections of morphisms $W_i \to W$, for $i\in
I$, whose images cover $W$ set theoretically.

\begin{remark}\label{rmk:fketalgebraic} (i) If $W \to U_L$
is a finite morphism of rigid analytic spaces, then for every
formal chart of $U$ the map $$\rho\colon W\times_{U_L}
\Spm\bigl(R\otimes_{\cO_K} L\bigr)\to \Spm\bigl(R\otimes_{\cO_K} L
\bigr)$$is finite by \cite[Th.~III.6.2]{Fresnel_VanderPut} so that
it is of the form $\Spm(B)\to \Spm\bigl(R\otimes_{\cO_K} L \bigr)$
for a $R\otimes_{\cO_K} L$-algebra $B$ which is finite as a
$R\otimes_{\cO_K} L$-module. Then, $\rho$ is finite and \'etale
over  $\Spm\bigl(L\{P^{\rm gp}\}\bigr)$ if and only if
$R\otimes_{\cO_K}L\bigl\{P^{\rm gp}\bigr\} \to B\bigl\{P^{\rm
gp}\bigr\}$ is a finite and separable extension of algebras. Since
this condition can be checked on $\Kbar$-points, this holds if and
only if $R\otimes_{\cO_K} L\bigl[P^{\rm gp}\bigr] \to
B\bigl[P^{\rm gp}\bigr]$ is finite and \'etale in the usual
sense.\smallskip

(ii) Let  $W \to U_L$ be a finite morphism of $L$-rigid analytic
spaces with $W$ smooth over $L$. Then, condition (2) holds  if and
only if there exist formal charts of $U_{\form}$ which cover
$U_{\form}$ and for which condition (2) holds.\smallskip

(iii) We remark that the definition in the algebraic case coincides with the one provided by the analogues of requirements (1) and (2). Indeed, given $U\in
X^{\ket}$ and $W\to U_L$ a Kummer \'etale cover, $W\to U_L$ is Kummer \'etale. Thus,  $W\to \Spec(L)$ is log smooth, and in fact smooth as the log structure on $L$
is trivial. The analogue of condition (2) holds thanks to \cite[Prop.~3.8]{katolog}. Viceversa assume that $W\to U_L$ is  a finite surjective morphism satisfying
conditions (1) and (2). Let $\iota\colon U_L^o\hookrightarrow U_L$ be the locus of triviality of the log structure and let $j\colon W^o \hookrightarrow W$ be its
inverse image in $W$. As $U_L$ is log regular, see  \ref{lemma:Ukfetisfppf}, the log structure on $U_L$ is defined by $\cO_{U_L}\cap \iota_\ast(\cO_{U_L^o}) \subset
\cO_{U_L}$ thanks to \cite[11.6]{katotoric}. As $W$ is smooth $\cO_W\cap j_\ast(\cO_{W^o}) \subset \cO_W$ defines a fine and saturated log structure on $W$,
cf.~\cite[\S 1.7]{IllusieKummer}. Using this log structure we get a map of log schemes $W\to U_L$ and, as $W\to U_L$ is finite and surjective, it is exact and log
\'etale, i.e., Kummer \'etale.
\end{remark}

Given a finite extension $K\subset L \subset L'\subset \Kbar$ the base change from $L$ to $L'$ provides a morphism of sites $U_L^{\fket} \to U_{L'}^{\fket}$. For
arbitrary extensions $K\subset L \subset \Kbar$, we then get a fibred site $U_\ast^{\fket}$ over the category of finite extensions of $K$ contained in $L$ in the
sense of \cite[\S VI.7.2.1]{SGAIV}.  We let $U_L^{\fket}$ be the site defined by the projective limit of the fibred site $U_\ast^{\fket}$; see \cite[Def.
VI.8.2.5]{SGAIV}.

\begin{remark} For example, one has the following explicit description of $U_\Kbar^{\fket}$. The objects in
$U_\Kbar^{\fket}$ consist of pairs $(\cW,L)$ where~$L$ is a finite extension of\/~$K$ contained in~$\Kbar$ and~$\cW\in U_L^{\fket}$. Given~$(\cW,L)$ and~$(\cW',L')$
define $\Hom_{U_\Kbar^{\fket}}\bigl((\cW',L'), (\cW,L)\bigr)$ to be the direct limit $\ds \lim_\rightarrow \Hom_{L''}\bigl(\cW'\otimes_{L'} L'', \cW\otimes_L
L''\bigr)$ over all finite extensions~$L''$ of~$K$, contained in $\Kbar$ and containing both~$L$ and~$L'$, of the morphisms $\cW'\otimes_{L'} L''\to \cW\otimes_L
L''$ as rigid analytic spaces over~$U_{L''}$.\end{remark}

\subsubsection{Faltings' site}
\label{sec:faltingssite}

 Let $K\subset L \subset \Kbar$ be any extension. Let $E_{X_L}$ be the
category defined as follows
\smallskip

i) the objects consist of pairs $\bigl(U, W\bigr)$ such that $U\in
X^{\ket}$ and $W\in U_L^{\fket}$; \smallskip

ii) a morphism $(U',W')\lra (U,W)$ in $E_{X_L}$ consists of a pair
$(\alpha,\beta)$, where $\alpha\colon U'\lra U$ is a morphism in
$X^{\ket}$  and $\beta\colon W'\lra W\times_{U_K} U'_K$ is a
morphism in $U_L^{',\fket}$.

\bigskip

The pair $(X,X_L)$ is a final object in $E_{X_L}$. Moreover, finite projective limits are representable in $E_{X_L}$ and, in particular, fibred products exist: the
fibred product of the objects $(U',W')$ and $(U'',W'')$ over $(U,W)$ is $\bigl(U'\times_U U'',W'\times_W W''\bigr)$ where $W'\times_W W''$ is the fibred product of
the base-changes of $W'$ and $W''$ to $\bigl(U'\times_U U''\bigr)_L^{\fket}$ over the base-change of $W$ to $\bigl(U'\times_U U''\bigr)_L^{\fket}$. See \cite[Prop.
2.6]{erratum}.

\bigskip
We say that a family $\{(U_i,W_i)\lra (U,W)\}_{i\in I}$ is a
covering family if either\smallskip

\noindent $\alpha$) $\{U_i\lra U\}_{i\in I}$ is a covering in
$X^{\fket}$ and $W_i\cong W\times_{U_K} U_{i,K}$ for every $i\in
I$. \smallskip

or\smallskip

\noindent $\beta$) $U_i\cong U$ for all $i\in I$ and $\{W_i\lra
W\}_{i\in I}$ is a covering in $U_L^{\fket}$.

\smallskip
We endow $E_{X_L}$ with the topology ${\rm T}_{X_L}$ generated by
the covering families described above and denote by $\fX_L$ the
associated site. We call ${\rm T}_{X_L}$ Faltings' topology and
$\fX_L$ Faltings' site associated to $X$. As in \cite[Lemma
2.8]{erratum} one proves that the so called  strict coverings of
$(U,W)$ (see definition \ref{def:striccoverings} below)
are cofinal in the collection of all covering families of
$(U,W)$. \smallskip

\begin{definition}\label{def:striccoverings} A family $\{(U_{ij},W_{ij})\lra (U,W)\}_{i\in I,j\in J}$ of
morphisms in $E_{X_L}$ is called a {\em strict covering family}
if\smallskip

\noindent a) For each $i\in I$ and for every $j\in J$ we have an object $U_i\in X^{\ket}$ and isomorphisms $U_i\cong U_{ij}$ in $X^{\fket}$.
\smallskip

\noindent b) $\{U_i\lra U\}_{i\in I}$ is a covering in
$X^{\fket}$.\smallskip

\noindent c) For every $i\in I$ the family $\{W_{ij}\lra
W\times_{U_K} U_{i,K}\}_{j\in J}$ is a covering in
$U_{L,i}^{\fket}$.
\end{definition}

This is not Faltings' original definition of the site given in
\cite{faltingsAsterisque}. We refer to
\cite{andreatta_iovita_comparison} for a discussion of the
differences between the two approaches and motivations for our
definition.

\subsubsection{Continuous Functors}
\label{sec:continuousfunctors} For $K\subset L \subset \Kbar$ we
let
$$v_{X,L}\colon X^{\ket}\lra \fX_L, \qquad z_{X,L}\colon X^{\et}\lra \fX_L$$be given by
$v_{X,L}(U):=\bigl(U,U_L\bigr)$ in the algebraic case and  by $v_{X,L}(U):= \bigl(U,U_K\bigr)$, viewing $U_K$ as an object of $U_{\Kbar}^{\fket}$, in the formal
case and similarly for $z_{X,L}$. We simply write $v_L$ and $z_L$.
\smallskip

\noindent Define $$\beta\colon \fX_K\lra  \fX_\Kbar$$by
$\beta(U,W)=\bigl(U,W\otimes_{K} \Kbar \bigr)$ (resp.~$\beta(U,W)$
equal to $\bigl(U,W\bigr)$ viewed in $\fX_{\Kbar}$) in the
algebraic (resp.~formal) setting. \smallskip

Assume we are in the algebraic case. Let $\widehat{X}$ be the $p$-adic formal scheme associated to $X$ and denote by $\mathfrak{\widehat{X}}_L$ Faltings' site
associated to the formal log scheme $\widehat{X}$. We then have a morphism
$$\gamma_L\colon \fX_L\lra \mathfrak{\widehat{X}}_L,$$sending $(U,W)$
to $\bigl(\widehat{U},W\vert_{\widehat{U}_L}\bigr)$. Here
$W\vert_{\widehat{U}_L}$ is defined as follows. Let $K\subset M$
be a finite extension, contained in $L$, where $W\to U_L$ is
defined. Let $W^{\rm an}\to U_M^{\rm an}$ be the associated finite
Kummer \'etale morphism of analytic spaces. Then
$W\vert_{\widehat{U}_L}$ is defined by restricting it to the open
immersion $\widehat{U}_M\subset U_M^{\rm an}$. We simply write
$\gamma$ if there is no confusion.

\bigskip

\noindent It is clear that the above functors send covering
families to covering families and commute with fiber products.  In
particular they define continuous functors of sites
by~\cite[Prop.~III.1.6]{SGAIV}. They also send final objects to
final objects so that they induce morphisms of the associated
topoi of sheaves.

\begin{remark}\label{rmk:alternative} We provide an alternative presentation of the
morphisms above for an arbitrary extension $K\subset L \subset \Kbar$.

For every finite extension $K\subset M$ in $L$, let $\mathfrak{\widehat{X}}_M$
(resp.~$\fX_M$) be Faltings' site associated to $\widehat{X}$ (resp.~$X$) over $M$.
Let $\gamma_M\colon \fX_M\to \mathfrak{\widehat{X}}_M$ be the morphism defined in
\ref{sec:continuousfunctors}. Given finite extensions $M\subset M'$ of $K$ in $L$
we have a natural morphism of sites $\hat{u}_{M',M}\colon \mathfrak{\widehat{X}}_M \to
\mathfrak{\widehat{X}}_{M'}$ (resp.~$u_{M',M}\colon\fX_M\to \fX_{M'}$) given
by $(U,W)\mapsto \bigl(U, W\otimes_M M'\bigr)$. Moreover, we have $\gamma_{M'}\circ u_{M',M} = \hat{u}_{M',M}\circ \gamma_M$.

Let $I_L$ be the category opposed to the category of finite extensions of $K$ contained in $L$. Then
$\mathfrak{\widehat{X}}_\bullet$ (resp.~$\fX_\bullet$) are
fibred sites over $I_L$ via the morphisms $\hat{u}$ (resp.~$u$) and $\gamma_\bullet\colon
\fX_\bullet\to \mathfrak{\widehat{X}}_\bullet$ defines a coherent morphism
of fibred sites; cf.~\cite[\S VI.7.2.1]{SGAIV}.  Then $\fX_L$ and $\mathfrak{\widehat{X}}_L$
are isomorphic to the projective limit site of  $\fX_\bullet$ and
$\mathfrak{\widehat{X}}_\bullet$ and $\gamma_L$ is induced by $\gamma_\bullet$; see \cite[Def. VI.8.2.5]{SGAIV}.

\end{remark}

\subsubsection{Geometric points} Following
\cite[Def.~4.1]{IllusieKummer} we define a log geometric point $s$ to be the spectrum of an algebraically closed field $k$ with log structure $M_s$ such that
multiplication by $n$ on $M_s/k^\ast$ is a bijection for every integer $n$ prime to the characteristic of $k$. A  log geometric point of $(X,N)$  is  a map of log
schemes from a log geometric point to $(X,N)$. For any such point $x \to (X,N)$, we let $\bigl(X_{x},N_{x}\bigr)$ be the log strict localization of $X$ at $x$ as in
\cite[\S 4.5]{IllusieKummer}: by definition it is the log strictly local log scheme defined as the inverse limit of $\bigl(U,N_U\bigr)$
(resp.~$\bigl(U_{\form},N_{\form})$) over the Kummer \'etale neighborhoods $U$ of $x$.

For a field extension $K\subset L$ in $\Kbar$ we define a geometric point of $\fX_L$ to be a pair $(x,y)$ where  $x$ is a log geometric point of $X$ and $y$ is a
log geometric point of $\bigl(X_{x},N_{x}\bigr)$ over $L$.

Given a presheaf $\cF$ on $\fX_L$ we define the stalk $\cF_{(x,y)}$ of $\cF$ on $\fX$ to be the direct limit $\lim \cF(U,W)$ over all pairs
$\bigl((U,x'),(W,y')\bigr)$ where $U$ is affine, $x'$ is a log geometric point of $U$ mapping to $x$ and $y'$ is a log geometric point of $W$ specializing to $x'$
and mapping to $y$. As in \cite[Prop. 3.4]{erratum} on proves that there are enough geometric points in $\fX_L$, i.e.~that a sequence of sheaves is exact if an only
if the induced sequence on stalks is exact for all geometric points $(x,y)$.

\subsubsection{The localization functors.}\label{sec:localization}
Let $U$ be a small connected affine object of $X^{\ket}$ and write
$U=\Spec(R_U)$ in the algebraic case and $U_{\form}:=\Spf(R_U)$ in
the formal case. Let $N_U$ be the induced log structure
($N_{U_{\rm form}}$ in the formal case).

Recall that $R_U$ is an integral domain. Let $\C_U$ be an algebraic closure of ${\rm Frac}(R_U)$ and let $\C_U^{\rm log}=(\C_U,N_\C)$  be a log geometric point of
$\bigl(\Spec(R_U),N_U\bigr)$ over $\C_U$. Let $\cG_{U_K}$ be the Kummer \'etale Galois group $\pi_1^{\log}\bigl(\Spec\big(R_{U}[p^{-1}]\big),\C_U^{\rm log}\bigr)$,
see \cite[\S 4.5]{IllusieKummer}, classifying Kummer \'etale covers of $\Spec\big(R_{U}[p^{-1}]\big)$. It follows from \ref{rmk:fketalgebraic} that  both in the
algebraic case and in the formal case the category $U_K^{\fket}$ is equivalent to the category of finite sets with continuous action of $\cG_{U_K}$. Write
$\bigl(\Rbar_U,\Nbar_U\bigr)$ for the direct limit of all the finite normal extensions $R_U \subset S$, all log structures $N_S$ on $\Spec(S_K)$ and all  maps
$(R_{U,K},N_{U,K})\to (S_K,N_S) \to (\C_U,N_\C)$ such that $(R_{U,K},N_{U,K})\to (S_K,N_S)$ is finite Kummer \'etale. Then we have an equivalence of categories
$$\Sh\bigl(U_K^{\fket}\bigr)\longrightarrow \Rep\bigl(\cG_{\cU_K}
\bigr),$$from the category of sheaves of abelian groups on $ U_K^{\fket}$ to the category of  discrete abelian groups with continuous action of $\cG_{U_K}$, defined
by $\cF \mapsto {\ds \lim_{\to}} \cF\bigl((S_K,N_S)\bigr)$. Composing with the restriction
$$\Sh\bigl(\fX_K\bigr)\longrightarrow \Sh\bigl(U_K^{\fket}\bigr)\longrightarrow  \Rep\bigl(\cG_{\cU_K}
\bigr)$$we obtain a functor which we simply write as $\cF\mapsto
\cF\bigl(\Rbar_U,\Nbar_U\bigr)$, called {\em localization
functor}. We also write
$$\Sh\bigl(\fX_K\bigr)^\N\longrightarrow   \Rep\bigl(\cG_{\cU_K}
\bigr),\qquad \cF=\bigl(\cF_n\bigr)_n\mapsto
\cF\bigl(\Rbar_U,\Nbar_U\bigr):={\ds \lim_{\infty \leftarrow n}}
\cF_n\bigl(\Rbar_U,\Nbar_U\bigr).$$

More generally we fix an extension $K\subset L \subset \Kbar$. Write $R_{U}\otimes_{\cO_K} L:=\prod_{i=1}^n R_{U,i}$ with $\Spec\bigl(R_{U,i}\bigr)$ connected and
let $N_{U,i}$ be the induced log structure. Fix a log geometric generic point $\bareta_i=\C_{U,i}^{\log}$ of $\bigl(\Spec\bigl(R_{U,i}\bigr),N_{U,i}\bigr)$ over
$\C_U$. Write $\bigl(\Rbar_{U,i},\Nbar_{U,i}\bigr)$ for the direct limit of all finite normal extensions $R_{U,i} \subset S$  taken over all morphisms
$(R_{U,i},N_{U,i})\to (S,N_S) \to (\C_{U,i},N_\C)$ such that $(R_{U,i},N_{U,i})\to (S,N_S)$ is finite Kummer \'etale. We let $\cG_{U_L,i}$ be the Galois group of
$R_{U,i}\subset \Rbar_{U,i}$. Eventually, put $\Rbar_U:=\prod_{i=1}^n \Rbar_{U,i}$ and $\Nbar_U:=\prod_{i=1}^n \Nbar_{U,i}$ and
$$\cG_{U_L}:=\prod_{i=1}^n \cG_{U_L,i}.$$For later purposes for
$L=\Kbar$ and for every $i$ write $\bigl(R_{U,\infty,i},
\Nbar_{U,\infty,i}\bigr)$ as the direct limit of the  Kummer
\'etale covers $(R_{U,i},N_{U,i}) \to (S,N_S)$ (mapping to
$(\C_{U,i},N_\C)$) of the form $S=R_{U,i}\otimes_{\Kbar[N_{U,i}]}
\Kbar\bigl[\frac{1}{n!} N_{U,i}\bigr]$ for varying $n\in\N$. We
let $R_{U,\infty}:=\prod_{i=1}^n R_{U,\infty,i}$ and
$\Nbar_{U,\infty}:=\prod_{i=1}^n \Nbar_{U,\infty,i}$. Let
$\cH_{U_\Kbar,i} $ be the group of automorphisms of $\Rbar_{U,i}$
as $R_{U,\infty,i}$-algebra. Let
$$\cH_{U_\Kbar}:=\prod_{i=1}^n \cH_{U_\Kbar,i}.$$

Let $\Rep\bigl(\cG_{U_L}\bigr)$ (resp.~$\Rep\bigl(\cG_{U_L}\bigr)^\N$) be the category of
discrete abelian groups (resp.~the category of inverse systems of finite
abelian groups indexed by~$\N$) with continuous action of $\cG_{U_L}$.
It follows from \ref{rmk:fketalgebraic} and \cite[\S 4.5]{IllusieKummer} that it is
equivalent to the category of sheaves (resp.~projective limits of sheaves) on
$U_L^{\fket}$. As before we have natural functors called {\em localization functors}
$$\Sh\bigl(\fX_L\bigr)\longrightarrow \Rep\bigl(\cG_{U_L}
\bigr)\qquad \hbox{{\rm and}}\qquad
\Sh\bigl(\fX_L\bigr)^\N\longrightarrow
\Rep\bigl(\cG_{U_L}\bigr)^\N$$defined as follows. If $\cG\in
\Sh\bigl(\fX_L\bigr)$ is a sheaf of abelian groups its
localization is $\ds
\cG\bigl(\Rbar_U,\Nbar_{U}\bigr):=\oplus_{i=1}^n \ds
\cG\bigl(\Rbar_{U,i},\Nbar_{U,i}\bigr)$ where
$\cG\bigl(\Rbar_{U,i},\Nbar_{U,i}\bigr):=\ds
\lim_{\to}\cG\bigl(U,(\Spec(S),N_S)\bigr)$ over all
$(R_{U,i},N_{U,i})\to (S,N_S)\subset
\bigl(\Rbar_{U,i},\Nbar_{U,i}\bigr)$ as before.

\subsubsection{The computation
of $\R^i v_{\ast}^{\rm
cont}$}\label{sec:computationofhigherdirectimages}

Let $K\subset L \subset \Kbar$. Let $\cF$ be a sheaf of abelian
groups on $\fX_L$.

\begin{proposition}\label{prop:RivastF}
The sheaf $\R^i v_{X,L,\ast}(\cF)$ is isomorphic to the sheaf on
$X^{\ket}$ associated to the contravariant functor whose values on
an affine connected open  $U\in X^{\ket}$ is $ {\rm
H}^i\bigl(\cG_{U_L},\cF\bigl(\Rbar_U,\Nbar_U\bigr)\bigr)$.\smallskip

Analogously, the sheaf $\R^i z_{X,L,\ast}(\cF)$ is isomorphic to the sheaf on $X^{\et}$ associated to the contravariant functor whose values on an affine connected
open   $U\in X^{\et}$ is $ {\rm H}^i\bigl(\cG_{U_L},\cF\bigl(\Rbar_U,\Nbar_U\bigr)\bigr)$.
\end{proposition}
\begin{proof} The proof is as in \cite[Thm.~3.6]{erratum}.
\end{proof}

Assume that we are in the algebraic case and that $X$ is proper
over $\cO_K$. Let $\widehat{X}$ be the associated formal scheme.
For every sheaf $\cL$ on $\fX_L$ we have a natural  morphism
$${\rm H}^i\bigl(\fX_L,\cL\bigr)\lra  {\rm
H}^i\bigl(\mathfrak{\widehat{X}}_L,\gamma^\ast\bigl(\cL\bigr)\bigr).
$$
\begin{proposition}\label{prop:GAGA} Let $\cL$ be a torsion  sheaf
on $\fX_L$. Then, the morphism above is an isomorphism.
\end{proposition}
\begin{proof} We first show how to reduce to the case that $L$ is a finite extension of $K$ in $\Kbar$.
Due to \ref{rmk:alternative} the sites $\fX_L$ and $\mathfrak{\widehat{X}}_L$ are identified
with the projective limit site of the sites $\fX_\bullet$ and
$\mathfrak{\widehat{X}}_\bullet$ fibred over the finite extensions of $K$ contained in $L$.
Furthermore $\gamma_L$ is induced by $\gamma_\bullet$. It follows from
\cite[\S VI.8.7.1]{SGAIV} and \cite[\S VI.8.7.3]{SGAIV} that
$${\rm H}^i\bigl(\fX_L,\cL\bigr)\cong \lim_{\to} {\rm H}^i\bigl(\fX_M,\cL\vert_{\fX_M}\bigr)$$and $${\rm
H}^i\bigl(\mathfrak{\widehat{X}}_L,\gamma_L^\ast(\cL)\bigr)\cong
\lim_{\to } {\rm
H}^i\bigl(\mathfrak{\widehat{X}}_M,\gamma_L^\ast(\cL)\vert_{\mathfrak{\widehat{X}}_M}\bigr),$$where
the direct limit is taken over the category of all finite
extensions $M$ of $K$ contained in $L$. Since
$\gamma_L^\ast(\cL)\vert_{\mathfrak{\widehat{X}}_M}\cong
\gamma_M^\ast\bigl( \cL\vert_{\fX_M}\bigr) $,  if we show that for
every $M$ the map ${\rm H}^i\bigl(\fX_M,\cL\bigr)\lra  {\rm
H}^i\bigl(\mathfrak{\widehat{X}}_M,\gamma^\ast_M\bigl(\cL\bigr)\bigr)$
is an isomorphism for every torsion sheaf $\cL$ on $\fX_M$, the
map ${\rm H}^i\bigl(\fX_L,\cL\bigr)\to {\rm
H}^i\bigl(\mathfrak{\widehat{X}}_L,\gamma_L^\ast(\cL)\bigr)$ is
also an isomorphism for every torsion sheaf $\cL$ on $\fX_L$. We
are then reduced to prove the proposition for $K\subset L$ a
finite extension contained in $\Kbar$. Consider the commutative
diagram
$$\begin{array}{ccc} \Sh\bigl(\mathfrak{\widehat{X}}_L\bigr) &
\stackrel{\gamma_\ast}{\lra} & \Sh\bigl(\fX_L\bigr) \cr z_{\widehat{X}, L,\ast}\big\downarrow & & z_{X,L,\ast}\big\downarrow \cr \Sh\bigl(\widehat{X}^{\et}\bigr)
&\stackrel{\nu}{\lra} & \Sh\bigl(X^{\et}\bigr).\cr\end{array}$$We have compatible spectral sequences
$${\rm H}^q\bigl(X^{\et}, \R^p z_{X,L,\ast}(\cL)\bigr)\Longrightarrow
{\rm H}^{p+q}(\fX_K,\cL)$$and
$${\rm H}^q\bigl(\widehat{X}^{\et}, \R^p
z_{\widehat{X}, L,\ast}(\gamma^\ast(\cL))\bigr)\Longrightarrow {\rm H}^{p+q}\bigl(\mathfrak{\widehat{X}}_K,\gamma^\ast(\cL)\bigr).$$It suffices to prove that the
natural map ${\rm H}^q\bigl(X^{\et}, \R^p z_{X,L,\ast}(\cL)\bigr)\lra {\rm H}^q\bigl(\widehat{X}^{\et}, \R^p z_{\widehat{X}, L,\ast}(\gamma^\ast(\cL))\bigr)$ is an
isomorphism. Due to~\cite[Cor.~1]{Gabber} and the fact that~$X$ is proper over~$\cO_K$ this follows if we show  that the natural map
$$\nu^\ast\bigl(\R^p z_{X,L,\ast}(\cL)\bigr)\cong \R^p
z_{\widehat{X}, L,\ast}(\gamma^\ast(\cL))$$is an isomorphism. This can be checked on stalks at  geometric points $x\in X_k$. Let $\cO_{X,x}^{\rm h}$
(resp.~$\cO_{\widehat{X},x}^{\rm h}$) be the henselization of $\cO_{X,x}$ (resp.~$\cO_{\widehat{X},x}$). Due to \ref{prop:RivastF} it suffices to prove that the map
from the Kummer \'etale covers of $\Spec(\cO_{X,x}^{\rm h}\otimes_{\cO_K} L)$ to the Kummer \'etale covers of $\Spec(\cO_{\widehat{X},x}^{\rm h}\otimes_{\cO_K} L)$,
given by base change, is an equivalence. In both cases the number of their connected components is finite and equal to the degree of the maximal unramified
extension $K'$ of $K$ contained in $L$. It thus suffices to show that their  Galois groups, by which we mean the product of the Galois groups of the  connected
components, are isomorphic. Such Galois groups are isomorphic to $[K':L]$ times the Galois groups of $\cO_{X,x}^{\rm h}\otimes_{\cO_{K'}} L$ and
$\cO_{\widehat{X},x}^{\rm h}\otimes_{\cO_{K'}} L$ respectively, which classify finite and normal extensions which are separable over the locus where the log
structure is trivial. By construction in both cases the log structures are defined by regular elements $Y_1,\ldots,Y_b\in \cO_{X,x}$. Hence, such Galois groups are
extensions of the Galois groups of $\cO_{X,x}^{\rm h}\otimes_{\cO_{K'}} L$ (resp.~of $\cO_{\widehat{X},x}^{\rm h}\otimes_{\cO_{K'}} L$) by the product of the
inertia groups ($\cong \widehat{\Z}$) at each of the prime ideals defined by $Y_i$ for those $i\in \{1,\ldots,b\}$ such that $Y_i$ is not a unit. Hence, we are
reduced to prove that the  Galois groups of $\cO_{X,x}^{\rm h}\otimes_{\cO_{K'}} L$ and of $\cO_{\widehat{X},x}^{\rm h}\otimes_{\cO_{K'}} L$  coincide. It suffices
to  show that the  Galois groups of $\cO_{X,x}^{\rm h}[p^{-1}]$ and of $\cO_{\widehat{X},x}^{\rm h}[p^{-1}]$ coincide. This follows from \cite[Thm.~5]{Elkik}.

\end{proof}

For every $U\in X^{\ket}$ affine connected define ${\rm H}^\ast\left(\cG_{U_L},\_ \right)$ to be the $\delta$-functor obtained by deriving the functor associating
to an inverse system of discrete $\cG_{U_L}$-modules $\{A_n\}_{n\in\N}$ the group $\ds \lim_{\infty \leftarrow n} A_n^{\cG_{U_L}}$. Consider an inverse system of
sheaves $\cF=\{\cF_n\}_n\in \Sh\bigl(\fX_L\bigr)^\N$ of abelian groups. Define ${\rm H}_{\rm Gal}^i\bigl(\cF\bigr)$ to be the sheaf associated to the contravariant
functor sending $U\in X^{\ket}$, affine connected, to ${\rm H}^i\left(\cG_{U_L},\{\cF_n(\Rbar_U)\}_n\right)$. One can also consider the sheaf ${\rm R}^i
v_{L,\ast}^{\rm cont}(\cF)$ obtained by deriving the functor $\cF\to \ds \lim_{\infty\leftarrow n} v_{L,\ast}(\cF_n)$. Then, proceeding as in \cite[Lemma
3.5]{erratum} and \cite[Lemma 3.17]{andreatta_iovita_comparison} one can show there is a functorial homomorphism of sheaves
$$f_i(\cF)\colon {\rm H}_{\rm Gal}^i\bigl(\cF\bigr) \lra {\rm R}^i
v_{L,\ast}(\cF).$$The next proposition, analogous to \ref{prop:RivastF}, provides a criterion under which the above morphism is an isomorphism. Assume that
$L=\Kbar$ and that $\{\cF_n\}_{n\in\N}$ is a sheaf of $A_{\rm inf}$--modules (resp.~of~$\{\cO_{\Kbar}/p^n \cO_{\Kbar} \}_n$--modules). For every small $U\in
X^{\et}$ we write $R_{U,\infty}$ as in \ref{sec:localization} and $R_{U,\infty,\cO_\Kbar}$ to be the normalization of $R_U$ in $R_{U,\infty}\Kbar\subset
\Rbar_U[p^{-1}]$. We write
$$\cH_{U_\Kbar}:=\gal\left(\Rbar_U[p^{-1}]/
R_{U,\infty,\cO_\Kbar}[p^{-1}]\right), \qquad \Gamma_{U_\Kbar}:=\gal\left(R_{U,\infty,\cO_\Kbar}[p^{-1}]/R_U \Kbar\right).$$We then have an exact sequence $$ 0 \lra
\cH_{U_\Kbar} \lra \cG_{U_\Kbar} \lra \Gamma_{U_\Kbar} \lra 0.$$As in \ref{sec:localization} we define $\cF\bigl(R_{U,\infty,\cO_\Kbar} \bigr):=\ds
\lim_{\infty\leftarrow n} \cF_n\bigl(R_{U,\infty,\cO_\Kbar}\bigr)$. They are $\Gamma_{U_\Kbar}$-modules.

Given  an $A_{\rm inf}$--module or an $\cO_{\Kbar}$ module, we say that it is {\it almost zero} if it is annihilated by any element of ideal ${\mathcal I}$
of~$A_{\rm inf}$ (resp.~the maximal ideal of~$\cO_{\Kbar}$) (see \S\ref{sec:classical} for the notation).

\begin{proposition}\label{prop:fiisisoo}
Assume that  for every small $U\in X^{\et}$ and every $n\in\N$ the following hold:\smallskip

(1)  the cokernel of $\cF_{n+1}(\Rbar_U)\to \cF_n(\Rbar_U)$ is almost zero;\smallskip

(2) for every~$q\geq 1$ the group ${\rm H}^q\left(\cH_{U_\Kbar},\cF_n\bigl(\Rbar_{U}\bigr)\right)$ is almost zero;\smallskip

(3)  the cokernel of the transition map $\cF_{n+1}(R_{U,\infty,\cO_\Kbar})\to \cF_n(R_{U,\infty,\cO_\Kbar})$ is almost zero;\smallskip

(4) for every covering~$Z\to U$ by small  objects in~$X^{\ket}$ and every~$q\geq 1$ the Chech cohomology group ${\rm H}^q\bigl(Z\to U,
\cF_n\bigl(R_{U,\infty,\cO_\Kbar}\otimes_{R_U} R_Z\bigr)\bigr)$ is almost zero.\smallskip

Then,  the morphism $f_i(\cF)$ has kernel and cokernel annihilated by any element of ${\mathcal I}^{2i}$ (resp.~any element of the maximal ideal of $\cO_{\Kbar}$).

\end{proposition}
\begin{proof} We follow the analogous proof given  in \cite[Thm.~6.12]{andreatta_iovita}.
See also
\cite[Lemma 3.19]{andreatta_iovita_comparison}. In (4) the notation
$\cF_n\bigl(R_{U,\infty,\cO_\Kbar}\otimes_{R_U} R_Z\bigr)$ stands for the following. Write
$R_{U,\infty,\cO_\Kbar}$ as a direct limit of normal $R_U \cO_{\Kbar}$-algebras $W$, finite and
Kummer \'etale after inverting $p$. Then
$\cF_n\bigl(R_{U,\infty,\cO_\Kbar}\otimes_{R_U} R_Z\bigr)$ is defined to be the direct limit
$\lim_W \cF_n(Z,W_Z)$, over all $W$'s, denoting by $W_Z$ the object of
$Z_\Kbar^{\rm fet}$ obtained from $W$ via the continuous map of sites $U_\Kbar^{\rm fet}\to Z_\Kbar^{\rm fet}$.
Note that $R_{Z,\infty,\cO_\Kbar}[p^{-1}]$ is a direct factor
in $R_{U,\infty,\cO_\Kbar}\otimes_{R_U} R_Z\bigl[p^{-1}\bigr]$, the group $\Gamma_{Z_\Kbar}$ is a
quotient of $\Gamma_{U_\Kbar}$ and
$R_{U,\infty,\cO_\Kbar}\otimes_{R_U} R_Z\bigl[p^{-1}\bigr] \cong {\rm Ind}_{\Gamma_{Z_\Kbar}}^{\Gamma_{U_\Kbar}}
R_{Z,\infty,\cO_\Kbar}$ is the induced
representation as $\Gamma_{U_\Kbar}$-module. Hence
$$\cF_n\bigl(R_{U,\infty,\cO_\Kbar}\otimes_{R_U} R_Z\bigr) \cong
{\rm Ind}_{\Gamma_{Z_\Kbar}}^{\Gamma_{U_\Kbar}} \cF_n\bigl(R_{Z,\infty,\cO_\Kbar}\bigr).$$Without loss of generality we may assume that $X=U$. Via the equivalence
of $U_\Kbar^{\rm fket}$ with the category of finite sets with action of $\cG_{U_\Kbar}$, we get a subtopology $U_\infty\subset U_\Kbar^{\rm fket}$ associated to the
category of finite sets with action of $\Gamma_{U_\Kbar}$.   Let $\fX_{\infty,\Kbar}\subset \fX_\Kbar$ be the subcategory consisting of pairs $(V,W)$ where $V\in
X^{\rm ket}$ and $W\in V_\Kbar^{\rm fket}$ is obtained from an object in $U_\infty$ via the continuous map of sites $U_\Kbar^{\rm fet}\to V_\Kbar^{\rm fet}$. It is
closed under fibred products and we endow it with the induced topology. The map $v_\Kbar$ factors as $v_\Kbar=\beta\circ \alpha$ via the continuous morphism of
sites
$$\alpha \colon X^{\rm ket}\lra \fX_{\infty,\Kbar}, \quad U\mapsto \big(U,U_\Kbar\big)$$and the continuous morphism of sites defined by the inclusion $\beta \colon
\fX_{\infty,\Kbar} \to \fX_\Kbar$. We can then compute $v_{\Kbar,\ast}^{\rm cont}$ as the composite of $\alpha^{\rm cont}_\ast \circ \beta_\ast^\N$; see
\ref{section:continuoussheaves} for the notation. We get a Leray spectral sequence
$${\rm R}^i \alpha_\ast^{\rm cont}\left({\rm R}^j
\beta_\ast^{\N}(\cF_n)_{n\in\N}\right)\Longrightarrow {\rm R}^{i+j} v_{\Kbar,\ast}^{\rm cont}(\cF_n).$$Note that ${\rm R}^i
\beta_\ast^{\N}(\cF_n)_{n\in\N}=\left({\rm R}^i \beta_\ast(\cF_n)\right)_{n\in\N}$.\medskip

{\it Step 1}: We claim that the group ${\rm R}^i \beta_\ast(\cF_n)$ is almost zero for $i\geq 1$.\smallskip

For $V\in X^{\ket}$ affine and $\cF$ a sheaf on $\fX_\Kbar$ we have $${\rm Ind}_{\Gamma_{V_\Kbar}}^{\Gamma_{U_\Kbar}} {\rm
H}^0\left(\cH_{V_\Kbar},\cF\bigl(\Rbar_{V},\Nbar_V\bigr)\right)\cong \beta_\ast(\cF) \bigl(R_{U,\infty,\cO_\Kbar}\otimes_{R_U} R_V\bigr),$$as representations of
$\Gamma_{U_\Kbar}$, functorially in $V$. As in \cite[Lemma 3.5]{erratum} one argues that for every $i$ we have a map $${\rm
Ind}_{\Gamma_{V_\Kbar}}^{\Gamma_{U_\Kbar}} {\rm H}^i\left(\cH_{V_\Kbar},\cF\bigl(\Rbar_{V},\Nbar_V\bigr)\right)\lra {\rm R}^i \beta_\ast(\cF)
\bigl(R_{U,\infty,\cO_\Kbar}\otimes_{R_U} R_V\bigr).$$A geometric point $(x,y)$ of $\fX_\Kbar$ defines a geometric point of $\fX_{\infty,\Kbar}$. Arguing as in
\cite[Thm.~3.6]{erratum} one proves that the  map above induces an isomorphism between the stalk ${\rm R}^i \beta_\ast(\cF)_{(x,y)} $ and $\lim_{x\in V} {\rm
Ind}_{\Gamma_{V_\Kbar}}^{\Gamma_{U_\Kbar}} {\rm H}^i\left(\cH_{V_\Kbar},\cF\bigl(\Rbar_{V},\Nbar_V\bigr)\right)$, where the direct limit is taken over all affine
neighborhoods $V$ of $x$. Since we have enough geometric points, this and Assumption (2) imply that ${\rm R}^i \beta_\ast(\cF_n)$ is almost zero for every $n$ and
every $i\geq 1$.\medskip

{\it Step 2}: The computation of ${\rm R}^i \alpha_\ast^{\rm cont} \beta_\ast(\cF_n)$.\smallskip

For every $i$ and $n\in \N$ consider the contravariant functor on $\fX_{\infty,\Kbar}$ associating to every affine connected $V\in X^{\ket}$ the group ${\cal
C}^i\bigl(\Gamma_{U_\Kbar},\beta_\ast(\cF_n)\bigl(R_{U,\infty,\cO_\Kbar}\otimes_{R_U} R_V\bigr)\bigr)$ of  continuous maps $\Gamma_{U_\Kbar}^{i+1}\lra
\cF_n\bigl(R_{U,\infty,\cO_\Kbar}\otimes_{R_U} R_V\bigr)$. Assumption (4) implies that the associated sheaf ${\cal
C}^i\bigl(\Gamma_{U_\Kbar},\beta_\ast(\cF_n)\bigr)$ has values on every affine connected $V\in X^{\ket}$ equal to the continuous maps $\Gamma_{U_\Kbar}^{i+1}\lra
\cF_n\bigl(R_{U,\infty,\cO_\Kbar}\otimes_{R_U} R_V\bigr)$, up to multiplication by any element of the ideal ${\mathcal I}$ of~$A_{\rm inf}$ (resp.~of the maximal
ideal of~$\cO_{\Kbar}$).  For $V\in X^{\rm ket}$ affine connected we have
$$\alpha_\ast^{\rm cont}\beta_\ast(\cF_n)(V) = \ds
\lim_{\infty\leftarrow n}\cF_n\bigl(R_{U,\infty,\cO_\Kbar}\otimes_{R_U} R_V\bigr)^{\Gamma_{U_\Kbar}}.$$In particular  up to multiplication by any element of
${\mathcal I}$  (resp.~of the maximal ideal of~$\cO_{\Kbar}$) we have a long exact sequence
$$0\lra \alpha_\ast\bigl(\beta_\ast(\cF_n)\bigr) \lra {\cal
C}^\bullet \bigl(\Gamma_{U_\Kbar},\beta_\ast(\cF_n)\bigr).$$For every $V\in X^{\ket}$ affine connected the group $\alpha_\ast^{\rm cont}\left({\cal C}^i
\bigl(\Gamma_{U_\Kbar},\beta_\ast(\cF_n)\bigr)\right)(V)$ coincides with the continuous cochains ${\rm C}^i\bigl(\Gamma_{U_\Kbar},\ds \lim_{\infty\leftarrow
n}\cF_n\bigl(R_{U,\infty,\cO_\Kbar}\otimes_{R_U} R_V\bigr)\bigr)$. To conclude the proof of the proposition it suffices to show that the higher direct images ${\rm
R}^j \alpha_\ast^{\rm cont}$ of ${\cal C}^i \bigl(\Gamma_{U_\Kbar},\beta_\ast(\cF_n)\bigr)_{n\in\N}$ are almost zero. We use the spectral sequence
$${\lim}^{(i)} \left({\rm R}^j
\alpha_\ast {\cal C}^h \bigl(\Gamma_{U_\Kbar},\beta_\ast(\cF_n)\bigr)\right)_{n\in\N}\Longrightarrow {\rm R}^{i+j} \alpha_\ast^{\rm cont}\left({\cal C}^h
\bigl(\Gamma_{U_\Kbar},\beta_\ast(\cF_n)\bigr)_{n\in\N}\right).$$Arguing as in \cite[Lemma 3.5 \& Thm.~3.6]{erratum} one proves that for any sheaf $\cF$ on
$\fX_{\infty,\Kbar}$ and any geometric point $x$ of $X^{\ket}$ the stalk ${\rm R}^j \alpha_\ast(\cF)_x$ is the limit $\lim_{x\in V} {\rm
H}^j\left(\Gamma_{U_\Kbar},\cF\bigl(R_{U,\infty,\cO_\Kbar}\otimes_{R_U} R_V\bigr)\right)$ over the affine connected  neighborhoods $V\in X^{\ket}$ of $x$. Up to
multiplication by any element of ${\mathcal I}$  (resp.~of the maximal ideal of~$\cO_{\Kbar}$) the group ${\rm H}^j\left(\Gamma_{U_\Kbar},{\cal C}^h
\bigl(\Gamma_{U_\Kbar},\cF_n\bigl(R_{U,\infty,\cO_\Kbar}\otimes_{R_U} R_V\bigr)\bigr)\right)$ coincides with the cohomology of ${\rm
H}^j\left(\Gamma_{U_\Kbar},\_\right)$ of the module of continuous maps $\Gamma_{U_\Kbar}^{h+1}\lra \cF_n\bigl(R_{U,\infty,\cO_\Kbar}\otimes_{R_U} R_V\bigr)$, which
is zero for $j\geq 1$. We deduce that ${\rm R}^j \alpha_\ast {\cal C}^h \bigl(\Gamma_{U_\Kbar},\beta_\ast(\cF_n)\bigr)$ is almost zero for $j\geq 1$. We are left to
prove that $\lim^{(i)} \left( \alpha_\ast {\cal C}^h \bigl(\Gamma_{U_\Kbar},\beta_\ast(\cF_n)\bigr)\right)_{n\in\N}$ is almost zero for $i\geq 1$. This follows
using Assumptions (3) and (4); we refer to the proof of \cite[Prop. 6.15(ii)]{andreatta_iovita} for details.

\end{proof}

\subsection{Fontaine's sheaves}
\label{sec:fontainesheaves}

In what follows we will use the following convention. Let $\cS$ be a site and let $\cA$ be a sheaf of commutative rings with identity on $\cS$ such that the
presheaf of units $\cA^\ast$ is a sheaf. We need the notion of logarithmic geometry in this general setting. We refer to \cite[\S 6]{GabberRamero} for the detailed
re-elaboration of \cite{katolog}. A prelog structure on $\cS$ is a sheaf of monoids $M$ and a morphism of multiplicative monoids $\alpha \colon M \to \cA$. A log
structure is a prelog structure such that $\alpha$ induces an isomorphism $\alpha^{-1}(\cA)^\ast\cong \cA^\ast$. The forgetful functor from the category of log
structures on $\cA$ to the category of prelog structures admits a left adjoint. We say that a log structure is {\em coherent} (resp.~{\em fine}, resp.~{\em fine and
saturated}) if there is an open covering $\{U_i\}_i$ of $\cS$ such that $M\vert_{U_i}\to \cA\vert_{U_i}$ is the log structure associated to a morphism of presheaves
of multiplicative monoids $P_i\to \cA\vert_{U_i}$ such that $P_i$ is a constant presheaf on $\cS\vert_{U_i}$ and $\Gamma(U_i,P_i)$ is finitely generated
(resp.~finitely generated and integral, resp.~finitely generated, integral and saturated) for every $i$. We refer to
 \cite{GabberRamero} for details.

If $\cA=\{\cA_n\}_n \in \Sh(\cS)^\N$ is a continuous sheaf of rings, a prelog structure (resp.~a log structure on $\cA$) is a continuous sheaf of monoids $M
=\{M_n\}_n$ and a morphism $\alpha=\{\alpha_n\}_n\colon M\to \cA$ of continuous sheaves such that each $\alpha_n\colon M_n \to \cA_n$ defines a prelog structure
(resp.~a log structure) on $\cA_n$. Also in this case the category of log structures admits a left adjoint. We say that a log structure is {\em coherent}
(resp.~{\em fine}, resp.~{\em fine and saturated}) if there is an open covering $\{U_i\}_{i\in I}$ of $\cS$ such that $M_n\vert_{U_i}\to \cA_n\vert_{U_i}$ is
coherent (resp.~fine, resp.~fine and saturated) for every $n\in\N$ and every $i\in I$.

Given sheaves (or continuous sheaves) of rings $\cA$ and $\cA'$ as above and prelog structures $\alpha\colon M\to \cA$ and $\alpha'\colon M'\to \cA'$, a morphism of
prelog structures is a morphism of sheaves of rings $f\colon \cA\to \cA'$ and a morphism of monoids $g\colon M\to M'$ such that $\alpha'\circ g=f\circ \alpha$. A
morphism of log structures is a morphism as prelog structures. We say that $(f,g)$ is {\em exact} if $M'$ is the log structure associated to the prelog structure
$\alpha'\circ g\colon M \to \cA'$.
\smallskip

{\em Examples:} It follows from \ref{Prop:OXisasheafonket}
that:\smallskip

(1) in the algebraic case $N_{X^{\ket}} \to \cO_{X^{\ket}}$
defines a fine and saturated log structure on $\cO_{X^{\ket}}$.
\smallskip

(2) in the formal case $N_{X^{\ket}_h} \to \cO_{X^{\ket}_h}$ for
$h\in\N$ and $N_{X^{\ket}_{\form}}\to \cO_{X^{\ket}_{\form}}$
define a log structure on $\cO_{X^{\ket}_h}$
(resp.~$\cO_{X^{\ket}_{\form}}$) which is fine and saturated.

\subsubsection{The sheaves $\cO_{\fX}$ and $\widehat{\cO}_{\fX}$}
Fix an extension $K\subset L \subset \Kbar$. In the algebraic case we define the presheaf of
$\cO_L$-algebras on $E_{X_L}$, denoted $\cO_{\fX_L}$, by
$$
\cO_{\fX_L}(U,W):=\mbox{ the normalization of
}\Gamma\bigl(U,\cO_U\bigr)\mbox{ in } \Gamma\bigl(W, \cO_W\bigr).
$$In the formal case the definition is the same replacing $\Gamma\bigl(U,\cO_U\bigr)$
with $\Gamma\bigl(U_{\form},\cO_{U_{\form}}\bigr)$. We also define the sub-presheaf of $\WW(k)$-algebras $\cO_{\fX_L}^{\rm un}$ of $\cO_{\fX_L}$ whose sections over
$(U,W)\in E_{X_L}$ consist of elements $x\in \cO_{\fX_L}(U,W)$ for which there exist   a Kummer \'etale morphism $U' \to U$ and a morphism $W\to U'_K$ over
$U_K$ such that $x$, viewed in $\Gamma\bigl(W, \cO_W\bigr)$, lies in the image of $\Gamma\bigl(U',\cO_{U'}\bigr)$. Then we have.

\begin{proposition}
\label{prop:sheaf} The presheaves $\cO_{\fX_L}$ and\/
$\cO_{\fX_L}^{\rm un}$ are sheaves. Moreover, $\cO_{\fX_L}^{\rm
un}$ is isomorphic to the sheaf
$v_{X,L}^\ast\bigl(\cO_{X^{\ket}}\bigr)$ in the algebraic case and
is isomorphic to the sheaf
$v_{X,L}^\ast\bigl(\cO_{X^{\ket}_{\form}}\bigr)$ in the formal
case.
\end{proposition}
\begin{proof} We prove the statements in the algebraic case. The proof in the formal case is
similar and left to the reader. We first prove that $\cO_{\fX_L}$
is a sheaf. Let $\{(U_\alpha, W_{\alpha,i})\lra
(U,W)\}_{\alpha,i}$ be a strict covering family. We set
$U_{\alpha\beta}:=U_\alpha\times_UU_\beta$ and $W_{\alpha\beta
ij}:=W_{\alpha,i}\times_WW_{\beta,j}$. We have the following
commutative diagram
$$
\begin{array}{cccccccccc}
0&\lra&\cO_{\fX_L}(U,W)&\stackrel{f}{\lra}&
\prod_{i,\alpha}\cO_{\fX_L}(U_\alpha,W_{\alpha,i})&\stackrel{g}{\lra}&
\prod_{(\alpha,i),(\beta,j)}\cO_{\fX_L}(U_{\alpha\beta}, W_{\alpha\beta ij})\\
&&\downarrow&&\downarrow&&\downarrow\\
0&\lra&\Gamma(W,
\cO_W)&\lra&\prod_{\alpha,i}\Gamma(W_{\alpha,i},\cO_{W_{\alpha,i}})&\lra&
\prod_{(\alpha,i),(\beta,j)} \Gamma(W_{\alpha\beta ij},
\cO_{\alpha\beta ij})
\end{array}
$$
Since  $\{U_\alpha\lra U\}_\alpha$ is a covering in $X^{\ket}$ and for every $\alpha$, the family $\{W_{\alpha,i}\lra W\times_UU_\alpha\}_i$ is a covering in
$\left(W\times_U U_{\alpha,M}\right)^{\fket}$ it follows  from \cite[Prop.~2.18]{niziol} that the bottom row of the above diagram is exact. Moreover the vertical
maps are all inclusions therefore $f$ is injective, i.~e.~$\cO_{\fX_L}$ is a separated presheaf. The rest of the argument proceeds as in
\cite[Prop.~2.11]{andreatta_iovita_comparison}.

Since $(U,U_L)$ is the initial object in the category of all pairs $(U',U'_L)$ admitting a morphism
$(U,W)\to (U',U'_L)$ in $\fX_L$, we conclude that
$v_{X,L}^\ast\bigl(\cO_{X^{\ket}}\bigr)$ is the sheaf on $\fX_M$ associated to the presheaf
$P(U,W):=\Gamma(U,\cO_U)$. In particular, we have a natural surjective
map of presheaves $P\to \cO_{\fX_L}^{\rm un}$ inducing a  morphism
$v_{X,M}^\ast\bigl(\cO_X\bigr)\to \cO_{\fX_L}^{\rm un}$. One proves that such morphism is an
isomorphism  as in \cite[Lemma 2.13]{andreatta_iovita_comparison} using that $P(U,W)=\Gamma(U,\cO_U)$ is
normal  for every $U$ and $W$ by \ref{lemma:Ukfetisfppf}.
\end{proof}

Denote by $\widehat{\cO}_{\fX_L}$ the inverse system of sheaves of
$\cO_L$-algebras $\left\{\cO_{\fX_L}/p^n\cO_{\fX_L}\right\}_n\in
\Sh(\fX_L)^\N$.

It follows from \ref{prop:sheaf} that each
$\cO_{\fX_L}/p^n\cO_{\fX_L}$ is a sheaf of
$v_{X,L}^\ast\bigl(\cO_{X^{\ket}}/p^n\cO_{X^{\ket}}\bigr)$
algebras so that we get  morphisms of monoids
$v_{X,L}^\ast\bigl(N_{X^{\ket}_{\form}}\bigr)\lra
\cO_{\fX_L}/p^n\cO_{\fX_L} $ which are compatible for varying
$n\in\N$. Proceeding as in \cite[\S 1.3]{katolog}, one obtains for
every $n$ an associated log structure $N_{\fX_L,n}\to
\cO_{\fX_L}/p^n\cO_{\fX_L}$ characterized by the fact that the
inverse image of $\bigl(\cO_{\fX_L}/p^n\cO_{\fX_L}\bigr)^\ast$ is
$\bigl(\cO_{\fX_L}/p^n\cO_{\fX_L}\bigr)^\ast$. We define
$$\widehat{N}_{\fX_L}:=\left\{N_{\fX_L,n}\right\} \lra
\widehat{\cO}_{\fX_L}$$to be the induced log structure. By
construction it is fine and saturated. For later purposes we
register the following result:

\begin{lemma} \label{lemma:FrobeniussurjectiveonOX/pOX}
Frobenius $\varphi$ is surjective on $\cO_{\fX_L}/p\cO_{\fX_L}$. For $L=\Kbar$ its  kernel is
$p^{1/p}\cO_{\fX_L}/p \cO_{\fX_L}$
\end{lemma}
\begin{proof} Using \ref{sec:localization} it suffices to prove
that Frobenius is surjective on $\Rbar_U/p\Rbar_U \to
\Rbar_U/p\Rbar_U$ with kernel $p^{1/p}\Rbar_U/p\Rbar_U$. This follows from \ref{cor:Frobonto} and the normality of $\Rbar_U$.

\end{proof}

{\em The sheaves $\WW_{s,L}$.} For~$s\in\N$ we define $\WW_{s,L}:=\WW_s\bigl(\cO_{\fX_L}/p\cO_{\fX_L}\bigr)$ as the sheaf of
sets~$\bigl(\cO_{\fX_L}/p\cO_{\fX_L}\big)^s$ with ring operations defined using the Witt polynomials.  Let $N_{{s,L}}$ be the following log structure. For $s=1$ we
let $N_{{1,L}}$ be the log structure associated to the log structure $N_{\fX_L,1} \to \WW_{1,L}=\cO_{\fX_L}/p\cO_{\fX_L}$. For general $s$ let $N_{{s,L}}$ be the
fibred product of monoids
$$\begin{array}{ccc} N_{{s,L}} & \lra & \cO_{\fX_L}/p\cO_{\fX_L}
\cr \big\downarrow & & \varphi^s\big\downarrow \cr N_{{1,L}} & \lra & \cO_{\fX_L}/p\cO_{\fX_L},\cr
\end{array}$$where $\varphi^s$ is Frobenius to the $s$-th
power.  Since the map $\varphi$ is surjective by \ref{lemma:FrobeniussurjectiveonOX/pOX} the map $\widetilde{\varphi}^s\colon N_{{s,L}}\lra N_{{1,L}}$ is surjective
with kernel $1+p^{\frac{1}{p^s}} \cO_{\fX_L}/p\cO_{\fX_L}$. If $U\in X^{\ket}$ is a small affine open, we have a chart $P\cong \N^{a+b} \to
N_{\fX_L,1}\vert_{(U,U_L)}$, provided by our Assumptions (\S\ref{sec:assumptions}). The surjectivity of $\varphi^s$ implies that it can be lifted to a map of
monoids $P \to N_{{s,L}}$ which provides a chart of $N_{{s,L}}$ locally over $(U,U_L)$ (compare with \cite[lemma 1.4.2]{tsujiinventiones}). We conclude that
$N_{{s,L}}$ is a fine and saturated sheaf of monoids.

Let $N_{\WW_{s,L}}\lra \WW_{s,L}$ be the log structure associated to the prelog structure defined as the composite of $N_{s,L}\to \cO_{\fX_L}/p\cO_{\fX_L}$ with the
Teichm\"uller lift $\cO_{\fX_L}/p\cO_{\fX_L}\to \WW_{s,L}$. Let
$$N_{\WW_L} \lra \bA_{\rm inf,L}^+$$ in\/~$\Sh(\fX_L)^\N$ be the inverse system of
sheaves of $\WW(k)$-algebras $\left\{\WW_{n,L}\right\}_n$ with the log structures $\bigl\{N_{\WW_{n,L}}\bigr\}_n$. The transition maps are defined as the composite
of the natural projection $\WW_{n+1,L}\to \WW_{n,L}$ and Frobenius on $\WW_{n,L}$ and the map induced by the natural morphisms $N_{\fX_L,t}\to N_{\fX_L,s}$ for
$t\geq s$. Arguing as before, one proves that for every $n$ the log structure $N_{\WW_{n,L}}$ is fine and saturated. Note that~$\bA_{\rm inf,L}^+$ and $N_{\WW_L}$
are endowed with a Frobenius operator, denoted by~$\varphi$, and that ~$\bA_{\rm inf,L}^+$ is a continuous sheaf of $\WW(k)$--algebras. We remark that if $L=\Kbar$
Frobenius is an isomorphism on~$\bA_{\rm inf,L}^+$ by \ref{lemma:FrobeniussurjectiveonOX/pOX}.

\smallskip

{\em The localizations.} Let $U\in X^{\ket}$ be a small affine open with underlying algebra $R_U$. Then, the localizations of the above defined sheaves in the sense
of \ref{sec:localization}, are
\smallskip

(1) $\cO_{\fX_L}\bigl(\Rbar_U\bigr)=\Rbar_U$;\smallskip

(2) $\widehat{\Rbar}_U\stackrel{\sim}{\lra} \widehat{\cO}_{\fX_L}\bigl(\Rbar_U\bigr)$. Moreover, the localization of $\{N_{\fX_L,n}\}_n$ defines on
$\widehat{\Rbar}_U$, via this isomorphism, the same log structure as the one associated to the prelog structure  $\psi_{R_U}\colon P' \lra R_U \to
\widehat{\Rbar}_U$ defined in \S\ref{section:localdescription}.

\smallskip

(3) $\WW\bigl(\widetilde{\bf E}^+\bigr) \stackrel{\sim}{\lra} \bA_{\rm inf,L}^+(\Rbar_U)$ where $\widetilde{\bf E}^+=\ds \lim_{\infty\leftarrow n} \Rbar_U/p
\Rbar_U$ is the projective limit taking Frobenius as transition map. Moreover, the localization of $\{N_{\WW_{n,L}}\}_n$ defines on $\WW\bigl(\widetilde{\bf
E}^+\bigr)$ the  same log structure as the one associated to the map of monoids $\psi_{\WW\bigl(\widetilde{\bf E}^+\bigr)}\colon P' \lra \WW\bigl(\widetilde{\bf
E}^+\bigr)$ defined in \S\ref{sec:Thetalocalized}.
\smallskip

Statement (1) is clear. In statements (2) and (3) we have natural morphisms due to (1). The proof that they are isomorphisms follows as in
\cite[Prop.~2.15]{andreatta_iovita_comparison} and the key ingredient is Faltings' almost purity theorem in the semistable case for $R_U$ (see \ref{prop:AE}). The
statements concerning the log structures follow as by construction $N_{\fX_L,n}$ and $N_{\WW_{n,L}}$ locally on $(U,U_L)$ admit charts compatible
with $\psi_{R_U}$ and $\psi_{\WW\bigl(\widetilde{\bf E}^+\bigr)}$ respectively.

\subsubsection{The morphism $\Theta$}\label{section:kertheta} One has a natural morphism of
continuous sheaves with log structures
$$\Theta_L:=\{\Theta_{L,n}\}_n\colon \bigl(\bA_{\rm inf,L}^+,N_{\WW_L}\bigr) \lra
\bigl(\widehat{\cO}_{\fX_L},\widehat{N}_{\fX_L}\bigr),$$which is
strict, i.~e.~it is such that the log structure
$\widehat{N}_{\fX_L}$ is the one associated to $N_{\WW_L}$ via
$\Theta_L$. For every $n\in\N$ the morphism $\Theta_{L,n}$ is the
morphism of sheaves associated to the following map of presheaves
$c_n$. For every object $(U,W)$ of $\fX_L$ if we put
$S=\cO_{\fX_L}(U,W)$, then
$$ c_n(U,W)\colon \WW_n(S/pS):=(S/pS)^n \lra S/p^nS, \qquad \bigl(s_0,s_1,\ldots,
s_{n-1}\bigr)\mapsto \sum_{i=0}^{n-1}p^i\tilde{s}_i^{p^{n-1-i}},$$where for every
$s\in S/pS$ we denote by $\tilde{s}$ a (any) lift of $s$ to $S/p^n S$. One proves
that $c_n(U,W)\bigl(s_0,s_1,\ldots, s_{n-1}\bigr)$ does not depend on the choice of the lifts
$\tilde{s}_i$ of $s_i$, that $c_n$ defines a map of presheaves and
that $c_{n+1}$ modulo $p^n$ is compatible with $c_n$; see \cite[\S 2.4]{andreatta_iovita_comparison}. Since the log structure on
$S/p^n S$ is the inverse image of
the log structure on $S/pS$, then $c_n$ is compatible and strict with respect to the  log structures.
Moreover, if we assume that $p^{1/p^{n-1}}\in S$, then,
$\xi_n:=[\overline{p}^{1/p^{n-1}}]-p$ is a well defined element of\/~$\WW_n(S/pS)$ and it generates the
kernel of\/~$c_n\colon\WW_n(S/pS)\lra S/p^nS$. For the proof we
refer to loc.~cit. Thus,

\begin{corollary}\label{cor:kerthetasheaves} We have
$\Ker\left(\Theta_\Kbar\colon \bA_{\rm inf,\Kbar}\lra
\hatcO_{\fX_\Kbar}\right)=\xi\cdot \bA_{\rm inf,\Kbar}$ as sheaves
in $\Sh(\fX_\Kbar)^\N$.
\end{corollary}
\subsubsection{The sheaf  $\bA_{\rm log}^\nabla$
}\label{sec:Alognabla} Recall from \cite[\S 2.5]{andreatta_iovita_comparison} that a $\WW(k)$--{\it divided power} ($\WW(k)$--DP) sheaf of algebras in $\Sh(\fX_L)$
or $\Sh(\fX_L)^\N$ is a triple $(\cF, \cI,\gamma)$ consisting of\enspace (1) a sheaf of $\WW(k)$--algebras $\cF\in \Sh(\fX_M)$ (resp.~an inverse system of sheaves
of $\WW(k)$--algebras $\{\cF_n\}\in \Sh(\fX_M)^\N$),\enspace (2) a sheaf of ideals $\cI\subset \cF$ (resp.~an inverse system of sheaves of ideals $\{\cI_n\subset
\cF_n\}$), \enspace (3) maps $\gamma_i\colon \cI \to \cI$ for~$i\in\N$ such that for every object $(\cU,\cW)$ the triple
$\bigl(\cF(\cU,\cW),\cI(\cU,\cW),\gamma_{(\cU,\cW)}\bigr)$ (resp.~for every~$n$ the triple $\bigl(\cF_n(\cU,\cW),\cI_n(\cU,\cW),\gamma_{(\cU,\cW)}\bigr)$) is a DP
algebra compatible with the standard divided power structure on the ideal $p\WW(k)$ in the sense of \cite[Ch.~3]{berthelot_ogus}. Given a sheaf of
$\WW(k)$--algebras~$\cG$ and an ideal~$\cJ\subset \cG$ (resp.~an inverse system of sheaves of $\WW(k)$--algebras~$\cG$ and ideals $\cJ\subset \cG$) the
$\WW(k)$--divided power envelope of~$\cG$ with respect to~$\cJ$ is a $\WW(k)$--divided power sheaf of algebras $(\cF, \cI,\gamma)$ and a morphism $\cG \to \cF$ of
sheaves (or inverse systems of sheaves) of $\WW(k)$--algebras, such that~$\cJ$ maps to~$\cI$, which is universal for morphisms as sheaves (or inverse systems of
sheaves) of $\WW(k)$--algebras from~$\cG$ to $\WW(k)$--divided power sheaves of algebras $\cF'$ such that $\cJ$ maps to the sheaf of ideals of $\cF'$ on which the
divided power structure is defined. \smallskip

Let $\cA$ and $\cA'$ be (continuous) sheaves of $\WW(k)$-algebras on $\fX$ endowed with fine log structures $M\to \cA$ and $M'\to \cA'$. Let $f\colon (\cA,M) \to
(\cA',M')$ be a morphism of sheaves of rings with log structures such that the morphism $\cA\to \cA'$ is surjective. We call the $\WW(k)$--{\em log-divided power
envelope}  of $(\cA,M)$ {\em with respect to} $f$ to be
\smallskip

(1)   a $\WW(k)$--divided power (continuous) sheaf of algebras $(\cF, \cI,\gamma)$ on $\fX$ and a fine log structure $H\to \cF$;\smallskip

(2)  a strict morphism  of log structures $(\cF,H)\to (\cA',M')$ such that $\cF/\cI\cong \cA'$ as (continuous) sheaves of rings;\smallskip

(3) a morphism  of log structures $(\cA,M)\to (\cF,H)$ such that the composite with $(\cF,H)\to (\cA',M')$ is $f$;\smallskip

(4) $(\cF, \cI,\gamma,H)$ is universal among objects satisfying (1), (2) and (3).

\smallskip Similarly, we call the {\em
log envelope} of $(\cA,M)$ {\em with respect to} $f$ to be a (continuous) sheaf  with log structures $(\cE,J)$ such that (2) and (3) hold and it is universal for
such properties.

\begin{lemma}\label{lemma:logpdenvelope} The
$\WW(k)$--log divided power envelope (resp.~the log envelope) of
$(\cA,M)$ with respect to $f$ exists.
\end{lemma}
\begin{proof} We argue as in \cite[Prop. 5.3]{katolog}.
Assume that the log envelope $(\cE,J)$ of $(\cA,M)$ with respect
to $f$ exists. Then, the $\WW(k)$--divided power envelope of $\cE$
with respect to the kernel of the morphism $\cE \to \cA'$ exists
by \cite[Thm. I.2.4.1]{berthelot1} and, together with the
log structure defined by $J$, it is the $\WW(k)$--log divided
power envelope of $(\cA,M)$ with respect to $f$. In particular,
the latter  exists.

We next prove that, under the assumption that $M\to \cA$ and $M'\to \cA'$ are fine, $(\cE,J)$ exists.
Due to the universal property it suffices to prove that
$(\cE,J)$ exists locally on $\fX$, cf.~\cite[Lemma 2.23]{andreatta_iovita_comparison}.
Let $V:=(U,W)\in \fX$ such that $(M,\cA)\vert_V$ and $(M',\cA')\vert_V$ admit
charts $P\to \cA\vert_V$ and $P'\to \cA'\vert_V$ with $P$ and $P'$ constant sheaves of monoids,
integral and finitely generated. Possibly after shrinking $V$ we may
also assume that $f$ is induced by a morphism of monoids $\alpha\colon P\to P'$.
Let $Q:=\bigl(\alpha^{\rm gp}\bigr)^{-1}(P')\subset P^{\rm gp}$. Let
$\cE:=\cA\vert_V\otimes_{\Z[P]} \Z[Q]$ with the log structure $J$ associated to the morphism of monoids $Q\to \cE$,
$q\mapsto 1\otimes q$. Then, we have natural
morphisms of sheaves of rings with log structures $(M,\cA)\vert_V \to (\cE,J)\to (M',\cA')\vert_V$ and
the latter is exact by construction. We leave to the reader
to check that $(\cE,J)$ has the required universal property.
\end{proof}

\bigskip

{\em The sheaf $\bA_{\rm cris}^\nabla$.} Let $\bA_{\rm cris,L}^\nabla:=\left\{\bA_{\rm cris,L,n}^\nabla \right\}_{n\in\N}$ be  the $\WW(k)$--divided power envelope
of $\bA_{\rm inf,L}^+$ with respect to $\Ker\left(\Theta_L\right)$. It is endowed with a Frobenius operator induced by Frobenius on $\bA_{\rm inf,L}^+$ and with a
decreasing filtration ${\rm Fil}^n \bA_{\rm cris,L}^\nabla$ for $n\in\Z$, defined by the divided power ideal, where we put ${\rm Fil}^n \bA_{\rm
cris,L}^\nabla=\bA_{\rm cris,L}^\nabla$ for $n\leq 0$.
\bigskip

{\em The sheaf $\bA_{\rm log,L}^\nabla$.} Let $\Theta_{\cO,L}$ be the morphism of continuous sheaves with log structure
$$\Theta_{\cO,L}:=\Theta_L\otimes \theta_\cO \colon \bA_{\rm inf,L}^+\otimes_{\WW(k)} \cO \lra
\widehat{\cO}_{\fX_L}.$$Let $\bA_{\rm
log,L}^\nabla:=\left\{\bA_{\rm log,L,n}^\nabla \right\}_{n\in\N}$
be the continuous sheaf defined as the $\WW(k)$--log divided power
envelope of $\bA_{\rm inf,L}^+\otimes_{\WW(k)} \cO $ with respect
to $\Theta_{\cO,L}$. It exists due to \ref{lemma:logpdenvelope}.
We have a natural morphism $\bA_{\rm cris ,L}^\nabla\lra \bA_{\rm
log,L}^\nabla$ compatible with log structures.

If $L=\Kbar$, the sheaf $\bA_{\rm cris ,\Kbar}^\nabla$ (resp.~$\bA_{\rm log ,\Kbar}^\nabla$) is a sheaf of
$A_{\rm cris}$-algebras (resp.~$A_{\rm log}$-modules)
where $A_{\rm cris}$ and $A_{\rm log}$ are the classical period rings of $\cO_K$. We further have the
following properties which are proven as in \cite[Prop.~2.24,
Lemma 2.26, Prop.~2.28]{andreatta_iovita_comparison}:\medskip

{\em Frobenius:} The Frobenius map~$\varphi\colon\WW_{n,L}\to \WW_{n,L} $ defines maps $\varphi\colon
\bA_{\rm cris,L}^\nabla\to \bA_{\rm cris,L}^\nabla$ and
$\varphi\colon \bA_{\rm log,L}^\nabla\to \bA_{\rm log,L}^\nabla$ which are compatible with the morphism
$\bA_{\rm cris ,L}^\nabla\lra \bA_{\rm
log,L}^\nabla$.\medskip

{\em Filtration:} We have a decreasing filtration ${\rm Fil}^n \bA_{\rm log,L}^\nabla$, for $n\in\Z$, defined by the divided power
ideal and compatible with the filtration on
$\bA_{\rm cris ,L}^\nabla$.\medskip

{\em Extension of scalars:} We have natural isomorphisms $\beta^\ast\left(\bA_{\rm cris,K}^\nabla\right)\cong \bA_{\rm cris,\Kbar}^\nabla $ and
$\beta^\ast\left(\bA_{\rm log,K}^\nabla\right)\cong \bA_{\rm log,\Kbar}^\nabla $ compatible with log
structures,  Frobenius and divided power structures and with
the morphism $\bA_{\rm cris ,L}^\nabla\lra \bA_{\rm log,L}^\nabla$.\medskip

{\em Explicit description:}  We have natural isomorphisms
$$\bA_{\rm cris,\Kbar}^\nabla\cong  \bA_{\rm
inf,\Kbar}^+\otimes_{\WW(k)} A_{\rm cris}, \qquad \bA_{\rm log,\Kbar}^\nabla \cong  \bA_{\rm inf,\Kbar}^+\otimes_{\WW(k)}
A_{\rm log}$$compatible with
the divided power structures,  log structures and Frobenius and such that the morphism
$\bA_{\rm cris,\Kbar}^\nabla\lra \bA_{\rm log,\Kbar}^\nabla$ is induced by the natural
morphism $A_{\rm cris} \to A_{\rm log}$. In particular,\medskip

(1) the $A_{\rm cris}$-linear derivation $d\colon A_{\rm log}\lra A_{\rm log}
\frac{dZ}{Z}$ defines on $\bA_{\rm log,\Kbar}^\nabla$ an $\bA_{\rm
cris,\Kbar}^\nabla$-linear derivation $$d\colon \bA_{\rm log,\Kbar}^\nabla\lra \bA_{\rm log,\Kbar}^\nabla
\frac{dZ}{Z}$$which is surjective and satisfies $\bA_{\rm
cris,\Kbar}^\nabla\cong \bA_{\rm log,\Kbar}^{\nabla,d=0}$;\medskip

(2) the inclusion $\bA_{\rm cris,\Kbar}^\nabla \subset \bA_{\rm log,\Kbar}^\nabla$ is split injective with left inverse defined
as the morphism which is the
identity on $\bA_{\rm cris,\Kbar}^\nabla$ and sends $(u-1)^{[n]}$ to $ 0$ for every $n\in\N$;\medskip

(3) the inclusion $\bA_{\rm cris,\Kbar}^\nabla \subset \bA_{\rm log,\Kbar}^\nabla$ is strict with respect to filtrations;\medskip

\medskip

{\em Localization:} For $U$ a small  object of $X^{\ket}$ with
underlying algebra $R_U$ we have

$$\bA_{\rm cris,L}^\nabla(\Rbar_U) \cong  {\rm A}_{\rm cris}^\nabla(R_U),
\qquad \bA_{\rm log,L}^\nabla(\Rbar_U) \cong {\rm A}_{\rm
log}^\nabla(R_U),$$compatibly with the action of $\cG_{U_L}$,
filtrations, Frobenius where ${\rm A}_{\rm cris}^\nabla(R_U)$  and
${\rm A}_{\rm log}^\nabla(R_U)$ are defined in
\ref{sec:AcrisnablaRbar}.

\subsubsection{The sheaf  $\bA_{\rm log}$}\label{sec:Alog}

Fix the following notation. \smallskip

({\it ALG}) For every $n\in\N$ we write $\Stilde_n:=
\Spec\bigl(\cO/(P_\pi(Z))^n\bigr)$ with the log structure
$\Mtilde_n$ defined by $\psi_\cO$.

\smallskip

({\it FORM}) For every $n\in\N$ we write $\Stilde_n:=
\Spec\bigl(\cO/(p,P_\pi(Z))^n\bigr)$ with the log structure
$\Mtilde_n$ defined by $\psi_\cO$.

\smallskip

In both cases we {\em assume} that a global deformation of $(X,M)$ to $\cO$ exists. More precisely, we assume that for every $n\in\N$ we have a log scheme
$\bigl(\Xtilde_n,\Ntilde_n\bigr)$ and log smooth morphism of log schemes of finite type $\tilde{f}_n\colon \bigl(\Xtilde_n,\Ntilde_n\bigr) \to
\bigl(\Stilde_n,\Mtilde_n\bigr)$ such that $\bigl(\Xtilde_n,\Ntilde_n\bigr)$ is isomorphic as log scheme over $\bigl(\Stilde_n,\Mtilde_n\bigr)$ to the fibred
product of $\bigl(\Xtilde_{n+1},\Ntilde_{n+1}\bigr)$ and $\bigl(\Stilde_n,\Mtilde_n\bigr)$ over $\bigl(\Stilde_{n+1},\Mtilde_{n+1}\bigr)$.

\begin{remark}\label{remark:defexists} Since $(X,M)$ is log smooth over $(S,N)$ (resp.~$(X_1,S_1)$ is
log smooth over $(S_i,N_1)$) it follows from
\cite[Prop.~3.14]{katolog}  that such a deformation of $(X,M)$ to
$\cO$ always exist Zariski locally on $X$. For example, it exists
if $X$ is affine and in such a case any two deformations are
isomorphic. It exists also if $X_1$ is of relative dimension $1$
over $S_1$.
\end{remark}

Given a small  object $U\in X^{\ket}$, for every $n\in\N$ we let
$\bigl(\widetilde{U}_n, \widetilde{H}_n\bigr)\lra
\bigl(\Xtilde_n,\Ntilde_n\bigr)$ be the unique Kummer \'etale
morphism deforming the morphism $U\to X$. Write
$\bigl(\widetilde{U}_{\rm form},\widetilde{H}_{\rm form}\bigr)$
for the formal scheme with log structure defined by
$\bigl(\widetilde{U}_n, \widetilde{H}_n\bigr)_{n\in\N}$. If
$\widetilde{U}_{\rm form}=\Spf(\widetilde{R}\bigr)$, we call a
{\em formal chart} of $\bigl(\widetilde{U}_{\rm
form},\widetilde{H}_{\rm form}\bigr)$ a chart
$$\WW(k) [P]\widehat{\otimes}_{\WW(k)[\N]} \cO
\stackrel{\psi_{\widetilde{R}}}{\longrightarrow}  \widetilde{R},
$$inducing a chart of $U$ as in \ref{sec:notation}.

\bigskip

{\em The sheaves $\widehat{\cO}_\Xtilde^{\rm DP}$ and $\omega^i_{\Xtilde}$.}\enspace Define the sheaf $\cO_{\Xtilde,n}$ on $X^{\ket}$ by setting
$\cO_{\Xtilde,n}(U):=\Gamma\bigl(\widetilde{U}_n,\cO_{\widetilde{U}_n} \bigr)$. Let $\bigl(\cO_{\Xtilde,n}/p^n \cO_{\Xtilde,n} \bigr)_{n\in\N} \in
\Sh\bigl(X^{\ket}\bigr)^\N$. Let
$$\theta_{\Xtilde,n}\colon  \cO_{\Xtilde,n}/p^n \cO_{\Xtilde,n}
\lra \cO_{X}/p \cO_X$$be the natural surjective map of sheaves of rings. It induces a strict morphism of log structures for every $n$. Let $\bigl(\cO_{\Xtilde,n}/p^n
\cO_{\Xtilde,n} \bigr)^{\rm DP}_{n\in\N}\in \Sh(X^{\ket})^\N$ be the continuous sheaf defined as the $\WW(k)$--log divided power envelope of $\cO_{\Xtilde,n}/p^n
\cO_{\Xtilde,n}$ with respect to the kernel of $\theta_{\Xtilde,n}$; see \ref{lemma:logpdenvelope}. Write $\widehat{\cO}_\Xtilde^{\rm DP}:=\ds
\lim_{\infty\leftarrow n} \bigl(\cO_{\Xtilde,n}/p^n \cO_{\Xtilde,n} \bigr)^{\rm DP}$. Then,
$$\widehat{\cO}_\Xtilde^{\rm
DP}\cong \cO_\Xtilde\widehat{\otimes}_{\cO} \cO\langle P_\pi(Z)\rangle,$$where the completion is taken with respect to the $(p,Z)$-adic topology (or equivalently
the $p$-adic topology). In particular, if $U$ is a small object of $X^{\ket}$ and if $\widetilde{R}$ is the algebra underlying $\widetilde{U}_{\rm form}$, we have
$\widehat{\cO}_\Xtilde^{\rm DP}(U)\cong \widetilde{R}_{\rm cris}$ in the notation of \ref{def:RlogRmax}.

\

Let $d$ be the relative dimension of $X$ over $\cO_K$. For every
integer~$0\leq i\leq d$ let $\omega^i_{\Xtilde_n/\Stilde_n}(U)$
(resp.~$\omega^i_{\Xtilde_n/\WW(k)}(U)$) be the module of global sections of
the sheaf of logarithmic K\"ahler differentials of
$\bigl(\widetilde{U}_n, \widetilde{H}_n\bigr)$ relative to
$\bigl(\Stilde_n,\Mtilde_n\bigr)$ (resp.~$\WW(k)$).
Let~$\omega^i_{\Xtilde/\cO}\in \Sh(X^{\ket})^\N$
(resp.~$\omega^i_{\Xtilde/\WW(k)}\in \Sh(X^{\ket})^\N$) be the
continuous sheaf
$\bigl(\omega^i_{\Xtilde_n/\Stilde_n}\bigr)_{n\in\N}$
(resp.~$\bigl(\omega^i_{\Xtilde_n/\WW(k)}\bigr)_{n\in\N}$).

\bigskip

{\em The sheaf $\bA_{\rm log,L}$.}\enspace  Let $\Theta_{\Xtilde,L}$ be the
morphism of continuous sheaves with log structure
$$\Theta_{\Xtilde,L}:=\Theta_L\otimes v_{X,L}^\ast\bigl(\theta_{\Xtilde}\bigr) \colon
\bA_{\rm inf,L}^+\otimes_{\WW(k)} v_{X,L}^\ast\bigl(\cO_{\Xtilde}\bigr) \lra \widehat{\cO}_{\fX_L}.$$Define $\bA_{\rm log,L}:=\left\{\bA_{\rm log,L,n}
\right\}_{n\in\N}$ as the $\WW(k)$--log divided power envelope of $\bA_{\rm inf,L}^+\otimes_{\WW(k)} v_{X,L}^\ast\bigl(\cO_{\Xtilde}\bigr)$ with respect to
$\Theta_{\Xtilde,L}$. It exists due to \ref{lemma:logpdenvelope}. It is endowed  with a decreasing filtration ${\rm Fil}^i \bA_{\rm log,L}$ for $i\in\Z$, defined by
the DP ideal, where  ${\rm Fil}^i \bA_{\rm log,L}^\nabla=\bA_{\rm log,L}^\nabla$ for $i\leq 0$. By construction we have a natural morphism $\bA_{\rm
log,L}^\nabla\lra \bA_{\rm log,L}$, compatible with the filtrations.

\bigskip

{\em Explicit description.}\enspace Let $U$ be a small  object of
$X^{\ket}$ and we fix compatible charts $$\psi_{\widetilde{R}}\colon
\WW(k) [P]\widehat{\otimes}_{\WW(k) [\N]}\cO \lra
\widetilde{R},\qquad \psi_{R}\colon \WW(k) [P]\otimes_{\WW(k)[\N]}
\cO_K \lra R$$for the log structure on $\widetilde{U}_{\rm
form}=\Spf\bigl(\widetilde{R} \bigr)$ (resp.~on $U=\Spec(R)$ in
the algebraic case and of $U_{\rm form}=\Spf(R)$ in the formal
case). Recall that $P=\N^a\times \N^b$. Let $e_1,\ldots,e_{a+b}$
be the standard generators of $P$ and  write
$$\widetilde{X}_i:=\psi_{\widetilde{R}}(e_i),\quad X_i=\psi_R(e_i) \quad \forall\, 1\leq i\leq
a\qquad  \widetilde{Y}_j:=\psi_{\widetilde{R}}(e_{a+j}), \quad
Y_j:=\psi_R(e_{a+j}) \quad \forall\, 1\leq j\leq b.$$

Let $Z_n\to U_K$ be the object in $U_{K}^{\fket}$ with underlying algebra  $R\times_{\WW(k)[P]} \WW(k)\left[\frac{1}{p^n}P\right] \otimes_{\WW(k)}
K'_{p^n}\bigl(\epsilon_{p^n}\bigr)$. Write $S_n:=\cO_{\fX_L}(U,Z_n)$ and $\widetilde{R}^{\rm kun}_n:=v_{X,L}^\ast\bigl(\cO_{\Xtilde}\bigr)(U,Z_n)$. Write
$\overline{X}_{i,n}:=X_i^{1/p^n}$ in $S_n/pS_n$ and $\bigl[\overline{X}_{i,n}\bigr]$ equal to the Teichm\"uller lift of $\overline{X}_{i,n}$ for $i=1,\ldots,a$.
Similarly put $\overline{Y}_{j,n}:=Y_j^{1/p^n}$ in $S_n/pS_n$ and $\bigl[\overline{Y}_{j,n}\bigr]$ equal to the Teichm\"uller lift of $\overline{Y}_{j,n}$  for
$j=1,\ldots,b$ in $\WW_n\bigl(S_n/pS_n\bigr)$. We also have the element $\pi_{p^n}\in S_n/pS_n$ and we write $\bigl[\overline{\pi}_{p^n}\bigr]$ for its
Teichm\"uller lift. Then:

\begin{proposition}\label{prop:loclstructure} The kernel of the map
$\WW_n\bigl(S_n/pS_n\bigr)\otimes_{\WW(k)} \widetilde{R}^{\rm kun}_n \to
S_n/p^n S_n $ defined by $\Theta_{\Xtilde,L}$  is the ideal
$\bigl(\xi_n,\bigl[\overline{X}_{i,n}\bigr]\otimes 1 -1\otimes
\widetilde{X}_i, \bigl[\overline{Y}_{j,n}\bigr]\otimes 1 -1\otimes
\widetilde{Y}_j\bigr)$ for $1\leq i\leq a $ and $1\leq j\leq b$ or
the ideal $\bigl(\xi_n,[\overline{\pi}_{p^n}]\otimes 1- 1
\otimes Z, \bigl[\overline{X}_{i,n}\bigr]\otimes 1 -1\otimes
\widetilde{X}_i, \bigl[\overline{Y}_{j,n}\bigr]\otimes 1 -1\otimes
\widetilde{Y}_j\bigr)$ for $2\leq i\leq a $ and $1\leq j\leq b$.
In particular, $$\bA_{\rm log,L,n}\vert_{(U,Z_n)} \cong \bA_{\rm
log,L,n}^\nabla\left\langle
v_{2,n}-1,\ldots,v_{a,n}-1,w_{1,n}-1,\ldots,w_{b,n}-1
\right\rangle $$
with
$v_i:=\frac{\bigl[\overline{X}_{i,n}\bigr]}{\widetilde{X}_i}$ for
$i=1,\ldots,a$ and
$w_j:=\frac{\bigl[\overline{Y}_{j,n}\bigr]}{\widetilde{Y}_j}$ for
$j=1,\ldots,b$.
\end{proposition}
\begin{proof} It follows from \ref{section:kertheta} that modulo $\xi_n$ the kernel of
$\Theta_{\Xtilde,L}$ on $\WW_n\bigl(S_n/pS_n\bigr)\otimes_{\WW(k)}
\widetilde{R}^{\rm kun}$ is the kernel of $S_n/p^n S_n
\otimes_{\WW(k)} \widetilde{R}^{\rm kun} \to S_n/p^n S_n$. The first
claim follows by an explicit computation,
cf.~\ref{lemma:structurBdR+}.

The second part of the proposition follows as in \cite[Lemma 2.30, Thm.~2.31]{andreatta_iovita_comparison}.

\end{proof}

\smallskip

{\em Extension of scalars.}\enspace We have a natural isomorphism $\beta^\ast\left(\bA_{\rm log,K}\right)\cong
\bA_{\rm log,\Kbar}$ compatible with log structures and
divided power structures and with the morphism $\bA_{\rm log ,L}^\nabla\lra \bA_{\rm log,L}$.

\medskip

{\em Frobenius.}\enspace For $U$ a small  object of $X^{\ket}$ as above let $F_U$ be the unique homomorphism $F_U\colon \widetilde{R}\to \widetilde{R}$ inducing Frobenius
modulo $p$ and compatible, via the chart $\psi_{\widetilde{R}}$, with the map $\WW(k)[P]\to \WW(k)[P] $ given by Frobenius on $\WW(k)$ and multiplication by $p$ on
$P$. This produces a Frobenius $F_U$ on on $v_{X,L}^\ast\bigl(\cO_{\Xtilde}\bigr)\vert_{(U,U_L)}$. Together with Frobenius on $\bA_{\rm inf,L}^+$ it defines a
Frobenius on $\bA_{\rm inf,L}^+\otimes_{\WW(k)} v_{X,L}^\ast\bigl(\cO_{\Xtilde}\bigr)\vert_{(U,U_L)}$ compatible with the log structures. Using
\ref{prop:loclstructure} one proves that it extends to a Frobenius morphism~$\varphi_U$ on $\bA_{\rm log,L}\vert_{(U,U_L)}$ compatible with Frobenius defined on
$\bA_{\rm log,L}^\nabla$ and with the log structures.
\medskip

{\em Localization.}\enspace For $U$ a small object of $X^{\ket}$ write $\widetilde{U}:=\Spf\bigl(\widetilde{R}\bigr)$ with induced log structure. Using
\ref{prop:loclstructure} one proves that
$$ \bA_{\rm log,L}^\nabla(\Rbar_U) \cong {\rm A}_{\rm
log}^\nabla\bigl(\widetilde{R}_U\bigr),$$compatibly with action of
$\cG_{U_L}$, filtrations, Frobenius. Here, ${\rm A}_{\rm
log}\bigl(\widetilde{R}_U\bigr)$ is the ring, with log structure,
defined in \S\ref{sec:AcrisnablaRbar}.

\subsubsection{Properties of $\bA_{\rm log}^\nabla$ and $\bA_{\rm log}$}

For $T=\cO$ or $T=\WW(k)$ consider the continuous sheaf
$v_{X,L}^\ast\bigl(\omega^i_{\Xtilde/T}\bigr)$  of locally free
$v_{X,L}^\ast(\cO_X)\cong \cO_{\fX_L}^{\rm un}$-modules over
$\fX_L$. The de Rham complex on~$\Xtilde_n$ for every $n\in\N$
defines a de Rham complex
$v_{X,M}^\ast\bigl(\omega^\bullet_{\Xtilde/T}\bigr)$ on~$\fX_L$.
We then get a complex $\bA_{\rm inf,L}^+\otimes_{\WW(k)}
v_{X,M}^\ast\bigl(\omega^\bullet_{\Xtilde/T}\bigr)$.\smallskip

{\it Convention:} In order to simplify the notation, for every
sheaf of $\cO_{\fX_L}^{\rm un}$-modules ${\cal E}$ and any sheaf
of $\cO_\Xtilde$-modules ${\cal M}$ we write
$\cE\otimes_{\cO_\Xtilde} \cM$ for $\cE\otimes_{\cO_{\fX_L}^{\rm
un}} v_{X,M}^\ast\bigl(\cM\bigr)$.\smallskip

One can prove as in \cite[\S2.7]{andreatta_iovita_comparison} that the
de Rham complex above  extends uniquely to a complex
$$\bA_{\rm log,L} \stackrel{\nabla^1_T}{\lra} \bA_{\rm log,L}
\otimes_{\cO_\Xtilde} \omega^1_{\Xtilde/T}
\stackrel{\nabla^2_T}{\lra} \bA_{\rm log,L} \otimes_{\cO_\Xtilde}
\omega^2_{\Xtilde/T} \lra \cdots .$$Using the description given in
in \ref{prop:loclstructure} we have that $\nabla^1_{\cO}$ is
$\bA_{\rm log,L}^\nabla$-linear and sends $(v_i-1)^{[n]}$ to $
-(v_i-1)^{[n-1]} v_i d\log \widetilde{X}_i$ for $i=2,\ldots,a$ and
sends $(w_j-1)^{[n]}$ to $-(w_j-1)^{[n-1]} w_j d\log
\widetilde{Y}_j$ for $j=1,\ldots,b$. Similarly,
$\nabla^1_{\WW(k)}$ is $\bA_{\rm cris,L}^\nabla$-linear and sends
$(v_i-1)^{[n]}$ to $ -(v_i-1)^{[n-1]} v_i d\log \widetilde{X}_i$
for $i=1,\ldots,a$ and sends $(w_j-1)^{[n]} $ to
$-(w_j-1)^{[n-1]} w_j d \log \widetilde{Y}_j$ for $j=1,\ldots,b$.
Moreover,

\begin{proposition}\label{prop:deRhamcomplex} Writing $\nabla_T:=\nabla_T^1$ and
$\nabla^i:=\nabla^i_\cO$ and $\nabla:=\nabla^1$ we have that:

\begin{enumerate}

\item[i.] for every $r\in \N$ the sequence $0\lra \Fil^r \bA_{\rm
log,L}^\nabla \lra \Fil^r \bA_{\rm log,L} \stackrel{\nabla}{\lra}
\Fil^{r-1} \bA_{\rm log,L} \otimes_{\cO_\Xtilde}
\omega^1_{\Xtilde/\cO} \stackrel{\nabla^2}{\lra} \Fil^{r-2}
\bA_{\rm log,L} \otimes_{\cO_\Xtilde} \omega^2_{\Xtilde/\cO}
\stackrel{\nabla^3}{\lra} \cdots$ is exact;

\item[i'.] for every $r\in \N$ the sequence $0\lra \Fil^r \bA_{\rm
cris,L}^\nabla \lra \Fil^r \bA_{\rm log,L}
\stackrel{\nabla_{\WW(k)}}{\lra} \Fil^{r-1} \bA_{\rm log,L}
\otimes_{\cO_\Xtilde} \omega^1_{\Xtilde/\WW(k)}
\stackrel{\nabla_{\WW(k)}^2}{\lra} \Fil^{r-2} \bA_{\rm log,L}
\otimes_{\cO_\Xtilde} \omega^2_{\Xtilde/\WW(k)}
\stackrel{\nabla_{\WW(k)}^3}{\lra} \cdots$ is exact;

\item[ii.] the natural inclusion $\bA_{\rm log,L}^\nabla\subset
\bA_{\rm log,L}$  identifies $\Ker(\nabla)$ with~$\bA_{\rm
log,L}^\nabla$;

\item[ii'.] the natural inclusion $\bA_{\rm cris,L}^\nabla\subset
\bA_{\rm log,L}$  identifies $\Ker\bigl(\nabla_{\WW(k)}\bigr)$
with~$\bA_{\rm cris,L}^\nabla$;

\item[iii.] {\rm (Griffiths' transversality)} we have
$\nabla_T\left(\Fil^r\bigl(\bA_{\rm log,L}\bigr)\right)\subset
\Fil^{r-1}\left(\bA_{\rm log,L}\right)\otimes_{\cO_\Xtilde}
\omega^1_{\Xtilde/T}$ for every~$r$;

\item[iv.] the connection $\nabla_T\colon \bA_{\rm log,L}\lra
\bA_{\rm log,L}\otimes_{\cO_\Xtilde} \omega^1_{\Xtilde/T}$ is
quasi--nilpotent;

\item[v.] let~$U$ be   small  and choose a Frobenius
$F_{\widetilde{U}}$ on a formal chart of $\widetilde{U}_{\rm
form}$. Then, Frobenius~$\varphi_U$ on~$\bA_{\rm
log,L}\vert_{(U,U_L)}$ is horizontal with respect
to~$\nabla_T\vert_{(U,U_L)}$ i.~e., $\nabla_T\vert_{(U,U_L)} \circ
\varphi_U=\bigl(\varphi_U\tensor d F_{\widetilde{U}}\bigr)\circ
\nabla_T\vert_{(U,U_L)}$.

\end{enumerate}

\end{proposition}
\begin{proof} The proof is formal and follows from the explicit description given
in \ref{prop:loclstructure}. We refer to \cite[Prop.~2.37]{andreatta_iovita_comparison} for details.

\end{proof}

\subsubsection{The sheaves $\bB_{\rm log}^\nabla$ and $\bB_{\rm log}$}
\label{sec:BlogBlognabla}

In this section we denote by $\bA$ any one of $\bA_{\rm cris,L}^\nabla$, $\bA_{\rm log,L}^\nabla$ or $\bA_{\rm log,L}$. For every integer $r$ define the continuous
sheaf $\bA(r):=\Z_p(r) \tensor_{\Z_p} \bA$ (thought of as ``$\bA t^{-r} $") with filtration $\Fil^i \bA(r):= \Z_p(r)\tensor_{\Z_p} \Fil^{i+r}\bA$ for $i\in\Z$. Let
the (local) Frobenius  $\varphi_r\colon \bA(r) \to \bA(pr)$ be defined by $p^{-r}$ times  the (local) Frobenius on $\bA$ (coming from the fact that
$\varphi(t)=pt$). This makes sense since $p^{-1}=(p-1)! \frac{t^{[p]}}{t^p}\in A_{\rm cris} \cdot t^{-p}$ so that  $p^{-r}$ is a well defined element of $\bA(pr)$.
Let the connection $$\nabla^{i+1}_\cO(r) \colon \bA_{\rm log,L}(r)\tensor_{\cO_\Xtilde}\omega^i_{\Xtilde/\cO} \lra \bA_{\rm
log,L}(r)\tensor_{\cO_\Xtilde}\omega^{i+1}_{\Xtilde/\cO}$$be the one obtained from the one $\bA_{\rm log,L}\tensor_{\cO_\Xtilde}\omega^i_{\Xtilde/\cO}$.  One
defines $\nabla^{i+1}_{\WW(k)}(r)$ in a similar way.

As in \cite[\S 2.8]{andreatta_iovita_comparison} one proves that, given integers $r\geq s$,  multiplication
by $t^{r-s}$ provide morphisms $\iota_{r,s}\colon
\bA(s)\to \bA(r)$ which respect all the above structures and satisfy $\iota_{u,r}\circ \iota_{r,s}=
\iota_{u,s}$ for integers $u\geq r \geq s$. Define $$\bB_{\rm
cris,L}^\nabla, \qquad \bB_{\rm log,L}^\nabla, \qquad \bB_{\rm log,L}$$in the category
${\rm Ind}\left(\Sh(\fX_M)^\N \right)$ of inductive systems of continuous
sheaves as the inductive systems of the sheaves $\bA_{\rm cris,L}^\nabla(r)$,
(resp.~$\bA_{\rm log,L}^\nabla(r)$, resp.~$\bA_{\rm log,L}(r)$) with respect to the
morphisms $\iota_{r,s}$ for $s\leq r$. They are endowed with a descending filtration $\Fil^n \bB_{\rm cris,L}^\nabla$,
$\Fil^n \bB_{\rm log,L}^\nabla$ and $\Fil^n
\bB_{\rm log,L}$ defined by the inductive systems $\Fil^n \bA_{\rm cris,L}^\nabla(r)$,
$\Fil^n \bA_{\rm log,L}^\nabla(r)$ (resp.~$\Fil^n \bA_{\rm log,L}(r)$) for
varying $r\in\Z$. Moreover, $\bB_{\rm cris,L}^\nabla$ and $\bB_{\rm log,L}^\nabla$
(resp.~$\bB_{\rm log,L}\vert_{(U,U_L)}$ for $U\in X^{\ket}$ small) are each endowed
with a Frobenius defined as the inductive limits of the Frobenii $\varphi_r$. We also get de Rham complexes

$$  \bB_{\rm  log,L} \stackrel{\nabla^1_T} {\lra}
\bB_{\rm  log,L} \otimes_{\cO_\Xtilde}\omega^1_{\Xtilde/T} \stackrel{\nabla^2_{T}}{\lra} \bB_{\rm  log,L}
\otimes_{\cO_\Xtilde}\omega^2_{\Xtilde/T} \lra \cdots $$for
$T=\cO$ or $T=\WW(k)$. As in \cite[Lemma 2.41]{andreatta_iovita_comparison} one proves the following:

\begin{lemma}\label{lemma:propbBcris} (1) Multiplication by $p$ is an isomorphism
on $\Fil^n \bB_{\rm cris,L}^\nabla$, $\bB_{\rm cris,L}^\nabla$, $\Fil^n \bB_{\rm log,L}^\nabla$, $\bB_{\rm log,L}^\nabla$,
$\Fil^n \bB_{\rm log,L}$ and $\bB_{\rm log,L}$. \smallskip

(2) For every $r\in\Z\cup\{-\infty\}$, putting $\Fil^{-\infty}
\bB_{\rm log,L}^\nabla=\bB_{\rm log,L}^\nabla$ and $\Fil^{-\infty}
\bB_{\rm log,L}=\bB_{\rm log,L}$ and $\nabla^i:=\nabla^i_\cO$, we
have exact sequences of inductive systems

$$0\lra \Fil^r
\bB_{\rm log,L}^\nabla \lra \Fil^r \bB_{\rm  log,L}
\stackrel{\nabla^1}{\lra}  \Fil^{r-1} \bB_{\rm  log,L}
\otimes_{\cO_\Xtilde}\omega^1_{\Xtilde/\cO}
\stackrel{\nabla^2}{\lra} \Fil^{r-2} \bB_{\rm  log,L}
\otimes_{\cO_\Xtilde}\omega^2_{\Xtilde/\cO} \lra \cdots $$and
$$0\lra \Fil^r
\bB_{\rm cris,L}^\nabla \lra \Fil^r \bB_{\rm  log,L}
\stackrel{\nabla^1_{\WW(k)}}{\lra}  \Fil^{r-1} \bB_{\rm  log,L}
\otimes_{\cO_\Xtilde}\omega^1_{\Xtilde/\WW(k)}
\stackrel{\nabla^2_{\WW(k)}}{\lra} \Fil^{r-2} \bB_{\rm  log,L}
\otimes_{\cO_\Xtilde}\omega^2_{\Xtilde/\WW(k)}\lra  \cdots .$$

(3) for $U\in X^{\ket}$ small,  Frobenius~$\varphi_\cU$
on~$\bB_{\rm log,L}\vert_{(U,U_L)}$ is horizontal with respect
to~$\nabla\vert_{(U,U_L)}$ and induces Frobenius on $\bB_{\rm
log,L}^\nabla\vert_{(U,U_L)}$.  \smallskip

(4) for $U\in X^{\ket}$ small, $\bB_{\rm
log,L}^{\nabla}(\Rbar_U)\cong B_{\rm log}^\nabla(\widetilde{R}_U)$
and $\bB_{\rm log,L}(\Rbar_\cU)\cong B_{\rm
log}(\widetilde{R}_U)$, as defined in \S\ref{sec:AcrisnablaRbar},
compatibly with Frobenius, filtrations, $\cG_{U_L}$-action and
connections.

\end{lemma}

\subsubsection{The sheaves $\overline{\bB}_{\rm log,\Kbar}^\nabla$ and $\overline{\bB}_{\rm log,\Kbar}$}\label{sec:defBbarlog}

Recall from \S\ref{sec:notation} that we have a natural map $f_\pi\colon B_{\rm log}\to B_{\rm dR}$,
with image $\overline{B}_{\rm log}$, sending $Z$ to $\pi$. Let $\overline{A}_{\rm log}$ be the image of $A_{\rm log}$.
Define $\overline{\bA}_{\rm log,\Kbar}^\nabla$  and $\overline{\bB}_{\rm log,\Kbar}^\nabla$ as the
quotient $\bA_{\rm log,\Kbar}^\nabla \otimes_{A_{\rm log}} \overline{A}_{\rm log} $ and
$\bB_{\rm log,\Kbar}^\nabla \otimes_{B_{\rm log}} \overline{B}_{\rm log} $, respectively, with image filtration.
Due to \S\ref{sec:Alognabla} we have
isomorphisms

$$\overline{\bA}_{\rm log,\Kbar}^\nabla\cong  \bA_{\rm inf,\Kbar}^+\otimes_{\WW(k)}
\overline{A}_{\rm log},\qquad \overline{\bB}_{\rm log,\Kbar}^\nabla\cong  \bA_{\rm inf,\Kbar}^+\otimes_{\WW(k)}
\overline{B}_{\rm log}$$ in $\Sh(\fX_\Kbar)^\N$ and in ${\rm Ind}\left(\Sh(\fX_\Kbar)^\N \right)$ respectively.
These isomorphisms preserve the filtrations. Similarly, define $$\overline{\bA}_{\rm log,\Kbar}:=\bA_{\rm log,\Kbar}
\otimes_{A_{\rm log}} \overline{A}_{\rm log},\quad \overline{\bB}_{\rm log,\Kbar}:=\bB_{\rm log,\Kbar}
\otimes_{B_{\rm log}} \overline{B}_{\rm log}$$with image filtration. As $f_\pi(\cO)=\cO_K$, we have
$\cO_\Xtilde\widehat{\otimes}_{\cO} \overline{A}_{\log}=\cO_X\otimes_{\cO_K} \overline{A}_{\log}$.
In particular $\overline{\bA}_{\rm log,\Kbar}$ and $\overline{\bB}_{\rm log,\Kbar}$ are $\cO_X\widehat{\otimes}_{\cO_K}
\overline{A}_{\log}$, resp.~$\cO_X\widehat{\otimes}_{\cO_K} \overline{B}_{\log}$-modules.
Due to \ref{prop:deRhamcomplex} and \ref{lemma:propbBcris} they are endowed with connections relative to
$\overline{A}_{\rm log}$, resp.~$\overline{B}_{\rm log}$ and the filtrations satisfies Griffiths' transversality.
Set $\Fil^{-\infty} \overline{\bA}_{\rm log,\Kbar}=\overline{\bA}_{\rm
log,\Kbar}$ and similarly for $\overline{\bA}_{\rm log,\Kbar}^\nabla$, $\overline{\bB}_{\rm
log,\Kbar}$ and $\overline{\bB}_{\rm log,\Kbar}^\nabla$.

\begin{lemma}\label{prop:deRhamcomplexBbarlog} (i) We have  $\overline{\bB}_{\rm cris,\Kbar}^\nabla\otimes_{K_0}
K \subset  \overline{\bB}_{\rm log,\Kbar}^\nabla$.\smallskip

(ii) Using the notation of \ref{prop:loclstructure} we have $$\overline{\bA}_{\rm log,\Kbar,n}
\vert_{(U,Z_n)} \cong\overline{\bA}_{\rm log,\Kbar,n}^\nabla\left\langle
v_{2,n}-1,\ldots,v_{a,n}-1,w_{1,n}-1,\ldots,w_{b,n}-1
\right\rangle;$$

(iii)  The de Rham complexes $0\lra \Fil^r \overline{\bA}_{\rm
log,\Kbar}^\nabla \lra \Fil^r \overline{\bA}_{\rm
log,\Kbar} \stackrel{\nabla}{\lra}
\Fil^r \overline{\bA}_{\rm
log,\Kbar} \otimes_{\cO_X}
\omega^1_{X/\cO_K} {\lra} \cdots$ and $0\lra \Fil^r \overline{\bB}_{\rm
log,\Kbar}^\nabla \lra \Fil^r \overline{\bB}_{\rm
log,\Kbar} \stackrel{\nabla}{\lra}
\Fil^{r-1} \overline{\bB}_{\rm log,\Kbar} \otimes_{\cO_X}
\omega^1_{X/\cO_K} {\lra} \cdots$ are exact for $r\in \Z\cup \{-\infty\}$;

\end{lemma}
\begin{proof} (i) It follows from \ref{lemma:BlogmodPpi}.

(ii) It is the analogue of \ref{prop:loclstructure}. The details are left to the reader.

(iii) The fact that we have sequences follows from \ref{prop:deRhamcomplex}(i) and  \ref{lemma:propbBcris}(2).
The exactness follows from (ii).

\end{proof}

\subsubsection{The monodromy
diagram}\label{section:monodromydiagram}

Consider the exact sequence $0 \lra \cO_\Xtilde \frac{dZ}{Z} \lra \omega^1_{\Xtilde/\WW(k)} \lra \omega^1_{\Xtilde/\cO} \lra 0$. It induces for every $i\geq 1 $ an
exact sequence $0 \lra \omega^{i-1}_{\Xtilde/\cO}\wedge \frac{dZ}{Z} \lra \omega^i_{\Xtilde/\WW(k)} \lra \omega^i_{\Xtilde/\cO} \lra 0$. This implies that the
following sequence of complexes is exact for every $n\in \Z\cup \{-\infty\}$:
$$0 \lra \Fil^{n-1} \bB_{\rm log,L}
\otimes_{\cO_\Xtilde} \omega^{\bullet-1}_{\Xtilde/\cO}\wedge
\frac{dZ}{Z}   \lra \Fil^n \bB_{\rm log,L} \otimes_{\cO_\Xtilde}
\omega^{\bullet}_{\Xtilde/\WW(k)}  \lra \Fil^n \bB_{\rm log,L}
\otimes_{\cO_\Xtilde} \omega^{\bullet}_{\Xtilde/\cO} \lra 0,
$$where $\omega^{i}_{\Xtilde/\cO}$ and
$\omega^{i}_{\Xtilde/\WW(k)}$ are set to be $0$ for $i<0$ and we
define $\Fil^n \bB_{\rm log,L}=\bB_{\rm log,L}$ for $n=-\infty$.
Taking the homology and using lemma \ref{lemma:propbBcris} we get the
exact sequence (of complexes)
$$0  \lra \Fil^{n-1} \bB_{\rm log,L}^\nabla
\frac{dZ}{Z}[-1]\lra \Fil^n \bB_{\rm cris,L}^\nabla  \lra \Fil^n
\bB_{\rm log,L}^\nabla  \lra 0.$$We let $N\colon \Fil^n \bB_{\rm
log,L}^\nabla\lra \Fil^n \bB_{\rm log,L}^\nabla$ be the morphism
defined by $d = N \frac{dZ}{Z}$. Then,\smallskip

(i) we have $N \circ \varphi=p \varphi\circ N$ on $\bB_{\rm
log,L}^\nabla$;\smallskip

(ii)  $N$ is surjective on $\Fil^n \bB_{\rm log,\Kbar}^\nabla$
with kernel $\Fil^n \bB_{\rm cris,\Kbar}^\nabla$. \smallskip

Indeed, it follows from \ref{prop:deRhamcomplex} that $\bA_{\rm
cris,\Kbar}^{\nabla} $ is the kernel of the monodromy operator on
$\bA_{\rm log,\Kbar}^\nabla$. We deduce that $\bB_{\rm
cris,\Kbar}^{\nabla} $ is the kernel of the monodromy operator
$N\colon \bB_{\rm log,\Kbar}^\nabla\lra \bB_{\rm
log,\Kbar}^\nabla$. Moreover, the monodromy operator $N$ is
surjective on $\Fil^\bullet \bA_{\rm log,\Kbar}^{\nabla}$ and,
hence, on $\Fil^\bullet \bB_{\rm log,\Kbar}^{\nabla}$ by the
explicit description of $\bA_{\rm log,\Kbar}^{\nabla}$ given in
\S\ref{sec:Alognabla}. By loc.~cit.~the inclusion $\bA_{\rm
cris,\Kbar}^{\nabla} \subset \bA_{\rm log,\Kbar}^\nabla$ is strict
so that the kernel of $N$ on $\Fil^n \bB_{\rm log,\Kbar}^\nabla$
is $\Fil^n \bB_{\rm cris,\Kbar}^\nabla$ as claimed.

\subsubsection{The fundamental exact diagram}\label{section:fundamentaldiagram}
Let us assume that we are in the formal case. The following commutative
diagram called the {\it crystalline fundamental diagram of
sheaves} has exact rows:

$$
\begin{array}{ccccccccccc} 0& \lra & \Q_p & \lra &\Fil^0
\bB_{\rm cris,\Kbar}^{\nabla} &\stackrel{1-\varphi}{\lra}&
\bB_{\rm
cris,\Kbar}^{\nabla}&\lra& 0\\
&&\cap&&\cap&&\Vert \\
0&\lra&(\bB_{\rm cris,\Kbar}^{\nabla})^{\varphi=1}&\lra&\bB_{\rm
cris,\Kbar}^{\nabla}&\stackrel{1-\varphi}{\lra}&\bB_{\rm
cris,\Kbar}^{\nabla}& \lra & 0.\cr
\end{array}$$We refer to
\cite[\S 2.9]{andreatta_iovita_comparison} for the proof.

We now consider the following diagram called the {\it fundamental
diagram of sheaves}:

\begin{equation}\label{display:fundamentalexactdiagram}
\begin{array}{ccccccccccc} 0 \lra & \Q_p & \lra & \Fil^0 \bB_{\rm
cris,\Kbar}^\nabla & \stackrel{(\varphi-1,N)}{\lra} &  \bB_{\rm
log,\Kbar}^\nabla \oplus  \bB_{\rm cris,\Kbar}^\nabla  &
\stackrel{(N,1-p\varphi)}{\lra} &\bB_{\rm log,\Kbar}^\nabla& \lra 0\\
&\cap&&\cap&&\cap &&\Vert\\
0 \lra & \bB_{\rm cris,\Kbar}^{\nabla,\varphi=1} & \lra & \bB_{\rm log,\Kbar}^\nabla &
\stackrel{(\varphi-1,N)}{\lra} & \bB_{\rm log,\Kbar}^\nabla\oplus \bB_{\rm
log,\Kbar}^\nabla  & \stackrel{(N,1-p\varphi)}{\lra} & \bB_{\rm log,\Kbar}^\nabla& \lra 0.\cr$$
\end{array}\end{equation}

\begin{lemma} Both rows in the fundamental diagram are exact sequences. \end{lemma}
\begin{proof}
Since $N \circ \varphi=p \varphi\circ N$, the rows define
sequences. It follows from \ref{section:monodromydiagram} that
$(\bB_{\rm cris,\Kbar}^{\nabla})^{\varphi=1}\cong \bB_{\rm
log,\Kbar}^{\nabla,N=0,\varphi=1}$, and similarly for $\Fil^0$,
and that $N$ is surjective on $\bB_{\rm log,\Kbar}^{\nabla}$. This
and the exactness in the crystalline fundamental diagram  imply
the exactness on the  left and on the right of both rows in the
fundamental diagram.

Since $N$ is surjective on $\bB_{\rm log,\Kbar}^{\nabla}$, to
prove the exactness at $ \bB_{\rm log,\Kbar}^\nabla\oplus \bB_{\rm
log,\Kbar}^\nabla$ of the second row it suffices to show that
$\varphi-1$ sends  $\bB_{\rm log,\Kbar}^{\nabla,N=0}$, which is
$\bB_{\rm cris,\Kbar}^{\nabla}$, surjectively onto $\bB_{\rm
log,\Kbar}^{\nabla,N=0}\cong \bB_{\rm cris,\Kbar}^{\nabla}$. This
follows from the exactness in the middle of the second row of the
crystalline fundamental diagram. We deduce that the second row in
the fundamental diagram is exact. The exactness of the first row is proven similarly.

\end{proof}

\subsubsection{Cohomology of $\bB_{\rm log}$ and $\overline{\bB}_{\rm log}$
}\label{section:cohBlog}  Let $\cO_{\Xtilde,\rm log}^{\rm geo}$ to be the image of $\cO_\Xtilde\widehat{\otimes}_{\cO}
B_{\rm log} \to v_{\Kbar,\ast}^{\rm cont} \bB_{\rm log,\Kbar}$ with image filtration, considering the composite of
the filtration on $B_{\rm log}$ and the $P_\pi(Z)$-adic filtration on $\cO_{\Xtilde}$. Due to
\ref{cor:injstrictinvariants}(3) it is a direct factor in $\cO_\Xtilde\widehat{\otimes}_{\cO} B_{\rm log}$.
The aim of this section is to prove the
following result. Write $\Fil^{-\infty} \bB_{\rm log,\Kbar}:=
\bB_{\rm log,\Kbar} $ and $\Fil^{-\infty} \overline{\bB}_{\rm log,\Kbar}:=
\overline{\bB}_{\rm log,\Kbar} $ with the notation of \S\ref{sec:defBbarlog}.

\begin{proposition}\label{prop:crysisacyclic} For every $r\in \Z\cup\{-\infty\}$
we have:

$$
{\rm R}^j v_{\Kbar,\ast}^{\rm cont} \left(\Fil^r \bB_{\rm log,\Kbar}
\right)=\begin{cases} 0 & \hbox{{\rm if }} j\geq 1 \cr \Fil^r \cO_{\Xtilde,\rm log}^{\rm geo} &
\hbox{{\rm if }} j=0 .\cr
\end{cases}$$Similarly

$$
{\rm R}^j v_{\Kbar,\ast}^{\rm cont} \left( \Fil^r
\overline{\bB}_{\rm log,\Kbar} \right) =\begin{cases} 0 & \hbox{{\rm if }} j\geq 1 \cr
\Fil^r\left(\cO_{\Xtilde,\rm log}^{\rm geo}\otimes_{B_{\rm log}}  \overline{B}_{\rm log}\right)&
\hbox{{\rm if }} j=0 .\cr
\end{cases}$$
\end{proposition}
\begin{proof} We prove the first statement.
Lemma \ref{lemma:fiisisoo}, \S\ref{sec:Alog} and \ref{thm:geometricacyclicity} imply that $ {\rm R}^j v_{\Kbar,\ast}^{\rm cont}
\left(\bA_{\rm log,\Kbar}(m) \right)$
is annihilated by a power of $t$  if  $j\geq 1$. From loc.~cit.~and \ref{cor:injstrictinvariants}(3) we also get the statement
for $r=-\infty$ and for ${\rm R}^0 v_{\Kbar,\ast}^{\rm cont} \left(\Fil^r
\bB_{\rm log,\Kbar} \right)$. For the statement concerning the vanishing of ${\rm R}^j v_{\Kbar,\ast}^{\rm cont}
\left(\Fil^r \bB_{\rm log,\Kbar} \right)$ for
$j\geq 1$ we argue as in \S\ref{section:invariantsFilBlog}.  As $t$ annihilates ${\rm Gr}^r \bA_{\rm log,\Kbar}(m)$, we conclude
that $ {\rm R}^j v_{\Kbar,\ast}^{\rm cont} \Fil^r \left(\bA_{\rm log,\Kbar}(m) \right)$ is annihilated by a power $t^N$ of $t$
depending on $m$ and $r$. Hence, the image of $ {\rm R}^j v_{\Kbar,\ast}^{\rm cont} \Fil^r \left(\bA_{\rm log,\Kbar}(m) \right)$
in $ {\rm R}^j v_{\Kbar,\ast}^{\rm cont} \Fil^{r-N} \left(\bA_{\rm log,\Kbar}(m+N) \right)$ is $0$. We are then left to prove that the
kernel of the map ${\rm R}^j v_{\Kbar,\ast}^{\rm cont} \Fil^{r} \left(\bA_{\rm log,\Kbar}(m+N)
\right)\to {\rm R}^j v_{\Kbar,\ast}^{\rm cont} \Fil^{r-N} \left(\bA_{\rm log,\Kbar}(m+N) \right)$ is annihilated by a power of $p$.
Proceeding by induction on $N$, it suffices to show that the cokernel of   ${\rm R}^{j-1} v_{\Kbar,\ast}^{\rm cont} \Fil^{r}
\left(\bA_{\rm log,\Kbar}(m) \right)\to {\rm R}^{j-1}
v_{\Kbar,\ast}^{\rm cont} {\rm Gr}^r \left(\bA_{\rm log,\Kbar}(m) \right)$ is annihilated by a power of $p$.
The latter cohomology group can be computed using \ref{lemma:fiisisoo} which implies that ${\rm R}^i v_{\Kbar,\ast}^{\rm cont}
{\rm Gr}^r \left(\bA_{\rm log,\Kbar}(m) \right)\cong {\rm H}_{\rm Gal}^i\bigl({\rm
Gr}^{r} \bA_{\rm log,\Kbar}(m)\bigr)$. It is sufficient to prove that for every $i\in\N$ the cokernel of the composite map
$${\rm H}_{\rm Gal}^i\bigl(\Fil^{r} \bA_{\rm log,\Kbar}(m)\bigr) \lra {\rm
R}^{i} v_{\Kbar,\ast}^{\rm cont} \Fil^{r} \left(\bA_{\rm log,\Kbar}(m) \right)\lra {\rm H}_{\rm Gal}^i\bigl({\rm Gr}^{r}
\bA_{\rm log,\Kbar}(m)\bigr) $$is
annihilated by a power of $p$. This is proved in \ref{lemma:invariantsFilBlog}.

The second statement is proved similarly.
\end{proof}

In the proof of proposition \ref{prop:crysisacyclic} we used the following lemma:

\begin{lemma}\label{lemma:fiisisoo} The assumptions in \ref{prop:fiisisoo} hold for the
sheaves (a) $\widehat{\cO}_{\fX_\Kbar}$;\enspace (b) $\bA_{\rm
log,\Kbar}(m)$ for every $m\in\Z$; \enspace (c) $\cF={\rm
Gr}^r\bA_{\rm log,\Kbar}(m)$ for every $m$ and $r\in \Z$; (d) $\overline{\bA}_{\rm
log,\Kbar}(m)$ for every $m\in\Z$; \enspace (e) $\cF={\rm
Gr}^r \overline{\bA}_{\rm log,\Kbar}(m)$ for every $m$ and $r\in \Z$.
\end{lemma}
\begin{proof} See \S\ref{sec:defBbarlog} for the definition of $\overline{\bA}_{\rm
log,\Kbar}(m)$.  Proceeding as in  \cite[Prop.~3.18]{andreatta_iovita_comparison} one reduces to the proof for
$\widehat{\cO}_{\fX_\Kbar}$ and the sheaves $\WW_{s,\Kbar}$, $s\in\N$, using \ref{prop:loclstructure} in cases
(b) and (c) and using \ref{prop:deRhamcomplexBbarlog}(ii) for cases (d) and (e). For  $\widehat{\cO}_{\fX_\Kbar}$ and  $\WW_{s,\Kbar}$
the proof follows arguing as in \cite[Thm 6.16(A)\&(B)]{andreatta_iovita}.

\end{proof}

Let $U\in X^{\ket}$ be a small object and let $\varphi$ be a Frobenius
on $\bB_{\rm log,K}\vert_{(U,U_K)}$.

\begin{lemma}\label{lemma:invariantsuptoFrobenius} There is a power $s$ of the
Frobenius morphism on $v_{K,\ast}^{\rm cont} \left(\bigl(\bB_{\rm log,K}\bigr)\vert_{(U,U_K)}\right)$, depending on the prime $p$,
which factors via the natural
inclusion $ \widehat{\cO}_\Xtilde^{\rm DP}\bigl[p^{-1}\bigr]\vert_{U}\subset v_{K,\ast}^{\rm cont}
\left(\bigl(\bB_{\rm log,K}\bigr)\vert_{(U,U_K)}\right)$. In fact
$s=1$ for $p\geq 3$ and $s=2$ for $p=2$.
\end{lemma}
\begin{proof} This follows from \ref{lemma:propbBcris} and
\ref{prop:arithemticinvariantsB}.

\end{proof}

\subsection{Semistable sheaves and their
cohomology}\label{sec:sssheaves}

As before we fix an extension $K\subset L \subset
\Kbar$.\smallskip

{\em $\Q_p$-adic \'etale  sheaves.} By a $p$--adic sheaf~$\cL$ on~$\fX_L^{\et}$ we mean a continuous system $\{\cL_n\}\in \Sh(\fX_L)^\N$ such that~$\cL_n$ is a locally constant sheaf of $\Z/p^n\Z$--modules, free of finite rank,  and~$\cL_n=\cL_{n+1}/p^n \cL_{n+1}$ for every~$n\in\N$. It is an abelian  tensor
category. Define $\Sh(\fX_L)_{\Q_p}$ to be the full subcategory of ${\rm Ind}\left(\Sh(\fX_L^{\et})^\N\right)$ consisting of inductive systems of the form
$(\cL)_{i\in\Z}$ where $\cL$ is a $p$-adic \'etale sheaf and the transition maps $\cL\to \cL$ are given by multiplication by $p$. It inherits from the category of
$p$-adic sheaves on~$\fX_L$ the structure of an abelian tensor category.

\subsubsection{The functor $\bDcrisgeo$}\label{sec:Dcrisgeo}

Given a $\Q_p$--adic sheaf $\cL$ on~$\fX_{\Kbar}$ define
$$\bDcrisgeo(\cL):=
v_{\Kbar,\ast}\Bigl(\cL\tensor_{\Z_p} \bB_{\rm log,\Kbar} \Bigr).$$It is a sheaf of $\cO_{\Xtilde,\rm log}^{\rm geo}$-modules
in~$\Sh\bigl(X^{\ket}\bigr)$, see \S\ref{section:cohBlog} for the notation. We get a functor
$$\bDcrisgeo\colon  \Sh\bigl(\fX_\Kbar\bigr)_{\Q_p} \lra  \Mod\left(\cO_{\Xtilde,\rm log}^{\rm geo}\right).$$ Then,\smallskip

(1) $\bDcrisgeo(\cL)$ is endowed with a decreasing filtration
$\Fil^n \bDcrisgeo(\cL):=v_{L,\ast}\left(\cL\tensor_{\Z_p}
\Fil^n\bB_{\rm log,\Kbar} \right)$ for $n\in\Z$;\bigskip

(2) $\bDcrisgeo(\cL)$ is endowed with a connection
$$\nabla_{\cL,\WW(k)}\colon \bDcrisgeo(\cL) \lra
\bDcrisgeo(\cL)\otimes_{\cO_\Xtilde} \omega^1_{\Xtilde/\WW(k)}$$
defined by  $v_{\Kbar,\ast}\bigl(1\otimes \nabla^1_{\WW(k)}\bigr)$
where $\nabla^1_{\WW(k)}$ is the connection on  $\bB_{\rm
log,\Kbar}$; \smallskip

(3) for every $U$ small and a choice of Frobenius on
$\widetilde{U}_{\rm form}$ we have a Frobenius operator
$\varphi_{\cL,U}\colon \bDcrisgeo(\cL)\vert_U\lra
\bDcrisM(\cL)\vert_U$ defined as $v_{\Kbar,\ast}\left(1\otimes
\varphi_U\right)$ where $\varphi_U$ is the Frobenius on  $\bB_{\rm
log,\Kbar}\vert_{(U,U_L)}$.

\subsubsection{Geometrically semistable sheaves}\label{sec:defsemgeositablesheaf}
A $\Q_p$-adic sheaf~$\cL=\{\cL_n\}_n$ on~$\fX_\Kbar$ is called
{\it geometrically semistable} if

\begin{enumerate}

\item[i.] there exists a coherent $\cO_\Xtilde
\widehat{\otimes}_{\cO} A_{\rm log}$-submodule $D(\cL)$ of
$\bDcrisgeo\bigl(\cL\bigr)$ such that:\smallskip

(a) it is stable under the connection $\nabla_{\cL,\WW(k)}$ and $\nabla_{\cL,\WW(k)}\vert_{D(\cL)}$ is  integrable and topologically nilpotent on
$D(\cL)$;\smallskip

(b) $\bDcrisgeo\bigl(\cL\bigr)\cong D(\cL)\tensor_{A_{\rm log}}
B_{\rm log}$;\smallskip

(c) there exist integers $h$ and $n\in\N$ such that for every small affine
$U$ the map $t^h \varphi_{\cL,U}$ sends $D(\cL)\vert_U$ to $D(\cL)\vert_U$ and multiplication by
$t^n$ on $D(\cL)\vert_U$ factors via $t^h \varphi_{\cL,U}$.

\item[ii.] $\bDcrisgeo\bigl(\cL\bigr)$ is locally free of finite
rank on $X^{\ket}$ as $\cO_{\Xtilde,\rm log}^{\rm geo}$-module.

\item[iii.] the natural map $\alpha_{\rm log,\cL}\colon \bDcrisgeo\bigl(\cL\bigr)\tensor_{\bigl(\cO_{\Xtilde,\rm log}^{\rm geo}\bigr)} \bB_{\rm
log,\Kbar}\lra \cL\tensor_{\Z_p} \bB_{\rm log,\Kbar}$ is an isomorphism in the category ${\rm Ind}\left(\Sh(\fX_\Kbar)^\N \right)$.

\end{enumerate}

We let $\Sh(\fX_\Kbar)_{\rm gs}$ be the full subcategory of
$\Q_p$-adic \'etale sheaves on $\fX_\Kbar$ consisting of
geometrically semistable sheaves.

\subsubsection{The functor $\bDcrisar$}\label{sec:Dcrisar}
{\it Assume} that $X$ is a  small affine so that a Frobenius
$F_{\widetilde{X}}$ on $\widehat{\cO}_\Xtilde^{\rm DP}$ and
$\varphi$ on $\bB_{\rm log,K}$ are defined. We get a  map
$v_{K,\ast}\bigl(\bB_{\rm log,K}\bigr)\lra
\widehat{\cO}_\Xtilde^{\rm DP}[p^{-1}]$ induced by $\varphi^{s}$;
cf.~\ref{lemma:invariantsuptoFrobenius}. Given a $\Q_p$--adic
sheaf $\cL$ on~$\fX_K$ define
$$\bDcrisar(\cL):=
v_{K,\ast}\Bigl(\cL\tensor_{\Z_p} \bB_{\rm log,K} \Bigr)\otimes^{\varphi^s}_{v_{K,\ast}\bigl(\bB_{\rm log,K}\bigr)}
\widehat{\cO}_\Xtilde^{\rm DP}[p^{-1}].$$Then,
$\bDcrisar(\cL)$ is a sheaf of $\widehat{\cO}_\Xtilde^{\rm DP}[p^{-1}]$-modules
in~$\Sh\bigl(X^{\ket}\bigr)$. As in \cite[Lemma 3.3]{andreatta_iovita_comparison}
one can prove that the sheaf $\bDcrisgeo(\cL)$ is endowed with an action of\/ $G_K$ and
$$\bDcrisar(\cL)=\bigl(\bDcrisgeo(\cL)\bigr)^{G_K}\otimes^{\varphi^s}_{v_{K,\ast}\bigl(\bB_{\rm
log,K}\bigr)} \widehat{\cO}_\Xtilde^{\rm DP}[p^{-1}].$$It follows
that $\bDcrisar$  defines a functor
$$\bDcrisar\colon \Sh\bigl(\fX_K\bigr)_{\Q_p} \lra
\Mod_{\widehat{\cO}_\Xtilde^{\rm DP}}.$$Moreover,\bigskip

(1) $\bDcrisar(\cL)$ is endowed with a decreasing filtration
$\Fil^n \bDcrisar(\cL)$, for $n\in\Z$, given by the inverse image
of $v_{K,\ast}\left(\cL\tensor_{\Z_p} \Fil^n\bB_{\rm log,K}
\right)$ via the map $\bDcrisar(\cL)\lra
v_{K,\ast}\left(\cL\tensor_{\Z_p} \bB_{\rm log,K} \right)$ induced
by $\varphi^{s}$ on $\bB_{\rm log,K}$;\bigskip

(2) $\bDcrisar(\cL)$ is endowed with a connection
$$\nabla_{\cL,\WW(k)}\colon \bDcrisar(\cL) \lra
\bDcrisar(\cL)\otimes_{\cO_\Xtilde} \omega^1_{\Xtilde/\WW(k)}$$
defined by  $v_{L,\ast}\bigl(1\otimes \nabla^1_{\WW(k)}\bigr)$
where $\nabla^1_{\WW(k)}$ is the connection on  $\bB_{\rm log,L}$.
We write $$\nabla_{\cL,\cO}\colon \bDcrisar(\cL) \lra
\bDcrisar(\cL)\otimes_{\cO_\Xtilde} \omega^1_{\Xtilde/\cO}$$for the
connection induced by the connection $\nabla^1_{\cO}$ on $\bB_{\rm
log,L}$; \bigskip

(3) we have a Frobenius operator $\varphi_{\cL}\colon
\bDcrisar(\cL)\lra \bDcrisar(\cL)$ defined as
$v_{L,\ast}\left(1\otimes \varphi\right)$ where $\varphi$ is the
Frobenius on  $\bB_{\rm log,K}$. By construction it is compatible
with the Frobenius $F_{\widetilde{X}}$ on
$\widehat{\cO}_\Xtilde^{\rm DP}$.\bigskip

{\em Localization of $\Q_p$-adic \'etale  sheaves.}\enspace Let $R_X$ be the algebra underlying the affine (formal) scheme
$X$ and let $\widetilde{R}_X$ be a deformation to
$\cO$ as in \ref{remark:defexists}. The localization $\cL_n(\Rbar_X) $ is given by a free
$\Z_p/p^n\Z$-module with continuous action of $\cG_{X_K}$ which we denote
by $V_X(\cL_n)$. Write $V_X(\cL)=\ds \lim_{\infty \leftarrow n} V_X(\cL_n)$. Define
$${\rm D}_{\rm log}^{\rm geo, cris}
\bigl(V_X(\cL)\bigr):=\left(V_X(\cL)\otimes_{\Z_p} {\rm B}^{\rm
cris}_{\rm log}(\widetilde{R}_X) \right)^{G_{X_\Kbar}} $$and
$${\rm D}^{\rm cris}_{\rm log}
\bigl(V_X(\cL)\bigr):=\left(V_X(\cL)\otimes_{\Z_p} {\rm B}^{\rm cris}_{\rm
log}(\widetilde{R}_X) \right)^{G_{X_K}}\otimes_{B_{\rm
cris}^{\log, \cG_R}}^{\varphi^{s}} \widetilde{R}_{X,\rm
cris}[p^{-1}],$$as in \S\ref{sec:Dcrislogmaxlog}. Since
$\widetilde{R}_{X,\rm cris}[p^{-1}]=\widehat{\cO}_\Xtilde^{\rm
DP}[p^{-1}](X)$, see \S\ref{sec:Alog}, it follows  from
\ref{lemma:propbBcris} that
$$\bDcrisgeo(\cL)(X)
\stackrel{\sim}{\lra} {\rm D}_{\rm log}^{\rm geo, cris}
\bigl(V_X(\cL)\bigr), \qquad \bDcrisar(\cL)(X)
\stackrel{\sim}{\lra} {\rm D}^{\rm cris}_{\rm log}
\bigl(V_X(\cL)\bigr)  $$as $\widetilde{R}_{X,\rm
cris}[p^{-1}]$-modules compatibly with Frobenius, filtrations,
connections.

\subsubsection{Semistable  sheaves}\label{sec:defsemsitablesheaf}
As in the previous section we assume that $X$ is a small affine. Following
\cite[Def.~1.1]{ogus} we denote by ${\rm
Coh}\bigl(\widehat{\cO}_\Xtilde^{\rm DP}\tensor_{\Z_p} \Q_p\bigr)$
the full subcategory of sheaves of
$\widehat{\cO}_\Xtilde^{\rm DP}$-modules isomorphic to
$F\tensor_{\Z_p} \Q_p$ for some coherent sheaf $F$ of
$\widehat{\cO}_\Xtilde^{\rm DP}$-modules on $X^{\ket}$. A
$\Q_p$-adic sheaf~$\cL=\{\cL_n\}_n$ on~$\fX_K$ is called {\it
semistable} if

\begin{enumerate}

\item[i.] $\bDcrisar\bigl(\cL\bigr)$ is in  ${\rm
Coh}\bigl(\widehat{\cO}_\Xtilde^{\rm DP}\tensor_{\Z_p}
\Q_p\bigr)$;

\item[ii.] the natural map $\alpha_{\rm log,\cL}\colon
\bDcrisar\bigl(\cL\bigr)\tensor_{\widehat{\cO}_\Xtilde^{\rm DP}} \bB_{\rm
log,K}\lra \cL\tensor_{\Z_p} \bB_{\rm log,K}$ is an isomorphism in
the category ${\rm Ind}\left(\Sh(\fX_K)^\N \right)$ of inductive
system of continuous sheaves.

\end{enumerate}

We let $\Sh(\fX_K)_{\rm ss}$ be the full subcategory of
$\Q_p$-adic \'etale sheaves on $\fX_K$ consisting of semistable
sheaves.

\begin{proposition}\label{prop:equivcris}
The following are equivalent:\smallskip

1) $\cL$ is semistable (resp.~geometrically semistable);\smallskip

2) for every small object~$U$ of~$X^{\ket}$ the
representation~$V_{U}(\cL)$  is semistable (resp.~geometrically
semistable) in the sense of \ref{prop:equivsemistable}
(resp.~\ref{prop:equivgeosemistable});\smallskip

3) there is a covering $\{U_i\}_i$ of~$X^{\et}$ by small objects  such that~$V_{U_i}(\cL)$ is semistable (resp.~geometrically semistable).\smallskip

In particular, if $\cL$ is a semistable sheaf on $\fX_K$ then
$\beta^\ast(\cL)$ is a geometrically semistable sheaf on
$\fX_\Kbar$ and $\bDcrisgeo\bigl(\beta^\ast(\cL)\bigr)\cong
\beta^\ast\bigl(\bDcrisar(\cL)\widehat{\otimes}_{\widehat{\cO}_\Xtilde^{\rm DP}} \cO_{\Xtilde,\rm log}^{\rm geo}\bigr) $.
\end{proposition}
\begin{proof} We refer to \cite[Prop.~3.7]{andreatta_iovita_comparison}
for the proof of the equivalences of (1), (2) and (3). The last assertion follows from this equivalence and \ref{sec:Vgeosemistable}.

\end{proof}

\subsubsection{The category of filtered Frobenius
isocrystals}\label{sec:isocrystals} Let $\cO_{\rm cris}$ be the $\WW(k)$-divided power envelope of $\cO$ with respect to the kernel of the morphism of
$\WW(k)$-algebras $\cO\to \cO_K$ sending $Z$ to $\pi$. It is endowed with the log structure coming from $\cO$. Following \cite[\S5]{katolog}, consider the site
$\bigl(X_0/\cO_{\rm cris}\bigr)^{\rm cris}_{\rm log}$, where $X_0:=X \times_{\cO_K} \Spec(\cO_K/p\cO_K)$,  consisting of quintuples
$\bigl(U,T,M_T,\iota,\delta\bigr)$ where\smallskip

(a) $U\to X_0$ is Kummer \'etale,\smallskip

(b) $\bigl(T,M_T\bigr)$ is a fine log scheme over $\cO_{\rm cris}$ (with its log structure) in which $p$ is locally nilpotent,\smallskip

(c) $\iota\colon U\to T$ is an exact closed immersion over
$\cO_{\rm cris}$,\smallskip

(d) $\delta$ is DP structure on the ideal defining the closed
immersion $U\subset T$, compatible with the DP structure on
$\cO_{\rm cris}$.

\

We let $\Crys(X_0/\cO)$  be the category of crystals  of finitely presented $\cO_{X_0/\cO_{\rm cris}}$-modules on $\bigl(X_0/\cO_{\rm cris}\bigr)^{\rm cris}_{\rm
log}$, cf.~\cite[Def 6.1]{katolog}.

\smallskip

Given a crystal $\cE$ let $\cE_n$ be the crystal $\cE_n:=\cE/p^n\cE$. It defines a $\cO_\Xtilde^{\rm
DP}/p^n\cO_\Xtilde^{\rm DP}$-module, endowed with integrable connection $\nabla_n$ relative to $\cO_{\rm cris}/p^n \cO_{\rm cris}$; see \cite[Thm.~6.2]{katolog}.
Let $\cE_\Xtilde:=\ds \lim_{\infty \leftarrow n} \cE_n$ be the finitely presented sheaf of $\widehat{\cO}_\Xtilde^{\rm DP}$-modules on $X_0^{\ket}$ with the
connection $\nabla_{\cE_\Xtilde}$ relative to $\cO_{\rm cris}$.
\smallskip

Let $\Isoc(X_0/\cO)$  be the category of {\it isocrystals}, i.e., the full subcategory of the category of inductive systems ${\rm Ind}\bigl(\Crys(X_0/\cO)\bigr)$
consisting of the inductive system $\cE \to \cE \to \cE \to \cdots$ where\enspace (1) $\cE$ is a crystal and the transition maps $\cE\to \cE$ are multiplication by
$p$; \enspace (2) $\cE_\Xtilde\bigl[p^{-1}\bigr]$ is a finite and projective sheaf of $\widehat{\cO}_\Xtilde^{\rm DP}\bigl[p^{-1}\bigr]$-modules locally on
$X_0^{\ket}$.
\smallskip

The absolute Frobenius on $X_0$ and the given Frobenius $\varphi_\cO$ on $\cO$ define a morphism of sites $$F\colon \bigl(X_0/\cO_{\rm cris}\bigr)^{\rm cris}_{\rm
log}\lra \bigl(X_0/\cO_{\rm cris}\bigr)^{\rm cris}_{\rm log}.$$Let ${\rm FIso}(X_0/\cO)$  be the category of {\em $F$-isocrystals} consisting of pairs
$(\cE,\varphi)$ where $\cE$ is an isocrystal and $\varphi\colon F^\ast(\cE) \to \cE$ is an isomorphism  of isocrystals.  \smallskip

We have two natural maps of $\WW(k)$-algebras endowed with log structures:\smallskip

i) $\cO_{\rm cris} \to \cO_K$, sending $Z$ to $\pi$;\smallskip

ii) $\cO_{\rm cris} \to \WW(k)^+$, sending $Z$ to $0$. Here $\WW(k)^+$ is $\WW(k)$ with the log structure associated to $\N \to \WW(k)$ given by $1\mapsto
0$.\smallskip

Both maps are compatible with log structures and divided powers, considering on $\cO_K$ and on $\WW(k)$ the standard DP on the ideal generated by $p$. Given a
crystal $\cE$ we denote by $\cE_{X}$ (resp.~$\cE^+$) the base change of $\cE$ via the map (i) (resp.~(ii)); see \cite[Prop.~5.8]{berthelot_ogus}. In particular,
$\cE_{X}$ defines a sheaf of $\widehat{\cO}_X$-modules endowed with an integrable connection $\nabla_{\cE_X}$ relative to $\cO_K$. Similarly, for an isocrystal
$\cE$ we let $\cE_{X_K}$ be the finite and projective sheaf of $\widehat{\cO}_X\otimes_{\cO_K} K$-modules obtained by base change of $\cE$. It comes equipped with an
integrable connection $\nabla_{\cE_{X_K}}$ defined by $\nabla_{\cE_X}$.  Base changing $\cE$
via the map (ii) we obtain an isocrystal $\cE^+$ in $\Isoc(X_0/\cO_K)$. As the map (ii) is Frobenius equivariant, if $\cE$ is a Frobenius crystal or isocrystal, then $\cE^+$ is also a Frobenius crystal
(resp.~isocrystal). Summarizing, given an isocrystal $\cE\in \Isoc(X_0/\cO)$ we get a composite functor
$$\Isoc(X_0/\cO)\lra {\rm Coh}\bigl(\widehat{\cO}_\Xtilde^{\rm
DP}\tensor_{\Z_p} \Q_p\bigr) \lra {\rm Coh}\bigl(\widehat{\cO}_X\tensor_{\cO_K} K\bigr),\quad \cE \mapsto \cE_\Xtilde\mapsto \cE_{X_K}.$$ Define ${\rm FIso}^{\rm
Fil}(X/\cO)$, called the category of {\em filtered Frobenius isocrystals}, to be the category whose objects are triples $\bigl(\cE,\varphi,\Fil^n \cE_{X_K}\bigr)$
where \smallskip

(a) $\bigl(\cE,\varphi\bigr)$ is an object of ${\rm FIso}(X_0/\cO)$;\smallskip

(b) the connection $\nabla_{\cE_\Xtilde}$ on $\cE_\Xtilde$ lifts to a connection $\nabla_{\cE_\Xtilde,\WW(k)}$ relative to $K_0$ such that Frobenius is horizontal with respect to $\nabla_{\cE_\Xtilde,\WW(k)}$;\smallskip

(c) $\Fil^n \cE_{X_K}$ is an exhaustive and descending filtration by finite and projective $\cO_{X_K}$-modules on  $\cE_{X_K}$ satisfying Griffiths'
transversality.\smallskip

It is naturally a tensor category.

\

{\em Cohomology of isocrystals.} Consider an object $\cE\in \Isoc(X_0/\cO)$. Define ${\rm H}^i\bigl(\bigl(X_0/\cO_{\rm cris}\bigr)^{\rm cris}_{\rm log},\cE\bigr)$
using the formalism of \ref{section:continuoussheaves} for the cohomology of inductive systems. It is a $\cO_{\rm cris}\bigl[p^{-1}\bigr]$-module. If $\cE$ is an
$F$-isocrystal, it is endowed with a  Frobenius $\varphi_{\cE,i}$. Define $${\rm H}^i\bigl(\bigl(X_0/\cO_{\rm cris}\bigr)^{\rm cris}_{\rm
log},\cE\bigr)^{\varphi-{\rm div}}$$as the image of the $\cO_{\rm cris}$-linearization $\varphi_{\cE,i}\otimes_{\cO_{\rm cris}}^\varphi \cO_{\rm cris}$.

Let $\cE_\Xtilde$ be the associated coherent $
\widehat{\cO}_\Xtilde^{\rm DP}$-module with connection
$\nabla_{\cE_\Xtilde}$.  It follows from \cite[Thm.~6.4]{katolog}
that we have a canonical isomorphism
$${\rm H}^i\bigl(\bigl(X_0/\cO_{\rm
cris}\bigr)^{\rm cris}_{\rm log},\cE\bigr)\cong {\rm H}^i_{\rm
dR}\bigl(X_0,\bigl(\cE_\Xtilde,\nabla_{\cE_\Xtilde}\bigr)\bigr)\bigl[p^{-1}\bigr]$$as
$\cO_{\rm cris}\bigl[p^{-1}\bigr]$-modules. Recall that by assumption
$\nabla_{\cE_\Xtilde}$  is the composite of
$\nabla_{\cE_\Xtilde,\WW(k)}$ and the surjection
$\omega^1_{\Xtilde/\WW(k)} \lra \omega^1_{\Xtilde/\cO}$. The exact
sequence
$$0 \lra \cO_\Xtilde \frac{dZ}{Z}\lra \omega^1_{\Xtilde/\WW(k)} \lra
\omega^1_{\Xtilde/\cO} \lra 0$$and the connection
$\nabla_{\cE_\Xtilde,\WW(k)}$ define a long exact sequence of
cohomology groups

$${\rm H}^i_{\rm
dR}\bigl(X_0,\bigl(\cE_\Xtilde,\nabla_{\WW(k)}\bigr)\bigr)\bigl[p^{-1}\bigr]\lra
{\rm H}^i_{\rm
dR}\bigl(X_0,\bigl(\cE_\Xtilde,\nabla_{\cE_\Xtilde}\bigr)\bigr)\bigl[p^{-1}\bigr]
\lra {\rm H}^i_{\rm
dR}\bigl(X_0,\bigl(\cE_\Xtilde,\nabla_{\cE_\Xtilde}\bigr)\bigr)\bigl[p^{-1}\bigr]
\frac{dZ}{Z}.$$In particular, ${\rm H}^i\bigl(\bigl(X_0/\cO_{\rm
cris}\bigr)^{\rm cris}_{\rm log},\cE\bigr)$ is endowed with a
logarithmic connection $\nabla_{\cE,i}$ relative to
$\omega^1_{\cO_{\rm cris}/\WW(k)}=\cO_{\rm cris} \frac{dZ}{Z}$.

\

{\it The relation with rigid cohomology.}\enspace  Frobenius on $\cO$ extends to a map $\cO_{\rm cris} \to \cO_{\rm cris}$ which factors via the natural map
$f\colon \cO_{\rm cris}\to \cO_{\rm max}=\WW[Z]\left\{\frac{P_\pi(Z)}{p} \right\} $; see \ref{rmk:AmaxfrobaAcris}. Let $g\colon \cO_{\rm max} \lra \cO_{\rm cris} $
be the induced map. Let $\Xtilde_{\max}:=\Xtilde\widehat{\otimes}_{\cO} \cO_{\rm max}$, where the completed tensor product is with respect to the $p$-adic topology.
Let $\cE_{\Xtilde,{\rm max}}:=\cE_{\Xtilde}\widehat{\otimes}_{\cO_{\rm cris}} \cO_{\rm max}$ be the base change of $\cE_{\Xtilde}$ via $f$. The connection
$\nabla_{\WW(k)}$ (resp.~$\nabla_{\cE_{\Xtilde}}$) defines a connection $\nabla_{\cE_{\Xtilde,{\rm max}},\WW(k)}$ (resp.~$\nabla_{\cE_{\Xtilde,{\rm max}}}$). Since
$\cE$ is an $F$-isocrystal, the base-change $\cE_{\Xtilde,{\rm max}}\widehat{\otimes}_{\cO_{\rm max}} \cO_{\rm cris}[p^{-1}]$ is isomorphic to
$\cE_{\Xtilde}[p^{-1}]$ as $\widehat{\cO}^{\rm DP}_{\Xtilde}[p^{-1}]$-modules with connection so that we get a natural map of $\cO_{\rm cris}$-modules
$$\alpha\colon {\rm H}^i_{\rm dR}\bigl(\Xtilde_{\max},\bigl(\cE_{\Xtilde,\rm max},\nabla_{\cE_{\Xtilde,\rm max}}\bigr)\bigr)\otimes_{\cO_{\rm max}} \cO_{\rm
cris}\bigl[p^{-1}\bigr]\lra {\rm H}^i_{\rm dR}\bigl(X_k,\bigl(\cE_\Xtilde,\nabla_{\cE_\Xtilde}\bigr)\bigr)\bigl[p^{-1}\bigr].$$The connection $\nabla_{\WW(k)}$
defines a connection $\nabla'$ on ${\rm H}^i_{\rm dR}\bigl(\Xtilde_{\max},\bigl(\cE_{\Xtilde,\rm max},\nabla_{\cE_{\Xtilde,\rm max}}\bigr)\bigr)\bigl[p^{-1}\bigr]$
and $\alpha$ is horizontal with respect to the connections on the two sides.

\begin{proposition}\label{prop:HicrisE}
Assume that $X$ is proper over $\cO_K$ and let $\bigl(\cE,\varphi,\Fil^n \cE_{X_K}\bigr)\in
{\rm FIso}^{\rm Fil}(X/\cO)$. \smallskip

(1) The map $\alpha$ is injective with image ${\rm
H}^i\bigl(\bigl(X_0/\cO_{\rm cris}\bigr)^{\rm cris}_{\rm
log},\cE\bigr)^{\varphi-{\rm div}}$;\smallskip

(2) the connection $\nabla_{\cE,i}$ is horizontal with respect to
Frobenius $\varphi_{\cE,i}$ on ${\rm H}^i\bigl(\bigl(X_0/\cO_{\rm
cris}\bigr)^{\rm cris}_{\rm log},\cE\bigr)$;\smallskip

(3) the module ${\rm H}^i\bigl(\bigl(X_0/\cO_{\rm cris}\bigr)^{\rm
cris}_{\rm log},\cE\bigr)^{\varphi-{\rm div}}$ is finite
and free as $\cO_{\rm cris}\bigl[p^{-1}\bigr]$-module and the
$\cO_{\rm cris}\bigl[p^{-1}\bigr]$-linearization of
$\varphi_{\cE,i}$ is an isomorphism;\smallskip

(4) the base change of ${\rm H}^i\bigl(\bigl(X_0/\cO_{\rm
cris}\bigr)^{\rm cris}_{\rm log},\cE\bigr)^{\varphi-{\rm
div}}$ via $\cO_{\rm cris}\lra K$, sending
$Z$ to $\pi$, is isomorphic to ${\rm H}^i_{\rm
dR}\bigl(X_0,\bigl(\cE_{X_K},\nabla_{\cE_{X_K}}\bigr)\bigr)$ as
$K$-vector space;\smallskip

(5) the base change of ${\rm H}^i\bigl(\bigl(X_0/\cO_{\rm cris}\bigr)^{\rm cris}_{\rm log},\cE\bigr)^{\varphi-{\rm div}}$ via $\cO_{\rm cris}\lra \WW(k)$, sending
$Z$ to $0$, coincides with ${\rm H}^i\bigl(\bigl(X_k/\WW(k)^+\bigr)^{\rm cris}_{\rm log},\cE^+\bigr)$ as $K_0$-modules, compatibly with Frobenius. The residue of
$\nabla_{\cE,i}$ defines a nilpotent operator $N_{\cE,i}$, called the monodromy operator, which satisfies $N_{\cE,i} \circ \varphi=p \varphi\circ
N_{\cE,i}$;\smallskip

(6) there is a unique isomorphism $${\rm
H}^i\bigl(\bigl(X_0/\cO_{\rm cris}\bigr)^{\rm cris}_{\rm
log},\cE\bigr)^{\varphi-{\rm div}}\cong {\rm
H}^i\bigl(\bigl(X_k/\WW(k)^+\bigr)^{\rm cris}_{\rm
log},\cE^+\bigr)\otimes_{\WW(k)} \cO_{\rm
cris}$$compatible with Frobenius and inducing the
identity modulo $Z$. Moreover,  via this isomorphism one has
$\nabla_{\cE,i}=N_{\cE,i} \otimes 1+ 1 \otimes d$.
\end{proposition}
\begin{proof} (2) Follows from the fact that $\nabla_{\cE_{\Xtilde},\WW(k)}$ is
horizontal with respect to Frobenius on $\cE$.

(5) The formula relating $N_{\cE,i}$ and $\varphi_{\cE,i}$ follows
from the fact that $\varphi_{\cE,i}$ is horizontal.\smallskip

Claims (3)--(6) follow if we prove claim (1) and the analogue
(3'), (4') etc., of (3), (4) etc.~for ${\rm H}^i_{\rm
dR}\bigl(\Xtilde_{\max},\bigl(\cE_{\Xtilde,\rm
max},\nabla_{\cE_{\Xtilde,\rm
max}}\bigr)\bigr)\bigl[p^{-1}\bigr]$. For (6) note that by
construction $\alpha$ commutes with Frobenius on the two
sides.\smallskip

(3') First of all, ${\rm H}^i_{\rm
dR}\bigl(\Xtilde_{\max},\bigl(\cE_{\Xtilde,\rm
max},\nabla_{\cE_{\Xtilde,\rm max}}\bigr)\bigr)\bigl[p^{-1}\bigr]$
is finite as $\cO_{\rm max}\bigl[p^{-1}\bigr]$-module. This
follows from the Hodge to de Rham spectral sequence using that the
cohomology of $\cE_{\Xtilde,\rm max}\otimes_{\cO_{\Xtilde{\rm
max}}}\omega^i_{\Xtilde_{\rm max}/\cO_{\rm max}}$ is coherent
since $\Xtilde_{\rm max}\to {\rm Spf}(\cO_{\rm max})$ is proper
and $\cO_{\rm max}$ is noetherian. Secondly, since it is endowed
with a connection $\nabla'$ with respect to the derivation on
$\cO_{\rm max}$, it is a projective $\cO_{\rm
max}\bigl[p^{-1}\bigr]$-module by
\cite[Prop.~8.8]{katz}.\smallskip

(4')-(5') Since $\cE_{\Xtilde,\rm max}\bigl[p^{-1}\bigr]$ is a
projective $\cO_{\Xtilde_{\max}}\bigl[p^{-1}\bigr]$-module, the
formation of $${\rm H}^i_{\rm
dR}\bigl(\Xtilde_{\max},\bigl(\cE_{\Xtilde,\rm
max},\nabla_{\cE_{\Xtilde,\rm
max}}\bigr)\bigr)\bigl[p^{-1}\bigr]$$commutes with base-change
from $\cO_{\rm max}\bigl[p^{-1}\bigr]$. In particular, ${\rm
H}^i_{\rm dR}\bigl(\Xtilde_{\max},\bigl(\cE_{\Xtilde,\rm
max},\nabla_{\cE_{\Xtilde,\rm max}}\bigr)\bigr)\bigl[p^{-1}\bigr]$
 coincides with ${\rm H}^i_{\rm
dR}\bigl(X_k,\bigl(\cE_{X_K},\nabla_{\cE_{X_K}}\bigr)\bigr)$ modulo  $P_\pi(Z)/p$ and
with ${\rm
H}^i\bigl(\bigl(X_k/\WW(k)^+\bigr)^{\rm cris}_{\rm
log},\cE^+\bigr)$ modulo $Z$ (by \cite[Thm.~6.4]{katolog}).\smallskip

(3')(continued) The Frobenius structure on $\cE$ defines a
$\cO_{\rm max} \bigl[p^{-1}\bigr]$-linear map $$F\colon {\rm
H}^i_{\rm dR}\bigl(\Xtilde_{\max},\bigl(\cE_{\Xtilde,\rm
max},\nabla_{\cE_{\Xtilde,\rm max}}\bigr)\bigr)\otimes_{\cO_{\rm
max}}^\varphi \cO_{\rm max} \bigl[p^{-1}\bigr] \lra {\rm H}^i_{\rm
dR}\bigl(\Xtilde_{\max},\bigl(\cE_{\Xtilde,\rm
max},\nabla_{\cE_{\Xtilde,\rm max}}\bigr)\bigr)\bigl[p^{-1}\bigr].
$$
The base change $F_{K_0'}$ of $F$ via any map of $\WW(k)$-algebras $\cO_{\rm max}
\bigl[p^{-1}\bigr]\lra K'_0$, with $K_0 \subset K_0'$  a finite
unramified extension,  is the map induced by Frobenius on the
cohomology of the isocrystal over $\Xtilde_{\rm max}\otimes_{
\cO_{\rm max}} K_0'$ defined by $\bigl(\cE_{\Xtilde,\rm
max},\nabla_{\cE_{\Xtilde,\rm max}}\bigr)\otimes_{ \cO_{\rm max}}
K_0'$. In particular, $F_{K_0'}$ is an isomorphism. Since the
maximal ideals of $\cO_{\rm max} \bigl[p^{-1}\bigr]$ defining an
unramified extension of $K_0$ are dense in $\Spec\bigl( \cO_{\rm
max} \bigl[p^{-1}\bigr]\bigr)$ and since $F$ is a map of
projective $\cO_{\rm max} \bigl[p^{-1}\bigr]$-modules of the same
rank, we conclude that $F$ is an isomorphism.\smallskip

(6') Let $\gamma_0\colon {\rm
H}^i\bigl(\bigl(X_k/\WW(k)^+\bigr)^{\rm cris}_{\rm
log},\cE^+\bigr)[p^{-1}]\to {\rm H}^i_{\rm
dR}\bigl(\Xtilde_{\max},\bigl(\cE_{\Xtilde,\rm
max},\nabla_{\cE_{\Xtilde,\rm max}}\bigr)\bigr)[p^{-1}]$ be any
map of $K_0$-vector spaces inducing the identity modulo $Z$. Its image spans ${\rm H}^i_{\rm
dR}\bigl(\Xtilde_{\max},\bigl(\cE_{\Xtilde,\rm
max},\nabla_{\cE_{\Xtilde,\rm max}}\bigr)\bigr)[p^{-1}]$ as
$\cO_{\rm max} \bigl[p^{-1}\bigr]$-module in a neighborhood of the
maximal ideal defined by $Z=0$. Possibly after composing
$\gamma_0$ with a power of $F$ we may assume that it spans it  as
$\cO_{\rm max} \bigl[p^{-1}\bigr]$-module. Write
$$\gamma=\sum_{n=0}^\infty F^i \circ \gamma_0 \circ F_0^{-i},$$where $F$ and $F_0$ are the two Frobenius morphisms. Fix
a basis ${\mathcal B}$ of ${\rm
H}^i\bigl(\bigl(X_k/\WW(k)^+\bigr)^{\rm cris}_{\rm
log},\cE^+\bigr)[p^{-1}]$ as $K_0$-vector space and take $s\in\N$
such that ${\rm det} F_0\in p^{-s} \WW(k)$. The image of
${\mathcal B}$ via $F\circ \gamma_0 \circ F_0^{-1} -\gamma_0$ is
contained in  $\frac{Z}{p^h} {\rm H}^i_{\rm
dR}\bigl(\Xtilde_{\max},\bigl(\cE_{\Xtilde,\rm
max},\nabla_{\cE_{\Xtilde,\rm max}}\bigr)\bigr)$ for some $h\in\N$
so that the power series $F^n\circ \gamma_0 \circ
F_0^{-n}-F^{n-1}\circ \gamma_0 \circ F_0^{-(n-1)}$ is contained in
$\frac{Z^{p^n}} {p^{h+(n-1)s}} {\rm H}^i_{\rm
dR}\bigl(\Xtilde_{\max},\bigl(\cE_{\Xtilde,\rm
max},\nabla_{\cE_{\Xtilde,\rm max}}\bigr)\bigr)$. Note that
$\frac{Z^{p^n}}{p^{p^n}}\in \cO_{\rm max}$ and $p^n-h+(n-1)s\to
\infty$ for $n\to \infty$. Thus, $\gamma=F\circ \gamma_0 \circ
F_0^{-1} -\gamma_0+\sum_{n=1}^\infty \bigl(F^n\circ \gamma_0 \circ
F_0^{-n}-F^{n-1}\circ \gamma_0 \circ F_0^{-(n-1)}\bigr)$ converges
and $\gamma$ is well defined. By construction $F \circ \gamma=
\gamma_0\circ F_0$ and $\gamma={\rm Id}$ modulo $Z$. This
implies that the image of $\gamma$ spans ${\rm H}^i_{\rm
dR}\bigl(\Xtilde_{\max},\bigl(\cE_{\Xtilde,\rm
max},\nabla_{\cE_{\Xtilde,\rm max}}\bigr)\bigr)[p^{-1}]$ as
$\cO_{\rm max} \bigl[p^{-1}\bigr]$-module. Hence,
$\gamma({\mathcal B})$ is a basis as well which provides the
analogue of the isomorphism in (6).

Given two such morphisms $\gamma$ and $\gamma'$ one argues that
$\gamma-\gamma' =F^n (\gamma-\gamma') F_0^{n-1}$ and the latter
converges to $0$ for $n\to \infty$ so that $\gamma=\gamma'$. For
the last formula in (6') it suffices to show that
$\nabla_{\cE,i}\circ \gamma- \gamma \circ \bigl(N_{\cE,i} \cdot d\log
Z+  d\big)=0$. The difference is $0$ modulo $Z$ and the composite with
$F_0$ has on the one hand the same image, since $F_0$ is an
isomorphism, and on the other hand is $0$ modulo $Z^p$. Iterating
this process we conclude that it is zero modulo $Z^{p^n}$ for
every $n$ and, hence, it must be $0$.\smallskip

(1)  Consider the commutative diagram $$\begin{array}{ccc} {\rm
H}^i_{\rm dR}\bigl(\Xtilde_{\max},\bigl(\cE_{\Xtilde,\rm
max},\nabla_{\cE_{\Xtilde,\rm max}}\bigr)\bigr)\otimes_{\cO_{\rm
max}}^\varphi \cO_{\rm cris}& \stackrel{F\otimes\cO_{\rm cris}}{
\lra} & {\rm H}^i_{\rm
dR}\bigl(\Xtilde_{\max},\bigl(\cE_{\Xtilde,\rm
max},\nabla_{\cE_{\Xtilde,\rm max}}\bigr)\bigr)\otimes_{\cO_{\rm
max}} \cO_{\rm cris}\cr \alpha\otimes^\varphi \cO_{\rm
cris}\big\downarrow & & \big\downarrow\alpha\cr {\rm H}^i_{\rm
dR}\bigl(X_k,\bigl(\cE_\Xtilde,\nabla_{\cE_\Xtilde}\bigr)\bigr)\bigl[p^{-1}\bigr]\otimes_{\cO_{\rm
cris}}^\varphi \cO_{\rm cris} & \stackrel{\varphi_{\cE,i}\otimes
\cO_{\rm cris}}{\lra} & {\rm H}^i_{\rm
dR}\bigl(X_k,\bigl(\cE_\Xtilde,\nabla_{\cE_\Xtilde}\bigr)\bigr)\bigl[p^{-1}\bigr].\cr
\end{array}$$The $\cO_{\rm cris}\bigl[p^{-1}\bigr]$-linearization of
$\varphi_{\cE,i}$ factors via
$\alpha$ as Frobenius on $\cE_\Xtilde$ factors via $\cE_{\Xtilde,\rm
max}$. Moreover,  $F\otimes \cO_{\rm
cris}$ is an isomorphism by (3'). We deduce that
$\alpha\otimes^\varphi \cO_{\rm cris}$ is split injective. Since
the map $\alpha\otimes^\varphi \cO_{\rm cris}$ is injective and
Frobenius $\varphi$ on $\cO_{\rm cris}$ is injective,  we conclude
that $\alpha$ is injective  and the linearization of Frobenius on
its image is an isomorphism. The proposition follows.
\end{proof}

\subsubsection{A geometric variant.}\label{sec:geoinvariant} Let $\overline{X}_0:=X \times_{\cO_K}\Spec\bigl(\cO_\Kbar/p\cO_\Kbar\bigr)$. Let $\bigl(\overline{X}_0/A_{\rm
log}\bigr)^{\rm cris}_{\rm log}$ and $\bigl(\overline{X}_0/A_{\rm
cris}\bigr)^{\rm cris}_{\rm log}$ be the site defined by replacing
$\cO$ with $A_{\rm log}$ with its log structure and divided power
structure (resp.~with $A_{\rm cris}$ with trivial
log structure). Let $A:=A_{\rm log}$ or $A_{\rm cris}$. Then,
proceeding as above, we let $\Crys\bigl(\overline{X}_0/A\bigr)$ be the
category of crystals of finitely presented $\cO_{\overline{X}_0/A}$-modules
on $\bigl(\overline{X}_0/A\bigr)^{\rm cris}_{\rm log}$. Given a crystal $\cE$
let $\cE_\Xtilde$ be the finitely presented sheaf of
$\cO_\Xtilde\widehat{\otimes}_{\cO} A$-modules on $\overline{X}_0^{\ket}$,
endowed with connection $\nabla$ relative to $A$, defined by the
inverse limit $\cE\vert_{(\overline{X}_0\subset \Xtilde_n\otimes_\cO A/p^n
A)}$. Write $B:=B_{\rm log}$ or $B_{\rm cris}$. Put
$\cE_{\Xtilde_B}:=\cE_\Xtilde[t^{-1}]$.

Let $\Isoc\bigl(\overline{X}_0/A\bigr)$, the category of {\em isocrystals}, be the full subcategory of the category of inductive systems ${\rm
Ind}\bigl(\Crys\bigl(\overline{X}_0/A\bigr)\bigr)$ consisting of the inductive system $\cE \to \cE \to \cE \to \cdots$
where \enspace (1) $\cE$ is a crystal and the transition
maps $\cE\to \cE$ are multiplication by $t$ (not by $p$ as before!), \enspace (2) $\cE_{\Xtilde_B}$ is a finite and projective sheaf of
$\cO_{\Xtilde,\rm log}^{\rm geo}$-modules on $\overline{X}_0^{\ket}$ with connections
$\nabla_{E_{\Xtilde_B}}$ relative to $B$. If $B=B_{\rm log}$ consider the map
$f_\pi\colon B_{\rm log}\to \overline{B}_{\rm log}$ sending $Z$ to $\pi$ defined in \ref{lemma:BlogmodPpi}.
Write $\cE_{X_K}:=\cE_{\Xtilde_B}\otimes_{B_{\rm log}}
\overline{B}_{\rm log}$; it is a $\cO_X\widehat{\otimes}_{\cO_K}
\overline{B}_{\rm log}$-module with connection $\nabla_{X_K}$ relative to $\overline{B}_{\rm log}$
obtained from $\nabla_{E_{\Xtilde_{B_{\rm log}}}}$.

The category of $F$-isocrystals ${\rm FIso}\bigl(\overline{X}_0/A\bigr)$
consists of pairs $(\cE,\varphi)$ where $\cE$ is an isocrystal and
$\varphi\colon F^\ast(\cE) \to \cE$ is an isomorphism  of
isocrystals.

Consider on $\cO_X\widehat{\otimes}_{\cO_K} \overline{B}_{\rm log}$ the filtration $\cO_X\widehat{\otimes}_{\cO_K} \Fil^\bullet
\overline{B}_{\rm log}$ defined by
the filtration on $\overline{B}_{\rm log}\subset B_{\rm dR}$ induced by the filtration on $B_{\rm dR}$.
Define by ${\rm FIso}^{\rm Fil}(X/A_{\rm log})$, called the category
of {\em filtered Frobenius isocrystals} the tensor category whose objects as triples $\bigl(\cE,\varphi,\Fil^n \cE_{X_K}\bigr)$ where\smallskip

\noindent (a) $\bigl(\cE,\varphi\bigr)$ is an $F$-isocrystal on
$\bigl(\overline{X}_0/A_{\rm cris}\bigr)^{\rm cris}_{\rm log}$;\smallskip

\noindent (b) $\Fil^n \cE_{X_K}$, for $n\in\Z$, is a descending filtration by $\cO_X\widehat{\otimes}_{\WW(k)} \overline{B}_{\rm log}$-modules on $\cE_{X_K}$ such
that

(i) $\Fil^h \bigl(\cO_X\widehat{\otimes}_{\WW(k)} \overline{B}_{\rm log}\bigr) \cdot \Fil^{n} \cE_{X_K} \lra \Fil^{n+h} \cE_{X_K} $,

(ii) the graded pieces are finite and projective
$\cO_X\widehat{\otimes}_{\cO_K} \C_p$-modules,

(iii) it satisfies Griffiths' transversality with respect to the connection $\nabla_{X_K}$.\smallskip

\

{\em Cohomology of isocrystals.} Consider an object $\cE\in \Isoc(\overline{X}_0/A_{\rm cris})$. As before we view
$\cE$ as an inductive system of inverse systems
$\{\cE_n\}_n\in {\rm Ind}\left(\Sh\left(\bigl(\overline{X}_0/A_{\rm log}\bigr)^{\rm cris}_{\rm log}\right)^{\N}\right)$ and we define
$${\rm H}^i\bigl(\bigl(\overline{X}_0/A_{\rm log}\bigr)^{\rm cris}_{\rm log},\cE\bigr)$$
using the formalism of \ref{section:continuoussheaves}. It is a
$B_{\rm log}$-module, endowed with a  Frobenius $\varphi_{\cE,i}$,
and we have a canonical isomorphism
$${\rm H}^i\bigl(\bigl(\overline{X}_0/\cO_{\rm
cris}\bigr)^{\rm cris}_{\rm log},\cE\bigr)\cong {\rm H}^i_{\rm dR}\left(\overline{X}_0,\bigl(\cE_{\Xtilde_{B_{\rm log}}},
\nabla_{E_{\Xtilde_{B_{\rm log}}}}\bigr)\right)$$as
$B_{\rm log}$-modules. Note that ${\rm H}^i_{\rm dR}\left(\overline{X}_0,\bigl(\cE_{X_K},\nabla_{X_K}\bigr)\right) $
is an $\overline{B}_{\rm log}$-module with a filtration,
the Hodge filtration, compatible with the filtration on $\overline{B}_{\rm log}$.
The surjective map $\cE_{\Xtilde_{B_{\rm log}}} \lra \cE_{X_K} $ induces a
morphism on cohomology, compatible with filtrations and $G_K$-action, $$  {\rm H}^i
\bigl(\bigl(\overline{X}_0/\cO_{\rm cris}\bigr)^{\rm cris}_{\rm log},\cE\bigr) \lra {\rm
H}^i_{\rm dR}\left(\overline{X}_0,\bigl(\cE_{X_K},\nabla_{X_K}\bigr)\right) .$$

\

\subsubsection{Properties of semistable sheaves.}\label{sec:defsemsitablesheafglobal}

We now drop the assumption that $X$ is a small affine and  deal with the general case.
Let $\cL$ be a $\Q_p$-adic sheaf~$\cL=\{\cL_n\}_n$ on~$\fX_K$. We
say that it is  {\it semistable} if there exists a covering
$\{U_i\}_{i\in I}$ of $X$ by small objects such that
$\cL\vert_{(U,U_K)}$ is semistable in the sense of
\ref{sec:defsemsitablesheaf}. For every $i$ we write
$\bDcrisar\bigl(\cL\bigr)_i$ for the $\widehat{\cO}_\Xtilde^{\rm
DP}[p^{-1}]\vert_{U_i}$-module with connection, Frobenius and
filtration associated to $\cL\vert_{(U,U_K)}$ in
\ref{sec:Dcrisar}. It follows from \ref{prop:equivcris} and
\ref{cor:changefrobenius} that we have a canonical isomorphism
$\bDcrisar\bigl(\cL\bigr)_i\vert_{U_i\cap U_j}\cong
\bDcrisar\bigl(\cL\bigr)_i\vert_{U_i\cap U_j}$, for every $i$ and
$j\in I$, as  $\widehat{\cO}_\Xtilde^{\rm
DP}[p^{-1}]\vert_{U_i\cap U_j}$-modules compatible with
connections and filtrations. In particular the modules
$\bDcrisar\bigl(\cL\bigr)_i$, $i\in I$ glue to a coherent
$\widehat{\cO}_\Xtilde^{\rm DP}[p^{-1}]$-module
$\bDcrisar\bigl(\cL\bigr)$ endowed with connection and a
filtration $\Fil^\bullet \bDcrisar\bigl(\cL\bigr)$. Moreover for
the same reason, for every small object $U\in X^{\ket}$ we have
that $\bDcrisar\bigl(\cL\bigr)\vert_U$ is the
$\widehat{\cO}_\Xtilde^{\rm DP}[p^{-1}]\vert_{U}$-module with
connection and filtration  defined in \ref{sec:Dcrisar}. In
particular once chosen a lift of Frobenius $F_{\widetilde{U}}$ on
$\widetilde{U}_{\rm form}$ we also get a Frobenius
$\varphi_{\cL,U}$ on $\bDcrisar\bigl(\cL\bigr)\vert_U$.

\begin{proposition}\label{prop:mainpropDcrisar}
Assume that $\cL$ is semistable. Then,\smallskip

(1) $\bDcrisar\bigl(\cL\bigr)$ is a projective
$\widehat{\cO}_\Xtilde^{\rm DP}[p^{-1}]$-module of finite
rank;\smallskip

(2) ${\rm Gr}^n \bDcrisar\bigl(\cL\bigr):=\Fil^n \bDcrisar\bigl(\cL\bigr)/\Fil^{n+1}\bDcrisar\bigl(\cL\bigr)$
are projective $\cO_X\otimes_{\WW(k)} K$-modules of
finite rank;\smallskip

(3) the connections $\nabla_{\cL,\WW(k)}$ and $\nabla_{\cL,\cO}$
are integrable and topologically nilpotent (with respect to the
special fiber $X_0$) and satisfy Griffiths' transversality with
respect to the filtration;\smallskip

(4) for every small object, Frobenius $\varphi_{\cL,U}$ on
$\bDcrisar\bigl(\cL\bigr)\vert_U$ is horizontal with respect to
the connections $\nabla_{\cL,\WW(k)}$ and $\nabla_{\cL,\cO}$
restricted to $U$ and is \'etale i.e., the map
$$\varphi_{\cL,U}\otimes 1\colon  \bDcrisar\bigl(\cL\bigr)\vert_U
\otimes_{\cO_{\widetilde{U}}^{\rm DP}}^{F_U}
\widehat{\cO}_{\widetilde{U}}^{\rm DP}\lra \bDcrisar\bigl(\cL\bigr)\vert_U$$is
an isomorphism of $\widehat{\cO}_{\widetilde{U}}^{\rm DP}\bigl[p^{-1}\bigr]$-modules;\smallskip

(5)  for every $n\in\Z$ the  morphism $${\rm Gr}^n \alpha_{\rm log,\cL} \colon \bigoplus_{a+b=n}
{\rm Gr}^a \bDcrisar(\cL)\otimes_{\cO_X\otimes K} {\rm Gr}^b
\bB_{\rm log,K}  \lra {\rm Gr }^n\left(\cL\tensor_{\Z_p} \bB_{\rm log,K}\right), $$
induced by $\alpha_{\rm log,\cL}$, is an isomorphism in ${\rm
Ind}\bigl(\Sh(\fX_K)^\N\bigr)$. In particular, $\alpha_{\rm log,\cL}$ is strict with respect to the
filtrations and it is compatible with Frobenii and
connections;\smallskip

(6) the map $\bDcrisar(\cL)\widehat{\tensor}_{\cO_\Xtilde^{\rm DP}} \cO_{\Xtilde,\rm log}^{\rm geo}\lra \bDcrisgeo(\cL)$
is  an isomorphism, strictly compatible with the filtrations and $\bDcrisgeo(\cL)$ is a direct summand in
$\bDcrisar(\cL)\widehat{\tensor}_{\cO_{\rm cris}} B_{\rm log}$ compatible with the filtrations.
See \S\ref{section:cohBlog} for the notation. It is isomorphic to $\bDcrisar(\cL)\widehat{\tensor}_{\cO_{\rm cris}} B_{\rm log}$
if $X_K$ is geometrically connected over $K$;\smallskip

(7)  there exists a coherent $\widehat{\cO}_\Xtilde^{\rm
DP}$-submodule $D(\cL)$ of $\bDcrisar\bigl(\cL\bigr)$ such that:

(7.i) it is stable under the connections and
$\nabla_{\cL,\WW(k)}\vert_{D(\cL)}$ is  integrable and
topologically nilpotent,

(7.ii) $\bDcrisar\bigl(\cL\bigr)\cong D(\cL)\tensor_{\Z_p} \Q_p$,

(7.iii) there exist integers $h$ and $n\in\N$ such that for every
small $U$ the map $p^h \varphi_{\cL,U}$ sends $D(\cL)\vert_U$ to
$D(\cL)\vert_U$, it is horizontal with respect to
$\nabla_{\cL,\WW(k)}\vert_{D(\cL)}$ and multiplication by $p^n$ on
$D(\cL)\vert_U$ factors via $p^h \varphi_{\cL,U}$.

\end{proposition}
\begin{proof} (1) follows from \ref{prop:equivsemistable}.

(2)-(4) follow from \ref{prop:propsemistable} after restricting to
small open affines of  $X$.

(5) The fact that ${\rm Gr}^n \alpha_{\rm log,\cL}$ is an
isomorphism follows from  \ref{prop:GrDdR}(3) for $\widetilde{\rm D}_{\rm dR}(V)$,
\ref{prop:BcrissubsetBdR}(4) and \ref{prop:propsemistable}(4). The compatibilities with Frobenius and connections are clear.

(6) the first claim follows from \ref{cor:arithmimpliesgeo}. The second follows from the fact that  $\cO_{\Xtilde,\rm log}^{\rm geo}$
is a direct summand in $\cO_\Xtilde\widehat{\otimes}_{\cO} B_{\rm log}$ and is isomorphic to it if
$X_K$ is geometrically connected over $K$ due to \ref{cor:injstrictinvariants}(3).

(7.i) and (7.ii) hold   after restricting to small open affines of
$X$ due to \ref{sec:relcrystals}. Claim (7.iii) holds by (4). Since $X$ is a noetherian
space and can be covered by finitely many small affines, the claim
follows.
\end{proof}

\begin{corollary}\label{cor:cEXtile} If $\cL$ is a semistable sheaf on~$\fX_K$,
there exists a unique filtered Frobenius isocrystal
$\bigl(\cE,\varphi,\Fil^\bullet \cE_{X_K}\bigr)$ such that\smallskip

(i)  $\bDcrisar(\cL)\cong \cE_\Xtilde$, compatibly with the
connections;\smallskip

(ii) $\Fil^n\cE_{X_K}$ is defined by the image of $\Fil^n
\bDcrisar(\cL)$ via the isomorphism in (i). Moreover,
$\Fil^n\cE_{X_K}$ and ${\rm Gr}^n \cE_{X_K}:=\Fil^n\cE_{X_K}/
\Fil^{n+1}\cE_{X_K}$ are locally free
$\widehat{\cO}_{X_K}$-modules of finite rank and the filtration on $\bDcrisar(\cL)$ is uniquely characterized by the fact
that its image in  $\cE_{X_K}$ is $\Fil^\bullet \cE_{X_K}$ and it satisfies Griffiths' tranversality with respect to $\nabla_{\cL,\WW(k)}$.
\smallskip

(iii) for every small affine $U$, writing $\widetilde{R}$ for the algebra
underlying $\widetilde{U}_{\rm form}$, the isomorphism in (i)
restricted to $U$ is compatible with Frobenii, the one on
$\bDcrisar(\cL)\vert_U$ given in \ref{sec:Dcrisar} and the one on
$\cE_\Xtilde\vert_U$ defined by the Frobenius $F_U$  on
$\widetilde{R}$.\smallskip

\end{corollary}
\begin{proof} The existence  of an isocrystal $\cE$
such that (i) holds follows from \ref{prop:mainpropDcrisar}(7).
The uniqueness follows from the characterization of crystals on
$\bigl(X_0/\cO_{\rm cris}\bigr)^{\rm cris}_{\rm log}$ in terms of
$\widehat{\cO}_\Xtilde^{\rm DP}$-modules given in \cite[Thm.
6.2]{katolog}.

(ii) provides the definition of the filtration. The fact that it satisfies Griffiths' transversality, that it consists of locally
free $\widehat{\cO}_{X_K}$-modules
and that its graded quotient also consist of locally free $\widehat{\cO}_{X_K}$-modules and the fact that we can recover the
original filtration on $\bDcrisar(\cL)$
follow from \ref{prop:propsemistable} and \ref{prop:GrDdR}(2).

(iii) the required property and \ref{prop:mainpropDcrisar}(4)
define $\varphi\vert_U$ on $\cE_\Xtilde\vert_U$, up to
multiplication by $p$, and hence on the crystal $\cE\vert_{U_k}$
by \cite[Thm. 6.2]{katolog}. We are left to show that the
$\varphi\vert_U$'s glue for varying $U$'s. This follows from
\ref{cor:changefrobenius}.

\end{proof}

By abuse of notation we simply write $$\bDcrisar\colon
\Sh\bigl(\fX_K\bigr)_{\rm ss} \lra {\rm FIso}^{\rm Fil}(X/\cO)$$for the
induced functor. We let ${\rm FIso}^{\rm Fil}(X/\cO)^{\rm adm} $, the
category of so called {\it admissible filtered Frobenius
isocrystals} be its essential image.
Let~$\underline{\cE}:=\bigl(\cE,\varphi,\Fil^n \cE_{X_K}\bigr)$ be
a filtered Frobenius isocrystal. Define
$$\bV_{\rm log}(\underline{\cE}):=\Fil^0\left(v_K^\ast(\cE_\Xtilde)
\tensor_{\cO_\Xtilde^{\rm DP}} \bB_{\rm
log,K}\right)^{\nabla_{\WW(k)}=0,\phi=1}\in {\rm
Ind}\left(\Sh(\fX_K)^\N\right).$$Here, we endow $\cE_\Xtilde$ with the filtrations provided in \ref{cor:cEXtile}(ii) and
$v_K^\ast(\cE_\Xtilde)
\tensor_{\cO_\Xtilde^{\rm DP}} \bB_{\rm
log,K}$ with the composite filtration.

\begin{proposition}\label{prop:BdCristannakianequivalence} The functor $\bDcrisar\colon
\Sh\bigl(\fX_K\bigr)_{\rm ss} \lra {\rm FIso}^{\rm Fil}(X/\cO)^{\rm adm}$
defines an equivalence of categories with left quasi-inverse
$\bV_{\rm log}$. Moreover,\smallskip

(i) if $\cL$ and $\cL'$ are semistable sheaves, then also
$\cL\otimes_{\Z_p} \cL'$ is semistable and
$\bDcrisar\bigl(\cL\otimes_{\Z_p} \cL'\bigr)\cong
\bDcrisar(\cL)\otimes_{\cO_{\widetilde{X}}^{\rm DP}}
\bDcrisar(\cL')$;\smallskip

(ii) if $\cL$ is a semistable sheaf, then also $\cL^\vee$ is
semistable and $\bDcrisar\bigl(\cL^\vee\bigr)\cong
\bDcrisar(\cL)^\vee$;\smallskip

(iii) if $\cL$ is a semistable sheaf, then  $\beta^\ast(\cL)$ is a geometrically semistable sheaf on
$\fX_\Kbar$ in the sense of \S\ref{sec:defsemgeositablesheaf} and $$\bDcrisgeo\bigl(\beta^\ast(\cL)\bigr)\cong
\beta^\ast\bigl(\bDcrisar(\cL)\widehat{\otimes}_{\cO_\Xtilde^{\rm DP}} \cO_{\Xtilde,\rm log}^{\rm geo}\bigr)$$
as filtered Frobenius isocrystals on $X_0$ relative to $A_{\rm cris}$ in the sense of \S\ref{sec:geoinvariant}.

\smallskip

In particular, $\Sh\bigl(\fX_K\bigr)_{\rm ss}$ and
${\rm FIso}^{\rm Fil}(X/\cO)^{\rm adm}$ are tannakian  categories and
$\bDcrisar$ defines an equivalence of abelian tensor categories.

\end{proposition}
\begin{proof}
It follows from \ref{section:fundamentaldiagram} and the
definition of semistable sheaf in
\ref{sec:defsemsitablesheafglobal} that  we have  isomorphisms of functors $\bV_{\rm
log}\circ \bDcrisar\cong {\rm Id}$ and $ \bDcrisar\circ \bV_{\rm
log}\cong {\rm Id}$ considering the categories $\Sh\bigl(\fX_K\bigr)_{\rm ss}$ and
${\rm FIso}^{\rm Fil}(X/\cO)^{\rm adm}$  respectively. In particular the functor $\bDcrisar$ is
fully faithful. Being essentially surjective by definition of
${\rm FIso}(X/\cO)^{\rm adm} $, we conclude that $\bDcrisar$ is
an equivalence of categories.

The fact that $\Sh\bigl(\fX_K\bigr)_{\rm ss}$ is closed under
tensor products and duals follow from \ref{prop:equivcris} and
\ref{prop:tensorss}. The fact that $\bDcrisar$ commutes with
tensor products and duals also follows from the description of
$\bDcrisar(\cL)$ on small affine formal subschemes given in  \ref{sec:Dcrisar} and
from \ref{prop:tensorss}. Claim (iii) has been proven in \ref{prop:equivcris} and
\ref{prop:mainpropDcrisar}(6). The Frobenius structure is defined for
$\bDcrisgeo\bigl(\beta^\ast(\cL)\bigr)$ only on small affines and is compatible with the one on
$\bDcrisar(\cL)$. This compatibility allows us to define a global Frobenius
structure on $\bDcrisgeo\bigl(\beta^\ast(\cL)\bigr)$ inherited from the Frobenius structure on $\bDcrisar(\cL)$.

\end{proof}

\subsubsection{Cohomology of semistable sheaves}\label{sec:cohomologyofsssheaves}

\begin{theorem}
\label{thm:formalcomparison} For $\cL$ a semistable sheaf on $\fX_K$ there is a
canonical isomorphism of
$\delta$-functors  from $\Sh(\fX_K)_{\rm ss}$:
$$
{\rm H}^i\bigl(\fX_\Kbar, \bL\otimes\bB_{\rm
log,\Kbar}^\nabla\bigr) \cong {\rm H}^i\bigl(\bigl(X_0/\cO_{\rm
cris}\bigr)^{\rm cris}_{\rm log}, \bDcrisgeo(\bL)\bigr),
$$of
$B_{\rm log}$-modules, compatible with action of $G_K$, Frobenius,
monodromy operator $N$  and strictly compatible with the
filtrations. In fact for every $r\in \Z$ we have isomorphisms of
$A_{\rm log}$-modules which are $G_K$-equivariant and compatible
for varying $r$'s and with the previous isomorphism,
$$
{\rm H}^i\bigl(\fX_\Kbar, \bL\otimes\Fil^r \bB_{\rm log,\Kbar}^\nabla\bigr)
\cong  \bH^i\bigl(X^{\ket},\Fil^{\rm r-\bullet} \bDcrisgeo(\bL)\otimes_{\cO_\Xtilde}
\omega^\bullet_{\Xtilde/\cO}\bigr).
$$
\end{theorem}

Here we write $\bL$ for $\beta^\ast(\bL)$ by abuse of notation. We identify $\bDcrisgeo(\bL)$ with the  Frobenius crystal
$\bDcrisar(\cL)\widehat{\otimes}_{\cO_\Xtilde^{\rm DP}} \cO_{\Xtilde,\rm log}^{\rm geo}$ using
\ref{prop:BdCristannakianequivalence}(iii). In particular, we have an isomorphism
$${\rm H}^i\bigl(\bigl(X_0/\cO_{\rm cris}\bigr)^{\rm cris}_{\rm
log}, \bDcrisgeo(\bL)\bigr)\cong \bH^i\bigl(X^{\ket}, \bDcrisgeo(\bL)\otimes_{\cO_\Xtilde}
\omega^\bullet_{\Xtilde/\cO}\bigr)$$as $B_{\rm log}$-modules. Note that
$\bL\otimes \bB_{\rm log,\Kbar}^\nabla$ is quasi-isomorphic to the complex $\bL\otimes\bB_{\rm log,\Kbar}
\otimes_{\cO_\Xtilde} \omega^\bullet_{\Xtilde/\cO}$ by
\ref{lemma:propbBcris} which is quasi-isomorphic to $\bB_{\rm log,\Kbar}\otimes_{\cO_{\Xtilde,\rm log}^{\rm geo}}
\bDcrisgeo(\bL)\otimes_{\cO_\Xtilde}
\omega^\bullet_{\Xtilde/\cO}$ by definition of semistable sheaf and \ref{prop:BdCristannakianequivalence}(iii).
Thus the fact that we have isomorphisms of $B_{\rm log}$-modules as claimed in the Theorem is a
formal consequence of \ref{lemma:propbBcris}.
\smallskip

{\it The filtrations.} The filtration on ${\rm H}^i\bigl(\fX_\Kbar, \bL\otimes\bB_{\rm log,\Kbar}^\nabla\bigr)$
is defined as the image of ${\rm
H}^i\bigl(\fX_\Kbar, \bL\otimes\Fil^r \bB_{\rm log,\Kbar}^\nabla\bigr)$. The filtration $\Fil^r {\rm H}^i
\bigl(\bigl(X_0/\cO_{\rm cris}\bigr)^{\rm cris}_{\rm log},
\bDcrisgeo(\bL)\bigr)$ on ${\rm H}^i\bigl(\bigl(X_0/\cO_{\rm cris}\bigr)^{\rm cris}_{\rm log},
\bDcrisgeo(\bL)\bigr)$ is defined as the image of
$\bH^i\bigl(X^{\ket},\Fil^{\rm r-\bullet} \bDcrisgeo(\bL)\otimes_{\cO_\Xtilde} \omega^\bullet_{\Xtilde/\cO}\bigr)$.
Due to \ref{lemma:propbBcris}  we are left to
prove that $ {\rm R}^j v_{\Kbar,\ast}^{\rm cont}\bigl(\bL\otimes\Fil^r \bB_{\rm log,\Kbar}\bigr)$
vanishes for every $j\geq 1$ and every $r\in\Z$. Note that
$\bL\otimes\Fil^n \bB_{\rm log,\Kbar}\cong \Fil^{n} \left(\bDcrisgeo(\bL) \otimes_{\cO_{\Xtilde,\rm log}^{\rm geo}}
\bB_{\rm log,\Kbar}\right)$ which is  $\Fil^n\bigl(\bDcrisar(\bL)
\otimes_{\widehat{\cO}_{\Xtilde}^{\rm DP}} \bB_{\rm log,\Kbar}\bigr)$ by
\ref{prop:BdCristannakianequivalence}(iii). We are reduced to prove the vanishing of  $ {\rm R}^j v_{\Kbar,\ast}^{\rm cont}
\Fil^{n} \left(\bDcrisar(\bL)
\otimes_{\widehat{\cO}_{\Xtilde}^{\rm DP}} \bB_{\rm log,\Kbar}\right)$ for every $j\geq 1$ and every
$r\in\Z$.  As $\bDcrisar(\bL)=\Fil^{h}\bDcrisar(\bL)$ for
$h$-small enough, we conclude that $ {\rm R}^j v_{\Kbar,\ast}^{\rm cont} \left(\Fil^{s}\bDcrisar(\bL)
\otimes_{\widehat{\cO}_{\Xtilde}^{\rm DP}} \Fil^t \bB_{\rm
log,\Kbar}\right)$ is $0$ for every $s\leq h$ and every $t\in\Z$ thanks to \ref{prop:crysisacyclic}.
Using \ref{prop:mainpropDcrisar}(2) and proceeding by induction on $s$
we get the vanishing for every $s$ and $t\in \Z$. We conclude using \ref{prop:mainpropDcrisar}(5).

\smallskip

{\it Galois action.} The Galois action on ${\rm H}^i_{\rm dR}\bigl(X^{\et}, \Fil^r\bDcrisgeo(\bL)\bigr)$ is induced by the Galois action on $\bDcrisgeo(\bL)$. The
Galois action on ${\rm H}^i(\fX_\Kbar, \bL\otimes\Fil^r\bB_{\rm log,\Kbar}^\nabla) $ arises via the isomorphism $\beta^\ast\bigl(\cL\tensor \Fil^r \bA_{\rm
log,K}^\nabla\bigr)=\cL\tensor \Fil^r \bA_{\rm log,\Kbar}^\nabla$. The verification of the compatibility with $G_K$-action is formal; see
\cite[Thm.~3.15]{andreatta_iovita_comparison} for details.\smallskip

{\it Frobenius:} The Frobenius on  ${\rm H}^i\left(\fX_\Kbar, \cL\tensor_{\Z_p} \bB_{\rm log,\Kbar}^\nabla\right)$ is induced by Frobenius on
$\bB_{\rm log,\Kbar}^\nabla$.
Frobenius on ${\rm H}^i\bigl(\bigl(X_0/\cO_{\rm cris}\bigr)^{\rm cris}_{\rm log}, \bDcrisgeo(\bL)\bigr)$ is defined by
Frobenius on $\bDcrisgeo(\bL)$ defined in \ref{prop:BdCristannakianequivalence}(iii). The
proof of the compatibility of the isomorphism with Frobenius is as in \cite[Thm.~3.15]{andreatta_iovita_comparison}.
We refer to loc.~cit.~for details. \smallskip

\smallskip

{\it Monodromy:} The monodromy operator $N$ on  ${\rm
H}^i\bigl(\fX_\Kbar, \bL\otimes\bB_{\rm log,\Kbar}^\nabla\bigr)$
is defined by the $\bB_{\rm cris,\Kbar}^\nabla$-linear derivation
$\bB_{\rm log,\Kbar}^\nabla \lra \bB_{\rm log,\Kbar}^\nabla
\frac{dZ}{Z}$ induced by the $\WW(k)$-linear derivation $\cO \to
\cO \frac{dZ}{Z}$ on $\cO$.

The monodromy on $\bH^i\bigl(X^{\ket}, \bDcrisgeo(\bL)\otimes_{\cO_\Xtilde}
\omega^\bullet_{\Xtilde/\cO}\bigr)$ is defined by taking the long exact sequence in
cohomology associated to the short  exact sequence
$$0 \lra  \bDcrisgeo(\bL) \otimes_{\cO_\Xtilde}
\omega^{\bullet-1}_{\Xtilde/\cO}\wedge \frac{dZ}{Z}   \lra \bDcrisgeo(\bL)
\otimes_{\cO_\Xtilde} \omega^{\bullet}_{\Xtilde/\WW(k)}  \lra \bDcrisgeo(\bL)
\otimes_{\cO_\Xtilde} \omega^{\bullet}_{\Xtilde/\cO} \lra 0$$of complexes
deduced from \ref{section:monodromydiagram}, relating the derivations
$\nabla_{\cL,\WW(k)}$ and $\nabla_{\cL,\cO}$. It provides for every $i$ a
map $N_{\bL,i}$: $$\bH^i\bigl(X^{\ket}, \bDcrisgeo(\bL)\otimes_{\cO_\Xtilde}
\omega^\bullet_{\Xtilde/\cO}\bigr)\to \bH^{i+1}\bigl(X^{\ket}, \bDcrisgeo(\bL)
\otimes_{\cO_\Xtilde} \omega^{\bullet-1}_{\Xtilde/\cO}\bigr)\frac{dZ}{Z} \cong
\bH^i\bigl(X^{\ket}, \bDcrisgeo(\bL)\otimes_{\cO_\Xtilde} \omega^\bullet_{\Xtilde/\cO}\bigr)\frac{dZ}{Z} .$$
The verification of the compatibility in
\ref{thm:formalcomparison} with the monodromy operator is a formal consequence of the following exact sequence
of complexes, see \ref{section:monodromydiagram},

$$\begin{array}{cccccccc} 0  \lra &\bB_{\rm log,L}^\nabla
\frac{dZ}{Z}[-1]&\lra &  \bB_{\rm cris,L}^\nabla &
\stackrel{d}{\lra}& \bB_{\rm log,L}^\nabla &  \lra 0\cr &
\big\downarrow & & \big\downarrow & & \big\downarrow & \cr
 0 \lra & \bB_{\rm log,L}
\otimes_{\cO_\Xtilde} \omega^{\bullet-1}_{\Xtilde/\cO}\wedge
\frac{dZ}{Z}  & \lra  & \bB_{\rm log,L} \otimes_{\cO_\Xtilde}
\omega^{\bullet}_{\Xtilde/\WW(k)}  &\lra & \bB_{\rm log,L}
\otimes_{\cO_\Xtilde} \omega^{\bullet}_{\Xtilde/\cO} &\lra 0.
\end{array}$$

\smallskip

{\it A variant:} We use the notation of \S\ref{sec:geoinvariant}.  Let $\bigl(\bDcrisgeo(\bL)_{X_K},\nabla_{\bDcrisgeo(\bL)_{X_K}}\bigr)$ be
$\bDcrisgeo(\bL)\otimes_{B_{\rm log}} \overline{B}_{\rm log}$. It is a sheaf of $\cO_X\widehat{\otimes}_{\cO_K} \overline{B}_{\rm log}$-modules with connection
relative to $\overline{B}_{\rm log}$ and filtration satisfying Griffiths' transversality. See loc.~cit. Recall that in \S\ref{sec:defBbarlog} we have defined
$\overline{\bB}_{\rm log,\Kbar}^\nabla$ as $\bB_{\rm log,\Kbar}^\nabla \otimes_{B_{\rm log}} \overline{B}_{\rm log} $ and similarly for $\overline{\bB}_{\rm
log,\Kbar}$.

\begin{theorem}\label{sec:formalcomparison} We have an isomorphism of $\delta$-functors:
$$
{\rm H}^i\bigl(\fX_\Kbar, \bL\otimes \overline{\bB}_{\rm log,\Kbar}^\nabla \bigr) \cong {\rm H}^i_{\rm
dR}\left(X^{\ket},\bigl(\bDcrisgeo(\bL)_{X_K},\nabla_{\bDcrisgeo(\bL)_{X_K}}\bigr)\right),
$$for $\cL$ a semistable sheaf on $\fX_K$. The above isomorphism is
$\overline{B}_{\rm log}$-linear, compatible with action of $G_K$ and strictly compatible with the filtrations.
\end{theorem}
\begin{proof}  This is a variant of \ref{thm:formalcomparison} using the quasi-isomorphism of complexes
$$\bL\otimes  \Fil^r \overline{\bB}_{\rm log,\Kbar}^\nabla  \cong \bL\otimes  \Fil^{r-\bullet} \overline{\bB}_{\rm log,\Kbar} \otimes_{\cO_X}
\omega^\bullet_{X/\cO_K}$$provided by \ref{prop:deRhamcomplexBbarlog},  the isomorphism $$\bL\otimes
\Fil^{r} \overline{\bB}_{\rm log,\Kbar}\cong \Fil^r\left( \overline{\bB}_{\rm log,\Kbar}\otimes_{\cO_{\Xtilde,\rm log}^{\rm geo}}
\bDcrisgeo(\bL)\right)\cong \Fil^r\left( \overline{\bB}_{\rm log,\Kbar}\otimes_{\cO_{\Xtilde,\rm log}^{\rm geo}}
\bDcrisgeo(\bL)_{X_K}\right) $$and the vanishing of ${\rm R}^j v_{\Kbar,\ast}^{\rm cont} \left( \Fil^r
\overline{\bB}_{\rm log,\Kbar} \right)$ for $j\geq 1$ and the fact that for $j=0$ it coincides with
$\Fil^r \left(\cO_{\Xtilde,\rm log}^{\rm geo}\widehat{\otimes}_{A_{\rm log}}\overline{B}_{\rm log}\right)$,  proven in \ref{prop:crysisacyclic}.
\end{proof}

\subsubsection{The comparison isomorphism for semistable sheaves in the proper case}

We assume that we are in the formal case and that there exists a
proper, geometrically connected and log smooth morphism $X^{\rm alg}\to {\rm
Spec}(\cO_K)$ whose associated $p$-adic logarithmic formal scheme
is $f\colon X \to {\rm Spf}(\cO_K)$. The main result of this
section is

\begin{theorem}\label{thm:mainthm} Let $\cL$ be a semistable sheaf
on $\fX_K$.  Then
${\rm H}^i\left(\fX_\Kbar, \bL\right)$ is a semistable representation of $G_K$
for every $i\ge 0$ and
$$D_{\rm st}\left({\rm H}^i\left(\fX_\Kbar, \bL\right)\right)\cong
{\rm H}^i\bigl(\bigl(X_k/\WW(k)^+\bigr)^{\rm cris}_{\rm log},
\bDcrisar(\bL)^+\bigr)$$compatibly with Frobenius and monodromy
operators and filtrations after extension of scalars to $K$. Such
an isomorphism is an isomorphism of $\delta$-functors on the category
of semistable sheaves. Moreover, ${\rm H}^i\left(\fX_\Kbar,
\_\right)$   satisfies K\"unneth formula for semistable sheaves on
$\fX_K$  and $D_{\rm
st}$ commutes with the K\"unneth formula.
\end{theorem}

The map of sites $u\colon \fX_K \lra X_K^{\ket}$, sending $(U,W)\mapsto W$ sends covering families to covering families, commutes with fibred products and sends the
final object to the final object. In particular it is continuous and the push-forward defines a morphism $u_\ast\colon \Sh(X_K^{\et}) \lra \Sh(\fX_K)$ which extends
to inductive systems of continuous sheaves. It is an immediate verification that it sends $\Q_p$-adic sheaves on $X_K^{\ket}$, defined in a  way similar to
\S\ref{sec:sssheaves}, to $\Q_p$-adic sheaves on $\fX_K$. Given any such sheaf $\bL$ we write $\bL$, by abuse of notation, also for its image $u_\ast(\bL)$ in
$\Sh(\fX_L)_{\Q_p}$. We get a map ${\rm H}^i\left(\fX_\Kbar, \bL\right) \lra {\rm H}^i\left(X_\Kbar^{\ket}, \bL\right) $.

\begin{theorem}\label{thm:finiteFaltings} (\cite[Thm.~9]{faltingsAsterisque}) The map above induces
an isomorphism ${\rm H}^i\left(\fX_\Kbar, \bL\right) \cong {\rm
H}^i\left(X_\Kbar^{\ket}, \bL\right) $ of $G_K$-modules. In particular, ${\rm
H}^i\left(\fX_\Kbar, \bL\right)$ is finite dimensional as $\Q_p$-vector space.
\end{theorem}

\begin{remark} Faltings' proof uses Poincar\'e duality for locally constant sheaves
on $\fX_\Kbar$ and on $X_\Kbar^{\et}$. If $X$ is smooth over $\cO_K$, one has a more direct proof suggested in \cite[Thm.~9]{faltingsAsterisque}. Via a Leray
spectral sequence argument, it amounts to prove that the higher direct images of $\Q_p$-adic sheaves with respect to the maps $\Sh(X_K^{\et}) \lra \Sh(X^{\ket})$
and $\Sh(\fX_K^{\et}) \lra \Sh(X^{\ket})$ coincide. This is worked out in detail in \cite[Prop.~4.9]{andreatta_iovita} if $X$ has trivial log structure  and in
Olsson \cite{Olsson}  in general.
\end{remark}

Let $\cL$ be a semistable sheaf. Write $$D_i(\cL):= {\rm
H}^i\bigl(\bigl(X_k/\WW(k)^+\bigr)^{\rm cris}_{\rm log},
\bDcrisar(\bL)^+\bigr).$$It is is a finite dimensional $K_0$-vector
space since $f$ is assumed to be proper. Moreover, thanks to
\ref{prop:HicrisE}, we have
$${\rm H}^i\bigl(\bigl(X_0/\cO_{\rm cris}\bigr)^{\rm cris}_{\rm log}, \bDcrisar(\bL)\bigr)^{\varphi-{\rm div}} \cong D_i(\cL)
\otimes_{\WW(k)} \cO_{\rm cris}\bigl[p^{-1}\bigr],$$where
$\varphi-{\rm div}$ stands for the image of Frobenius linearized.
The above isomorphism is compatible with the Frobenii and with the logarithmic connections relative to
$\cO_{\rm cris}$-modules. Consider the natural morphisms
\begin{equation}\label{gammaigammaibar}
\begin{array}{ccc}
{\rm H}^i\bigl(\bigl(X_0/ \cO_{\rm cris}\bigr)^{\rm cris}_{\rm log}, \bDcrisar(\bL)
\bigr)^{\varphi-{\rm div}}\widehat{\otimes}_{\cO_{\rm cris}} B_{\rm log} &
\stackrel{\gamma^i_\cL}{\lra} & {\rm H}^i\bigl(\bigl(X_0/\cO_{\rm cris}\bigr)^{\rm cris}_{\rm log},
\bDcrisgeo(\bL)\bigr)^{\varphi-{\rm div}}\cr \alpha_i
\big\downarrow & & \big\downarrow \beta_i \cr  {\rm H}^i_{\rm dR}\bigl(X_0,\bigl(\bDcrisar(\bL)_{X_K},
\nabla_{\bDcrisar(\bL)_{X_K}}\bigr)\bigr)\otimes_K
 \overline{B}_{\rm log} & \stackrel{\overline{\gamma}^i_\cL}{\lra} & {\rm H}^i_{\rm
dR}\left(\overline{X}_0,\bigl(\bDcrisgeo(\bL)_{X_K},\nabla_{\bDcrisgeo(\bL)_{X_K}}\bigr)\right).\cr\end{array}
\end{equation} Here the top row is deduced from the isomorphism $$\bDcrisgeo(\bL)\cong \bDcrisar(\cL)\widehat{\tensor}_{\cO_{\rm cris}}
B_{\rm log}$$which follows from \ref{prop:mainpropDcrisar}(6) and the assumption that $X^{\rm alg}_K$
is geometrically connected over $K$. It is compatible with the Frobenii,
connections $\nabla_{\cL,i}$ relative to $B_{\rm log}$, filtrations and $G_K$-actions. The bottom row is defined by the natural map $B_{\rm log} \to
\overline{B}_{\rm log} $, defined in  \ref{lemma:BlogmodPpi}, which induces the map
$\cO_{\rm cris} \to \cO_K$. It is a morphism of filtered $\overline{\bB}_{\rm
log}$-modules. Let ${\cal C}_{\rm cris}$ be the complex
$${\cal C}_{\rm cris}\colon\qquad (N,1-p\varphi)\colon \bB_{\rm log,\Kbar}^\nabla \oplus
\bB_{\rm cris,\Kbar}^\nabla \lra \bB_{\rm log,\Kbar}^\nabla$$ and let
${\cal C}_{\rm log}$ be the complex
$${\cal C}_{\rm log}\colon\qquad (N,1-p\varphi)\colon \bB_{\rm log,\Kbar}^\nabla \oplus
\bB_{\rm log,\Kbar}^\nabla \lra \bB_{\rm log,\Kbar}^\nabla.$$

\begin{proposition}\label{prop:Nonto} (1) The derivation
$N_{\bL,i}$ on ${\rm H}^i\bigl(\bigl(X_0/\cO_{\rm cris}\bigr)^{\rm
cris}_{\rm log}, \bDcrisgeo(\bL)\bigr)$, defined in \S\ref{sec:cohomologyofsssheaves},
is surjective  with kernel
isomorphic to ${\rm H}^i\bigl(\fX_\Kbar, \bL\otimes \bB_{\rm
cris,\Kbar}^\nabla \bigr)$ as $B_{\rm cris}$-modules, compatible
with Frobenii. The same result applies to  ${\rm H}^i\bigl(\bigl(X_0/\cO_{\rm
cris}\bigr)^{\rm cris}_{\rm log},
\bDcrisgeo(\bL)\bigr)^{\varphi-{\rm div}}$.
\smallskip

(2) For every $i$ we have exact sequences
$$0\lra  {\rm H}^i\bigl(\fX_\Kbar, \bL\otimes {\cal
C}_{\rm cris}\bigr) \lra {\rm H}^i\bigl(\bigl(X_0/\cO_{\rm
cris}\bigr)^{\rm cris}_{\rm log}, \bDcrisgeo(\bL)\bigr)\oplus {\rm
H}^i\bigl(\bigl(X_0/\cO_{\rm cris}\bigr)^{\rm cris}_{\rm log},
\bDcrisgeo(\bL)\bigr)^{N_{\bL,i}=0}\lra $$ $$ \stackrel{(N_{\bL,i},1-p\varphi)}{\lra}
{\rm H}^i\bigl(\bigl(X_0/\cO_{\rm cris}\bigr)^{\rm cris}_{\rm
log}, \bDcrisgeo(\bL)\bigr) \lra 0$$and

$$0\lra  {\rm H}^i\bigl(\fX_\Kbar, \bL\otimes {\cal
C}_{\rm log}\bigr) \stackrel{s_i}{\lra} {\rm H}^i\bigl(\bigl(X_0/\cO_{\rm
cris}\bigr)^{\rm cris}_{\rm log}, \bDcrisgeo(\bL)\bigr)\oplus {\rm
H}^i\bigl(\bigl(X_0/\cO_{\rm cris}\bigr)^{\rm cris}_{\rm log},
\bDcrisgeo(\bL)\bigr)\lra $$ $$ \stackrel{(N_{\cL,i},1-p\varphi)}{\lra}
{\rm H}^i\bigl(\bigl(X_0/\cO_{\rm cris}\bigr)^{\rm cris}_{\rm
log}, \bDcrisgeo(\bL)\bigr) \lra 0.$$In particular, the natural map  ${\rm H}^i\bigl(\fX_\Kbar, \bL\otimes {\cal
C}_{\rm cris}\bigr)\lra {\rm H}^i\bigl(\fX_\Kbar, \bL\otimes {\cal
C}_{\rm log}\bigr)$ is injective for every $i$.\smallskip

(3) The morphisms $\gamma^i_\cL$ and $\overline{\gamma}^i_\cL$ are
isomorphisms and $\overline{\gamma}^i_\cL$ is strictly  compatible
with the filtrations. \smallskip

(4) The morphisms $\alpha_i$ and $\beta_i$ in (\ref{gammaigammaibar}) are surjective.
\end{proposition}
\begin{proof}
We identify the cohomology group ${\rm H}^i\bigl(\fX_\Kbar,
\bL\otimes\bB_{\rm log}^\nabla\bigr)$ with ${\rm
H}^i\bigl(\bigl(X_0/\cO_{\rm cris}\bigr)^{\rm cris}_{\rm log},
\bDcrisgeo(\bL)\bigr)$ using \ref{thm:formalcomparison}. Let
$(\cE,\nabla)$ be the module with connection on $\Xtilde_{\rm
max}$ associated to $\bDcrisar(\bL)$; see \ref{prop:HicrisE}. As
explained in \ref{sec:isocrystals} and using that the isomorphism
$\bDcrisar(\bL)\widehat{\otimes}_{\cO_{\rm cris}} B_{\rm log}\cong
\bDcrisgeo(\bL)$ provided by \ref{prop:mainpropDcrisar}(6), we conclude that we have an isomorphism
$${\rm H}^i\bigl(\bigl(X_0/\cO_{\rm cris}\bigr)^{\rm cris}_{\rm log}, \bDcrisgeo(\bL)\bigr)\cong {\rm H}^i_{\rm dR}\bigl(X_k,
\cE\widehat{\otimes} B_{\rm log}\bigr)$$compatible with the
connection relative to $B_{\rm cris}$. Frobenius on $B_{\rm log}$
factors via the natural map $f\colon B_{\rm log}\to B_{\rm max}$
and  $g\colon B_{\rm max} \to B_{\rm log}$ so that the image of
Frobenius on ${\rm H}^i\bigl(\bigl(X_0/\cO_{\rm cris}\bigr)^{\rm
cris}_{\rm log}, \bDcrisgeo(\bL)\bigr)$  factors, using the
identifications above, via $${\rm H}^i_{\rm dR}\bigl(X_k,
\cE\widehat{\otimes}_{\cO_{\rm max}} B_{\rm log}\bigr)\cong {\rm
H}^i_{\rm dR}\bigl(X_k, \cE\bigr)\widehat{\otimes}_{\cO_{\rm max}}
B_{\rm log}.$$The last isomorphism comes from the fact that ${\rm
H}^i_{\rm dR}\bigl(X_k, \cE\bigr)$ is a finite $\cO_{\rm
max}$-module and  $\cO_{\rm max} \to A_{\rm max}$ is almost flat
by \ref{lemma:AtildeRlogff}. We conclude from \ref{prop:HicrisE} that $\gamma^i_\cL$
is an isomorphism. The proof that $\overline{\gamma}^i_\cL$ is an
isomorphism is similar using the isomorphism
$$\bDcrisar(\bL)_{X_K} \widehat{\otimes}_{\cO_K} \overline{B}_{\rm log}\cong \bDcrisgeo(\bL)_{X_K},$$of filtered
$\cO_X\widehat{\otimes}_{\cO_K} \overline{B}_{\rm log}$-modules endowed
with a connection relative to $\overline{B}_{\rm log}$, obtained via the
base change $B_{\rm log} \to \overline{B}_{\rm log}$ (inducing the
map $\cO\to \cO_K$ sending $Z$ to $\pi$).

\smallskip

(1) The derivation $N\colon \bB_{\rm log,\Kbar}^\nabla \lra \bB_{\rm log,\Kbar}^\nabla$ is surjective. Its kernel is $\bB_{\rm cris,\Kbar}^\nabla$
and the inclusion $\bB_{\rm cris,\Kbar}^\nabla\subset \bB_{\rm log,\Kbar}^\nabla$ is split injective; see \ref{sec:Alognabla}.
Identifying ${\rm H}^i\bigl(\bigl(X_0/\cO_{\rm cris}\bigr)^{\rm
cris}_{\rm log}, \bDcrisgeo(\bL)\bigr)$  with ${\rm
H}^i\bigl(\fX_\Kbar, \bL\otimes\bB_{\rm log,\Kbar}^\nabla\bigr)$ and
using that the map induced by $N$ on the latter is split
surjective for every $i$, the first part of the claim follows.
Since $N \circ \varphi=p \varphi\circ N$ and $\varphi$ is an isomorphism on $\bB_{\rm
cris,\Kbar}^\nabla$, we conclude that $N$
preserves the image of Frobenius and the last part of the claim follows.

If (3) holds, then  since $N$ is nilpotent on $D_i(\cL)$ and it is
surjective on $B_{\rm log}$, the monodromy operator is surjective
on ${\rm H}^i\bigl(\bigl(X_0/\cO_{\rm cris}\bigr)^{\rm cris}_{\rm
log}, \bDcrisgeo(\bL)\bigr)^{\varphi-{\rm div}}$. This concludes
the proof of (1).\smallskip

(2) The given short exact sequence, beside the exactness on the
left and on the right, is obtained from the long exact sequence
relating the cohomology of $\bL\otimes\cC_{\rm cris}$ and $\bL\otimes\cC_{\rm log}$ with the cohomology groups ${\rm
H}^i\bigl(\fX_\Kbar, \bL\otimes\bB_{\rm log,\Kbar}^\nabla\bigr)$
identified with ${\rm H}^i\bigl(\bigl(X_0/\cO_{\rm
cris}\bigr)^{\rm cris}_{\rm log}, \bDcrisgeo(\bL)\bigr)$. The connecting
homomorphisms ${\rm H}^i\bigl(\bigl(X_0/\cO_{\rm cris}\bigr)^{\rm
cris}_{\rm log}, \bDcrisgeo(\bL)\bigr)\to {\rm
H}^{i+1}\bigl(\fX_\Kbar, \bL\otimes {\cal C}_{\rm cris}\bigr)$ and ${\rm H}^i\bigl(\bigl(X_0/\cO_{\rm cris}\bigr)^{\rm
cris}_{\rm log}, \bDcrisgeo(\bL)\bigr)\to {\rm
H}^{i+1}\bigl(\fX_\Kbar, \bL\otimes {\cal C}_{\rm log}\bigr)$  are zero for
every $i$ due to (1). As  ${\rm H}^i\bigl(\bigl(X_0/\cO_{\rm cris}\bigr)^{\rm
cris}_{\rm log}, \bDcrisgeo(\bL)\bigr)^{N_{\cL,i}=0}$ coincides with ${\rm H}^i\bigl(\fX_\Kbar, \bL\otimes \bB_{\rm
cris,\Kbar}^\nabla \bigr)$ by (1), the conclusion follows.\smallskip

(3) We prove that $\overline{\gamma}^i_\cL$ is strict on filtrations, i.e.~that it induces an isomorphism on the various
steps of the filtrations.  We argue as in
the proof of \cite[Prop.~3.25]{andreatta_iovita_comparison}. Since $X$ is proper and algebrizable over
$\cO_K$, by GAGA there exists a $\cO_{X^{\rm alg}}$ module
$\cE_{\cO_K}$ with logarithmic and integrable connection $\nabla$ algebrizing the coherent
$\widehat{\cO}_X$-module $D(\cL)\otimes_{\cO_{\rm cris}} \cO_K$ (see
\ref{prop:mainpropDcrisar} for the definition of $D(\cL)$). Its base change $\cE_K:=\cE_{\cO_K}\otimes_{\cO_K}K$
algebrizes $\bDcrisar(\cL)\otimes_{\cO_{\rm cris}}
K$, viewed as a module with connection  on the rigid analytic space $X_K$, so that the filtration on
$\bDcrisar(\cL)_{X_K}:=\bDcrisar(\cL)\otimes_{\cO_{\rm cris}}
K$ defines unique filtrations $\Fil^\bullet \cE_K$ on $\cE_K$ and $\Fil^\bullet \cE_{\cO_K}:=
\Fil^\bullet \cE_K\cap \cE_{\cO_K}$ on $\cE_{\cO_K}$ satisfying
Griffiths' transversality. By GAGA and \ref{prop:HicrisE} we have isomorphisms
$${\rm H}^i\bigl(\bigl(X_k/\WW(k)^+\bigr)^{\rm cris}_{\rm log},
\bDcrisar(\bL)^+\bigr)^{\varphi-{\rm div}}\otimes_{\WW(k)} K \cong
{\rm H}^i_{\rm dR}\bigl(X_K^{\rm alg},\cE_K\bigr)$$and
${\rm H}^i_{\rm
dR}\left(X_k,\bigl(\bDcrisgeo(\bL)_{X_K},\nabla_{\bDcrisgeo(\bL)_{X_K}}\bigr)\right)\cong
{\rm H}^i_{\rm dR}\left(X_K^{\rm alg},\cE_{\cO_K}\widehat{\otimes}_{\cO_K}
\overline{B}_{\rm log}\right)$, as filtered $\overline{B}_{\rm log}$-modules.

It then suffices to prove that the isomorphism of  $ \overline{B}_{\rm log}$-modules
$$g_i\colon {\rm H}^i_{\rm
dR}\bigl(X_K^{\rm alg},\cE_K\bigr)\otimes_{K}\overline{B}_{\rm log} \lra {\rm H}^i_{\rm dR}
\bigl(X_K^{\rm alg},\cE_{\cO_K}\widehat{\otimes}_{\cO_K}
\overline{B}_{\rm log}\bigr)$$is strict with respect to the filtrations. As in the proof of
\cite[Prop.~3.25]{andreatta_iovita_comparison} one shows by a direct
computation that this holds for $i=2 d $ and $\cE=\Omega^d_{X_K/K}$ where $d=\dim X_K$ and  $\Omega^d_{X_K/K}$ are the usual
K\"ahler differentials. In this
case both groups are isomorphic to $\overline{B}_{\rm log}(-d)$ where $(-d)$ stands for the shift in the filtration  by $d$.
Let $X_K^{\rm alg,o}\subset X_K^{\rm
alg}$ be the maximal open subset where the log structure is trivial. The morphism of filtered $\overline{B}_{\rm log}$-modules
$${\rm H}^i_{\rm dR,!}\bigl(X_K^{\rm
alg,o},\cE_K\bigr)\otimes_{K} \overline{B}_{\rm log} \lra {\rm H}^i_{\rm dR,!}\bigl(X_K^{\rm alg,o},\cE_{\cO_K}\widehat{\otimes}_{\cO_K}
\overline{B}_{\rm
log}\bigr)$$of compactly supported cohomology is compatible with the previous one and Poincar\'e duality.
By \cite[Prop. 2.5.3]{MSaito}
Poicar\'e duality provides
an isomorphism
$${\rm H}^i_{\rm dR}\bigl(X_K^{\rm alg},\cE_K\bigr)\lra \Hom_{K}\left({\rm H}^{2d-i}_{\rm dR,!}\bigl(X_K^{\rm alg,o},\cE_K\bigr),
K(-d)\right)$$of filtered
$K$-vector spaces, strict with respect to the filtrations. Then,
$${\rm H}^i_{\rm dR}\bigl(X_K^{\rm alg},\cE_K\bigr)\otimes_{K}
\overline{B}_{\rm log}\lra \Hom_{K}\left({\rm H}^{2d-i}_{\rm dR,!}\bigl(X_K^{\rm alg,o},\cE_K\bigr),
\overline{B}_{\rm log}(-d)\right)$$is a strict isomorphism of
filtered $\overline{B}_{\rm log}$-modules. Since it factors via $g_i$, also $g_i$ must be strict with respect to the filtrations.

(4) As $\gamma^i_\cL$ and $\overline{\gamma}^i_\cL$ are isomorphisms and $\alpha_i$ is surjective by
\ref{prop:HicrisE}(4), it follows that also $\beta_i$ is surjective.

\end{proof}

Consider the diagram obtained by tensoring  $\bL$ with the fundamental exact diagram (\ref{display:fundamentalexactdiagram}) in \S \ref{section:fundamentaldiagram}.
Taking the long exact sequence in cohomology, we get a commutative diagram of $\Q_p$-modules endowed with continuous action of $G_K$, whose rows are exact:
\begin{equation}\label{longfundsequence}
\begin{array}{cccccccccc} \cdots  \lra  & {\rm
H}^i(\fX_\Kbar, \bL) & \lra &{\rm H^i}\bigl(\fX_\Kbar, \bL\otimes
{\rm Fil}^0 \bB_{\rm log,\Kbar}^\nabla\bigr)& \lra & {\rm
H}^i\bigl(\fX_\Kbar,
\bL\otimes {\cal C}_{\rm cris} \bigr) & \lra  \cdots\\
 &\big\downarrow & & \big\downarrow &&\big\downarrow \\
\cdots \lra &{\rm H}^i\left(\fX_\Kbar, \bL\otimes\bB_{\rm
cris,\Kbar}^{\nabla,\varphi=1}\right) & \lra & {\rm
H}^i\bigl(\bigl(X_0/\cO_{\rm cris}\bigr)^{\rm cris}_{\rm log},
\bDcrisgeo(\bL)\bigr) & \lra & {\rm H}^i\bigl(\fX_\Kbar,
\bL\otimes {\cal C}_{\rm log}\bigr) & \lra \cdots
\end{array}\end{equation} Here we have used  the above \ref{thm:formalcomparison} to
identify ${\rm H}^i\bigl(\fX_\Kbar, \bL\otimes\bB_{\rm log,\Kbar}^\nabla\bigr)\cong {\rm H}^i\bigl(\bigl(X_0/\cO_{\rm
cris}\bigr)^{\rm cris}_{\rm log}, \bDcrisgeo(\bL)\bigr)$,
compatibly with monodromy operators, Frobenii and $G_K$-action.

Recall that $$D_i(\cL)\otimes_{\WW(k)} B_{\rm log}\cong {\rm
H}^i\bigl(\bigl(X_0/\cO_{\rm cris}\bigr)^{\rm cris}_{\rm log},
\bDcrisgeo(\bL)\bigr)^{\varphi-{\rm div}}$$compatible with
monodromy operators, Frobenii and $G_K$-actions. It is also
compatible with filtrations where the latter is endowed  with the
filtration induced from ${\rm H}^i\bigl(\bigl(X_0/\cO_{\rm
cris}\bigr)^{\rm cris}_{\rm log}, \bDcrisgeo(\bL)\bigr)$. Then, with the notation of \ref{prop:admis}, we have:

\begin{proposition}\label{prop:pieces} (1) The isomorphism $D_i(\cL)\otimes_{\WW(k)} B_{\rm log}\cong {\rm
H}^i\bigl(\bigl(X_0/\cO_{\rm cris}\bigr)^{\rm cris}_{\rm log},
\bDcrisgeo(\bL)\bigr)^{\varphi-{\rm div}}$ is compatible with filtrations, monodromy operators,
Frobenii and $G_K$-actions.\smallskip

(2)  We have a homomorphism of $G_K$-modules  $$ \frac{{\rm H}^i\bigl(\fX_\Kbar, \bL\otimes\bB_{\rm cris,\Kbar}^\nabla
\bigr)}{\Fil^0 {\rm H}^i\bigl(\fX_\Kbar,
\bL\otimes\bB_{\rm cris,\Kbar}^\nabla \bigr)} \stackrel{\iota}{\lra} \frac{{\rm H}^i\bigl(\fX_\Kbar,
\bL\otimes\overline{\bB}_{\rm log,\Kbar}^\nabla\bigr)}{\Fil^0
{\rm H}^i\bigl(\fX_\Kbar, \bL\otimes\overline{\bB}_{\rm log,\Kbar}^\nabla\bigr)}\stackrel{\rho}{\lra}
V^1_{\rm st}\left(D_i(\cL)\otimes_{\WW(k)} B_{\rm
log}^{G_K}\right),$$where $\iota$ is injective and $\rho$ is an isomorphism.
\smallskip

(3) The image of $u_i\colon {\rm H}^i\left(\fX_\Kbar, \bL\otimes\bB_{\rm cris,\Kbar}^{\nabla,\varphi=1}\right) \lra  {\rm
H}^i\bigl(\bigl(X_0/\cO_{\rm cris}\bigr)^{\rm cris}_{\rm log},
\bDcrisgeo(\bL)\bigr) $ is contained in $D_i(\cL)\otimes_{\WW(k)}
B_{\rm log}$ and its image is $$V_{\rm
log}^0\left(D_i(\cL)\otimes_{\WW(k)} B_{\rm
log}^{G_K}\right):=\left(D_i(\cL)\otimes_{\WW(k)} B_{\rm
log}\right)^{N=0,\varphi=1}$$which coincides with ${\rm H}^i\bigl(\bigl(X_0/\cO_{\rm cris}\bigr)^{\rm cris}_{\rm log},
\bDcrisgeo(\bL)\bigr)^{N=0,\varphi=1}\cong {\rm H}^i\bigl(\fX_\Kbar, \bL\otimes \bB_{\rm
cris,\Kbar}^\nabla\bigr)^{\varphi=1}$. \smallskip

(4) The $G_K$-submodule $V_{\rm log}\left(D_i(\cL)\otimes_{\WW(k)}
B_{\rm log}^{G_K}\right)$ of $D_i(\cL)\otimes_{\WW(k)} B_{\rm
log}$ coincides with  the image of ${\rm H}^i(\fX_\Kbar,
\bL)$;\smallskip

(5.i) $V_{\rm log}\left(D_i(\cL)\otimes_{\WW(k)} B_{\rm
log}^{G_K}\right)$ is finite dimensional as $\Q_p$-vector space
and it is a semistable representation of $G_K$ for every $i$;\smallskip

(5.ii) the maps ${\rm H^{j}}\bigl(\fX_\Kbar, \bL\otimes \Fil^0 \bB_{\rm cris,\Kbar}^{\nabla} \bigr)
\lra {\rm H}^{j}\left(\fX_\Kbar, \bL\otimes\bB_{\rm cris,\Kbar}^{\nabla}\right)$ are injective for every $j$;\smallskip

(5.iii) the morphism $\iota$ in (2) is an isomorphism and we have a long exact sequence
$$\cdots\to {\rm H^{i}}\bigl(\fX_\Kbar, \bL\bigr) \lra
{\rm H}^i\left(\fX_\Kbar, \bL\otimes\bB_{\rm cris,\Kbar}^{\nabla,\varphi=1}\right){\lra} V^1_{\rm
log}\left(D_i(\cL)\otimes_{\WW(k)} B_{\rm log}^{G_K}\right)\lra
{\rm H^{i+1}}\bigl(\fX_\Kbar, \bL\bigr)\to \cdots$$

\end{proposition}

\begin{proof} (1) follows from  \ref{prop:Nonto}(3).
The existence of $\iota$ follows from \ref{prop:deRhamcomplexBbarlog}(i). The fact that $\rho$ is an isomorphism
follows from  \ref{prop:Nonto}(3) and
\ref{sec:formalcomparison}. As $\overline{\bB}_{\rm log,\Kbar}^\nabla\cong  \bA_{\rm inf,\Kbar}^+\otimes_{\WW(k)}
\overline{B}_{\rm log} $ and $\bB_{\rm
cris,\Kbar}^\nabla\cong  \bA_{\rm inf,\Kbar}^+\otimes_{\WW(k)} {B}_{\rm cris}$ by
\S\ref{sec:defBbarlog} and \S\ref{sec:Alognabla} and as ${B}_{\rm cris}/
\Fil^0 {B}_{\rm cris} \cong \overline{B}_{\rm log}/ \Fil^0 \overline{B}_{\rm log}$ by \ref{lemma:BlogmodPpi},
we get an isomorphism  ${\bB}_{\rm
cris,\Kbar}^\nabla/\Fil^0 {\bB}_{\rm cris,\Kbar}^\nabla \cong \overline{\bB}_{\rm log,\Kbar}^\nabla/\Fil^0
\overline{\bB}_{\rm log,\Kbar}^\nabla$. Using the
inclusion $$\frac{{\rm H}^i\bigl(\fX_\Kbar, \bL\otimes\bB_{\rm cris,\Kbar}^\nabla \bigr)}{\Fil^0
{\rm H}^i\bigl(\fX_\Kbar, \bL\otimes\bB_{\rm cris,\Kbar}^\nabla
\bigr)} \subset {\rm H}^i\bigl(\fX_\Kbar, \bL\otimes\bB_{\rm cris,\Kbar}^\nabla/\Fil^0 \bB_{\rm cris,\Kbar}^\nabla \bigr)
\cong {\rm H}^i\bigl(\fX_\Kbar,
\bL\otimes\overline{\bB}_{\rm log,\Kbar}^\nabla/\Fil^0 \overline{\bB}_{\rm log,\Kbar}^\nabla \bigr),$$
which contains ${{\rm H}^i\bigl(\fX_\Kbar,
\bL\otimes\overline{\bB}_{\rm log,\Kbar}^\nabla\bigr)}/{\Fil^0 {\rm H}^i\bigl(\fX_\Kbar, \bL\otimes
\overline{\bB}_{\rm log,\Kbar}^\nabla\bigr)}$ as a submodule, we
deduce that $\iota$ is injective.

(3) Since Frobenius is the identity on ${\rm H}^i\left(\fX_\Kbar,
\bL\otimes\bB_{\rm cris,\Kbar}^{\nabla,\varphi=1}\right)$, the first
claim is clear. The composite of the map  $t_i\colon {\rm
H^i}\bigl(\fX_\Kbar, \bL\otimes \bB_{\rm log,\Kbar}^\nabla\bigr)
\lra {\rm H}^i\bigl(\fX_\Kbar, \bL\otimes {\cal C}_{\rm log}\bigr)$ with the
map $s_i$ in \ref{prop:Nonto}(2) is identified with the map
$$(\varphi-1,N)\colon {\rm H^i}\bigl(\fX_\Kbar, \bL\otimes \bB_{\rm
log,\Kbar}^\nabla\bigr) \lra {\rm H^i}\bigl(\fX_\Kbar, \bL\otimes
\bB_{\rm log,\Kbar}^\nabla\bigr)\oplus {\rm H^i}\bigl(\fX_\Kbar,
\bL\otimes \bB_{\rm log,\Kbar}^\nabla\bigr).$$Due to
\ref{prop:Nonto}(2) the kernel of $t_i$, which is the image of $u_i$,
coincides with the kernel of $s_i\circ t_i$. By \ref{prop:Nonto}(1) the kernel of
$s_i\circ t_i$ is ${\rm H^i}\bigl(\fX_\Kbar, \bL\otimes \bB_{\rm
log,\Kbar}^\nabla\bigr)^{\varphi=1}$. This proves the second claim except for the last isomorphism
in the display. To get this it suffices to remark that ${\rm
H}^i\bigl(\bigl(X_0/\cO_{\rm cris}\bigr)^{\rm cris}_{\rm log},
\bDcrisgeo(\bL)\bigr)^{\varphi=1}\subset {\rm
H}^i\bigl(\bigl(X_0/\cO_{\rm cris}\bigr)^{\rm cris}_{\rm log},
\bDcrisgeo(\bL)\bigr)^{\varphi-{\rm div}}$. The conclusion
follows.

(4) An element $x$ in $V_{\rm st}\left(D_i(\cL)\otimes_{\WW(k)} B_{\rm log}^{G_K}\right)$ is in the image of ${\rm H}^i\left(\fX_\Kbar,
\bL\otimes\bB_{\rm cris,\Kbar}^{\nabla,\varphi=1}\right)$ by (3). Thanks to the injectivity of $\iota$ proven in (2)
the element $x$  is also the image of some $y\in {\rm
H}^i\left(\fX_\Kbar, \bL\otimes \Fil^0 \bB_{\rm cris,\Kbar}^{\nabla}\right)$ by (2). This implies that
$(\varphi-1)(y)=0$ using the long exact sequence in cohomology
defined by tensoring the first diagram in \S \ref{section:fundamentaldiagram} with $\cL$.  We conclude that $y$ is
in the image of ${\rm H}^i(\fX_\Kbar, \bL)$ as
wanted.

(5.i) This follows from (4) and \ref{prop:admis}.

(5.ii) Let $Q$ be the kernel of the  map ${\rm H^j}\bigl(\fX_\Kbar, \bL\otimes {\rm Fil}^0 \bB_{\rm cris,
\Kbar}^\nabla\bigr)\lra D_j(\cL)\otimes_{\WW(k)} B_{\rm
log}$. Then, $Q$ is a $B_{\rm cris}$-module. Since $B_{\rm cris}$ contains the maximal unramified extension
$\Q_p^{\rm un}$ of $\Q_p$, then $Q$ can be considered as
a vector space over $\Q_p^{\rm un}$. A diagram chase in (\ref{longfundsequence}) and the last assertion in \ref{prop:Nonto}(2) imply that $Q$ is in
the image of ${\rm H}^j(\fX_\Kbar, \bL)$ which is a finite dimensional $\Q_p$-vector space by \ref{thm:finiteFaltings}. Hence $Q$ must be trivial.

(5.iii) Using the long exact sequence in cohomology  associated to  the exact sequence $0\lra \bL\otimes
\Fil^0\bB_{\rm cris,\Kbar}^{\nabla} \lra \bL\otimes \bB_{\rm
cris,\Kbar}^{\nabla} \lra \bL\otimes \bB_{\rm cris,\Kbar}^{\nabla}/\Fil^0 \bB_{\rm cris,\Kbar}^{\nabla} \lra 0$
and  (5.ii) we get that $$\frac{{\rm H}^i\bigl(\fX_\Kbar, \bL\otimes\bB_{\rm cris,\Kbar}^\nabla \bigr)}{\Fil^0
{\rm H}^i\bigl(\fX_\Kbar, \bL\otimes\bB_{\rm cris,\Kbar}^\nabla
\bigr)} \cong {\rm H}^i\bigl(\fX_\Kbar, \bL\otimes\bB_{\rm cris,\Kbar}^\nabla/\Fil^0 \bB_{\rm cris,\Kbar}^\nabla \bigr).$$
This and the argument in (2) imply that $\iota$ is an isomorphism.  As
$$\bL\otimes \bB_{\rm cris,\Kbar}^{\nabla}/\Fil^0 \bB_{\rm cris,\Kbar}^{\nabla}\cong  \bigl(\bL\otimes
\bB_{\rm cris,\Kbar}^{\nabla,\varphi=1}\bigr)/\bL$$by \S\ref{section:fundamentaldiagram} we deduce the second claim by considering the cohomology
of the exact sequence $0\lra \bL \lra \bL\otimes \bB_{\rm
cris,\Kbar}^{\nabla,\varphi=1} \lra \bigl(\bL\otimes \bB_{\rm
cris,\Kbar}^{\nabla,\varphi=1}\bigr)/ \bL \lra 0$.

\end{proof}

\begin{corollary}\label{cor:criterion}
The filtered $(\varphi,N)$-module  $D_i(\cL)\otimes_{\WW(k)} B_{\rm log}^{G_K}$ is admissible
and it is associated to the semistable representation
${\rm H^i}\bigl(\fX_\Kbar, \bL\bigr)$ of $G_K$; \smallskip

\end{corollary}
\begin{proof} Thanks to \ref{prop:admis}, \enspace
(1) the filtered $(\varphi,N)$-module $D_i(\cL)\otimes_{\WW(k)} B_{\rm log}^{G_K}$ is admissible if and only if \enspace (2)
the map $\delta\left(D_i(\cL)\right)$ is surjective. Assume that (1) holds. The map $$h_i\colon {\rm H^{i}}\bigl(\fX_\Kbar,
\bL\bigr)\lra V_{\rm st}\left(D_i(\cL)\otimes_{\WW(k)} B_{\rm log}^{G_K}\right)$$
is surjective by \ref{prop:pieces}(4). Its kernel coincides with $$\frac{{\rm H^i}\bigl(\fX_\Kbar,
\bL\otimes \bB_{\rm cris,\Kbar}^\nabla\bigr)}{ (\varphi-1) {\rm H^i}\bigl(\fX_\Kbar, \bL\otimes \bB_{\rm
cris,\Kbar}^\nabla\bigr)} \cong \frac{{\rm H}^i\bigl(\bigl(X_0/\cO_{\rm cris}\bigr)^{\rm cris}_{\rm log},
\bDcrisgeo(\bL)\bigr)^{N=0}}{(\varphi-1) {\rm
H}^i\bigl(\bigl(X_0/\cO_{\rm cris}\bigr)^{\rm cris}_{\rm log}, \bDcrisgeo(\bL)\bigr)^{N=0}}$$
by \ref{prop:Nonto}(1) using the long exact sequence in cohomology associated to the crystalline fundamental
diagram in \S\ref{section:fundamentaldiagram}. Note that $\bigl(D_i(\cL)\otimes_{\WW(k)}
B_{\rm log}\bigr)^{N=0}/(\varphi-1) \bigl(D_i(\cL)\otimes_{\WW(k)} B_{\rm log}\bigr)^{N=0}$ is $0$ since
$D_i(\cL)\otimes_{\WW(k)} B_{\rm log}^{G_K}$ is admissible.
By definition $\bigl(D_i(\cL)\otimes_{\WW(k)} B_{\rm log}\bigr)^{N=0}$ contains the image of Frobenius on
${\rm H}^i\bigl(\bigl(X_0/\cO_{\rm cris}\bigr)^{\rm
cris}_{\rm log}, \bDcrisgeo(\bL)\bigr)^{N=0}$. Hence, $\varphi-1$ on ${\rm H}^i\bigl(\bigl(X_0/\cO_{\rm cris}\bigr)^{\rm cris}_{\rm log},
\bDcrisgeo(\bL)\bigr)^{N=0}/\bigl(D_i(\cL)\otimes_{\WW(k)} B_{\rm log}\bigr)^{N=0}$ is the operator $-1$ which is an isomorphism.
We conclude that the map
$\varphi-1$ is surjective on ${\rm H}^i\bigl(\bigl(X_0/\cO_{\rm cris}\bigr)^{\rm cris}_{\rm log},
\bDcrisgeo(\bL)\bigr)^{N=0}$. Thus the map $h_i$ is an isomorphism.

\smallskip

We are left to prove that one of these equivalent statements is true. Due to (2) it suffices to show that the map $\delta\left(D_i(\cL)\right)$
is surjective. Let $V:=V_{\rm st}\left(D_i(\cL)\otimes_{\WW(k)}
B_{\rm log}^{G_K}\right)$ and put
$D':=\bigl(V\otimes_{\Q_p} B_{\rm st}\bigr)^{G_K}$. Due to \ref{prop:admis} the
filtered $(\varphi,N)$-module $D'$ is admissible and $V=V_{\rm st}\left(D'\otimes_{\WW(k)} B_{\rm log}^{G_K}\right)$.
It then suffices to prove that $D'=D_i(\cL)$.
We argue as in \cite[Prop. 5.6\& Prop. 5.7]{colmez_fontaine}. Let $D:=D_i(\cL)/D'$. Consider the commutative
diagram $$\begin{array}{ccccccc}  &  & & 0 & & 0 \\ &  & & \downarrow & & \downarrow \\
0\lra & V & \lra & V_{\rm log}^0\left(D'\otimes_{\WW(k)} B_{\rm
log}^{G_K}\right) & \lra & V_{\rm log}^1\left(D'\otimes_{\WW(k)}
B_{\rm log}^{G_K}\right) &
\lra 0\cr & \Vert & & \downarrow & & \downarrow \\
0\lra & V & \lra & V_{\rm log}^0\left(D_i(\cL)\otimes_{\WW(k)}
B_{\rm log}^{G_K}\right) & \lra & V_{\rm
log}^1\left(D_i(\cL)\otimes_{\WW(k)} B_{\rm log}^{G_K}\right) &
\\  &  & & \downarrow & & \downarrow \\
&  &  & V_{\rm log}^0\left(D\otimes_{\WW(k)} B_{\rm
log}^{G_K}\right) & \lra & V_{\rm log}^1\left(D\otimes_{\WW(k)}
B_{\rm log}^{G_K}\right) \cr
&  & & \downarrow & & \downarrow \\ &  & & 0 & & 0, \\
\end{array}$$
where the first line is exact since $D'$ is admissible and  the
columns are exact by \ref{prop:admis}. Thus the map $\delta(D)$ is
injective and its cokernel coincide with the cokernel of
$\delta\left(D_i(\cL)\right)$.

Let $h$ be the dimension of $D$ as $K_0$-vector space. Fix a basis $\{d_1,\ldots,d_h\}$ adapted to the filtration and for every $j=1,\ldots,h$ let $i_j$ be such
that $d_j\in \Fil^{i_j}D\backslash \Fil^{i_j+1}D $. Fix $r$ such that $r> i_j$ for every $j$. Let $n$ be the dimension of the $\Q_p$-vector space ${\rm
H^{i+1}}\bigl(\fX_\Kbar, \bL\bigr)$. We consider the Galois twist of $\delta(D)$ by $\Q_p(r)$. Due to \ref{prop:pieces}(3)\&(5.iii) we have ${\rm
Coker}\bigl(\delta(D)\bigr)(r)\cong {\rm Coker}\bigl(\delta(D_i(\cL))\bigr)(r) \subset {\rm H^{i+1}}\bigl(\fX_\Kbar, \bL\bigr)(r)$ so that its dimension is bounded
by $n$.  Let $K \subset K'$ be a totally ramified extension of degree $s >0$. Then, $V_{\rm log}^0\left(D\otimes_{\WW(k)} B_{\rm log}^{G_K}\right)(r)\subset
\bigl(D\otimes_{K_0} B_{\rm st}\bigr)(r)$. Since $t$ is invertible in  $B_{\rm st}$ we have $\bigl(D\otimes_{K_0} B_{\rm st}\bigr)(r) \cong B_{\rm st}^h$ so that
its $G_{K'}$-invariants are $K_0^h$. On the other hand, $V_{\rm log}^1\left(D\otimes_{K_0} B_{\rm log}^{G_K}\right)(r)= \oplus_{j=1}^h \bigl(B_{\rm dR}/B_{\rm
dR}^+\bigr) t^{r-i_j}\otimes d_j$ as Galois module, i.e., it is isomorphic to $\oplus_{j=1}^h B_{\rm dR}/ \Fil^{r-i_j} B_{\rm dR}$. In particular, its
$G_{K'}$-invariants coincide with $\oplus_{j=1}^h K'$; see \cite[\S 1.5]{colmez_fontaine}. Hence, ${\rm H}^0\left(G_{K'},{\rm Coker}\bigl(\delta(D)(r)\bigr)\right)$
has dimension as $K_0$-vector space at least $(s-1) h$. On the other hand, it is bounded by $n$. Since $s$ can be chosen arbitrarily large, the only possibility is
that $h=0$ so that $D=0$ as wanted.

\end{proof}

\begin{proof}(of theorem \ref{thm:mainthm})\enspace It follows from
\ref{cor:criterion} that ${\rm H^i}\bigl(\fX_\Kbar, \bL\bigr)$ is
a semistable representation of $G_K$ with associated filtered
$(\varphi,N)$-module $D_i(\cL)$. Since semistable representations
of $G_K$ form an abelian tensor category and $D_{\rm st}$ is exact
and since ${\rm H}^i\bigl(\bigl(X_k/\WW(k)^+\bigr)^{\rm cris}_{\rm
log}, \_\bigr)$ is a $\delta$-functor, the statement of
\ref{thm:mainthm} regarding the isomorphism as $\delta$-functors
is clear. The functoriality is also clear. Note that ${\rm
H}^i\bigl(\bigl(X_k/\WW(k)^+\bigr)^{\rm cris}_{\rm log},
\bDcrisar(\bL)^+\bigr)$ satisfies K\"unneth formula by
\cite[Thm.~6.12]{katolog}. The category of semistable sheaves is
closed under tensor products and $\bDcrisar$ commutes with tensor
products by \ref{prop:BdCristannakianequivalence}. Thus the
compatibility with K\"unneth formula holds as well.

\smallskip

\end{proof}

\section{Relative Fontaine's theory}

\subsection{Notations. First properties}\label{section:localdescription}

Let $R$ be an $\cO_K$--algebra. Assume that there exist a positive
integer~$\alpha$,  non--negative integers $a$ and~$b$ and elements
$X_1,\ldots,X_a$ and $Y_1,\ldots,Y_b\in R$ such that $X_1\cdots
X_a=\pi^\alpha$ and the properties numbered (1),(2),(3),(4) below hold. We start by
defining the monoids
$P_a:=\N^a$ and $P_b:=\N^b$, put $P:=P_a\times P_b$ and we let
$$\psi_R\colon  P \lra R,\qquad
\mbox{ be defined by }\bigl(\alpha_1,\ldots,\alpha_a,\beta_1,\ldots,\beta_b\bigr)\mapsto
\prod_{i=1}^a X_i^{\alpha_i}\prod_{j=1}^b Y_j^{\beta_j}.$$It
induces a morphism of $\cO_K$--algebras $\psi_R\colon\cO_K[P]\to
R$.  We then get a commutative diagram of morphisms of
$\cO_K$--algebras

$$\begin{array}{ccc}
\cO_K[P] & \stackrel{\psi_R}{\longrightarrow} & R \cr \uparrow & &
\uparrow\cr \cO_K[\N] &\stackrel{\psi_\alpha}{\longrightarrow} &
\cO_K,\cr
\end{array}$$where the left vertical map is induced by the map
$\N\to P=P_a\times P_b$, $n\mapsto
\bigl((n,\ldots,n),(0,\ldots,0)\bigr)$ and $\psi_\alpha\colon
\cO_K[\N]\to \cO_K$ is a morphism of $\cO_K$-algebras sending
$\N\ni 1\mapsto \pi^\alpha$. We assume that the following hold.

\begin{enumerate}

\item[(1)] $R$ is excellent;

\item[(2)] the map  $\Psi_R\colon \cO_K[P]\tensor_{\cO_K[\N]}
\cO_K \to R $ induced by $\psi_R$ has geometrically regular
fibers;

\item[(3)] $R$ is obtained from $\cO_K[P]\tensor_{\cO_K[\N]} \cO_K$ as a succession of extensions $R^{(0)}=\cO_K[P]\tensor_{\cO_K[\N]} \cO_K \subset \cdots \subset
R^{(n)}=R$ such that

\smallskip ({\it ALG}) $R^{(i+1)}$ is obtained from $R^{(i)}$ by
\enspace (loc) localizing with respect to a multiplicative system
or \enspace (\'et) by an \'etale extension.

\smallskip

({\it FORM}) each $R^{(i+1)}$ is $p$--adically complete and separated and it is obtained from $R^{(i)}$ as \enspace (loc) the $p$-adic completion of the
localization with respect to a multiplicative system, \enspace (\'et) the $p$-adic completion of an \'etale extension, \enspace (comp) the completion with respect
to an ideal containing $p$.

\item[(4)] For every subset $J_a\subset \{1,\ldots,a\}$ and every subset $J_b\subset \{1,\ldots,b\}$ the ideal of $R$ generated by $\psi_R\bigl(\N^{J_a}\times
\N^{J_b}\bigr)$ is a prime ideal of $R$,  the ideal of $R$ generated by $\psi_R\bigl(P_a\bigr)$ is not the unit ideal and the image of the monoid $\cO_\Kbar^\ast
\cdot \psi_R\bigl(P_b\bigr )$ is saturated in $R\otimes_{\cO_K} \cO_\Kbar$.
\end{enumerate}

In both cases we consider the log structure on $\Spec(R)$ induced by the one on the spectrum of~$\cO_K[P]\tensor_{\cO_K[\N]} \cO_K$ considering the fibred product
log structure. Here we take on $\Spec\bigl(\cO_K[P]\bigr)$ (resp.~$\Spec\bigl(\cO_K[\N]\bigr)$) the log structure associated to the prelog structure $P\to \cO_K[P]$
(resp.~$\N\to \cO_K[\N]$) and we take on $\Spec(\cO_K)$ the log structure associated to the prelog structure $\N \to \cO_K$ sending $1$ to $ \pi$. In particular the
structure map of schemes $\Spec(R) \to \Spec(\cO_K)$ extends to a morphism of log schemes.
\smallskip

More explicitly let $P'$ be the submonoid of $\frac{1}{\alpha}
P_a\times P_b\subset \Q^a\times P_b$ given by
$$P':=\frac{1}{\alpha}\N+ P\subset \frac{1}{\alpha} P_a\times
P_b,$$where $\frac{1}{\alpha}\N$ is diagonally embedded in
$\frac{1}{\alpha} P_a$.

\begin{lemma}\label{lemma:structureofP'}
(a) The monoid $P'$ is the amalgamated sum of monoids
$P\oplus_{\N} \N$ via the maps \enspace(i) $\N \to P$ given by
$\N\ni n\mapsto \bigl((n,\ldots,n),(0,\ldots,0)\bigr)\in P_a\times
P_b=P$; \enspace (ii) $\N \to \N$ given by $\N\ni n\mapsto \alpha
n$.

\smallskip

(b) The monoid $P'$ is fine and saturated.\smallskip

(c) The structural morphism of log schemes $q\colon
\bigl(\Spec\bigl(\cO_K[P]\tensor_{\cO_K[\N]} \cO_K\bigr),P'\bigr)
\lra \bigl(\Spec(\cO_K),\N\bigr)$ is log smooth of relative
dimension $a+b-1$, in the sense of \cite[\S3.3]{katolog}.

\end{lemma}
\begin{proof} (a) By construction we  have a surjective morphism
of monoids $f\colon \N \oplus P \to P'$ given by the inclusion $P\subset P'$ and the map $\N \to P'$ sending $n$ to $ \bigl(\frac{1}{\alpha}
(n,\ldots,n),(0,\ldots,0)\bigr)$. The equalizer of $f$ is $\N$ mapping to $\N \oplus P$ via the maps given in (i) and (ii). By the universal properties of the
amalgamated sum we thus get an isomorphism $P\oplus_{\N} \N\cong P'$.

(b) The group $P^{', \rm gp}$ generated by $P'$ is the subgroup
$\frac{1}{\alpha}\Z+ P^{\rm gp}$ of $\frac{1}{\alpha} P_a^{\rm
gp}\times P_b^{\rm gp}$ which is torsion free and abelian. In
particular, $P'$ is integral and it  is clearly finitely generated. We
prove now that $P'$ is saturated.

Every element $a\in P^{', \rm gp}$ can be written as $a=
\bigl(h/\alpha+h_1, \ldots, h/\alpha+h_a,m_1,\ldots,m_b)$ for a
unique positive integer $0\leq h\leq \alpha-1$ and unique integers
$h_1,\ldots,h_a,m_1,\ldots,m_b$. It lies in $P'$ if an only if
$h_1,\ldots,h_a,m_1,\ldots,m_b\in\N$. Let $0\neq \beta \in\N$ be
such that $\beta a\in P'$. Write $\beta h= r \alpha + h'$, the
division by $\alpha$ with reminder $0\leq h'\leq \alpha-1$. Then,
$\beta a=\bigl(h'/\alpha+(\beta h_1+r), \ldots, h/\alpha+(\beta
h_a+r),\beta m_1,\ldots,\beta m_b)$ so that $\beta h_1+r, \ldots,
\beta h_a+r,\beta m_1,\ldots,\beta m_b\in \N$. This implies that
$m_1,\ldots,m_b$ are non-negative, and hence lie in $\N$, and that
for every $1\leq i \leq a$ we have $r+\beta h_i\geq 0$,
i.e.~$\alpha \beta h_i\geq -r \alpha=h'-\beta h\geq -\beta h$. We
conclude that $\alpha h_i \geq - h > -\alpha$ so that  $h_i>-1$,
i.e.~$h_i\in \N$. This implies that $a\in P'$ to start with.

(c) The map of monoids $\iota\colon \N \to P'$ sending $n$ to $ \bigl(\frac{1}{\alpha} (n,\ldots,n),(0,\ldots,0)\bigr)$ is injective. At the level of associated
groups $\iota^{\rm gp}$ remains injective and the quotient  is isomorphic to $\Z^{a+b-1}$. Thus, $q$ is log smooth of the
claimed relative dimension by \cite[Prop. 3.4]{katolog}.

\end{proof}

For every $n\in\N$ write $R_n=R$ if~$n=0$ and let $R_n$ be
$$R_n:=R\otimes_{\cO_K} \cO_{K'_n}\bigl[X_1^{\frac{1}{n!}},\ldots,X_a^{\frac{1}{
n!}},Y_1^{\frac{1}{n!}},\ldots,Y_b^{\frac{1}{n!}}\bigr]/\bigl( X_1^{\frac{1}{n!}}\cdots X_a^{\frac{1}{ n!}}-\pi^{\frac{\alpha}{n!}}\bigr) $$for $n\geq 1$. Then,
$\Spec\bigl(R_n\bigr)$ has a log structure $N_n$ defined by the chart $\cO_K\bigl[\frac{1}{n!}P]\tensor_{\cO_K\bigl[\frac{1}{n!}\N\bigr]} \cO_{K_n'}\to R_n$
considering on $\cO_K\bigl[\frac{1}{n!}P]\tensor_{\cO_K\bigl[\frac{1}{n!}\N\bigr]} \cO_{K_n'}$ the fibred product log structure, where:\smallskip

(i) we endow $\Spec\bigl(\cO_K\bigl[\frac{1}{n!}\N\bigr]\bigr)$ and $\Spec\bigl(\cO_K\bigl[\frac{1}{n!}P]\bigr)$ with the log structures having $\frac{1}{n!}\N\to
\cO_K\bigl[\frac{1}{n!}\N\bigr]$ and  respectively $\frac{1}{n!}P\lra \cO_K\bigl[\frac{1}{n!}P]$ as charts;\smallskip

(ii) the log structure on $\cO_{K_n'}$ is the one associated to the map $\N \to \cO_{K_n'}$ defined by $1\mapsto \pi^{\frac{1}{n!}}$;\smallskip

(iii) the map $\cO_K\bigl[\frac{1}{n!}\N\bigr]\to \cO_K\bigl[\frac{1}{n!}P]$ is the map of $\cO_K$-algebras defined by $\frac{d}{n!}\mapsto \frac{1}{n!}(d
,\ldots,d,0,\ldots,0)$;
\smallskip

(iv) the map $\cO_K\bigl[\frac{1}{n!}\N\bigr]\to \cO_{K_n'}$ is the map of $\cO_K$-algebras defined by $\frac{1}{n!}\mapsto \pi^{\frac{\alpha}{n!}}$;\smallskip

(v) the map $\cO_K\bigl[\frac{1}{n!}P] \to R_n$ is the map of $\cO_K$-algebras defined by
$$\frac{1}{n!}\bigl(\alpha_1,\ldots,\alpha_a,\beta_1,\ldots,\beta_b\bigr)\mapsto
\prod_{i=1}^a X_i^{\frac{\alpha_i}{n!}}\prod_{j=1}^b
Y_j^{\frac{\beta_j}{n!}}$$

Equivalently, proceeding as in the case $n=1$ treated before, we
have an isomorphism
$$\cO_K\bigl[\frac{1}{n!}P]\tensor_{\cO_K\bigl[\frac{1}{n!}\N\bigr]}
\cO_{K_n'}\cong \cO_{K_n'}\bigl[\frac{1}{n!} P'\bigr]$$and the log structure on $\Spec\bigl(R_n\bigr)$ is the one associated to the morphism of monoids
$\frac{1}{n!} P'\to R_n$. We also define
$$R^o:=R\bigl[X_1^{\frac{1}{\alpha}},\ldots,X_a^{\frac{1}{\alpha}}\bigr]/\bigl( X_1^{\frac{1}{\alpha}}\cdots X_a^{\frac{1}{
\alpha}}-\pi\bigr)  \subset R_\alpha$$with log structure on $\Spec\bigl(R^o\bigr)$ associated the morphism of monoids $\bigl(\frac{1}{\alpha} P_a\bigr) \times
P_b\to R^o$ sending $\frac{1}{\alpha}\bigl(u_1,\ldots,u_a) \times (v_1,\ldots,v_b\bigr)$ to $\prod_{i=1}^a X_i^{\frac{u_i}{\alpha}}\prod_{j=1}^b Y_j^{{v_j}}$. We
consider it as a log scheme over $\Spec(\cO_K)$ where the map on log structures is associated to the map of monoids $\N \to \bigl(\frac{1}{\alpha} P_a\bigr) \times
P_b$ sending $n\in \N$ to ${\frac{1}{\alpha}}(n,\ldots,n,0\ldots,0)$.

\subsubsection{First properties of $R_n$}\label{lemma:Rinftyflat}
The following hold:\begin{itemize}

\item[(1)] $R_n$ and $\widehat{R}_n$ (resp.~$R^o$ and $\widehat{R}^o$) are flat $\cO_{K_n'}$--algebras (resp.~$\cO_K$-algebras);

\item[(2)] the extension $R\to R_n$ is $\pi^{\alpha}$-flat, i.e., the base change of an injective morphism of $R$-modules has kernel annihilated by $\pi^{\alpha}$.

\item[(3)] $R_n$ is a Cohen-Macaulay ring and, in particular, it is normal. It is regular if $\alpha=1$.

\item[(4)] $R^o$ is a regular ring. Moreover, $R$ is a direct summand in $R^o$ as $R$-module and $\pi^\alpha R^o$ is contained in a finite and free $R$-submodule of $R^o$. We have $R=R^o$ if and only if
$\alpha=1$.

\item[(5)] $R$ is an integral domain.

\end{itemize}

\begin{proof}

Since $R_n$ (resp.~$R^o$) is noetherian  claims (1)--(2) for $\widehat{R}_n$ (resp.~$\widehat{R}^o$) follow from the claims (1)--(2) for~$R_n$ (resp.~for $R^o$). By
construction $R_n$ is the tensor product of $ \cO_K[P'] \to R$, induced by $\psi_R$, and $R_n':=\cO_K\bigl[\frac{1}{n!}P']$. Thus it suffices to prove claim (2) for
the tower defined by $R_n'$ for $n\in\N$. Since the map $\Psi_R$ is flat by assumption, it suffices to prove  claim (1) for $R_n'$ and similarly we can replace
$R^o$ with $R^{'o}$. \smallskip

(1) We prove it for $R_n'$ leaving the analogous proof for $R^{'o}$ to the reader. Since the element $X_1^{\frac{1}{n!}}\cdots
X_s^{\frac{1}{n!}}-\pi^{\frac{\alpha}{n!}} $ is irreducible in $\cO_{K_n'}\bigl[X_1^{\frac{1}{n!}},\ldots, X_a^{\frac{1}{n!}},Y_1^{\frac{1}{n!}},\ldots,
Y_b^{\frac{1}{n!}}\bigr]$, which is a UFD, it defines a prime ideal. Since the quotient is $R_n'$, the latter is an integral domain and, hence, it is
$\pi^{\frac{1}{n!}}$-torsion free. The claim follows.
\smallskip

(2) Let $A_n\subset R_n'$ be the $R'$-subalgebra generated by $\pi^{\frac{1}{n!}}$,  $X_2^{\frac{1}{n!}}$, $\ldots$, $X_a^{\frac{1}{n!}}$, $Y_1^{\frac{1}{n!}}$ and
$Y_b^{\frac{1}{n!}}$. Since $\pi^{\frac{\alpha}{ n!}}=X_1^{\frac{1}{n!}}\cdots X_a^{\frac{1}{n!}}$ in $R_{n}'$, we compute that
$$\pi^{\alpha} X_1^{\frac{\rho}{n!}}=X_1^{\alpha} X_2^{\alpha-\frac{\rho}{n!} }\cdots X_a^{\alpha-\frac{\rho}{n!}} \pi^{\frac{\alpha\rho}{ n!}}\in A_n.$$ Hence,
$\pi^\alpha R_{n}' \subset A_n \subset R_{n}'$. Furthermore, $A_{n+1}$  is finite and free as $A_n$--module for every $n$. Indeed, since both $A_{n+1}$ and $A_n$
are flat as $\cO_{K_n'}$--modules, it suffices to prove that the elements $\pi^{\frac{\rho}{n!}} X_\beta Y_\gamma$ are linearly independent over
$A_n\bigl[p^{-1}\bigr]=R_n'\bigl[p^{-1}\bigr]$. Since $K_{n+1}'=K_n'\bigl[\pi^{\frac{1}{ (n+1)!}}\bigr]$ is an extension of degree $n+1$, we need only to show that
the elements $X_{\beta'} Y_\gamma:= \prod_{i=2}^{a} X_i^{\frac{\beta_i}{(n+1)!}} \prod_{j=1}^b Y_j^{\frac{\gamma_j}{(n+1)!}}$, with
$\beta'=(\beta_2,\ldots,\beta_a)$ and $0 \leq \beta_i < n+1$ for every $2\leq i\leq a$ and with $\gamma=(\gamma_1,\ldots,\gamma_b)$ and $0 \leq \gamma_j < n+1$ for
every $1\leq j\leq b$, are linearly independent over ${\rm Frac}(R_n')\otimes_{{K_n'}}K_{n+1}'$ and this is clear.
\smallskip

(3) We prove that $R_n$ is Cohen-Macaulay   for $n=0$. The general case follows in the same way after replacing $K$ with $K_n'$ and $R$ with $R_n$.  Assume we are
in the algebraic case. Since the map $\cO_K \to R$ is the base-change of the map $\cO_K[\N]\to \cO_K[P]$, which considered as a map of log schemes is log smooth, it
defines itself a log smooth map. Since $\cO_K $ with the log structure defined by $\pi$ is log regular, \cite[Thm.~8.2]{katotoric} implies that $R$, with its log
structure, is log regular. Then, \cite[Thm.~4.1]{katotoric} implies that $R$ is Cohen-Macaulay and normal as claimed.

In the formal case, due to \cite[Thm.~4.1]{katotoric} it suffices to prove that $R$, with its log structure, is log regular. By construction $R$ is obtained from
$R^{(0)}:=\cO_K[P]\tensor_{\cO_K[\N]} \cO_K$ by taking successive extensions $R^{(i)}\subset R^{(i+1)}$ each given by (\'et) the $p$--adic completion of an \'etale
extension, (loc) the $p$-adic completion of a localization or (comp) the completion with respect to an ideal containing $p$. Since $R^{(0)}$ is log regular by the
argument provided in the algebraic case, it suffices to prove that if $R^{(i)}$ is log regular, then $R^{(i+1)}$ is. We may assume that $R=R^{(i+1)}$. By
\cite[Prop.~7.1]{katotoric} to prove the log regularity of $R$ it suffices to show it at maximal ideals of $R$. Since $R$ is $p$-adically complete and separated,
any such contains $p$. Due to \cite[Thm.~3.1(1)]{katotoric} the log regularity of $R$ at $Q$ is expressed in terms of the completed local ring $\widehat{R}_Q$ of
$R$ at  $Q$, with the induced log structure. Set $Q^{(i)}:=Q\cap R^{(i)}$. Then, $\widehat{R^{(i)}}_{Q^{(i)}}\subset \widehat{R}_{Q}$ is a finite and \'etale
extension in case (\'et) or it is an isomorphism in the other two cases. Since $R^{(i)}$ is log regular by assumption, then \cite[Thm.~3.1(1)]{katotoric} holds for
$\widehat{R^{(i)}}_{Q^{(i)}}$ and hence it holds also for $\widehat{R}_{Q}$ as wanted.

\smallskip
Assume next that $\alpha=1$. We may assume that $n=0$. In case (ALG) the map $\Psi_R$ is the composite of localization and \'etale morphisms. Thus to prove the
regularity of $R_n$ it suffices to  show that $R_n'$ is regular. Since $R_n'[p^{-1}]$ is a smooth $K_n'$--algebra, it is regular. We are left to prove that the
localizations of $R_n'$ at prime ideals containing $p$ are regular. In case (FORM) the map $\Psi_R$ is the composite of $p$--adically formally \'etale morphisms,
$p$-adic completions of localizations and completions with respect to ideals containing $p$. Since it suffices to check regularity for the localization at maximal
ideals and the maximal ideals of a $p$--adically complete ring contain~$p$, it suffices to prove that the $p$--adic completion $\widehat{R}_n'$ of $R_n'$ (for
$R_n'$ as in (ALG)) is regular.

Let $\cal P$ be a prime ideal of $R_n'$ or $\widehat{R}_n'$
containing $p$. Then, it contains
$\pi^{\frac{1}{n!}}=X_1^{\frac{1}{n!}}\cdots X_a^{\frac{1}{n!}}$
and, hence, it contains $X_i^{\frac{1}{n!}}$ for some
$i=1,\ldots,a$. Note that $X_i^{\frac{1}{n!}}$ is a regular
element in $R_n'$ and in $\widehat{R}_n'$ i.e., it is not a zero
divisor. Otherwise $\pi^{\frac{1}{n!}}$ would be a zero divisor.
But this is impossible due to (1). Since $R_n'/ X_i^{\frac{1}{n!}}
R_n'\cong \widehat{R}_n'/ X_i^{\frac{1}{n!}} \widehat{R}_n'$ is a
 smooth $k$--algebra, we deduce that $R_{n,{\cal P}}$
(resp.~$\widehat{R}_{n,{\cal P}}$) is regular as
claimed.\smallskip

(4) The regularity of $R^o$ follows arguing as in the proof of (3). Clearly $R=R^o$ if and only if $\alpha=1$. One shows that in general $\pi^\alpha R^o\subset \oplus R \prod_{i=2}^a X_i^{\frac{\beta_i}{\alpha}}
\prod_{j=1}^b Y_j^{\frac{\gamma_j}{\alpha}}$ with $\beta=(\beta_2,\ldots,\beta_a)$ and $0 \leq \beta_i < \alpha$ for every $2\leq i\leq a$  and with
$\gamma=(\gamma_1,\ldots,\gamma_b)$ and $0 \leq \gamma_j < \alpha$ for every $1\leq j\leq b$
proceeding as in the proof of (2). The proof that $R$ is a direct summand in $R^o$ as $R$-module is reduced to the case that $R=R'=\cO_K[P']$ and
$R^o=R^{'o}=\cO_K\big[\bigl((1/\alpha) P_a\bigr)\times  P_b \big]$. This follows as $R^{'o}=R'\oplus D$ where $D= \sum_{x\in Q} R' x$ with $Q$ the subset of
$(1/\alpha) P_a$ of elements which are not diagonal (i.e., of the form $(u,\ldots,u)\in (1/\alpha) P_a$) and which do not lie in $P_a\subset (1/\alpha) P_a$.

(5) It follows from assumption (4) in \ref{section:localdescription}, taking $J_a=J_b=\emptyset$, that $R$ is an integral domain.

\end{proof}

Let $\Omega$ be an algebraic closure of ${\rm Frac}(R)$. Fix compatible $R$-algebra morphisms $R_n\to \Omega$ for $n\in\N$ and write $$R_\infty\subset \Omega$$for
the union of their images.

\begin{lemma}\label{lemma:Rinftyff} (i) The $R^o$-algebra $R_\infty$ is flat as $R^o$-module.\smallskip

(ii) For every $n$ the image of $R_n\to R_\infty$ is a direct factor of $R_\infty$ and it is a finite $R$-module.
\end{lemma}
\begin{proof} (ii) Since $R_n$ is noetherian and normal,
$R_n$ is the product of finitely many normal integral domains and the image $R_n''\subset \Omega$ of $R_n$ in $\Omega$ is then one of these factors. We conclude
that $\ds \lim_n R_n$ is a product of integral normal domains and its image $R_\infty=\cup_n R_n''$ in $\Omega$ is one of these factors.

(i) We claim that $A:=\ds \lim_n R_n$ is a flat $R^{o}$-module. Thanks to (ii) this proves that $R_\infty$ is flat as $R^o$-module.  Proceeding as in
\ref{lemma:Rinftyflat} we reduce to the case that $R=R'=\cO_K[P']$. In this setting we prove that $A$ is in fact free as $R^{'o}$-module with basis given by the
elements $\prod_{i=1}^{a} X_i^{\frac{\beta_i}{\alpha}} \prod_{j=1}^b Y_j^{{\gamma_j}}$, with rational numbers $0\leq \beta_i <1$ and $0\leq \gamma_j< 1$ for
$i=1,\ldots,a$ and $j=1,\ldots,b$. Indeed, as $\pi^{\alpha c}=X_1^{c}\cdots X_a^{c}$ for every positive $c\in \Q$, then $\cO_{K_\infty'}\subset A$ so that these
elements are generators of $A$ as $R^{'o}$-module.  They also form a basis  over $R^{'o}[p^{-1}]$. As $R^{'o}$ is $p$-torsion free, they form a basis of $A$ as
$R^{'o}$-module.
\end{proof}

\subsubsection{The ring $\widehat{\overline{R}}$}
\label{sec:widehatR}
Let $\cS$ be the set of $R$--subalgebras $S$ of~$\Omega$ such that
\smallskip

(1) $R[Y_1^{\pm 1},\ldots,Y_b^{\pm 1},p^{-1}] \subset S\tensor_R
R[Y_1^{\pm 1},\ldots,Y_b^{\pm 1},p^{-1}]$ is \'etale;\smallskip

(2) $S$ is finite as $R$-module and $S$ is a normal
domain.\smallskip

Then $\cS$  is a directed set with respect to the inclusion.
Define $\overline{R}$ to be the direct limit
$\overline{R}:=\lim_{S\in \cS} S$ and $\widehat{\overline{R}}$ to
be the $p$--adic completion of $\overline{R}$. Put $\cG_R$ to be
the Galois group of $\overline{R}\bigl[p^{-1}\bigr]$ over
$R\bigl[p^{-1}\bigr]$. We endow the $R$--algebra
$\widehat{\overline{R}}$ with the  log structure induced from the
given one on~$R$.

Let $\cS_\infty$ be the set of extensions $R_\infty\subset S_\infty \subset \Omega$ such that $S_\infty$ is finite and \'etale over $R_\infty\bigl[p^{-1}\bigr]$ and
such that $S_\infty$ is normal. Every $S_\infty\in \cS_\infty$  is contained in $\overline{R}$. On the other hand it follows from Abyhankar's lemma \cite[Prop.
XIII.5.2]{SGAI} that every normal extension $R_\infty\subset T$ finite and \'etale over $R_\infty[Y_1^{\pm 1},\ldots,Y_b^{\pm 1},p^{-1}]$ is in fact finite and
\'etale over $R_\infty[p^{-1}]$. Hence, $${\ds \lim_{S\in \cS}} S=:\overline{R}={\ds \lim_{S_\infty \in \cS_\infty}} S_\infty$$For every $S_\infty\in \cS_\infty$
let $ e_{S_\infty}\in S_\infty\tensor_{R_\infty} S_\infty\bigl[p^{-1}\bigr]$ be the canonical idempotent splitting the multiplication map
$S_\infty\tensor_{R_\infty} S_\infty\bigl[p^{-1}\bigr]\to S_\infty\bigl[p^{-1}\bigr]$. We have:

\begin{proposition}\label{prop:AE} For for every~$n\in\N$, the element $\pi^{\frac{1}{ n!}}e_{S_\infty}$ is
in the image of $ S_\infty\tensor_{R_\infty} S_\infty$.
\end{proposition}
\begin{proof} The claim follows from Faltings' Almost
Purity Theorem \cite[Thm.~4]{faltingsAsterisque}. See also \cite[\S 9]{GabberRamero}, especially Theorem 9.6.34.

\end{proof}

It is in the proof of the proposition that  assumptions (1) and (2) on the ring $R$ made in \S\ref{section:localdescription} are used and they might be relaxed
using recent work of P.~Scholze. Let ${m}_{\overline{R}}$ be the ideal of $\overline{R}$ generated by $\pi^{\frac{1}{ n!}}$ for $n\in\N$. Then,

\begin{corollary}\label{cor:overlineRaf}
The extension $R_\infty \subset \overline{R}$ is almost flat. In particular, the extension $R^o\subset \overline{R}$ is ${m}_{\overline{R}}$--flat.
\end{corollary}
\begin{proof} It follows from \ref{prop:AE} and
\ref{lemma:Rinftyff}.

\end{proof}

\begin{proposition}\label{prop:Rbarff} The following hold:\smallskip

(1) the ring $\widehat{\overline{R}}$ is $p$-torsion free and
reduced, the map $\overline{R}\to \widehat{\overline{R}}$ is
injective and $p \widehat{\overline{R}}\cap \overline{R}=p
\overline{R}$;

(2) the extension
$\widehat{R}\bigl[p^{-1}\bigr]\subset\widehat{\overline{R}}\bigl[p^{-1}\bigr]$
is faithfully flat.
\end{proposition}
\begin{proof} (1) The claim, except that $\widehat{\overline{R}}$ is reduced,
is the analogue of \cite[Prop. 2.0.2]{brinon} and the same proof
applies. In our case the key ingredient is that $\overline{R}$
is the direct limit of finite and normal extensions of $R$. Let
$x\in \widehat{\overline{R}}$ be such that $x^n=0$. Since
$\widehat{\overline{R}}$ is $p$-torsion free if $x\neq 0$ we may
assume that $x$ is not divisible by any element of the maximal
ideal of $\cO_{\widehat{\overline{K}}}$. Since $\overline{R}/p
\overline{R}= \widehat{\overline{R}}/p \widehat{\overline{R}}$, we
may write $x=y+p z$ with $x\in \overline{R}$. Then, $y^n\in p
\overline{R}$ and we deduce from the normality of $\overline{R}$
that $y p^{-\frac{1}{n}}\in \overline{R}$. This implies that
$p^{\frac{1}{n}}$ divides $y$ and hence $x$. This leads to a
contradiction. \smallskip

(2) We start with the flatness. It follows from \ref{lemma:Rinftyff} and \cite[Cor.~9.2.7]{brinon}, with $\Lambda=R^o$ and $B=\overline{R}$, that the extension
$\widehat{R}^o\bigl[p^{-1}\bigr]\subset\widehat{\overline{R}}\bigl[p^{-1}\bigr]$ is flat (note that in loc.~cit.~one does not need $\Lambda$ to be $p$--adically
complete). Due to \S\ref{lemma:Rinftyflat} we have that $\pi^\alpha \widehat{R}^o$ is contained in a finite and free $\widehat{R}$-submodule of $\widehat{R}^o$. Thus the map $\widehat{R}[p^{-1}] \to \widehat{R}^o[p^{-1}]$ is flat.

To conclude the proof of the proposition we are left to show that
the image of ${\rm
Spec}\bigl(\widehat{\overline{R}}\bigl[p^{-1}\bigr]\bigr)\to {\rm
Spec}\bigl(\widehat{R}\bigl[p^{-1}\bigr]\bigr)$ contains the
maximal ideals of $\widehat{R}\bigl[p^{-1}\bigr]$. Let $\cP
\subset \widehat{R}$ be a prime ideal such that $\cP
\widehat{R}\bigl[p^{-1}\bigr]$ is maximal. Arguing as in \cite[Thm.
3.2.3]{brinon} one concludes that the ideal $\cP
\widehat{\overline{R}}\bigl[p^{-1}\bigr]$ is not the whole ring
$\widehat{\overline{R}}\bigl[p^{-1}\bigr]$. In particular there
exists a maximal ideal $\cQ$ of
$\widehat{\overline{R}}\bigl[p^{-1}\bigr]$ such that $\cP \subset
\cQ \cap \widehat{R}\bigl[p^{-1}\bigr]$. Since $\cP$ is maximal
the last inclusion is an equality and $\cP$ is the image of $\cQ$.

\end{proof}

\begin{corollary}\label{cor:Frobonto} Frobenius is surjective on $S_\infty/pS_\infty$
for every $S_\infty\in \cS_\infty$ and, in particular, on
$\overline{R}/p\overline{R}$. Moreover, $\pi^{\frac{1}{p}}
S_\infty$ is a finitely generated $R_\infty$-module.
\end{corollary}
\begin{proof}
If $R$ is $p$--adically complete the proof is as in \cite[Prop.~2.0.1]{brinon}. In the general case we proceed as follows. We claim that Frobenius $\varphi$ on
$R_\infty'/p R_\infty'$ is surjective. Recall that $R/p R$ is  obtained as a chain $R'/pR'=R^{(0)}/pR^{(0)} \subset \cdots \subset R^{(n)}/pR^{(n)}=R/pR$ where
$R^{(i+1)}/pR^{(i+1)}$ is obtained from $R^{(i)}/pR^{(i)}$ by taking a localization or an \'etale extension or the completion with respect to an ideal. One then
proves by induction on $i$ that in each case Frobenius on $R^{(i+1)}/p R^{(i+1)}$ induces an isomorphism $R^{(i+1)}/p R^{(i+1)}\otimes_{R^{(i)}/pR^{(i)}}^\varphi
R^{(i)}/p R^{(i)} \lra R^{(i+1)}/p R^{(i+1)}$. In particular Frobenius provides an isomorphism $R/p R\otimes_{R'/pR'}^\varphi R'/p R' \lra R/p R$. Since $R_\infty/p
R_\infty$ is a direct factor of $R_\infty'/p R_\infty'\otimes_{R'} R$ by \ref{lemma:Rinftyff}, Frobenius on $R_\infty/p R_\infty$ induces an isomorphism $R_\infty/p
R_\infty\otimes_{R_\infty'}^\varphi R_\infty'/p R_\infty' \lra R_\infty/p R_\infty$. In particular, since Frobenius $\varphi$ on $R_\infty'/p R_\infty'$ is
surjective then Frobenius is surjective also on $R_\infty/p R_\infty$.

Let $R_\infty\subseteq S_\infty$ be an extension in $\cS_\infty$.
Write $\pi^{\frac{1}{p}}e_{S_\infty}$ as a finite sum of elements
$\sum_{i=1}^n x_i\tensor y_i$ with $x_i$ and $y_i\in S_\infty$.
Then, for every $z\in S_\infty$ we have $\pi^{\frac{1}{p}} z= \sum
{\rm Tr}_{S_\infty/ R_\infty}(zx_i)  y_i $, i.e.,
$\pi^{\frac{1}{p}} S_\infty$ is generated by $y_1,\ldots,y_n$ as
$R_\infty$-module proving the last statement. We also compute
$\pi^{\frac{1}{p}} e_{S_\infty}=\bigl(\pi^{\frac{1}{
p^2}}e_{S_\infty}\bigr)^p=\sum_i x_i^p\tensor y_i^p + p w$ with
$w\in S_\infty\tensor_{R_\infty} S_\infty$. For every $z\in
S_\infty$ we have $\sum_i {\rm Tr}_{S_\infty/ R_\infty} (z x_i^p)
y_i^p= \pi^{\frac{1}{p}} z + p w'$ with $w'\in S_\infty$. Write
${\rm Tr}_{S_\infty/ R_\infty} (z x_i^p)=\alpha_i^p + p\beta_i$
with $\alpha_i$ and $\beta_i\in R_\infty$. Then,
$\pi^{\frac{1}{p}} z=\bigl(\sum_i \alpha_i y_i\bigr)^p+p w''$ for
some $w''\in S_\infty$. Since $S_\infty$ is normal and the $p$--th
power of $\gamma_i=\bigl(\sum_i \alpha_i
y_i\bigr)/\pi^{\frac{1}{p^2}} $ lies in $S_\infty$, then
$\gamma_i$ lies in $S_\infty$ and $z=\gamma_i^p +
\frac{p}{\pi^{1/p}} w'' $. This proves that Frobenius is
surjective on $S_\infty/ \frac{p}{\pi^{1/p}} S_\infty$ and, hence,
it is surjective on $S_\infty$.

\end{proof}

\subsubsection{The lift of the Frobenius tower $\widetilde{R}_\infty$}
\label{sec:widetildeRinfty}

Put $\cO:=\WW(k)\bigl[\!\bigl[Z\bigr]\!\bigr]$. Let $M_\cO$ be the log structure on $\Spec(\cO)$ associated to the prelog structure $\psi_\cO\colon \N \to \cO$
given by $1 \mapsto Z$. Let $\psi_\cO\colon \WW(k)[\N] \to \cO$ be the associated map of $\WW(k)$--algebras. Note that $\cO_K\cong \cO/\bigl(P_\pi(Z)\bigr)$ so that
$\cO_K$ has a natural structure of $\cO$--algebra, compatible with log structures. Put
$$\widetilde{R}^{(0)}:=\cO[P]\widehat{\otimes}_{\cO[\N]} \cO$$where
the completion is taken with respect to the ideal $\bigl(P_\pi(Z)\bigr)$, the map $\cO[\N]\to \cO[P]$ is the morphism of $\cO$-algebras defined by $ \N\ni n \mapsto
\bigl((n,\ldots,n),(0,\ldots,0)\bigr)\in P$ and the map $\cO[\N] \to\cO$ is the morphism of $\cO$-algebras defined by $\N\ni n\mapsto Z^{\alpha n}$. All these maps
are compatible with log structures taking on $\cO[P]$ (resp.~$\cO[\N]$) the prelog structures given by $P$ (resp.~$\N$). We consider on $\widetilde{R}^{(0)}$ the
log structure induced by the fibred product log structure on $\cO[P]\otimes_{\cO[\N]} \cO$. It is associated to the prelog structure $P'\to \widetilde{R}^{(0)}$
where $P':=P\oplus_{\N} \N$ where $\N\to P$ is defined above and the map $\N \to \N$ is multiplication by $\alpha$. Note that
$\widetilde{R}^{(0)}/\bigl(P_\pi(Z)\bigr)\cong R^{(0)}=\cO_K[P]\tensor_{\cO_K[\N]} \cO_K$ compatibly with log structures so that we can view $R^{(0)}$ as
$\widetilde{R}^{(0)}$-algebra.

\begin{lemma}\label{lemma:tildeRn} There exists a unique chain of
$\widetilde{R}^{(0)}$-algebras $\widetilde{R}^{(0)}\subset \widetilde{R}^{(1)}\subset \ldots \subset \widetilde{R}^{(n)}$, complete and separated with respect to
the ideal $\bigl(P_\pi(Z)\bigr)$ in case (ALG) and with respect to the ideal $\bigl(P_\pi(Z),p\bigr)$ for $i\geq 1$ in case (FORM), lifting the chain of
$\widetilde{R}^{(0)}$-algebras $R^{(0)} \subset R^{(1)} \subset \ldots\subset R=R^{(n)}$ modulo $\bigl(P_\pi(Z)\bigr)$.
\end{lemma}
\begin{proof}
We construct $\widetilde{R}^{(i)}$ proceeding by induction on $i$. Assume that $\widetilde{R}^{(i)}$ has been constructed. If $R^{(i+1)}$ is obtained from $R^{(i)}$
by (the $p$-adic completion of) an \'etale extension $R^{(i')}$ of $R^{(i)}$ then we put $\widetilde{R}^{(i+1)}$ to be the $\bigl(P_\pi(Z)\bigr)$--adic completion
(resp.~$\bigl(p,P_\pi(Z)\bigr)$-adic completion) of the \'etale extension of $\widetilde{R}^{(i)}$ lifting $R^{(i)}\subset R^{(i')}$. If $R^{(i+1)}$ is obtained
from $R^{(i)}$ by (the $p$-adic completion of) the localization of $R^{(i)}$ with respect to a multiplicative set $U_i$, we let $\widetilde{U}_i$ be the set of
elements of $\widetilde{R}^{(i)}$ reducing to $U_i$ and we let $\widetilde{R}^{(i+1)}$ be the $\bigl(P_\pi(Z)\bigr)$-adic completion
(resp.~$\bigl(p,P_\pi(Z)\bigr)$-adic completion) of $\widetilde{R}^{(i)}\bigl[\widetilde{U}_i^{-1}\bigr]$. If $R^{(i+1)}$ is obtained from $R^{(i)}$ by completing
with respect to an ideal $I_i$ (containing $p$), we let $\widetilde{I}_i$ be the inverse image of $I_i$ in $\widetilde{R}^{(i)}$ and we let $\widetilde{R}^{(i+1)}$
be the $\widetilde{I}_i$-adic completion of $\widetilde{R}^{(i)}$. We leave to the reader to check the uniqueness.
\end{proof}

We put $\widetilde{R}:=\widetilde{R}^{(n)}$ and we let $\widetilde{X}_1,\ldots,\widetilde{X}_a$ and $\widetilde{Y}_1,\ldots,\widetilde{Y}_b\in \widetilde{R}$ be the
elements so that the induced prelog structure $\psi_{\widetilde{R}}\colon P' \to \widetilde{R}$ restricted to $P\subset P'$ is the morphism of monoids
$\bigl(\alpha_1,\ldots,\alpha_a,\beta_1,\ldots,\beta_b\bigr)\mapsto \prod_{i=1}^a \widetilde{X}_i^{\alpha_1}\prod_{j=1}^b \widetilde{Y}_j^{\beta_j}$. Note that we
have a commutative diagram of morphisms of $\cO$--algebras

$$\begin{array}{ccc}
\cO[P'] & \stackrel{\psi_{\widetilde{R}}}{\longrightarrow} &
\widetilde{R} \cr \uparrow & & \uparrow\cr \cO[\N]
&\stackrel{\psi_\cO}{\longrightarrow} & \cO.\cr
\end{array}$$Let   $\cO_K\left\{\frac{Z}{\pi}-1\right\}$
be the ring of $\pi$-adically convergent power series in
$\frac{Z}{\pi}-1$. Since the power series in $Z$ with coefficients
in $\cO_K$ can be expressed as power series in $Z-\pi=\pi
\bigl(\frac{Z}{\pi}-1\bigr)$, then
$\cO_K\left\{\frac{Z}{\pi}-1\right\}$ is a
$\cO_K\otimes_{\WW(k)}\cO$-algebra.

\begin{lemma}\label{lemma:RtildeandR} There exists an isomorphism of
$\cO_K\left[\!\left[\frac{Z}{\pi}-1\right]\!\right]$-algebras
$$
\widetilde{R}\widehat{\otimes}_{\cO} \Bigl(\cO_K\left[\!\left[\frac{Z}{\pi}-1\right]\!\right]\Bigr) \cong R\widehat{\otimes}_{\cO_K}
\Bigl(\cO_K\left[\!\left[\frac{Z}{\pi}-1\right]\!\right]\Bigr),$$where $\widehat{\otimes}$ stands for the $\bigl(\frac{Z}{\pi}-1\bigr)$-adic completion.
\end{lemma}
\begin{proof} We construct compatible isomorphisms of
$\cO_K\left[\!\left[\frac{Z}{\pi}-1\right]\!\right]$-algebras between $\widetilde{R}^{(n)}\widehat{\otimes}_{\cO}
\cO_K\left[\!\left[\frac{Z}{\pi}-1\right]\!\right]$ and $R^{(n)}\widehat{\otimes}_{\cO_K} \cO_K\left[\!\left[\frac{Z}{\pi}-1\right]\!\right]$ by induction on $n$.
The inductive step follows from the construction of $\widetilde{R}^{(n)}$ given in \ref{lemma:tildeRn}. We just prove the case $n=0$.

Recall that  $\widetilde{R}^{(0)}$ is the $\bigl(P_\pi(Z)\bigr)$-adic completion (resp.~$\bigl(P_\pi(Z),p\bigr)$-adic completion in case (FORM)) of
$$S_0:=\cO[P]\otimes_{\cO[\N]} \cO \cong \cO\bigl[\widetilde{X}_1,\ldots,\widetilde{X}_a,\widetilde{Y}_1,\ldots,\widetilde{Y}_b\bigr]/
\bigl(\widetilde{X}_1\cdots\widetilde{X}_a-Z^\alpha\bigr).$$Then, $S_0\otimes_{\cO} \cO_K\left[\!\left[\frac{Z}{\pi}-1\right]\!\right]\cong
\cO_K\left[\!\left[\frac{Z}{\pi}-1\right]\!\right] [\widetilde{X}_1,\ldots,\widetilde{X}_a,\widetilde{Y}_1,\ldots,\widetilde{Y}_b\bigr]/
\bigl(\widetilde{X}_1\cdots\widetilde{X}_a-Z^\alpha\bigr)$. Since $Z= \pi u$ with $u=\frac{Z}{\pi} $ a unit of $\cO_K\left[\!\left[\frac{Z}{\pi}-1\right]\!\right]$,
we have
$$S_0\cong \cO_K\left[\!\left[\frac{Z}{\pi}-1\right]\!\right]
[\widetilde{X}_1',\ldots,\widetilde{X}_a,\widetilde{Y}_1,\ldots,\widetilde{Y}_b\bigr]/ \bigl(\widetilde{X}_1'\cdots\widetilde{X}_a-\pi^\alpha\bigr)$$with
$\widetilde{X}_1'= u^{-\alpha} \widetilde{X}_1$.  There is a map of $\cO_K\left[\!\left[\frac{Z}{\pi}-1\right]\!\right] $-algebras to $R^{(0)}\otimes_{\cO_K}
\cO_K\left[\!\left[\frac{Z}{\pi}-1\right]\!\right]$ sending $\widetilde{X}_1'$ to $X_1$, $\widetilde{X}_i$ to $X_i$ for $i=2,\ldots,a$ and $\widetilde{Y}_j$ to
$Y_j$ for $j=1,\ldots,b$. It is an isomorphism. This proves the case $n=0$.
\end{proof}

For every $n\in\N$ write $\widetilde{R}_n$ for
$$\widetilde{R}_n:=\widetilde{R}\bigl[\widetilde{X}_1^{\frac{1}{n!}},\ldots,\widetilde{X}_a^{\frac{1}{
n!}},\widetilde{Y}_1^{\frac{1}{n!}},\ldots,\widetilde{Y}_b^{\frac{1}{n!}},Z^{\frac{1}{ n!}}\bigr]/\bigl( \widetilde{X}_1^{\frac{1}{n!}}\cdots
\widetilde{X}_a^{\frac{1}{ n!}}-Z^{\frac{\alpha}{n!}}\bigr) .$$Let
$$\widetilde{R}^o:=\widetilde{R}\bigl[\widetilde{X}_1^{\frac{1}{\alpha}},\ldots,\widetilde{X}_a^{\frac{1}{
\alpha}}\bigr]/\bigl( \widetilde{X}_1^{\frac{1}{\alpha}}\cdots \widetilde{X}_a^{\frac{1}{ \alpha}}-Z\bigr)$$and define morphism of monoids
$\psi_{\widetilde{R}^o}\colon \bigl(\frac{1}{\alpha} P_a\bigr) \times P_b\to \widetilde{R}^o$ sending $\frac{1}{\alpha}\bigl(u_1,\ldots,u_a) \times
(v_1,\ldots,v_b\bigr)$ to $\prod_{i=1}^a X_i^{\frac{u_i}{\alpha}}\prod_{j=1}^b Y_j^{{v_j}}$. As in \S\ref{lemma:Rinftyflat} and in \ref{lemma:Rinftyff} one proves
that

\begin{lemma}\label{lemmma:Rtildeinftyflat} The following hold:

\item[(1)] the rings $\widetilde{R}_n$ and $\widetilde{R}^o$ are noetherian and flat $\cO$--algebras;

\item[(2)]  $\widetilde{R}_{n}$ is $Z^\alpha$-flat as $\widetilde{R}$-module and the direct limit $\ds\lim_{n\to \infty} \widetilde{R}_n$ is a flat
$\widetilde{R}^o$-module with basis provided by the elements
$$\widetilde{X}_u \widetilde{Y}_v:=\prod_{i=1}^a \widetilde{X}_i^{\frac{u_i}{\alpha}} \prod_{j=1}^b
\widetilde{Y}_j^{v_j}$$with  $u=(u_1,\ldots,u_a)$ and $0 \leq u_i < 1$ rational number for every $1\leq i\leq a$  and with $v=(v_1,\ldots,v_b)$ and $0 \leq v_j < 1$
rational number for every $1\leq j\leq b$.

\item[(3)] $\widetilde{R}_n$ are Cohen-Macaulay rings and they are regular if $\alpha=1$. In particular, they are normal.

\item[(4)] $\widetilde{R}^o$ is a regular ring. Moreover, $\widetilde{R}$ is a direct summand in $\widetilde{R}^o$ as $\widetilde{R}$-module and $Z^\alpha
\widetilde{R}^o$ is contained in a finite and free $\widetilde{R}$-submodule of $\widetilde{R}^o$. Furthermore, $\widetilde{R}=\widetilde{R}^o$ if and only if $\alpha=1$.

\item[(5)] $\widetilde{R}$ is an integral domain.

\end{lemma}

\subsubsection{The map $\Theta$.}\label{sec:Thetalocalized}

For any normal subring $S\subset \overline{R}$ we put
$$\widetilde{\bf E}^+_S:=\lim_\leftarrow S/pS \subset \widetilde{\bf E}^+:=\lim_\leftarrow \overline{R}/p\overline{R}$$
where the projective limits are taken with respect to Frobenius
$x\mapsto x^p$. The fact that the natural map $\widetilde{\bf
E}^+_S\to \widetilde{\bf E}^+$ is injective follows from the
assumption that $S$ is normal.  We write the elements of
$\widetilde{\bf E}^+$ as sequences $(x_0,x_1,\ldots)$.

The rings $\widetilde{\bf E}^+_S$ and $\widetilde{\bf E}^+$ are
rings of characteristic $p$. In fact, they are $k$--algebras via
the map $k\ni x\mapsto (x,x^{1/p},x^{1/p^2},\ldots)$. For any
$(x_0,x_1,\ldots)\in \widetilde{\bf E}^+$ there is a unique
sequence of elements $\bigl(x^{(0)},x^{(1)},\ldots\bigr)$ of
elements in $\widehat{\overline{R}}$ such that
$\bigl(x^{(n+1)}\bigr)^p=x^{(n)}$  and $x^{(n)}\equiv x_n$
modulo~$p$ for every $n\in\N$; see \cite[\S II.1.2.2]{Fontaineperiodes}. In fact due
to \ref{cor:Frobonto} if $S_\infty\in \cS_\infty$ then we have an
isomorphism of multiplicative monoids
$$\widetilde{\bf E}^+_{S_\infty}\cong
\left\{\bigl(x^{(0)},x^{(1)},\ldots\bigr)\in \widehat{S_\infty}^\N\vert
\bigl(x^{(n+1)}\bigr)^p=x^{(n)},\,
\forall n\in\N\right\}.$$In particular, if $R_\infty\subset
S_\infty$ is Galois with group $H$ after inverting $p$ the trace
map $\sum_{\sigma\in H} \sigma$ induces maps
\begin{equation}\label{EtildeSinftySinftyhatN}
{\rm Tr}_{S_\infty/R_\infty}\colon \widetilde{\bf
E}^+_{S_\infty}\lra \widetilde{\bf E}^+_{R_\infty}, \quad  {\rm
Tr}_{S_\infty/R_\infty}\colon \WW\bigl(\widetilde{\bf
E}^+_{S_\infty}\bigr)\lra \WW\bigl(\widetilde{\bf
E}^+_{R_\infty}\bigr).
\end{equation}

\

Let $\Theta$ be the map from the Witt vectors of $\widetilde{\bf
E}^+$ to $\widehat{\overline{R}}$
$$\Theta\colon \WW\bigl(\widetilde{\bf E}^+\bigr) \longrightarrow
\widehat{\overline{R}}, \qquad (a_0,a_1,a_2,\ldots)\mapsto \sum_{n=0}^\infty p^n a_n^{(n)}$$The elements $\varepsilon:=(1,\epsilon_p,\epsilon_{p^2},\ldots)$,
$\overline{p}:=(p, p^{1/p},\ldots )$, $\overline{\pi}:=(\pi, \pi^{1/p},\ldots )$, $\overline{X}_i:=(X_i,X_i^{1/p},\ldots)$ for $i=1,\ldots,a$ and
$\overline{Y}_j:=(Y_j,Y_j^{1/p},\ldots)$ for $j=1,\ldots,b$ all define elements of $\widetilde{\bf E}^+_{R_\infty}$. The images of their Teichm\"uller lifts via
$\Theta$ is $$\Theta\bigl(\bigl[\varepsilon\bigr] \bigr)=1,\quad \Theta\bigl(\bigl[\overline{p}\bigr] \bigr)=p,\quad \Theta\bigl(\bigl[\overline{\pi}\bigr]
\bigr)=\pi,\quad \Theta\bigl(\bigl[\overline{X}_i\bigr] \bigr)=X_i,\quad \Theta\bigl(\bigl[\overline{Y}_j\bigr] \bigr)=Y_j.$$We endow $\WW\bigl(\widetilde{\bf
E}^+\bigr)$ with the log structure defined by the prelog structure $$\psi_{\WW\bigl(\widetilde{\bf E}^+\bigr)}\colon P'=P\oplus_{\N} \N \lra \WW\bigl(\widetilde{\bf
E}^+\bigr)$$sending $P\ni \bigl(\alpha_1,\ldots,\alpha_a,\beta_1,\ldots,\beta_b\bigr)$ to $ \prod_{i=1}^a \bigl[\overline{X}_i\bigr]^{\alpha_i}\prod_{j=1}^b
\bigl[\overline{Y}_j\bigr]^{\beta_j}$ and $n\in \N$ to $\bigl[\overline{\pi}\bigr]^n$. Recall that the $R$--algebra $\widehat{\overline{R}}$ is endowed with the log
structure induced from the given one on~$R$. Since $\Theta\bigl(\bigl[\overline{X}_i\bigr] \bigr)=X_i$ and $\Theta\bigl(\bigl[\overline{Y}_j\bigr] \bigr)=Y_j$, we
conclude that $\Theta$ respects the given log structures. Write
$$q^\prime:=1+[\varepsilon]^{\frac{1}{p}}+\cdots +
[\varepsilon]^{\frac{p-1}{p}}, \qquad
\xi:=\bigl[\overline{p}\bigr]-p.$$

\begin{lemma}\label{lemma:KerTheta} (1) The morphism $\Theta$
is a surjective homomorphism of $\WW(k)$--algebras. The kernel is
generated by
$$\Ker(\Theta)=\left(P_\pi\bigl(\bigl[\overline{\pi}\bigr]\bigr)\right)=\bigl(q^\prime\bigr)=\bigl(
\xi\bigr).$$

(2) The kernel of the morphism of $\cO_K$-algebras $1\otimes
\Theta\colon \cO_K\otimes_{\WW(k)} \WW\bigl(\widetilde{\bf
E}^+\bigr) \lra \widehat{\overline{R}}$ is generated by
$\xi_\pi:=1\otimes \bigl[\overline{\pi}\bigr]-\pi\otimes 1 $.

\smallskip
(3) The same holds for the induced map $\Theta_{S_\infty}\colon
\WW\bigl(\widetilde{\bf E}^+_{S_\infty}\bigr) \longrightarrow
\widehat{S_\infty}$ for every $S_\infty\in \cS_\infty$.
\end{lemma}
\begin{proof} (1) The proof that $\Theta$ is a homomorphism of
$\WW(k)$--algebras proceeds as in \cite[Prop 5.1.1-5.1.2]{brinon}.
Since $\WW\bigl(\widetilde{\bf E}^+\bigr)$ is $p$-adically
complete, to prove that $\Theta$ is surjective it suffices to show
that it is surjective modulo $p$ and this follows from
\ref{cor:Frobonto}. Note that $q^\prime$, $\xi$ and
$P_\pi\bigl(\bigl[\overline{\pi}\bigr]\bigr)$ lie in
$\Ker(\Theta)$. Moreover $q^\prime$ and $\xi$ generate the same
ideal in Fontaine's ring $A_{\rm inf}=\WW\bigl(\widetilde{\bf
E}^+_{\cO_\Kbar}\bigr)$ which is contained in
$\WW\bigl(\widetilde{\bf E}^+\bigr)$. Since
$\WW\bigl(\widetilde{\bf E}^+\bigr)$ is $p$-adically complete and
separated and $\widehat{\overline{R}}$ is $p$-torsion free, to
prove the claims regarding the kernel it suffices to show that for
any $x\in \Ker(\Theta)$ there exists $z$ such that $x=\xi z$
(resp.~$x=P_\pi\bigl(\bigl[\overline{\pi}\bigr]\bigr) z$) modulo
$p$. If $e$ is the degree of $P_\pi(Z)$ then
$P_\pi\bigl(\bigl[\overline{\pi}\bigr]\bigr)\equiv
\bigl[\overline{\pi}\bigr]^e$, modulo $p$,  which is
$\overline{p}$ up to a unit. Thus, it suffices to show the claim
regarding $\xi$ and this follows as in \cite[Prop
5.1.2]{brinon}.

(2) It follows from (1) that $\cO_K\otimes_{\WW(k)}
\WW\bigl(\widetilde{\bf E}^+\bigr)/(\xi)\cong
\cO_K\otimes_{\WW(k)} \widehat{\overline{R}}$. Thus the kernel of
$1\otimes \Theta$ is generated by $\xi_\pi$ and $\xi$ and we
are left to show that $\xi$ is a multiple of $\xi_\pi$. Note
that $p= u \pi^e$ for some unit $u\in \cO_K$ and
$\bigl[\overline{p}\bigr]= [\overline{u}]
\bigl[\overline{\pi}\bigr]^e$  for some unit $\overline{u}\in
\widetilde{\bf E}^+_{\cO_{K_\infty'}}$. Since
$\Theta([\overline{u}])=u$ we conclude that $[\overline{u}]-u \in
(\xi_\pi, \xi)$. Hence, $ \xi=
[\overline{u}]\bigl([\overline{\pi}]^e- \pi^e\bigr) +
\bigl([\overline{u}]-u\bigr)\pi^e$ so that $\xi \in
\bigl(\xi_\pi,\pi^e \xi\bigr)$ i.~e.,  $\xi (1+\pi^e a)= b
\xi_\pi$ for some $a$ and $b$ in  $\cO_K\otimes_{\WW(k)}
\WW\bigl(\widetilde{\bf E}^+\bigr)$. Since the latter  is
$\pi^e$-adically complete and separated, then $1+\pi^e a$ is a
unit and the claim follows.

The proof of (3) is analogous to the proofs of (1) and (2) and is
left to the reader.

\end{proof}

\subsubsection{The ring ${\bf A}_{\widetilde{R}_n}^+$.}
\label{sec:AwidetildeRn}

Recall that $\cO:=\WW(k)\bigl[\!\bigl[Z\bigr]\!\bigr]$. Consider on $\WW\bigl(\widetilde{\bf E}^+\bigr)$ the structure of $\cO$-algebra given by $Z\mapsto
\bigl[\overline{\pi}\bigr]$ where the latter is the Teichm\"uller lift of the element $(\pi,\pi^{\frac{1}{p}},\cdots)$. Using the prelog structure
$\psi_{\WW\bigl(\widetilde{\bf E}^+\bigr)}$ and the fact that $\bigl[\overline{X}_1\bigr]\cdots \bigl[\overline{X}_a\bigr]= \bigl[\overline{\pi}\bigr]^\alpha$, we
deduce that $\WW\bigl(\widetilde{\bf E}^+\bigr)$ is endowed with the structure of $\widetilde{R}^{(0)}$-algebra via the map of $\cO$-algebras
$$\widetilde{R}^{(0)}\lra \WW\bigl(\widetilde{\bf E}^+\bigr)$$sending
$ \widetilde{X}_i$ to $\bigl[\overline{X}_i\bigr]$ for $1\leq i\leq a$ and $\widetilde{Y}_j$ to $\bigl[\overline{Y}_j\bigr]$ for $1\leq j\leq b$. In particular the
log structure on $\WW\bigl(\widetilde{\bf E}^+\bigr)$ is the one induced by the log structure on $\widetilde{R}^{(0)}$.

\smallskip

{\it Convention:} In what follows given an element $a\in \widetilde{\bf E}^+$ and $n\in \N$ for typographical reasons we write $[a]^{\frac{1}{n}}$ to denote
$\left[a^{\frac{1}{n}} \right]$ where $[a]$ is the Teichm\"uller lift of $a$.

\begin{lemma}\label{lemma:A+R} (1) The elements
$\big(\bigl[\overline{\pi}\bigr],p\bigr)$ form a regular sequence
in $\WW\bigl(\widetilde{\bf E}^+\bigr)$. Moreover  $\xi$
is also a regular element.
\smallskip

(2) There exists a unique morphism $\widetilde{R}\lra \WW\bigl(\widetilde{\bf E}^+\bigr)$ of ${\widetilde{R}}^{(0)}$-algebras such that the reduction modulo
$P_\pi(Z)$ induces the inclusion $ \widetilde{R}/\bigl(P_\pi(Z)\bigr)\cong R \subset \widehat{\overline{R}}\cong \WW\bigl(\widetilde{\bf
E}^+\bigr)/\bigl(P_\pi([\overline{\pi}])\bigr)$ (using \ref{lemma:tildeRn} and \ref{lemma:KerTheta}).

For every $n\in\N$ there exists a unique morphism of $\widetilde{R}$-algebras ${\widetilde{R}}_n\lra \WW\bigl(\widetilde{\bf E}^+\bigr)$ (resp.~$\widetilde{R}^o\to
\WW\bigl(\widetilde{\bf E}^+\bigr)$) sending $\widetilde{X}_i^{\frac{1}{(n+1)!}}$ to $\big[\overline{X}_i\bigr]^{\frac{1}{(n+1)!}}$ for $i=1,\ldots,a$
 and $\widetilde{Y}_j^{\frac{1}{(n+1)!}}$ to $
\big[\overline{Y}_j\bigr]^{\frac{1}{(n+1)!}}$ for $j=1,\ldots,b$ and $Z^{\frac{1}{(n+1)!}}$ to $\big[\overline{\pi}\bigr]^{\frac{1}{(n+1)!}}$
(resp.~$\widetilde{X}_i^{\frac{1}{\alpha}}$ to $\big[\overline{X}_i\bigr]^{\frac{1}{\alpha}}$ for $i=1,\ldots,a$).\smallskip

(3) The $(Z,p)$-adic completion ${\widetilde{R}}_n'$ (resp.~${\widetilde{R}}^{o'}$) of the image of ${\widetilde{R}}_n$ (resp.~${\widetilde{R}}^o$) in
$\WW\bigl(\widetilde{\bf E}^+\bigr)$ is a direct factor of the $(Z,p)$-adic completion $\widehat{\widetilde{R}}_n$ (resp.~$\widehat{\widetilde{R}}^o$) of
${\widetilde{R}}_n$ (resp.~of $\widetilde{R}^o$). It coincides with the $(Z,p)$-adic completion $\widehat{\widetilde{R}}$ of ${\widetilde{R}}$ for $n=0$.

\smallskip

(4) The $(Z,p)$-adic completion of $\widetilde{R}_\infty:=\ds
\lim_n \widetilde{R}_n'$ maps isomorphically onto
$\WW\bigl(\widetilde{\bf E}^+_{R_\infty}\bigr)$.

\smallskip

(5) The subring $\widehat{\widetilde{R}}\subset \WW\bigl(\widetilde{\bf E}^+_{R_\infty}\bigr)$ is stable under Frobenius and the induced morphism $\varphi$ extends
uniquely to a morphism, denoted $\varphi$, on $\widehat{\widetilde{R}}_n$ and on $\widehat{\widetilde{R}}^o$ sending $\widetilde{X}_i^{\frac{1}{(n+1)!}} $ to $
\big[\overline{X}_i\bigr]^{\frac{p}{(n+1)!}}$ for $i=1,\ldots,a$ and $\widetilde{Y}_j^{\frac{1}{(n+1)!}}$ to $ \big[\overline{Y}_j\bigr]^{\frac{p}{(n+1)!}}$ for
$j=1,\ldots,b$ and $Z^{\frac{1}{(n+1)!}} $ to $ \big[\overline{\pi}\bigr]^{\frac{p}{(n+1)!}}$. It has the property that the maps $\widehat{\widetilde{R}}_n \to
\WW\bigl(\widetilde{\bf E}^+_{R_\infty}\bigr) $ and $\widehat{\widetilde{R}}^o\to \WW\bigl(\widetilde{\bf E}^+_{R_\infty}\bigr) $ commute with the morphism
$\varphi$.

\end{lemma}
\begin{proof} (1) Note that $\widetilde{\bf
E}^+_{R_\infty}$ is identified as a multiplicative monoid with a
submonoid of $\widehat{\overline{R}}^\N$. Since
$\widehat{\overline{R}}$ is reduced by \ref{prop:Rbarff} and
multiplication by $p$ on $\WW\bigl(\widetilde{\bf
E}^+_{R_\infty}\bigr)$ is the composite of Frobenius and
Vershiebung, we deduce that  $p$ is a regular element of
$\WW\bigl(\widetilde{\bf E}^+_{R_\infty}\bigr)$. Since
$\widehat{\overline{R}}$ is $\pi$-torsion free for every $n$, then
$\widetilde{\bf E}^+_{R_\infty}$ is $\overline{\pi}$-torsion free.
This proves the first part of the claim.

For the second part one proceeds as in \cite[Prop. 5.1.5]{brinon}.
Assume that $x:=(x_0,x_1,\ldots)\in \WW\bigl(\widetilde{\bf
E}^+\bigr)$ is such that $x\neq 0$ and $x \xi=0$. Let $n$ be the
minimal integer such that~$x_n\neq 0$. Dividing $x$ by $p^n$ we
may assume that $n=0$. In particular, $0\neq x_0 \in
\widetilde{\bf E}^+$ and $\overline{p} x_0=0$. Since
$\widetilde{\bf E}^+$ is the inverse limit $\lim
\widehat{\overline{R}}$ with the natural multiplication and since
$\widehat{\overline{R}}$ is $p$--torsion free, we deduce that
$x_0=0$ (absurd).
\smallskip

(2)-(3) We prove the claims for $\widetilde{R}_n$; the statements for $\widetilde{R}^o$ are proven in the same way. It follows from (1) that
$\WW\bigl(\widetilde{\bf E}^+_{R_\infty}\bigr)/\big(\bigl[\overline{\pi}\bigr],p\bigr)\cong \widetilde{\bf E}^+_{R_\infty}/\bigl(\overline{\pi}\bigr) \cong
R_\infty/\pi R_\infty$. By construction $\widetilde{R}_n/(Z,p) \cong R_n/pR_n$. The image $R_n'$ of $R_n \to R_\infty$ is a direct factor of $R_n$ as $R_n$ is
normal and noetherian. This defines a direct factor of $\widetilde{R}_n/(Z,p)$ and, hence, a direct factor $\widetilde{R}_n'$ of $\widehat{\widetilde{R}}_n$.

First of all we construct injective morphisms of $\widehat{\widetilde{R}}^{(0)}$-algebras $\widehat{\widetilde{R}}^{(i)}\lra \WW\bigl(\widetilde{\bf
E}^+_{R_\infty}\bigr)$ by induction on $i$. Assume that we have constructed a morphism $\widehat{\widetilde{R}}^{(i)} \lra \WW\bigl(\widetilde{\bf
E}^+_{R_\infty}\bigr)$ of ${\widetilde{R}}^{(0)}$-algebras inducing the natural inclusion $\widetilde{R}^{(i)}/(Z,p) \widetilde{R}^{(i)} \subset \WW\bigl(\widetilde{\bf
E}^+_{R_\infty}\bigr)/\big(\bigl[\overline{\pi}\bigr],p\bigr)$ and with the property required in (2). Then, $\widetilde{R}^{(i)} \subset \widetilde{R}^{(i+1)}$ is
obtained as (the $\bigl(p,P_\pi(Z)\bigr)$-completion of) an \'etale extension, a localization or the completion with respect to an ideal containing
$\bigl(p,P_\pi(Z)\bigr)$. In each case, one proves by induction on $m$ that  the map $\widetilde{R}^{(i+1)}/(Z,p) \widetilde{R}^{(i+1)} \subset
\WW\bigl(\widetilde{\bf E}^+_{R_\infty}\bigr)/\big(\bigl[\overline{\pi}\bigr],p\bigr)$ extends uniquely to a morphism of $\widetilde{R}^{(i)}$-algebras
$$\widetilde{R}^{(i+1)}/\big(\bigl[\overline{\pi}\bigr],p\bigr)^m\lra
\WW\bigl(\widetilde{\bf E}^+_{R_\infty}\bigr)/\big(\bigl[\overline{\pi}\bigr],p\bigr)^m.$$ Passing to the limit over $m\in\N$ we get the morphism
$\widehat{\widetilde{R}}^{(i+1)}\lra \WW\bigl(\widetilde{\bf E}^+_{R_\infty}\bigr)$. Reducing modulo $P_\pi(Z)$ and using uniqueness one proves that such map has
the property required in (2).

The existence and uniqueness of the morphism
${\widetilde{R}}_n\lra \WW\bigl(\widetilde{\bf E}^+\bigr)$ for
$n\in \N$ as required in (2) is clear. Note that $(p,Z)$ are
regular elements in $\widetilde{R}_n$ and in
$\WW\bigl(\widetilde{\bf E}^+_{R_\infty}\bigr)$ by (1). Moreover,
$\widetilde{R}_n/(p,Z) \to \WW\bigl(\widetilde{\bf
E}^+_{R_\infty}\bigr)/\big(\bigl[\overline{\pi}\bigr],p\bigr)\cong
R_\infty/\pi R_\infty$ factors via the direct factor
$\widetilde{R}_n'/(p,Z) $ which injects in $R_\infty/\pi
R_\infty$. Thus, the map ${\widetilde{R}}_n\lra
\WW\bigl(\widetilde{\bf E}^+\bigr)$ factors via
${\widetilde{R}}_n'\lra \WW\bigl(\widetilde{\bf E}^+\bigr)$ and
the latter is injective.

\smallskip

(4) Since $\WW\bigl(\widetilde{\bf
E}^+_{R_\infty}\bigr)/\big(\bigl[\overline{\pi}\bigr],p\bigr)$
coincides with $\cup_n \widetilde{R}_n'/(p,Z)$, the statement
follows.
\smallskip

(5) The proof proceeds as in (2). First of all one proves by induction on $i$ that the $(p,Z)$-adic completion of the image of ${\widetilde{R}}^{(i)}\to
\WW\bigl(\widetilde{\bf E}^+_{R_\infty}\bigr)$ is stable under $\varphi$. This is clear for $i=0$. For the inductive step one recalls that the $(p,Z)$-adic
completion of $\widetilde{R}^{(i)} \subset \widetilde{R}^{(i+1)}$ is obtained as the $\bigl(p,P_\pi(Z)\bigr)$-completion of an \'etale extension, a localization or
the completion with respect to an ideal containing $\bigl(p,P_\pi(Z)\bigr)$. In each case one checks that this is preserved by $\varphi$. One verifies  that the
extension of $\varphi$ to $\widehat{\widetilde{R}}_n$, given in (5), is well defined and that the morphism  $\widehat{\widetilde{R}}_n \to \WW\bigl(\widetilde{\bf
E}^+_{R_\infty}\bigr)$ commutes with $\varphi$ on the two sides. The details are left to the reader.
\end{proof}

\begin{definition}\label{def_AR+} We write ${\bf A}_{\widetilde{R}_n}^+$ (resp.~${\bf A}_{\widetilde{R}^o}^+$)
for the $\bigl(p,[\overline{\pi}]\bigr)$-adic completion of the image of $\widetilde{R}_n$ (resp.~$\widetilde{R}^o$) in $\WW\bigl(\widetilde{\bf E}^+\bigr)$.
\end{definition}

We write ${\cal I}$ for the ideal of $\WW\bigl(\widetilde{\bf E}^+\bigr)$ generated by
$[\varepsilon]^{\frac{1}{p^n}}-1$ for $n\in\N$ and by the Teichm\"uller lifts $[x]$ for
$x\in \widetilde{\bf E}^+$ such that $x^{(0)}\in
{m}_{\overline{R}}$. Following Fontaine (cf.~\cite[Def
9.2.1]{brinon}), we say that the extension ${\bf
A}_{\widetilde{R}_n}^+ \lra \WW\bigl(\widetilde{\bf E}^+\bigr)$ is
${\cal I}^m$-flat for $m\in\N$ if, given an injective map of ${\bf
A}_{\widetilde{R}_n}^+$-modules $M\lra N$ the induced map
$M\otimes_{{\bf A}_{\widetilde{R}_n}^+}\WW\bigl(\widetilde{\bf
E}^+\bigr)\lra N\otimes_{{\bf
A}_{\widetilde{R}_n}^+}\WW\bigl(\widetilde{\bf E}^+\bigr) $ has
kernel annihilated by ${\cal I}^m$.

\begin{proposition}\label{prop:A+Rtildeff} The extension
${\bf A}_{\widetilde{R}^o}^+ \lra \WW\bigl(\widetilde{\bf E}^+\bigr)$ is ${\cal I}^9$-flat. Moreover, ${\bf A}_{\widetilde{R}^o}^+$ is finite and
$[\overline{\pi}]^\alpha$-flat as ${\bf A}_{\widetilde{R}}^+$-module and ${\bf A}_{\widetilde{R}}^+$ is a direct summand in ${\bf A}_{\widetilde{R}^o}^+$ as ${\bf
A}_{\widetilde{R}}^+$-module.
\end{proposition}
\begin{proof}

Thanks to \ref{lemmma:Rtildeinftyflat} and \ref{lemma:A+R} the extension ${\bf A}_{\widetilde{R}^o}^+\to \widetilde{R}_\infty$ is flat. As $\WW\bigl(\widetilde{\bf
E}^+_{R_\infty}\bigr) $ is the $\big(p,[\overline{\pi}]\bigr)$-completion of $\widetilde{R}_\infty$ and $\big(p,[\overline{\pi}]\bigr)$ is a regular sequence in
$\widetilde{R}_\infty$ by loc.~cit., the extension ${\bf A}_{\widetilde{R}^o}^+/(p^n)\to \WW\bigl(\widetilde{\bf E}^+_{R_\infty}\bigr)/(p^n)$ is flat  by
\cite[Thm.~9.2.6]{brinon} for every $n\in\N$. Taking the limit over $n\in\N$ and arguing as in the proof of \cite[Prop.~9.2.5 \& Thm.~9.2.6]{brinon}, we conclude
that ${\bf A}_{\widetilde{R}^o}^+\to \WW\bigl(\widetilde{\bf E}^+_{R_\infty}\bigr)$ is flat.

For $R_\infty\subset S_\infty (\subset \Omega)$  normal and finite and \'etale after inverting $p$ the extension $\WW\bigl(\widetilde{\bf
E}^+_{R_\infty}\bigr)\subset \WW\bigl(\widetilde{\bf E}^+_{S_\infty}\bigr)$ is almost \'etale by \ref{prop:AE} and, hence, ${\cal I}$-flat. As
$\WW\bigl(\widetilde{\bf E}^+\bigr)$ is the $\big(p,[\overline{\pi}]\bigr)$-completion of the union of all the rings $\WW\bigl(\widetilde{\bf E}^+_{S_\infty}\bigr)$
and $\big(p,[\overline{\pi}]\bigr)$ is a regular sequence, arguing as above and using \cite[Prop.~9.2.5 \& Thm.~9.2.6]{brinon},  we conclude that ${\bf
A}_{\widetilde{R}^o}^+ \to \WW\bigl(\widetilde{\bf E}^+\bigr)$ is ${\cal I}^9$-flat.

The other statements follow from \ref{lemmma:Rtildeinftyflat} and \ref{lemma:A+R}.

\end{proof}

Extending  $\cO$--linearly (resp.~$R$-linearly,
resp.~$\widetilde{R}$-linearly) the morphism $\Theta$ we get a
homomorphisms of $\cO$--algebras (resp.~$R$-algebras,
resp.~$\widetilde{R}$--algebras)
$$\Theta_{\log}\colon \WW\bigl(\widetilde{\bf E}^+\bigr)\tensor_{\WW(k)} \cO
\longrightarrow \widehat{\overline{R}},\quad
\Theta_{R,\log}\colon\WW\bigl(\widetilde{\bf
E}^+\bigr)\tensor_{\WW(k)} R \longrightarrow
\widehat{\overline{R}},\quad \Theta_{\widetilde{R},\log}\colon
\WW\bigl(\widetilde{\bf E}^+\bigr)\tensor_{\WW(k)} \widetilde{R}
\longrightarrow \widehat{\overline{R}}.$$We consider on
$\WW\bigl(\widetilde{\bf E}^+\bigr)\tensor_{\WW(k)}  \cO$
(resp.~$\WW\bigl(\widetilde{\bf E}^+\bigr)\tensor_{\WW(k)}R$,
resp.~$\WW\bigl(\widetilde{\bf E}^+\bigr)\tensor_{\WW(k)}
\widetilde{R}$) the log structure defined as the product of the
log structures on $\WW\bigl(\widetilde{\bf E}^+\bigr)$ and
on~$\cO$ (resp.~on $R$, resp.~on $\widetilde{R}$). Then,
$\Theta_{\log}$, $\Theta_{R,\log}$ and
$\Theta_{\widetilde{R},\log}$ respect the log structures.

\subsection{The rings ${\rm B}_{\rm dR}$}\label{sec:BdR(R)}

Define ${\rm A}_{\rm inf}\bigl(R/\cO\bigr)$ (resp.~${\rm A}_{\rm inf}\bigl(R/R\bigr)$, resp.~${\rm A}_{\rm inf}\bigl(R/\widetilde{R}\bigr)$)  as the completion of
$\WW\bigl(\widetilde{\bf E}^+\bigr)\otimes_{\WW(k)} \cO$ with respect to the ideal $\Theta_{\log}^{-1}\bigl(p \widehat{\Rbar}\bigr)$ (resp.~of
$\WW\bigl(\widetilde{\bf E}^+\bigr)\otimes_{\WW(k)} R$ with respect to the ideal $\Theta_{R,\log}^{-1}\bigl(p \widehat{\Rbar}\bigr)$, resp.~of
$\WW\bigl(\widetilde{\bf E}^+\bigr)\otimes_{\WW(k)} \widetilde{R}$ with respect to the ideal $\Theta_{\widetilde{R},\log}^{-1}\bigl(p \widehat{\Rbar} \bigr)$) with
the induced log structures. Denote by
$$\Theta_{\log}\colon {\rm A}_{\rm inf}\bigl(R/\cO\bigr) \to \widehat{\Rbar}, \quad \Theta_{R,\log}\colon {\rm A}_{\rm inf}\bigl(R/R\bigr)\to  \widehat{\Rbar},
\quad \Theta_{\widetilde{R},\log}\colon {\rm A}_{\rm inf}\bigl(R/\widetilde{R}\bigr) \to \widehat{\Rbar}$$the maps induced by $\Theta_{\log}$
(resp.~$\Theta_{R,\log}$, resp.~$\Theta_{\widetilde{R},\log}$).

Define ${\rm B}_{{\rm dR},n}^{\nabla,+}\bigl(R\bigr)$ (resp.~${\rm B}_{{\rm dR},n}^{\nabla,+}\bigl(\widetilde{R}\bigr)$, resp. ${\rm B}_{{\rm
dR},n}^+\bigl(R\bigr)$, resp.~${\rm B}_{{\rm dR},n}^+\bigl(\widetilde{R}\bigr)$) to be the algebra underlying the $n$--th log infinitesimal neighborhood of the
closed immersion of log schemes defined by $\Theta\tensor \WW(k)\bigl[p^{-1}\bigr]$ (resp.~$\Theta_{\log} \tensor \WW(k)\bigl[p^{-1}\bigr]$, resp.
$\Theta_{R,\log}\tensor \WW(k)\bigl[p^{-1}\bigr]$, resp. $\Theta_{\widetilde{R},\log}\tensor \WW(k)\bigl[p^{-1}\bigr]$) in the sense of \cite[Rmk. 5.8]{katolog}.
Put $${\rm B}_{\rm dR}^{\nabla,+}\bigl(R\bigr):=\ds \lim_{\infty \leftarrow n} {\rm B}_{{\rm dR},n}^{\nabla,+}\bigl(R\bigr), \quad {\rm B}_{\rm
dR}^{\nabla,+}\bigl(\widetilde{R}\bigr):=\ds \lim_{\infty \leftarrow n} {\rm B}_{{\rm dR},n}^{\nabla,+}\bigl(\widetilde{R}\bigr)$$and similarly $${\rm B}_{\rm
dR}^+\bigl(R\bigr):=\ds \lim_{\infty \leftarrow n} {\rm B}_{{\rm dR},n}^+\bigl(R\bigr), \quad {\rm B}_{\rm dR}^+\bigl(\widetilde{R}\bigr):=\ds \lim_{\infty
\leftarrow n} {\rm B}_{{\rm dR},n}^+\bigl(\widetilde{R}\bigr).$$Note that $\Ker(\Theta)$ contains the element $[\varepsilon]-1$ with $\widetilde{\bf E}^+\ni
\varepsilon:=(1,\epsilon_p,\epsilon_{p^2},\ldots)$. In particular, ${\rm B}_{\rm dR}^{\nabla,+}$, and hence ${\rm B}_{\rm dR}^+\bigl(\widetilde{R}\bigr)$, contains
Fontaine's element $t:=\log [\varepsilon]$. Put $${\rm B}_{\rm dR}^\nabla\bigl(R\bigr):={\rm B}_{\rm dR}^{\nabla,+}\bigl(R\bigr)\bigl[t^{-1}\bigr], \qquad {\rm
B}_{\rm dR}^\nabla\bigl(\widetilde{R}\bigr):={\rm B}_{\rm dR}^{\nabla,+}\bigl(\widetilde{R}\bigr)\bigl[t^{-1}\bigr],$$
$${\rm B}_{\rm dR}\bigl(R\bigr):={\rm B}_{\rm
dR}^+\bigl(R\bigr)\bigl[t^{-1}\bigr],\qquad {\rm B}_{\rm
dR}\bigl(\widetilde{R}\bigr):={\rm B}_{\rm
dR}^+\bigl(\widetilde{R}\bigr)\bigl[t^{-1}\bigr].$$

{\it Filtrations:} We endow  ${\rm B}_{\rm
dR}^{\nabla,+}\bigl(R\bigr)$ (resp.~${\rm B}_{\rm
dR}^+\bigl(R\bigr)$, resp.~${\rm B}_{\rm
dR}^{\nabla,+}\bigl(\widetilde{R}\bigr)$,  resp.~${\rm B}_{\rm
dR}^+\bigl(\widetilde{R}\bigr)$) with the $\Ker(\Theta)$-adic
(resp. $\Ker(\Theta_{R,\log})$-adic,
resp.~$\Ker(\Theta_{\log})$-adic,
resp.~$\Ker(\Theta_{\widetilde{R},\log})$-adic) filtration.

\medskip

{\it Galois action:} Note that $\cG_R$ acts continuously on the  rings above, preserving the filtration.

\medskip

We  extend the filtrations as
follows. Let ${\rm B}_{\rm dR}^+$ be ${\rm B}_{\rm
dR}^{\nabla,+}\bigl(R\bigr)$ (resp.~${\rm B}_{\rm
dR}^+\bigl(R\bigr)$, resp.~${\rm B}_{\rm
dR}^{\nabla,+}\bigl(\widetilde{R}\bigr)$,  resp.~${\rm B}_{\rm
dR}^+\bigl(\widetilde{R}\bigr)$) with the given filtration ${\rm
Fil}^r {\rm B}_{\rm dR}^+$. Set ${\rm B}_{\rm dR}:={\rm B}_{\rm dR}^+[t^{-1}]$ and
$$\Fil^0 {\rm B}_{\rm dR}:= \sum_{n=0}^\infty t^{-n} {\rm Fil}^n
{\rm B}_{\rm dR}^+,\qquad \Fil^r {\rm B}_{\rm dR}:=t^r \Fil^0 {\rm
B}_{\rm dR}\, \forall r\in\Z.$$

\subsubsection{Explicit descriptions}

Following \cite[Pf Prop. 4.10(1)]{katolog} let $T:=\bigl\{(a,b)\in \Z\times \Z\vert a+b\in\N\bigr\}$ and let  $Q$ be the inverse image of~$P'$ in~$P^{',\rm
gp}\times P^{',\rm gp}$ via the sum~$P^{',\rm gp}\times P^{',\rm gp}\to P^{',\rm gp}$. Put $\bigl({\rm A}_{\rm inf}\bigl(R/\cO\bigr)\bigr)^{\rm log}:={\rm A}_{\rm
inf}\bigl(R/\cO\bigr)\tensor_{\Z[\N\times \N]} \Z\bigl[T\bigr]$. The map $\Theta_{\log}$ extends to a map
$$\Theta_{\log}'\colon \bigl({\rm A}_{\rm inf}\bigl(R/\cO\bigr)\bigr)^{\rm log} \to \widehat{\overline{R}}.$$Similarly, put
$\bigl({\rm A}_{\rm inf}\bigl(R/\widetilde{R}\bigr)\bigr)^{\rm log}:={\rm A}_{\rm inf}\bigl(R/\widetilde{R}\bigr)\tensor_{\Z[P'\times P']} \Z\bigl[Q\bigr]$ and
extend $\Theta_{\widetilde{R},\log}$ to
$$\Theta_{\widetilde{R},\log}'\colon  \bigl({\rm A}_{\rm inf}\bigl(R/\widetilde{R}\bigr)\bigr)^{\rm log} \to \widehat{\overline{R}}.$$Then,
${\rm B}_{\rm dR}^{\nabla,+}\bigl(\widetilde{R}\bigr)$ is the $\Ker\bigl(\Theta_{\log}'\bigr)$--adic completion of $\bigl({\rm A}_{\rm
inf}\bigl(R/\cO\bigr)\bigr)^{\rm log}\bigl[p^{-1}\bigr]$ and ${\rm B}_{\rm dR}^+\bigl(\widetilde{R}\bigr)$ is the
$\Ker\bigl(\Theta_{\widetilde{R},\log}'\bigr)$--adic completion of $\bigl({\rm A}_{\rm inf}\bigl(R/\widetilde{R}\bigr)\bigr)^{\rm log}\bigl[p^{-1}\bigr]$. One
proceeds similarly for ${\rm B}_{\rm dR}^+\bigl(R\bigr)$.

We make these definitions more explicit.  Consider the elements
$$u:=\frac{\bigl[\overline{\pi}\bigr]}{Z},
\qquad v_i:=\frac{\bigl[\overline{X}_i\bigr]}{\widetilde{X}_i},\qquad w_j:=\frac{\bigl[\overline{Y}_j\bigr]}{\widetilde{Y}_j}$$for $i=1,\ldots,a$ and
$j=1,\ldots,b$. Then, $\bigl({\rm A}_{\rm inf}\bigl(R/\cO\bigr)\bigr)^{\rm log}$ is generated by $u$ and $u^{-1}$ as ${\rm A}_{\rm inf}\bigl(R/\cO\bigr)$--algebra,
i.e., $$ \bigl({\rm A}_{\rm inf}\bigl(R/\cO\bigr)\bigr)^{\rm log} \cong {\rm A}_{\rm inf}\bigl(R/\cO\bigr) [u, u^{-1}]$$and $\Ker\bigl(\Theta_{\log}'\bigr)=(u-1)$.
Similarly, $\bigl({\rm A}_{\rm inf}\bigl(R/\widetilde{R}\bigr)\bigr)^{\rm log}$ is generated as ${\rm A}_{\rm inf}\bigl(R/\widetilde{R}\bigr)$--algebra by $u$, the
elements $v_i$ for $i=1,\ldots,a$, and $w_j$ for $j=1,\ldots,b$ and by their multiplicative inverses
$$\bigl({\rm A}_{\rm inf}\bigl(R/\widetilde{R}\bigr)\bigr)^{\rm log} \cong {\rm A}_{\rm inf}\bigl(R/\widetilde{R}\bigr)
\bigl[u^{\pm 1}, v_1^{\pm 1},\ldots, v_a^{\pm 1},w_1^{\pm 1},\ldots, w_b^{\pm 1}\bigr].$$\medskip

For later purposes we generalize these constructions. Set $$\bigl(\WW\bigl(\widetilde{\bf E}^+\bigr)\tensor_{\WW(k)} \cO\bigr)^{\rm log}:=\WW\bigl(\widetilde{\bf
E}^+\bigr)\tensor_{\WW(k)} \cO\tensor_{\Z[\N\times \N]} \Z\bigl[T\bigr].$$The map $\Theta_{\log}$ extends to a map
$$\Theta_{\log}'\colon \bigl(\WW\bigl(\widetilde{\bf E}^+\bigr)\tensor_{\WW(k)}
\cO\bigr)^{\rm log} \to \widehat{\overline{R}}.$$As above $ \WW\bigl(\widetilde{\bf E}^+\bigr)\tensor_{\WW(k)} \cO\tensor_{\Z[\N\times \N]} \Z\bigl[T\bigr]\cong
\WW\bigl(\widetilde{\bf E}^+\bigr)\tensor_{\WW(k)} \cO [u, u^{-1}]$ and $\Ker\bigl(\Theta_{\log}'\bigr)=(u-1)$. Similarly, set $\bigl(\WW\bigl(\widetilde{\bf
E}^+\bigr)\tensor_{\WW(k)} \widetilde{R}\bigr)^{\rm log}:=\WW\bigl(\widetilde{\bf E}^+\bigr)\tensor_{\WW(k)} \widetilde{R}\tensor_{\Z[P'\times P']} \Z\bigl[Q\bigr]$
and extend $\Theta_{\widetilde{R},\log}$ to
$$\Theta_{\widetilde{R},\log}'\colon  \bigl(\WW\bigl(\widetilde{\bf E}^+\bigr)\tensor_{\WW(k)}
\widetilde{R}\bigr)^{\rm log} \to \widehat{\overline{R}}.$$Then,  $\WW\bigl(\widetilde{\bf E}^+\bigr)\tensor_{\WW(k)} \widetilde{R}\tensor_{\Z[P'\times P']}
\Z\bigl[Q\bigr] \cong \WW\bigl(\widetilde{\bf E}^+\bigr)\tensor_{\WW(k)} \widetilde{R}\bigl[u^{\pm 1}, v_1^{\pm 1},\ldots, v_a^{\pm 1},w_1^{\pm 1},\ldots, w_b^{\pm
1}\bigr]$.

\begin{lemma}\label{lemma:structurBdR+} (1) The sequence $(\xi,u-1)$
(resp.~$\left(P_\pi\bigl([\overline{\pi}]\bigr),u-1\right)$) is regular and it generates the kernel of $\Ker\bigl(\Theta_{\log}'\bigr)$ in
$\bigl(\WW\bigl(\widetilde{\bf E}^+\bigr)\tensor_{\WW(k)} \cO\bigr)^{\rm log}$.\smallskip

(2) The sequence $\bigl(\xi,u-1,v_2-1,\ldots,v_a-1,w_1-1,\ldots,w_b-1\bigr)$ is regular and it generates the kernel of
$\Ker\bigl(\Theta_{\widetilde{R},\log}'\bigr)$ in $\bigl(\WW\bigl(\widetilde{\bf E}^+\bigr)\tensor_{\WW(k)} \widetilde{R}\bigr)^{\rm log}$.
\end{lemma}
\begin{proof}
It follows from \ref{lemma:KerTheta} that
$P_\pi\bigl([\overline{\pi}]\bigr)$ and $\xi$ generate the same
ideal.

(1) Due to \ref{lemma:A+R} the element $\xi$ is not a zero
divisor in $\WW\bigl(\widetilde{\bf E}^+\bigr)\tensor_{\WW(k)}
\cO$. Since $\WW\bigl(\widetilde{\bf E}^+\bigr)\tensor_{\WW(k)}
\cO/(\xi)\cong \widehat{\overline{R}} \tensor_{\WW(k)} \cO$ and
$1\otimes Z$ is not a zero divisor in it, we deduce that $\xi$
is not a zero divisor in $\WW\bigl(\widetilde{\bf
E}^+\bigr)\tensor_{\WW(k)}
\cO\bigl[u^{-1}\bigr]=\WW\bigl(\widetilde{\bf
E}^+\bigr)\tensor_{\WW(k)} \cO [T]/\bigl([\overline{\pi}] T-
Z\bigr)$.

Note that $ \WW\bigl(\widetilde{\bf E}^+\bigr)\tensor_{\WW(k)} \cO[u,u^{-1}]/(\xi) \cong \widehat{\overline{R}} \tensor_{\WW(k)} \cO[u,u^{-1}]$ with $(1\otimes Z)
u=\pi\otimes 1$. Modulo $(u-1)$ this coincides with $\widehat{\overline{R}}$. Moreover, such ring injects in $\widehat{\overline{R}}[\![ Z ]\!][u,u^{-1}]$ which
injects in $\widehat{\overline{R}}[p^{-1}](\!( Z )\!)$. It then suffices to show that $(u-1)$ is not a zero divisor in the latter or equivalently that $Z
(u-1)=\pi-Z$ is not a zero divisor. This is clear since $\pi$ is a unit in $\widehat{\overline{R}}[p^{-1}]$.

\smallskip

(2) Since~$\widetilde{R}^{(i+1)}$ is obtained from~$\widetilde{R}^{(i)}$ by completing with respect to some ideal, localizing or taking \'etale extensions, it is a
flat $\widetilde{R}^{(i)}$-module. Thus, the regularity of the sequence, given in (2), in the ring $\WW\bigl(\widetilde{\bf E}^+\bigr)\tensor_{\WW(k)}
\widetilde{R}^{(i)}\bigl[u^{\pm 1}, v_1^{\pm 1},\ldots, v_a^{\pm 1},w_1^{\pm 1},\ldots, w_b^{\pm 1}\bigr]$ holds if it holds for $i=0$. Since  $\widetilde{R}^{(0)}$
is flat as an algebra over $\widetilde{R}':=\cO\bigl[\widetilde{X}_1,\ldots,\widetilde{X}_a, \widetilde{Y}_1,\ldots,\widetilde{Y}_b\bigr]
/(\widetilde{X}_1\cdots\widetilde{X}_a-Z^\alpha)$ it suffices to prove the regularity for $\widetilde{R}'$ instead of $\widetilde{R}$. Note that
$\WW\bigl(\widetilde{\bf E}^+\bigr)\tensor_{\WW(k)} \widetilde{R}'\bigl[u^{\pm 1}, v_1^{\pm 1},\ldots, v_a^{\pm 1},w_1^{\pm 1},\ldots, w_b^{\pm 1}\bigr]$ is
isomorphic to $\WW\bigl(\widetilde{\bf E}^+\bigr)\tensor_{\WW(k)}\cO\bigl[u^{\pm 1}, v_1^{\pm 1},\ldots,v_a^{\pm 1}, w_1^{\pm 1},\ldots,w_b^{\pm 1}\bigr]
/\bigl(v_1\cdots v_a-u^\alpha\bigr)$ which is
$$\WW\bigl(\widetilde{\bf E}^+\bigr)\tensor_{\WW(k)}\cO\bigl[u^{\pm
1}, v_2^{\pm 1},\ldots,v_a^{\pm 1}, w_1^{\pm 1},\ldots,w_b^{\pm 1}\bigr]$$since $v_1= u^\alpha v_2^{-1}\cdots v_a^{-1}$. Thus, if $\xi$ and $u-1$ is a regular
sequence of $\WW\bigl(\widetilde{\bf E}^+\bigr)\tensor_{\WW(k)}\cO[u^{\pm 1}]$ generating $\Ker\bigl(\Theta_{\log}'\bigr)$ then also (2) holds for $\widetilde{R}'$
in place of $\widetilde{R}$. In particular, the regularity claimed in (2) follows and  we are left to prove that the sequence given in (2) generates the ideal
$\Ker\bigl(\Theta_{\widetilde{R},\log}'\bigr)$.

\smallskip
Due to (1) the ring $\WW\bigl(\widetilde{\bf E}^+\bigr)\tensor_{\WW(k)} \widetilde{R}\bigl[u^{\pm 1}, v_1^{\pm 1},\ldots, v_a^{\pm 1},w_1^{\pm 1},\ldots, w_b^{\pm
1}\bigr]$ modulo $(\xi,u-1)$ coincides with $\overline{R} \widehat{\tensor}_{\cO_K} R \bigl[v_2^{\pm 1},\ldots, v_a^{\pm 1},w_1^{\pm 1},\ldots, w_b^{\pm 1}\bigr]$.
Consider the quotient $B$ modulo the ideal $J:=\bigl(v_2-1,\ldots,v_a-1,w_1-1,\ldots,w_b-1\bigr)$. To show that $B\cong \widehat{\overline{R}}$, by base changing
via $R \lra \widehat{\overline{R}}$, it is sufficient to prove that $R\tensor_{\cO_K} R \bigl[v_2^{\pm 1},\ldots, v_a^{\pm 1},w_1^{\pm 1},\ldots, w_b^{\pm
1}\bigr]/J $ coincides with $R$; here and below we still denote by $J$ the ideal generated by $\bigl(v_2-1,\ldots,v_a-1,w_1-1,\ldots,w_b-1\bigr)$. This follows
showing by induction on $i$ that $R^{(i)}\tensor_{\cO_K} R^{(i)} \bigl[v_2^{\pm 1},\ldots, v_a^{\pm 1},w_1^{\pm 1},\ldots, w_b^{\pm 1}\bigr]/J \cong R^{(i)}$. The
inductive step is left to the reader using the fact that $R^{(i+1)}$ (resp.~$\widetilde{R}^{(i+1)}$) is obtained from $R^{(i)}$ (resp.~$\widetilde{R}^{(i)}$)  by
completing with respect to some ideal, localizing or taking \'etale extensions. The essential case is $i=0$ and in this case we may replace $R^{(0)}$ with $R'=
\cO_K\bigl[X_1,\ldots,X_a, Y_1,\ldots,Y_b\bigr] /(X_1\cdots X_a-\pi^\alpha)$. Then, $R'\tensor_{\cO_K} R' \bigl[v_2^{\pm 1},\ldots, v_a^{\pm 1},w_1^{\pm 1},\ldots,
w_b^{\pm 1}\bigr]/J \cong R'\bigl[v_2^{\pm 1},\ldots, v_a^{\pm 1},w_1^{\pm 1},\ldots, w_b^{\pm 1}\bigr]/J\cong R'$ and the claim follows.

\end{proof}

\begin{proposition}\label{cor:BdRstr} The following properties hold
\begin{enumerate}

\item[(1)] ${\rm B}_{\rm
dR}^{\nabla,+}\bigl(\widetilde{R}\bigr)\cong {\rm B}_{\rm
dR}^{\nabla,+}\bigl(R\bigr)[\![ u-1]\!]$;

\item[(2)] ${\rm B}_{\rm dR}^+\bigl(\widetilde{R}\bigr)\cong {\rm
B}_{\rm dR}^{\nabla,+}\bigl(R\bigr)\bigl[\!\bigl[
v_1-1,\ldots,v_a-1,w_1-1,\ldots,w_b-1\bigr]\!\bigr]$;

\item[(3)] ${\rm B}_{\rm dR}^+\bigl(R\bigr)\cong {\rm
B}_{\rm dR}^{\nabla,+}\bigl(R\bigr)\bigl[\!\bigl[
v_2-1,\ldots,v_a-1,w_1-1,\ldots,w_b-1\bigr]\!\bigr]$;

\item[(4)] ${\rm B}_{\rm dR}^+\bigl(\widetilde{R}\bigr) \cong {\rm
B}_{\rm dR}^+\bigl(R\bigr) [\![u-1]\!]\cong {\rm B}_{\rm
dR}^+\bigl(R\bigr) [\![Z-\pi]\!]$;

\item[(5)] the filtration on ${\rm B}_{\rm
dR}^{\nabla,+}\bigl(R\bigr)$ is the $t$-adic filtration. In
particular, ${\rm Gr}^\bullet {\rm B}_{\rm
dR}^{\nabla,+}\bigl(R\bigr)=\widehat{\overline{R}}[p^{-1}][t]$
with grading given by the degree in $t$;

\item[(6)] the filtration on ${\rm B}_{\rm
dR}^{+}\bigl(\widetilde{R}\bigr)$ is the
$\bigl(t,v_1-1,\ldots,v_a-1,w_1-1,\ldots,w_b-1\bigr)$-adic
filtration. In particular, ${\rm Gr}^\bullet {\rm B}_{\rm
dR}^{+}\bigl(\widetilde{R}\bigr)=\widehat{\overline{R}}[p^{-1}]\left[t,
v_1-1,\ldots,v_a-1,w_1-1,\ldots,w_b-1\right]$ with grading given by
the degree as polynomials in $t,
v_1-1,\ldots,v_a-1,w_1-1,\ldots,w_b-1$. Therefore, $${\rm
Gr}^\bullet {\rm B}_{\rm
dR}\bigl(\widetilde{R}\bigr)=\widehat{\overline{R}}[p^{-1}]\left[t,t^{-1},
\frac{v_1-1}{t},\ldots,\frac{v_a-1}{t},\frac{w_1-1}{t},\ldots,\frac{w_b-1}{t}\right]$$
with grading given by the degree in $t$. Similarly,  $${\rm
Gr}^\bullet {\rm B}_{\rm
dR}\bigl(R\bigr)=\widehat{\overline{R}}[p^{-1}]\left[t,t^{-1},
\frac{v_2-1}{t},\ldots,\frac{v_a-1}{t},\frac{w_1-1}{t},\ldots,\frac{w_b-1}{t}\right]$$
with grading given by the degree in $t$;

\item[(7)] the filtration on ${\rm B}_{\rm
dR}\bigl(\widetilde{R}\bigr)$ is exhaustive and separated and
${\rm Fil}^r {\rm B}_{\rm dR}\bigl(\widetilde{R}\bigr)\cap {\rm
B}_{\rm dR}^+\bigl(\widetilde{R}\bigr)={\rm Fil}^r {\rm B}_{\rm
dR}^+\bigl(\widetilde{R}\bigr)$ for every $r\in\N$.

\end{enumerate}

\end{proposition}
\begin{proof}
The proofs of (1), (2) and (3) are similar and follow closely the proof of \cite[Prop. 5.2.2]{brinon}. We only sketch the proof of (1) and (2) and we refer to
loc.~cit.~for the details. We certainly have morphisms $\iota\colon {\rm B}_{\rm dR}^{\nabla,+}\bigl(R\bigr)[\![ u-1]\!] \to {\rm B}_{\rm
dR}^{\nabla,+}\bigl(\widetilde{R}\bigr)$ and $f\colon {\rm B}_{\rm dR}^{\nabla,+}\bigl(R\bigr)\bigl[\!\bigl[ v_1-1,\ldots,v_a-1,w_1-1,\ldots,w_b-1\bigr]\!\bigr]
\lra {\rm B}_{\rm dR}^+\bigl(\widetilde{R}\bigr)$. Notice that ${\rm B}_{\rm dR}^{\nabla,+}\bigl(R\bigr)[\![ u-1]\!]$ has the structure of an $\cO$-algebra as we
can send $Z$ to  $[\overline{\pi}] u^{-1}$. Similarly, ${\rm B}_{\rm dR}^{\nabla,+}\bigl(R\bigr)\bigl[\!\bigl[ v_1-1,\ldots,v_a-1,w_1-1,\ldots,w_b-1\bigr]\!\bigr]$
has the structure of $\widetilde{R}^{(0)}$-algebra. Indeed, it is a $\cO$-algebra since $W=[\overline{\pi}]^\alpha \cdot v_1^{-1} \cdots v_a^{-1}$ lies in this
ring. Since the equation $X^\alpha=W$ has the solution $1$ modulo $(t,u-1,v_1-1,\ldots,v_a-1,w_1-1,\ldots,w_b-1)$, by Hensel's lemma it admits a solution $Z'$.
Then, the structure of $\cO$-algebra is defined by sending $Z$ to $Z'$. The structure of $\widetilde{R}^{(0)}$-algebra is given by the structure of $\cO$-algebra
and by sending $\widetilde{X}_i$ to $ \bigl[\overline{X}_i\bigr] v_i^{-1}$ for $i=1,\ldots,a$ and $\widetilde{Y}_j$ to $\bigl[\overline{Y}_j\bigr] w_j^{-1} $ for
$j=1,\ldots,b$. Arguing as in \ref{lemma:tildeRn} one proves that there is a unique extension to an $\widetilde{R}$-algebra structure compatible via the morphism
$\Theta$ with the $\widehat{R}$-structure on $\widehat{\overline{R}}[p^{-1}]$. This provides morphisms ${\rm B}_{\rm dR}^{\nabla,+}\bigl(\widetilde{R}\bigr)\to {\rm
B}_{\rm dR}^{\nabla,+}\bigl(R\bigr)[\![ u-1]\!]$ and ${\rm B}_{\rm dR}^+\bigl(\widetilde{R}\bigr)\lra {\rm B}_{\rm dR}^{\nabla,+}\bigl(R\bigr)\bigl[\!\bigl[
v_1-1,\ldots,v_a-1,w_1-1,\ldots,w_b-1\bigr]\!\bigr]$ which are proven to be inverse to $\iota$ and $f$ respectively, see loc.~cit.

For (4) we notice that $u^\alpha=\prod_{i=1}^a v_i$. Since the
$v_i$'s are  unit in ${\rm B}_{\rm
dR}^+\bigl(\widetilde{R}\bigr)$, the first formula follows from
(2) and (3). The second formula follows remarking that
$u=\frac{Z}{[\overline{\pi}]}= \frac{\pi}{[\overline{\pi}]} \cdot
\frac{Z}{\pi}$ so that
$$u-1=\frac{\pi}{[\overline{\pi}]} \left(\frac{Z}{\pi}-1\right)+
\left(\frac{\pi}{[\overline{\pi}]}-1\right).$$The claim follows
since $\frac{\pi}{[\overline{\pi}]}-1$ lies in ${\rm Fil}^1 {\rm
B}_{\rm dR}^+\bigl(\cO_K\bigr)$.

For (5) one  needs to prove that $t$ is not a zero divisor in
${\rm B}_{\rm dR}^{\nabla,+}\bigl(R\bigr)$. Note that $t {\rm
B}_{\rm dR}^{\nabla,+}\bigl(R\bigr)=\xi {\rm B}_{\rm
dR}^{\nabla,+}\bigl(R\bigr)$, as this holds already for $R=\Z_p$
due to \cite[\S II.1.5.4]{Fontaineperiodes}. One is then left to
prove that $\xi$ is not a zero divisor in ${\rm B}_{\rm
dR}^{\nabla,+}\bigl(R\bigr)$. Arguing as in \cite[Prop.
5.1.4]{brinon} one reduces to prove that $\widehat{\overline{R}}$
has no non-trivial $p$-torsion. This has been proven in
\ref{prop:Rbarff}.

(6) follows from (5), (2) and (3).

(7) follows arguing as in \cite[Prop. 5.2.8 \& Cor. 5.2.9]{brinon}.
\end{proof}

\subsubsection{Connections.}

Put $\widehat{\omega}^1_{R/\cO_K}:=\lim_n \omega^1_{R/\cO_K}/p^n
\omega^1_{R/\cO_K}$ where $\omega^1$ denotes the module of logarithmic
K\"ahler differentials. Then, $\widehat{\omega}^1_{R/\cO_K}\cong
\oplus_{i=2}^a R {\rm dlog} X_i \oplus_{j=1}^b R {\rm dlog} Y_j$.
Similarly, let
$$\widehat{\omega}^1_{\widetilde{R}/\WW(k)}:=\lim_{\infty\leftarrow
n} \omega^1_{\widetilde{R}/\WW(k)}/\bigl(p,P_\pi(Z)\bigr)^n \omega^1_{\widetilde{R}/\WW(k)},\qquad \widehat{\omega}^1_{\widetilde{R}/\cO}:=\lim_{\infty\leftarrow n}
\omega^1_{\widetilde{R}/\cO}/\bigl(p,P_\pi(Z)\bigr)^n \omega^1_{\widetilde{R}/\cO}.$$We have $\widehat{\omega}^1_{\widetilde{R}/\WW(k)}\cong \widehat{\widetilde{R}}
{\rm dlog} Z \oplus_{i=2}^a \widehat{\widetilde{R}} {\rm dlog} \widetilde{X}_i \oplus_{j=1}^b \widehat{\widetilde{R}} {\rm dlog} \widetilde{Y}_j$, where
$\widehat{\widetilde{R}}$ is the $\bigl(p,P_\pi(Z)\bigr)$-adic completion  of $\widetilde{R}$, as ${\rm dlog}  \widetilde{X}_1=\alpha {\rm dlog} Z + \sum_{i=2}^a
{\rm dlog} \widetilde{X}_i$. We also have $\widehat{\omega}^1_{\cO/\WW(k)}\cong \cO {\rm dlog} Z$. We have an exact sequence
$$0\lra \widetilde{R}\widehat{\otimes}_{\cO}
\widehat{\omega}^1_{\cO/\WW(k)} \lra \widehat{\omega}^1_{\widetilde{R}/\WW(k)}\lra \widehat{\omega}^1_{\widetilde{R}/\cO} \lra 0.$$

\smallskip

Using \ref{cor:BdRstr} define the {\it connections}
$$\nabla_R\colon {\rm B}_{\rm
dR}^+\bigl(R\bigr) \lra {\rm B}_{\rm
dR}^+\bigl(R\bigr)\otimes_{\widehat{R}}
\widehat{\omega}^1_{R/\cO_K},$$ $$\nabla_{\widetilde{R}}\colon
{\rm B}_{\rm dR}^+\bigl(\widetilde{R}\bigr) \lra {\rm B}_{\rm
dR}^+\bigl(\widetilde{R}\bigr)\otimes_{\widehat{\widetilde{R}}}
\widehat{\omega}^1_{\widetilde{R}/\WW(k)}$$and
$$\nabla_{\widetilde{R}/\cO}\colon {\rm B}_{\rm dR}^+\bigl(\widetilde{R}\bigr)
\lra {\rm B}_{\rm dR}^+\bigl(\widetilde{R}\bigr)\otimes_{\widehat{R}} \widehat{\omega}^1_{\widetilde{R}/\cO}$$to be the ${\rm B}_{\rm
dR}^{\nabla,+}\bigl(R\bigr)$-linear (resp.~${\rm B}_{\rm dR}^{\nabla,+}\bigl(\widetilde{R}\bigr)$-linear) map given by sending $(v_i-1)$ to $-v_i {\rm dlog}
\widetilde{X}_i$ for $i=1,\ldots,a$ and $(w_j-1)$ to $ -w_j {\rm dlog} \widetilde{Y}_j$ for $j=1,\ldots,b$. These connections extend to the rings  ${\rm B}_{\rm
dR}\bigl(R\bigr)$ and ${\rm B}_{\rm dR}\bigl(\widetilde{R}\bigr)$.

\begin{lemma}\label{lemma:nabla=0} We have:
\begin{enumerate}

\item[(1)] The above connections commute with the action of
$\cG_R$, are integrable and satisfy Griffiths' transversality with
respect to the filtrations;

\item[(2)] ${\rm B}_{\rm dR}^{\nabla}\bigl(R\bigr)={\rm B}_{\rm
dR}\bigl(R\bigr)^{\nabla_R=0}\cong {\rm B}_{\rm
dR}\bigl(\widetilde{R}\bigr)^{\nabla_{\widetilde{R}/\WW(k)}=0}$;

\item[(3)] ${\rm B}_{\rm
dR}^{\nabla}\bigl(\widetilde{R}\bigr)={\rm B}_{\rm
dR}\bigl(\widetilde{R}\bigr)^{\nabla_{\widetilde{R}/\cO}=0}$.

\end{enumerate}

\noindent The same statements apply for the rings with $+$.
\end{lemma}
\begin{proof}
By definition the connections are integrable and satisfy
Griffiths' transversality. To prove that they commute with Galois
it suffices to prove that the induced derivation $N_i$ equal to
$\widetilde{X}_i \frac{\partial}{\partial \widetilde{X}_i}$ for
$i=1,\ldots,a$ and $N_i$ equal to $\widetilde{Y}_{i-a}
\frac{\partial}{\partial \widetilde{Y}_{i-a}}$ for
$i=a+1,\ldots,a+b$ commute with $\cG_R$. Let $X_i=v_i$ if $1\leq
i\leq a$ and $w_j$ if $i=j+a$ for some $1\leq j\leq b$. Since
$N_i$ acts trivially on $v_j-1$ for $j\neq i$ and on $w_j-1$ for
$j+a\neq i$ it suffices to prove that for every $g\in \cG_R$ we
have $g\bigl(N_i (X_i-1)^n\bigr)=N_i(g(X_i)-1)^n$. Since $N_i$
satisfies Leibniz' rule it suffices to consider the case $n=1$.
Then $g(X_i)=[\varepsilon]^{c_i(\gamma)} X_i$ for suitable
$c_i(\gamma)\in \Z_p^\ast$ and, as $N(X_i)=-X_i$, the formula is
readily verified.

(2) and (3) are a formal consequence of \ref{cor:BdRstr}.

\end{proof}

\subsubsection{Flatness and Galois invariants.}

Let  $\widehat{\widetilde{R}[p^{-1}]}$ be the $P_\pi(Z)$-adic
completion of $\widetilde{R}\bigl[p^{-1}\bigr]$.

\begin{lemma} \label{lemma:RtildeR} We have isomorphisms $\widehat{\cO[p^{-1}]}\cong K[\![Z-\pi]\!]$
and $\widehat{\widetilde{R}[p^{-1}]}\cong
\widehat{R}[p^{-1}]\left[\!\left[\frac{Z}{\pi}-1\right]\!\right]$
as $\widehat{\cO[p^{-1}]}$-algebras.
\end{lemma}
\begin{proof} Note that the
$P_\pi(Z)$-adic completion $\widehat{\cO[p^{-1}]} $ of
$\cO[p^{-1}]$ is a complete dvr with residue field $K$. In
particular, it is a $K$-algebra by Hensel's lemma and, hence,
$\widehat{\cO[p^{-1}]}$ is isomorphic to $K[\![Z-\pi]\!]$. Thus,
$\widehat{\widetilde{R}[p^{-1}]}$ is a $K[\![Z-\pi]\!]$-algebra.
Since it is $Z-\pi$-adically complete and separated and
$\widetilde{R}\bigl[p^{-1}\bigr]/ (Z-\pi)\cong
\widehat{R}[p^{-1}]$, the proof of the second isomorphism is a
variant of the proof of \ref{lemma:RtildeandR} and is left to the
reader.
\end{proof}

Recall that $\cG_R$ is the Galois group of
$\overline{R}\bigl[p^{-1}\bigr]$ over $R\bigl[p^{-1}\bigr]$. Then,

\begin{proposition}\label{prop:BdRff} The extensions
$\widehat{R}\bigl[p^{-1}\bigr]\subset {\rm B}_{\rm dR}(R)$ and
$\widehat{\widetilde{R}\bigl[p^{-1}\bigr]}\subset {\rm B}_{\rm
dR}(\widetilde{R})$ are faithfully flat. Moreover,
$\widehat{R}\bigl[p^{-1}\bigr]= {\rm B}_{\rm dR}(R)^{\cG_R}$ and $
\widehat{\widetilde{R}\bigl[p^{-1}\bigr]}= {\rm B}_{\rm
dR}(\widetilde{R})^{\cG_R}$.
\end{proposition}
\begin{proof} We prove the first assertion for $\widetilde{R}$.
The assertion concerning $R$ follows remarking that ${\rm B}_{\rm
dR}(R)\cong {\rm B}_{\rm dR}(\widetilde{R})/\bigl( P_\pi(Z)\bigr)$
so that $\widehat{R}\bigl[p^{-1}\bigr]\subset {\rm B}_{\rm dR}(R)$
is obtained from the extension
$\widehat{\widetilde{R}\bigl[p^{-1}\bigr]}\subset {\rm B}_{\rm
dR}(\widetilde{R})$  by tensoring with
$\widehat{\widetilde{R}\bigl[p^{-1}\bigr]}\to
\widehat{\widetilde{R}\bigl[p^{-1}\bigr]}/(P_\pi(Z))=\widehat{R}\bigl[p^{-1}\bigr]
$.

We first prove that ${\rm B}_{\rm dR}^+(\widetilde{R})/(t)$ is a
faithfully flat
$\widehat{\widetilde{R}\bigl[p^{-1}\bigr]}$-algebra.  It follows
from \ref{cor:BdRstr} that ${\rm B}_{\rm dR}^+(\widetilde{R})/t
{\rm B}_{\rm dR}^+(\widetilde{R})$ is isomorphic to
$\widehat{\overline{R}}[p^{-1}] \bigl[\!\bigl[
v_1-1,\ldots,v_a-1,w_1-1,\ldots,w_b-1\bigr]\!\bigr]$. This is a
faithfully flat $\widehat{R}[p^{-1}] \bigl[\!\bigl[
v_1-1,\ldots,v_a-1,w_1-1,\ldots,w_b-1\bigr]\!\bigr]$-algebra since
$\widehat{R}[p^{-1}]\subset \widehat{\overline{R}}[p^{-1}]$ is
faithfully flat by \ref{prop:Rbarff}. Furthermore,
$\widehat{R}[p^{-1}] \bigl[\!\bigl[
v_1-1,\ldots,v_a-1,w_1-1,\ldots,w_b-1\bigr]\!\bigr]$ is the
completion of $R\otimes_{\WW(k)}
\widetilde{R}\bigl[p^{-1}\bigr]\bigl[ v_1^{\pm 1},\ldots,v_a^{\pm
1},w_1^{\pm 1},\ldots,w_b^{\pm 1}\bigr]$ with respect to the
kernel of the map $R\otimes_{\WW(k)}
\widetilde{R}\bigl[p^{-1}\bigr]\bigl[ v_1^{\pm 1},\ldots,v_a^{\pm
1},w_1^{\pm 1},\ldots,w_b^{\pm 1}\bigr]\lra
\widehat{R}\bigl[p^{-1}\bigr]$. Such kernel is given by
$\bigl(P_\pi(Z), v_2-1,\ldots,v_a-1,w_1-1,\ldots,w_b-1\bigr]$.
Thus, such completion coincides with
$\widehat{\widetilde{R}\bigl[p^{-1}\bigr]}\bigl[\!\bigl[
v_2-1,\ldots,v_a-1,w_1-1,\ldots,w_b-1\bigr]\!\bigr]$ which is a
faithfully flat
$\widehat{\widetilde{R}\bigl[p^{-1}\bigr]}$-algebra.

Since $t$ is a regular element of ${\rm B}_{\rm
dR}^+(\widetilde{R})$, it follows by induction on $i$ that ${\rm
B}_{\rm dR}^+(\widetilde{R})/(t^i)$ is the successive extension of
flat $\widehat{\widetilde{R}\bigl[p^{-1}\bigr]}$-modules and,
hence, it is flat itself. Since
$\widehat{\widetilde{R}\bigl[p^{-1}\bigr]}$ is noetherian, one
concludes as in \cite[Thm. 5.4.1]{brinon} that ${\rm B}_{\rm
dR}^+(\widetilde{R})$ is a flat
$\widehat{\widetilde{R}\bigl[p^{-1}\bigr]}$-module. Since
$\Spec\bigl({\rm B}_{\rm dR}^+(\widetilde{R})/(t)\bigr)\lra
\Spec\bigl(\widehat{\widetilde{R}\bigl[p^{-1}\bigr]} \bigr)$ is
surjective then  $\Spec\bigl({\rm B}_{\rm
dR}^+(\widetilde{R})\bigr)\lra
\Spec\bigl(\widehat{\widetilde{R}\bigl[p^{-1}\bigr]} \bigr)$ is
surjective as well and
$\widehat{\widetilde{R}\bigl[p^{-1}\bigr]}\subset {\rm B}_{\rm
dR}^+(\widetilde{R})$  is faithfully flat. Arguing as in
\cite[Thm. 5.4.1]{brinon}, the faithful flatness of ${\rm B}_{\rm
dR}^+(\widetilde{R})/(t)$ as
$\widehat{\widetilde{R}\bigl[p^{-1}\bigr]}$-algebra implies that
the assertion of the Proposition regarding $ {\rm B}_{\rm
dR}^+\bigl(\widetilde{R}\bigr)$ implies the assertion regarding $
{\rm B}_{\rm dR}\bigl(\widetilde{R}\bigr)$.

\smallskip

We are left to compute the invariants. Recall from \ref{cor:BdRstr} that ${\rm B}_{\rm dR}(\widetilde{R}) \cong {\rm B}_{\rm dR}^+(R)[\![ u-1]\!] [t^{-1}]$. Since
$\bigl[\overline{\pi}\bigr]$ is invertible in ${\rm B}_{\rm dR}^+(R)$, then ${\rm B}_{\rm dR}^+(R)[\![ u-1]\!]\cong {\rm B}_{\rm dR}^+(R)\bigl[\!\bigl[
Z-\bigl[\overline{\pi}\bigr]\bigr]\!\bigr]$. Note that $\bigl[\overline{\pi}\bigr]-\pi \in {\rm Fil}^1  {\rm B}_{\rm dR}^+(R)$ so that  ${\rm B}_{\rm
dR}^+(R)\bigl[\!\bigl[ Z-\bigl[\overline{\pi}\bigr]\bigr]\!\bigr]\cong {\rm B}_{\rm dR}^+(R)\bigl[\!\bigl[ Z-\pi\bigr]\!\bigr]$. We conclude that ${\rm B}_{\rm
dR}(\widetilde{R})\subset  {\rm B}_{\rm dR}(R)\bigl[\!\bigl[ Z-\pi\bigr]\!\bigr]$. Since $Z-\pi$ is fixed by $\cG_R$, if we prove that ${\rm B}_{\rm
dR}(R)^{\cG_R}=\widehat{R}[p^{-1}]$, we conclude that ${\rm B}_{\rm dR}(\widetilde{R})^{\cG_R}$ is contained in $\widehat{R}[p^{-1}]\bigl[\!\bigl[ Z-\pi
\bigr]\!\bigr] \cong \widehat{\widetilde{R}[p^{-1}]}$. Since it also contains $\widehat{\widetilde{R}[p^{-1}]}$, it  coincides with
$\widehat{\widetilde{R}[p^{-1}]}$.

We are then left to show that ${\rm B}_{\rm
dR}(R)^{\cG_R}=\widehat{R}[p^{-1}]$. The proof proceeds as in
\cite[Prop.~5.2.12]{brinon}. Consider the exact sequence $$0\lra
{\rm Fil}^{r+1} {\rm B}_{\rm dR}(R) \lra {\rm Fil}^r {\rm B}_{\rm
dR}(R) \lra {\rm Gr}^r {\rm B}_{\rm dR}(R) \lra 0.$$As ${\rm
Gr}^r {\rm B}_{\rm
dR}\bigl(R\bigr)=t^r \widehat{\overline{R}}[p^{-1}]\left[
\frac{v_2-1}{t},\ldots,\frac{v_a-1}{t},\frac{w_1-1}{t},\ldots,\frac{w_b-1}{t}\right]$
by \ref{cor:BdRstr} one shows as
in \cite[Prop.~4.1.4\& Cor.~4.1.5]{brinon} that ${\rm H}^i\bigl(\cG_R,{\rm
Gr}^r {\rm B}_{\rm dR}(R)\bigr)$ is $\widehat{R[p^{-1}]}$ if $i=0$
and $1$ and $r=0$ and it is $0$ otherwise. We refer to loc.~cit.~for the details using
\ref{cor:cohoRbarmodp}(ii) in place of  \cite[Prop.~3.1.3]{brinon}.

In particular, ${\rm
H}^0\bigl(\cG_R,{\rm Fil}^r {\rm B}_{\rm dR}(R)\bigr)\cong {\rm
H}^0\bigl(\cG_R,{\rm Fil}^{r+1} {\rm B}_{\rm dR}(R)\bigr)$ for $r\neq 0$ which implies that
${\rm H}^0\bigl(\cG_R,{\rm Fil}^1 {\rm B}_{\rm dR}(R)\bigr)=0$ and ${\rm
H}^0\bigl(\cG_R,{\rm B}_{\rm dR}(R)\bigr)={\rm
H}^0\bigl(\cG_R,{\rm Fil}^0 {\rm B}_{\rm dR}(R)\bigr)$
since the filtration  on ${\rm B}_{\rm dR}(R)$ is separated and exhaustive by
\ref{cor:BdRstr}. Thus, ${\rm H}^0\bigl(\cG_R,{\rm Fil}^0 {\rm
B}_{\rm dR}(R)\bigr)\subset {\rm H}^0\bigl(\cG_R,{\rm Gr}^0 {\rm
B}_{\rm dR}(R)\bigr)$ which is $\widehat{R[p^{-1}]}$. Since
$\widehat{R[p^{-1}]}\subset {\rm H}^0\bigl(\cG_R,{\rm Fil}^0 {\rm
B}_{\rm dR}(R)\bigr)$, the claim follows.

Alternatively, one can argue in the same way using ${\rm B}_{\rm dR}(\widetilde{R})$ instead of
${\rm B}_{\rm dR}(R)$. Thanks to \ref{prop:BcrissubsetBdR}(4) and
\ref{thm:geometricacyclicity}(iii)  one deduces  that ${\rm H}^i\bigl(G_R,{\rm Gr}^r {\rm B}_{\rm dR}(\widetilde{R})\bigr)=0$
for $i\geq 1$ and every $r\in \N$ and
is a direct summand in $R\otimes_{\cO_K} {\rm Gr}^r B_{\rm log}$ for $i=0$. Here $G_R\subset \cG_R$ is the geometric
Galois group and $B_{\rm log}$ is the classical
ring of periods introduced in \S\ref{sec:notation}. As ${\rm Gr}^r B_{\rm log}=\sum_{a+b=r}{\rm Gr}^a
B_{\rm dR}\cdot (Z-\pi)^b$, see \S\ref{sec:classical}, one
deduces that ${\rm H}^i\bigl(G_R,{\rm Gr}^r {\rm B}_{\rm dR}(R)\bigr)$ is $0$ for $i\geq 1$ and every
$r\in \N$ and is a direct summand in $R\otimes_{\cO_K} {\rm
Gr}^r B_{\rm dR}$ for $i=0$. One deduces that ${\rm H}^i\bigl(\cG_R,{\rm Gr}^r {\rm B}_{\rm dR}(R)\bigr)$ is
$\widehat{R[p^{-1}]}$ if $i=0$ and $1$ and $r=0$ and it
is $0$ otherwise from the analogous result for  the cohomology of ${\rm Gr}^r B_{\rm dR}$.

\end{proof}

\begin{corollary}\label{cor:conenctioninducestandardconnection}
The connection $\nabla_R$ (resp.~$\nabla_{\widetilde{R}}$,
resp.~$\nabla_{\widetilde{R}/\cO}$) induces the standard
derivation $d\colon R\bigl[p^{-1}\bigr] \lra
\widehat{\omega}^1_{R/\cO_K}\bigl[p^{-1}\bigr]$ (resp.~$d\colon
\widetilde{R}\bigl[p^{-1}\bigr] \lra
\widehat{\omega}^1_{\widetilde{R}/\WW(k)}\bigl[p^{-1}\bigr]$,
resp.~$d\colon \widetilde{R}\bigl[p^{-1}\bigr] \lra
\widehat{\omega}^1_{\widetilde{R}/\cO}\bigl[p^{-1}\bigr]$);
\end{corollary}
\begin{proof} It follows from \ref{lemma:nabla=0} that the
connections are $\cG_R$-equivariant. Due to \ref{prop:BdRff}, upon taking invariants, we get maps as claimed. We only need to verify that they coincide with the
standard derivations. It suffices to prove that they send $X_i$ to $d X_i$  and $Y_j$ to $d Y_j$ (resp.~$\widetilde{X}_i$ to $ d\widetilde{X}_i$ and
$\widetilde{Y}_j$ to $ \widetilde{Y}_j$). This is clear.

\end{proof}

\subsection{The functors ${\rm D}_{\rm dR}$ and $\widetilde{{\rm D}}_{\rm
dR}$. De Rham representations.}
\label{sec:DdeRam}

Let $V$ be a finite dimensional $\Q_p$--vector space endowed with
a continuous action of $\cG_R$. We write $${\rm D}_{\rm
dR}(V):=\left(V\otimes_{\Q_p} {\rm B}_{\rm
dR}\bigl(R\bigr)\right)^{\cG_R}, \qquad \widetilde{{\rm D}}_{\rm
dR}(V):=\left(V\otimes_{\Q_p} {\rm B}_{\rm
dR}\bigl(\widetilde{R}\bigr)\right)^{\cG_R}.$$Then ${\rm D}_{\rm
dR}(V)$ is a $\widehat{R}[p^{-1}]$-module and $\widetilde{{\rm
D}}_{\rm dR}(V)$ is a $\widehat{\widetilde{R}[p^{-1}]}$--module.
The filtrations and the connections on ${\rm B}_{\rm
dR}\bigl(R\bigr)$ and on ${\rm B}_{\rm
dR}\bigl(\widetilde{R}\bigr)$ induce  exhaustive decreasing
filtrations ${\rm Fil}^n{\rm D}_{\rm dR}(V)$ and resp.~${\rm
Fil}^n\widetilde{{\rm D}}_{\rm dR}(V)$ for $n\in\Z$ and integrable
connections
$$\nabla\colon {\rm D}_{\rm dR}(V) \lra {\rm D}_{\rm dR}(V)\otimes_{\widehat{R}}
\widehat{\omega}^1_{\widehat{R}/\cO_K},\qquad \widetilde{\nabla}\colon
\widetilde{{\rm D}}_{\rm dR}(V) \lra \widetilde{{\rm D}}_{\rm
dR}(V)\otimes_{\widehat{\widetilde{R}}[p^{-1}]}
\widehat{\omega}^1_{\widetilde{R}/\WW(k)}$$such that the
filtrations satisfy Griffiths' transversality.

\begin{definition}\label{def:deRhamrep} We say that a representation
$V$ of $\cG_R$ is de Rham if $${\rm D}_{\rm dR}(V)
\otimes_{\widehat{R}[p^{-1}] } {\rm B}_{\rm dR}\bigl(R\bigr)\lra
V\otimes_{\Q_p}{\rm B}_{\rm dR}\bigl(R\bigr)$$is an isomorphism of
${\rm B}_{\rm dR}\bigl(R\bigr)$-modules.
\end{definition}

Recall from \ref{lemma:RtildeR}  that $\widehat{\widetilde{R}[p^{-1}]}\cong
\widehat{R}[p^{-1}][\![Z-\pi]\!]$ as filtered rings. Then,

\begin{lemma}\label{lemma:eqdeRham} Given a representation $V$ of $\cG_R$,
we have a functorial isomorphism of filtered
$\widehat{\widetilde{R}[p^{-1}]}$-modules endowed with connection
${\rm D}_{\rm dR}(V)\otimes_{\widehat{R}[p^{-1}] }
\widehat{\widetilde{R}[p^{-1}]} \lra \widetilde{{\rm D}}_{\rm
dR}(V)$. Thus, \smallskip

(1) the filtration on $\widetilde{{\rm D}}_{\rm dR}(V)$ is the composite of
the filtration on ${\rm D}_{\rm dR}(V)$ and the $(Z-\pi)$-adic filtration. In particular,
considering the map $$\rho\colon \widetilde{{\rm D}}_{\rm
dR}(V) \longrightarrow \widetilde{{\rm D}}_{\rm
dR}(V)/(Z-\pi)\cong {\rm D}_{\rm dR}(V)$$the filtration  $\Fil^n {\rm D}_{\rm dR}(V)$ is the image of
$\Fil^n \widetilde{{\rm D}}_{\rm dR}(V)$. Viceversa $\Fil^\bullet \widetilde{{\rm D}}_{\rm dR}(V)$ is the
unique filtration such that the image via $\rho$ is $\Fil^\bullet {\rm D}_{\rm dR}(V)$ and it satisfies Griffiths'
transversality with respect to $\widetilde{\nabla}$. It  is characterized by the property that for every $n\in \N$
$$\Fil^n \widetilde{{\rm D}}_{\rm dR}(V):=\left\{ x\in  \widetilde{{\rm D}}_{\rm
dR}(V)\vert \rho(x)\in \Fil^n {\rm D}_{\rm dR}(V),\quad \frac{\partial (x)}{\partial (Z-\pi)} \in \Fil^{n-1}
\widetilde{{\rm D}}_{\rm dR}(V)\right\};$$ \smallskip

(2) $V$ is de Rham if and only if $$\widetilde{{\rm D}}_{\rm
dR}(V) \otimes_{\widehat{\widetilde{R}}[p^{-1}] } {\rm B}_{\rm
dR}\bigl(\widetilde{R}\bigr)\lra V\otimes_{\Q_p}{\rm B}_{\rm
dR}\bigl(\widetilde{R}\bigr)$$is an isomorphism of ${\rm B}_{\rm
dR}\bigl(\widetilde{R}\bigr)$-modules.
 \end{lemma}
\begin{proof}
Recall from \ref{cor:BdRstr} that  ${\rm B}_{\rm
dR}^+\bigl(\widetilde{R}\bigr)\cong {\rm B}_{\rm
dR}^+\bigl(R\bigr) [\![ Z-\pi]\!]$. This isomorphism is compatible with the isomorphism $\widehat{\widetilde{R}[p^{-1}]}\cong
\widehat{R}[p^{-1}][\![Z-\pi]\!]$ via the inclusion $\widehat{\widetilde{R}[p^{-1}]}\subset {\rm B}_{\rm
dR}^+\bigl(\widetilde{R}\bigr)$ and $\widehat{R}[p^{-1}][\![Z-\pi]\!]\subset {\rm B}_{\rm
dR}^+\bigl(R\bigr) [\![ Z-\pi]\!]$. These isomorphisms  are strict with respect to the
filtrations where $\widehat{R}[p^{-1}][\![Z-\pi]\!]$ is endowed
with the $(Z-\pi)$-filtration and ${\rm B}_{\rm dR}^+\bigl(R\bigr)
[\![ u-1]\!]$ is endowed with the filtration composite of the
filtration on ${\rm B}_{\rm dR}^+\bigl(R\bigr)$ and the
$(Z-\pi)$-adic filtration.
We deduce that the natural application
${\rm D}_{\rm dR}(V)\otimes_{\widehat{R}[p^{-1}] }
\widehat{\widetilde{R}[p^{-1}]} \lra \widetilde{{\rm D}}_{\rm
dR}(V)$ is an isomorphism of filtered
$\widehat{\widetilde{R}[p^{-1}]}$-modules endowed with connection. This proves the first claim.
Claims (1) and (2) follow. For the formula in (1) compare with \cite[p. 207]{breuil}. We remark that in
(1) the condition ${\partial x}/{\partial (Z-\pi)} \in \Fil^{n-1} \widetilde{{\rm D}}_{\rm dR}(V)$
is equivalent to  $\widetilde{\nabla}(x)\in \Fil^{n-1} \widetilde{{\rm D}}_{\rm dR}(V)\otimes_{\widehat{\widetilde{R}}[p^{-1}]}
\widehat{\omega}^1_{\widetilde{R}/\WW(k)}$.
\end{proof}

\begin{proposition}\label{prop:GrDdR} Let $V$ be a de Rham representation of $\cG_R$ of dimension $n$. Then,\smallskip

(1) ${\rm D}_{\rm dR}(V)$ (resp.~$\widetilde{\rm D}_{\rm dR}(V)$) are finite and projective
$\widehat{R}[p^{-1}] $-module (resp.~$\widehat{\widetilde{R}}[p^{-1}] $) of rank $n$;\smallskip

(2)  the $\widehat{R}[p^{-1}] $-modules ${\rm Fil}^r{\rm D}_{\rm
dR}(V)$, ${\rm Gr}^r{\rm D}_{\rm dR}(V):={\rm Fil}^r{\rm
D}_{\rm dR}(V)/{\rm Fil}^{r+1}{\rm D}_{\rm dR}(V)$ and ${\rm Gr}^r\widetilde{{\rm D}}_{\rm dR}(V):={\rm Fil}^r\widetilde{\rm
D}_{\rm dR}(V)/{\rm Fil}^{r+1}\widetilde{\rm D}_{\rm dR}(V)$ are finite and
projective for every $r\in\N$;\smallskip

(3) for every $r\in\N$ the natural maps $$\oplus_{a+b=n}  {\rm
Gr}^a{\rm D}_{\rm dR}(V) \otimes_{\widehat{R}[p^{-1}] }  {\rm
Gr}^b {\rm B}_{\rm dR}\bigl(R\bigr) \lra V\otimes_{\Q_p}  {\rm
Gr}^n {\rm B}_{\rm dR}\bigl(R\bigr)$$and $$\oplus_{a+b=n}  {\rm
Gr}^a\widetilde{\rm D}_{\rm dR}(V) \otimes_{\widehat{R}[p^{-1}] }  {\rm
Gr}^b {\rm B}_{\rm dR}\bigl(\widetilde{R}\bigr) \lra V\otimes_{\Q_p}  {\rm
Gr}^n {\rm B}_{\rm dR}\bigl(\widetilde{R}\bigr)$$are isomorphisms.  \smallskip

In particular, the isomorphisms  ${\rm D}_{\rm dR}(V)
\otimes_{\widehat{R}[p^{-1}] } {\rm B}_{\rm dR}\bigl(R\bigr)\lra
V\otimes_{\Q_p}{\rm B}_{\rm dR}\bigl(R\bigr)$ and $\widetilde{{\rm D}}_{\rm
dR}(V) \otimes_{\widehat{\widetilde{R}}[p^{-1}] } {\rm B}_{\rm
dR}\bigl(\widetilde{R}\bigr)\lra V\otimes_{\Q_p}{\rm B}_{\rm
dR}\bigl(\widetilde{R}\bigr)$ are strict with
respect to the filtrations.\end{proposition}
\begin{proof} The last claim follows from the others and \ref{lemma:eqdeRham}.
Claim (1) for ${\rm D}_{\rm dR}(V)$ follows from the assumption that $V$ is de Rham and the fact, proven in \ref{prop:BdRff},
that  the extension
$\widehat{R}\bigl[p^{-1}\bigr]\subset {\rm B}_{\rm dR}(R)$ is faithfully flat. The statement for
$\widetilde{\rm D}_{\rm dR}(V)$ follows from  \ref{lemma:eqdeRham}.

If (3) holds and since the extension
$\widehat{R}\bigl[p^{-1}\bigr]\subset\widehat{\overline{R}}\bigl[p^{-1}\bigr]$
is faithfully flat by \ref{prop:Rbarff} and since $ {\rm Gr}^b
{\rm B}_{\rm dR}\bigl(R\bigr)$ is a free
$\widehat{\overline{R}}\bigl[p^{-1}\bigr]$-module by
\ref{cor:BdRstr}, it follows that each  ${\rm Gr}^r{\rm D}_{\rm
dR}(V)$ is finite and projective as $\widehat{R}[p^{-1}] $-module.
Arguing by induction one gets Claim (2) for ${\rm Fil}^r{\rm
D}_{\rm dR}(V)$ as well. Statements (2) and (3) for ${\rm Gr}^\bullet\widetilde{{\rm D}}_{\rm dR}(V)$   follow from \ref{lemma:eqdeRham}.

We are left to prove (3). Let $T$ be the set of minimal prime ideals of $R$ over the ideal $(\pi)$ of $R$. For any such $\cP$ let $\overline{T}_\cP$ be the set of
minimal prime ideals of $\overline{R}$ over the ideal $\cP$. For any  $\cP\in T$ denote by $\widehat{R}_\cP$ the $p$-adic completion of the localization of $R$ at
$\cP\cap R$. It is a dvr.   For $\cQ\in \overline{T}_\cP$ let $\overline{R}(\cQ)$  be the normalization of $\overline{R}_\cQ$ in an algebraic closure of its
fraction field and let $\widehat{\overline{R}}(\cQ)$ be its $p$-adic completion. Let $ {\rm B}_{\rm dR,\cQ}\bigl(R\bigr)$ be the ring defined using the extension
$\widehat{R}_\cP\subset \widehat{\overline{R}}(\cQ)$ and let $\cG_{R,\cQ}$ be the decomposition group of $\cG_R$ at $\cQ$. It is the Galois group of $R_\cP\subset
{\overline{R}}(\cQ)$. The normality of $\overline{R}$ implies that the map $\overline{R}/p \overline{R} \lra \prod_{\cQ} \overline{R}_\cQ/p \overline{R}_\cQ$, where
the product is taken over all $\cP$ and all $\cQ$, is injective.  This and \ref{cor:BdRstr} imply that the map $ {\rm B}_{\rm dR}\bigl(R\bigr) \ra \prod_\cQ {\rm
B}_{\rm dR,\cQ}\bigl(R\bigr)$ is injective on graded rings and, hence, it is injective. It is naturally a map of $\cG_R$-modules considering the action of $\cG_R$
on the prime ideals $\cQ$'s; see \cite[Rmk.~3.3.2]{brinon} for a description of the action on $\prod_\cQ \widehat{\overline{R}}(\cQ)$. In particular,  the map
$f\colon {\rm D}_{\rm dR}(V)\lra \prod_\cQ {\rm D}_{\rm dR,\cQ}(V)$ where ${\rm D}_{\rm dR,\cQ}(V):=\bigl(V\otimes_{\Q_p} {\rm B}_{\rm
dR,\cQ}\bigl(R\bigr)\bigr)^{\cG_{R,\cQ}}$ is injective.

By \cite[Rmk.~3.3.2]{brinon} the group $\cG_R$ acts transitively on  $\overline{T}_\cP$ and, for every $\cQ$ and $\cQ'\in \overline{T}_\cP$, any $h\in \cG_R$
sending $\cQ$ to $\cQ'$ induces an isomorphism $\widehat{\overline{R}}(\cQ)\cong \widehat{\overline{R}}(\cQ')$ and hence ${\rm B}_{\rm dR,\cQ}\bigl(R\bigr)\cong {
\rm B}_{\rm dR,\cQ'}\bigl(R\bigr)$. This induces an isomorphism between ${\rm D}_{\rm dR,\cQ}(V)$ and ${\rm D}_{\rm dR,\cQ'}(V)$  as filtered
$\widehat{R}_\cP[p^{-1}]$-modules for any $\cQ$ and $\cQ'$ over $\cP$. As the elements in the image of $f$ are fixed under the action of $\cG_R$ and Claims (2) and
(3)  are known for $R$ formally smooth over $\cO_K$ by \cite[Prop. 8.3.2]{brinon}, ${\rm Gr}^a{\rm D}_{\rm dR}(V) $ is zero apart for finitely many $a$'s. The
morphism $f$ is strict with respect to the filtrations by \ref{cor:BdRstr} so that it induces an injective morphism ${\rm Gr}^a{\rm D}_{\rm dR}(V) \lra \prod_\cQ
{\rm Gr}^a{\rm D}_{\rm dR,\cQ}(V)$ for every $a\in\N$. Since the map in (3) is injective for ${\rm D}_{\rm dR,\cQ}(V)$ for every $\cQ$, we conclude that the map
displayed in (3) is injective for every $n\in\N$. This implies that it is an isomorphism and (3) follows.

\end{proof}

\subsection{The rings ${\rm B}^{\rm cris}_{\rm log}$ and ${\rm B}^{\rm max}_{\rm log}$}\label{sec:AcrisnablaRbar}

Define ${\rm A}_{\rm cris}^\nabla(R)$ to be the $p$--adic
completion of the logarithmic divided power envelope
$\bigl(\WW\bigl(\widetilde{\bf E}^+\bigr)\bigr)^{\rm DP}$ of
$\WW\bigl(\widetilde{\bf E}^+\bigr)$ with respect to
$\Ker\bigl(\Theta\bigr)$ (compatible with the canonical divided
power structure on $p \WW\bigl(\widetilde{\bf E}^+\bigr)$).
We define~${\rm A}_{\rm max}^\nabla(R)$ to be the $p$-adic completion
of the $\WW\bigl(\widetilde{\bf E}^+\bigr)$-subalgebra of
$\WW\bigl(\widetilde{\bf E}^+\bigr)[p^{-1}]$ generated by
$p^{-1}\Ker(\Theta)$. We have a natural map ${\rm A}_{\rm
cris}^\nabla(R)\lra {\rm A}_{\rm max}^\nabla(R)$. The element
$t=\log\bigl([\varepsilon]\bigr)$ is well defined in ${\rm A}_{\rm
cris}^\nabla$. Define $ {\rm B}_{\rm cris}^\nabla(R):= {\rm
A}_{\rm cris}^\nabla(R)\bigl[t^{-1}\bigr]$ and $ {\rm B}_{\rm
max}^\nabla(R):= {\rm A}_{\rm max}^\nabla(R)\bigl[t^{-1}\bigr]$.
\smallskip

Let ${\rm A}_{\rm log}^{\rm cris,\nabla}(R)$ be the $p$--adic completion of the log divided power envelope $\bigl(\WW\bigl(\widetilde{\bf E}^+\bigr)\tensor_{\WW(k)}
\cO\bigr)^{\rm logDP}$ of $\WW\bigl(\widetilde{\bf E}^+\bigr)\tensor_{\WW(k)} \cO$ with respect to $\Ker\bigl(\Theta_{\log}\bigr)$ (compatible with the canonical
divided power structure on $p \WW\bigl(\widetilde{\bf E}^+\bigr)\tensor_{\WW(k)} \cO$) in the sense of \cite[Def. 5.4]{katolog}. Here we consider on
$\WW\bigl(\widetilde{\bf E}^+\bigr)\tensor_{\WW(k)} \cO$ its log structure. Define ${\rm A}_{\rm log}^{\rm max,\nabla}(R)$ to be the $p$--adic completion of the
$\bigl(\WW\bigl(\widetilde{\bf E}^+\bigr)\tensor_{\WW(k)} \cO\bigr)^{\rm log}$-subalgebra of $\bigl(\WW\bigl(\widetilde{\bf E}^+\bigr)\tensor_{\WW(k)}
\cO\bigr)^{\rm log}[p^{-1}]$ generated by $p^{-1}\Ker\bigl(\Theta_{\log}'\bigr)$. We have a natural map ${\rm A}_{\rm log}^{\rm cris,\nabla}(R)\lra {\rm A}_{\rm
log}^{\rm max,\nabla}(R)$. We define
$${\rm B}_{\rm log}^{\rm cris,\nabla}(R):=
{\rm A}_{\rm log}^{\rm cris,\nabla}(R)\bigl[t^{-1}\bigr]\quad
\hbox{{\rm and}}\quad {\rm B}_{\rm log}^{\rm max,\nabla}(R):= {\rm
A}_{\rm log}^{\rm max,\nabla}(R)\bigl[t^{-1}\bigr].$$
\smallskip

Let ${\rm A}^{\rm cris}_{\rm log}(\widetilde{R})$ be the $p$--adic completion of the log divided power envelope $\left(\WW\bigl(\widetilde{\bf
E}^+\bigr)\tensor_{\WW(k)} \widetilde{R}\right)^{\rm logDP}$ of $\WW\bigl(\widetilde{\bf E}^+\bigr)\tensor_{\WW(k)} \widetilde{R}$ with respect to
$\Ker\bigl(\Theta_{\widetilde{R},\log}\bigr)$ (compatible with the canonical divided power structure on $p \WW\bigl(\widetilde{\bf E}^+\bigr)\tensor_{\WW(k)}
\widetilde{R} $) in the sense of \cite[Def. 5.4]{katolog}. Let ${\rm A}^{\rm max}_{\rm log}(\widetilde{R})$ be the $p$--adic completion of  the
$\bigl(\WW\bigl(\widetilde{\bf E}^+\bigr)\tensor_{\WW(k)} \widetilde{R}\bigr)^{\rm log}$-subalgebra of $\bigl(\WW\bigl(\widetilde{\bf E}^+\bigr)\tensor_{\WW(k)}
\widetilde{R}\bigr)^{\rm log}[p^{-1}]$ generated by $p^{-1}\Ker\bigl(\Theta_{\widetilde{R},\log}'\bigr)$. We have a natural map ${\rm A}^{\rm cris}_{\rm
log}(\widetilde{R}) \lra {\rm A}^{\rm max}_{\rm log}(\widetilde{R}) $. Define
$$ {\rm B}^{\rm cris}_{\rm log}(\widetilde{R}):=
{\rm A}^{\rm cris}_{\rm
log}(\widetilde{R})\bigl[t^{-1}\bigr]\quad\hbox{{\rm and}}\quad
{\rm B}^{\rm max}_{\rm log}(\widetilde{R}):= {\rm A}^{\rm
max}_{\rm log}(\widetilde{R})\bigl[t^{-1}\bigr] .$$

\subsubsection{Explicit descriptions of ${\rm B}^{\rm cris}_{\rm log}$ and ${\rm B}^{\rm max}_{\rm log}$}

\begin{lemma}\label{lemma:structureAlog} The ring
${\rm A}_{\rm log}^{\rm cris,\nabla}(R)$ coincides with the
$p$-adic completion of\/ the divided power envelope of
$\WW\bigl(\widetilde{\bf E}^+\bigr)[u]$ with respect to the ideal
$\bigl(\xi,u-1\bigr)=
\bigl(P_\pi\bigl([\overline{\pi}]\bigr),u-1\bigr)$. Hence, $${\rm
A}_{\rm log}^{\rm cris,\nabla}(R)\cong {\rm A}_{\rm
cris}^\nabla(R)\bigl\{\langle u-1\rangle\bigr\},$$the $p$-adic
completion of the divided power algebra ${\rm A}_{\rm cris}^\nabla(R)\langle
u-1\rangle$.\smallskip

Similarly,  ${\rm A}_{\rm log}^{\rm max,\nabla}(R)\cong
\WW\bigl(\widetilde{\bf
E}^+\bigr)\left\{\frac{W}{p},\frac{u-1}{p}\right\}$,  the
$p$-adically convergent power series ring  in the variables
$W=\xi$ (or $W=P_\pi\bigl([\overline{\pi}]\bigr)$) and $u-1$. In
particular,  $${\rm A}_{\rm log}^{\rm max,\nabla}(R)\cong {\rm
A}_{\rm max}^\nabla(R)\left\{\frac{u-1}{p}\right\}.$$
\end{lemma}
\begin{proof} We prove the claims for ${\rm A}_{\rm cris}^\nabla$
and ${\rm A}_{\rm log}^{\rm cris,\nabla}$. Those for ${\rm A}_{\rm
max}^\nabla$ and ${\rm A}_{\rm log}^{\rm max,\nabla}$ follow
similarly. It is clear that ${\rm A}_{\rm
cris}^\nabla(R)\bigl\{\langle u-1\rangle\bigr\}$ is the $p$-adic
completion of the DP envelope of $\WW\bigl(\widetilde{\bf
E}^+\bigr)[u]$ with respect to the ideal $\bigl(\xi,u-1\bigr)$.
There is a map of $\WW\bigl(\widetilde{\bf E}^+\bigr)[u]$-algebras
$$f\colon {\rm A}_{\rm cris}^\nabla(R)\bigl\{\langle
u-1\rangle\bigr\}\lra {\rm A}_{\rm log}^{\rm cris,\nabla}(R)$$by universal property of divided power envelopes. By definition ${\rm A}_{\rm log}^{\rm
cris,\nabla}(R)$ is the $p$-adic completion of $\WW\bigl(\widetilde{\bf E}^+\bigr)^{\rm log DP}$ and the latter is the DP envelope of $\WW\bigl(\widetilde{\bf
E}^+\bigr)\otimes_{\WW(k)}\cO\bigl[u, u^{-1} \bigr]$ with respect to the kernel of the map to $\Theta_{\log}'$. Such kernel is $\bigl(\xi,u-1\bigr)$ by
\ref{lemma:structurBdR+}. Note that $u=1+(u-1)$ has $\sum_{i=0}^\infty (-1)^i i! (u-1)^{[i]}$ as multiplicative inverse in ${\rm A}_{\rm
cris}^\nabla(R)\bigl\{\langle u-1\rangle\bigr\}$. Furthermore, $Z=\bigl[\overline{\pi}\bigr](u^{-1}-1)+ \bigl[\overline{\pi}\bigr] $. If $e$ is the degree of
$P_\pi(Z)$ then $\bigl[\overline{\pi}\bigr]^e=\nu \bigl[\overline{p}\bigr]$ with $\nu$ a unit of $\WW\bigl(\widetilde{\bf E}^+\bigr)$. This implies that $Z^e-\nu p$
admits divided powers in ${\rm A}_{\rm cris}^\nabla(R)\bigl\{\langle u-1\rangle\bigr\}$ so that, since the latter is $p$-adically complete, power series in $Z$
converge in it. We thus get a map $g \colon  {\rm A}_{\rm log}^{\rm cris,\nabla}(R)\lra {\rm A}_{\rm cris}^\nabla(R)\bigl\{\langle u-1\rangle\bigr\}$ which is the
inverse of $f$.
\end{proof}

\begin{corollary}\label{display:Alognabla}
For every $n\in\N$ the morphisms
$$
\WW_n\bigl(\widetilde{\bf E}^+\bigr)
\bigl\{\delta_0,\delta_1,\ldots\bigr\}/
\bigl(p\delta_0-\xi^p,p\delta_{m+1}-\delta_m^p\bigr)_{m\in\N}\lra
{\rm A}_{\rm cris}^\nabla(R)/p^n {\rm A}_{\rm cris}^\nabla(R)$$and
$$ {\rm A}_{\rm cris}^\nabla(R)/p^n {\rm A}_{\rm cris}^\nabla(R) [u]\bigl\{\rho_0,\rho_1,\ldots\bigr\}/
\bigl(p\rho_0-(u-1)^p,p\rho_{m+1}-\rho_m^p\bigr)_{m\in\N}\lra {\rm A}_{\rm log}^{\rm cris,\nabla}(R)/p^n {\rm A}_{\rm log}^{\rm cris,\nabla}(R),$$sending $\delta_m$
to $\gamma^{m+1}(\xi)$ and $\rho_m$ to $\gamma^{m+1}(u-1)$ with $\gamma(x):=(p-1)! x^{[p]}$, are isomorphisms. In particular, ${\rm A}_{\rm cris}^\nabla(R)$ and
${\rm A}_{\rm log}^{\rm cris,\nabla}(R)$ are $p$-torsion free. Also ${\rm A}_{\rm max}^\nabla(R)$ and ${\rm A}_{\rm log}^{\rm max,\nabla}(R)$ are $p$-torsion free.
\end{corollary}
\begin{proof} For the first claim one argues as for the proof of \cite[Prop. 6.1.1 \& Cor.
6.1.2]{brinon}. If ${\rm A}_{\rm cris}^\nabla(R)$ is $p$-torsion
free, then ${\rm A}_{\rm log}^{\rm cris,\nabla}(R)$ is $p$-torsion
free as well thanks to \ref{lemma:structureAlog}. One proves that
${\rm A}_{\rm cris}^\nabla(R)$ is $p$-torsion free as in
\cite[Prop. 6.1.3]{brinon}. The fact that ${\rm A}_{\rm
max}^\nabla(R)$ and ${\rm A}_{\rm log}^{\rm max,\nabla}(R)$ are
$p$-torsion free is clear.
\end{proof}

\begin{lemma}\label{lemma:structAlog} The natural map
$${\rm A}_{\rm log}^{\rm cris,\nabla}(R)\left\{\langle
v_2-1,\ldots,v_a-1,w_1-1,\ldots,w_b-1\rangle \right\} \lra  {\rm
A}^{\rm cris}_{\rm log}(\widetilde{R})$$is an isomorphism. The map
$${\rm A}_{\rm log}^{\rm max,\nabla}(R)\left\{
\frac{v_2-1}{p},\ldots,\frac{v_a-1}{p},\frac{w_1-1}{p},\ldots,\frac{w_b-1}{p}\right\}
\lra  {\rm A}^{\rm max}_{\rm log}(\widetilde{R})$$is an
isomorphism.

\end{lemma}

\begin{proof} We prove the claims for ${\rm A}^{\rm cris}_{\rm log}$.
Those for ${\rm A}^{\rm max}_{\rm log}$ follow similarly. We
follow \cite[Prop. 6.1.5]{brinon}. Recall that ${\rm A}^{\rm
cris}_{\rm log}(\widetilde{R})$ is the $p$-adic completion of
$\left(\WW\bigl(\widetilde{\bf E}^+\bigr)\tensor_{\WW(k)}
\widetilde{R}\right)^{\rm logDP}$. The latter is the DP envelope
of
$$\WW\bigl(\widetilde{\bf E}^+\bigr)\tensor_{\WW(k)}
\widetilde{R}\left[\frac{p}{\bigl[\overline{p}\bigr]},\frac{\bigl[\overline{p}\bigr]}{p},
\frac{X_2}{\bigl[\overline{X}_2\bigr]},\frac{\bigl[\overline{X}_2\bigr]}{X_2},\ldots, \frac{X_a}{\bigl[\overline{X}_a\bigr]},\frac{\bigl[\overline{X}_a\bigr]}{X_a},
\frac{Y_1}{\bigl[\overline{Y}_1\bigr]},\frac{\bigl[\overline{Y}_1\bigr]}{Y_1},\ldots,
\frac{Y_b}{\bigl[\overline{Y}_b\bigr]},\frac{\bigl[\overline{Y}_b\bigr]}{Y_b}\right]$$with respect to the ideal $\bigl(\xi,
u-1,v_2-1,\ldots,v_a-1,w_1-1,\ldots,w_b-1\bigr)$ which is the kernel of the map to $\widehat{\overline{R}}$ by \ref{lemma:structurBdR+}. Note that  ${\rm A}_{\rm
log}^{\rm cris,\nabla}(R)$ is an $\cO$-algebra. Consider the structure of $\cO[P']$-algebra on ${\rm A}_{\rm log}^{\rm cris,\nabla}(R)\left\{\langle
v_2-1,\ldots,v_a-1,w_1-1,\ldots,w_b-1\rangle \right\}$ given by sending $X_i$ to $\bigl[\overline{X}_i\bigr] v_i+ \bigl[\overline{X}_i\bigr]$ for $i=2,\ldots,a$ and
$Y_j$ to $\bigl[\overline{Y}_j\bigr] w_j+ \bigl[\overline{Y}_j\bigr]$ for $j=1,\ldots,b$. Using that $v_i$ is invertible in ${\rm A}_{\rm log}^{\rm
cris,\nabla}(R)\left\{\langle v_2,\ldots,v_a,w_1,\ldots,w_b\rangle \right\}$ for $i=2,\ldots,a$, it sends $X_1$ to $\bigl[\overline{X}_1 \bigr] u^\alpha
\prod_{i=2}^a v_i^{-1}$. The ring ${\rm A}_{\rm log}^{\rm cris,\nabla}(R)\left\{\langle v_2-1,\ldots,v_a-1,w_1-1,\ldots,w_b-1 \rangle \right\}$ is
$$\frac{{\rm A}_{\rm log}^{\rm cris,\nabla}(R)\bigl\{
v_i,h_{i,0},h_{i,1},\cdots,w_j,\ell_{j,0},\ell_{j,1},\cdots\bigr\}_{
i=2,\ldots,a,\, j=1,\ldots,b}}{ \bigl(
ph_{i,0}-(v_i-1)^p,ph_{m+1,0}-h_{i,m}^p,p\ell_{j,0}-(w_i-1)^p,p\ell_{m+1,0}-\ell_{j,m}^p\bigr)}$$with
$h_{i,m}\mapsto \gamma^{m+1}(v_i)$ and $h_{j,m}\mapsto
\gamma^{m+1}(w_j)$ where $\gamma\colon x \mapsto \frac{x^p}{p}$.
See \cite[Prop. 6.1.2]{brinon}. Put $${\cal
A}:=\frac{\widetilde{\bf E}^+/\overline{p}^p \widetilde{\bf
E}^+\bigl[u-1,v_2-1,\ldots,v_a-1,w_1-1,\ldots,w_b-1\bigr]}{
\bigl((u-1)^p,(v_2-1)^p,\ldots,(v_a-1)^p,(w_1-1)^p,\ldots,(w_b-1)^p\bigr)}$$and
${\cal
I}:=(\overline{p},u-1,v_2-1,\ldots,v_a-1,w_1-1,\ldots,w_b-1)$. It
follows from \ref{display:Alognabla} that

$${\rm A}_{\rm log}^{\rm cris,\nabla}(R)\left\{\langle
v_2-1,\ldots,v_a-1,w_1-1,\ldots,w_b-1\rangle \right\}/(p) \cong$$
$$\cong \frac{{\cal A} \bigl[\delta_0,\delta_1,
\ldots,\rho_0,\rho_1,\ldots,\ldots h_{i,0},h_{i,1},\cdots,\ell_{j,0},\ell_{j,1},\cdots\bigr]_{ i=2,\ldots,a,\, j=1,\ldots,b}}{\bigl(
\delta_m^p,\rho_m^p,\ell_{j,m}^p\bigr)_{i=2,\ldots,a,j=1,\ldots,b,m\in\N}}.$$ Then, ${\rm A}_{\rm log}^{\rm cris,\nabla}(R)\left\{\langle
v_2-1,\ldots,v_a-1,w_1-1,\ldots,w_b-1\rangle \right\}$ modulo $p$ is an ${\cal A}$--algebra and ${\cal A}$ is an $\cO[P']$--algebra. The ideal of ${\cal I}$ is
nilpotent. Furthermore, ${\cal A}/{\cal I}\cong \overline{R}/p\overline{R}$ as $\cO[P']$--algebras. Since $\widetilde{R}/p\widetilde{R}$ is a successive extension
of $\cO[P']/p\cO[P']$--algebras obtained taking localizations, \'etale extensions and completions with respect to ideals,  there exists a unique morphism of
$\cO[P']$--algebras $\widetilde{R}\to {\cal A}$ inducing on $\overline{R}/p\overline{R}$ its natural structure of $\widetilde{R}$--algebra. This also provides ${\rm
A}_{\rm log}^{\rm cris,\nabla}(R)\left\{\langle v_2-1,\ldots,v_a-1,w_1-1,\ldots,w_b-1\rangle \right\}$ modulo $p$ with a structure of $\widetilde{R}$--algebra. By
induction on $n$ we get unique maps  $\widetilde{R} \lra {\rm A}_{\rm log}^{\rm cris,\nabla}(R)\left\{\langle v_2-1,\ldots,v_a-1,w_1-1,\ldots,w_b-1\rangle
\right\}/(p^n)$ of $\cO[P']$--algebras, compatible for varying $n$, inducing via the natural map $${\rm A}_{\rm log}^{\rm cris,\nabla}(R)\left\{\langle
v_2-1,\ldots,v_a-1,w_1-1,\ldots,w_b-1\rangle \right\}/(p^n)\to \overline{R}/p^n\overline{R}$$the natural structure of $\widetilde{R}$--algebra on
$\overline{R}/p^n\overline{R}$. Hence,
$${\rm A}_{\rm log}^{\rm cris,\nabla}(R)\left\{\langle
v_2-1,\ldots,v_a-1,w_1-1,\ldots,w_b-1\rangle \right\}$$is a
$\WW\bigl(\widetilde{\bf E}^+\bigr)\tensor_{\WW(k)}
\widetilde{R}$--algebra.  By the universal property  of ${\rm
A}^{\rm cris}_{\rm log}(\widetilde{R})$ such morphism extends
uniquely to a morphism
$$f\colon {\rm A}^{\rm cris}_{\rm log}(\widetilde{R})\lra {\rm A}_{\rm log}^{\rm cris,\nabla}(R)\left\{\langle
v_2-1,\ldots,v_a-1,w_1-1,\ldots,w_b-1\rangle \right\}.$$Consider the natural map $g\colon {\rm A}_{\rm log}^{\rm cris,\nabla}(R)\left\{\langle
v_2-1,\ldots,v_a-1,w_1-1,\ldots,w_b-1\rangle \right\}\lra {\rm A}^{\rm cris}_{\rm log}(\widetilde{R})$ of ${\rm A}_{\rm log}^{\rm cris,\nabla}(R)$-algebras. By
construction it is a morphism of $\WW\bigl(\widetilde{\bf E}^+\bigr)\tensor_{\WW(k)}\cO[P']$--algebras. Arguing as before, one concludes that it is a morphism of
$\WW\bigl(\widetilde{\bf E}^+\bigr)\tensor_{\WW(k)} \widetilde{R}$--algebras since this holds  modulo $p^n$ by induction on $n$. One verifies that since the
composites $g\circ f$ and $f\circ g$ are morphisms of $\WW\bigl(\widetilde{\bf E}^+\bigr)\tensor_{\WW(k)} \widetilde{R}$--algebras, they are the identities on

$${\rm
A}^{\rm cris}_{\rm log}(\widetilde{R}) \mbox{ and on } {\rm A}_{\rm log}^{\rm cris,\nabla}(R)\left\{\langle v_2-1,\ldots,v_a-1,w_1-1,\ldots,w_b-1\rangle \right\}
$$
respectively, by the universal properties of divided power envelopes. This concludes the proof.
\end{proof}

\begin{remark}\label{rmk:alpha=1} We have morphisms $${\rm A}_{\rm
cris}^\nabla(R)\left\{\langle v_1-1,\ldots,v_a-1,w_1-1,\ldots,w_b-1\rangle \right\}\lra {\rm A}^{\rm cris}_{\rm log}(\widetilde{R})$$and $${\rm A}_{\rm
max}^\nabla(R)\left\{\frac{v_1-1}{p},\ldots,\frac{v_a-1}{p},\frac{w_1-1}{p},\ldots,\frac{w_b-1}{p}\right\} \lra  {\rm A}^{\rm max}_{\rm log}(\widetilde{R}).$$Using
the isomorphism ${\rm A}_{\rm log}^{\rm cris,\nabla}(R)\cong {\rm A}_{\rm cris}^\nabla(R)\bigl\{\langle u-1\rangle\bigr\}$ of \ref{lemma:structureAlog}, the first
map is a map of ${\rm A}_{\rm cris}^\nabla(R)$-algebras sending $v_1$ to $u^\alpha v_2^{-1} \cdots v_a^{-1}$ and being the identity on the $v_i$'s for $i\geq 2$ and
on the $w_j$'s. Hence, using \ref{lemma:structAlog} we conclude that the above morphisms are isomorphisms if $\alpha=1$.
\end{remark}

\begin{corollary}\label{display:Alogtorionfree} The rings
${\rm A}^{\rm cris}_{\rm log}(\widetilde{R})$  and ${\rm A}^{\rm
max}_{\rm log}(\widetilde{R})$ are $p$-torsion free.
\end{corollary}
\begin{proof} This follows from \ref{lemma:structAlog} and
\ref{display:Alognabla}.

\end{proof}

\subsubsection{Galois action, filtrations, Frobenii,
connections}\label{GaloisfiltFrob}

Write ${\rm A}={\rm A}_{\rm cris}^\nabla\bigl(R\bigr)$ or ${\rm
A}_{\rm max}^\nabla\bigl(R\bigr)\bigl[p^{-1}\bigr]$, ${\rm A}_{\rm
log}^\nabla:={\rm A}^{\rm cris,\nabla}_{\rm log}\bigl(R\bigr)$ or
${\rm A}^{\rm max,\nabla}_{\rm
log}\bigl(\widetilde{R}\bigr)\bigl[p^{-1}\bigr]$ and ${\rm A}_{\rm
log}:={\rm A}^{\rm cris}_{\rm log}\bigl(\widetilde{R}\bigr)$ or
${\rm A}^{\rm max}_{\rm
log}\bigl(\widetilde{R}\bigr)\bigl[p^{-1}\bigr]$. Put ${\rm
B}={\rm A}\bigl[t^{-1}\bigr]$, ${\rm B}_{\rm log}^\nabla={\rm
A}_{\rm log}^\nabla\bigl[t^{-1}\bigr]$ and ${\rm B}_{\rm log}={\rm
A}_{\rm log}\bigl[t^{-1}\bigr]$.

\

{\it Galois action:} The Galois action of $\mathcal{G}_R$ on $
\WW\bigl(\widetilde{\bf E}^+\bigr)$ extends to an action on the
rings ${\rm A}$, ${\rm A}_{\rm log}^\nabla$ and ${\rm A}_{\rm
log}$ which are continuous for the $p$-adic topology. For every
$\sigma\in \mathcal{G}_R$ we have $\sigma(t)=\chi(\sigma) t$ with
$\chi\colon \mathcal{G}_R\to \Z_p^\ast$ the cyclotomic character.
Thus the action of $\mathcal{G}_R$  extends to an action on ${\rm
B}$, ${\rm B}_{\rm log}^\nabla$ and ${\rm B}_{\rm log}$.

\

{\it Filtrations:} Note that the rings ${\rm A}_{\rm
cris}^\nabla(R)$ and ${\rm A}^{\rm cris}_{\rm
log}(\widetilde{R})$, with and without $\nabla$, are endowed with
the divided power filtrations which are decreasing and exhaustive. Similarly,
${\rm A}_{\rm max}(R)$ and ${\rm A}^{\rm max}_{\rm
log}(\widetilde{R})$, with and without $\nabla$, are endowed with
the $p^{-1}\Ker\bigl(\Theta_{\rm log}'\bigr)$-adic filtrations
which are compatible with those on ${\rm A}_{\rm cris}^\nabla(R)$
and ${\rm A}^{\rm cris}_{\rm log}(\widetilde{R})$. Set ${\rm
Fil}^r {\rm B}:= \sum_{n\in \Z} t^n {\rm Fil}^{r-n} {\rm A}$,
${\rm Fil}^r {\rm B}_{\rm log}^\nabla:= \sum_{n\in \Z} t^n {\rm
Fil}^{r-n} {\rm A}_{\rm log}^\nabla$ and ${\rm Fil}^r {\rm B}_{\rm
log}:= \sum_{n\in \Z} t^n {\rm Fil}^{r-n} {\rm A}_{\rm log}$ for
every $r\in\Z$.

\

{\it Frobenii:} Let $\varphi_\cO\colon \cO\to \cO$ be the
Frobenius morphism inducing the usual Frobenius on $\WW(k)$ and
$Z\mapsto Z^p$. Let $\varphi_{\widetilde{R}}\colon
\widetilde{R}\lra  \widetilde{R}$ be the unique morphism which
lifts Frobenius modulo $p$ and is compatible via the chart
$\psi_{\widetilde{R}}\colon \cO[P']\to \widetilde{R}$ with the
morphism $\cO[P']\lra \cO[P']$ which is $\varphi_\cO$ on $\cO$ and
gives multiplication by $p$ on $P'$. Then, $\varphi\otimes
\varphi_{\cO}$ on $\WW\bigl(\widetilde{\bf
E}^+\bigr)\tensor_{\WW(k)} \cO$ extends to Frobenius morphisms
$\varphi$ on ${\rm A}$,  ${\rm A}_{\rm log}^\nabla$ and ${\rm
A}_{\rm log}$. They are compatible with respect to the natural
morphisms between these rings. Since $\varphi(t)=p t$, the Frobenii
extend to compatible morphisms on ${\rm B}_{\rm cris}^{\nabla}$,
${\rm B}_{\rm log}^{\nabla}$ and ${\rm B}_{\rm log}$.

\

{\it Connections:} Using \ref{lemma:structAlog} define the ${\rm A}$-linear connections
$$\nabla_{\widetilde{R}/\WW(k)}\colon {\rm A}_{\rm log}
\lra {\rm A}_{\rm log}\otimes_{\widehat{\widetilde{R}}}
\widehat{\omega}^1_{\widetilde{R}/\WW(k)},$$characterized by the
property that for every $m\in\N$ we have
$$\nabla\bigl((y-1)^{[m]}\bigr)=(y-1)^{[m-1]},\qquad \nabla\bigl((y-1)^{m} p^{-m}\bigr)=(m p^{-1})
(y-1)^{m-1} p^{-(m-1)}$$for $y=u$, $v_1,\ldots,v_a$, or
$w_1,\ldots,w_b$.  One defines similarly
$$\nabla_{\widetilde{R}/\cO}\colon {\rm A}_{\rm log} \lra {\rm
A}_{\rm log}\otimes_{\widehat{\widetilde{R}}} \widehat{\omega}^1_{\widetilde{R}/\cO}$$as the
${\rm A}_{\rm log}^\nabla$-linear connections  characterized by the
formula above for $y=v_2,\ldots,v_a$, or $w_1,\ldots,w_b$.\smallskip

These connections are compatible for the natural morphisms ${\rm A}_{\rm log}^{\rm cris}(\widetilde{R})\lra
{\rm A}_{\rm log}^{\rm max}(\widetilde{R})$.  They
extend to connections on ${\rm B}_{\rm log}$. We will also prove in \ref{cor:Frobeniusishorizontal} that
Frobenius on ${\rm B}_{\rm log}$ is horizontal with respect
to the connections $\nabla_{\widetilde{R}/\WW(k)}$ and $\nabla_{\widetilde{R}/\cO}$.

\begin{corollary}\label{cor:nabla=0} The following hold:\smallskip

(1) the connections $\nabla_{\widetilde{R}/\WW(k)}$ and
$\nabla_{\widetilde{R}/\cO}$  are $\cG_R$-equivariant, they are
integrable and they satisfy Griffiths' transversality with respect
to the given filtrations;\smallskip

(2) the connections $\nabla_{\widetilde{R}/\WW(k)}$ and
$\nabla_{\widetilde{R}/\cO}$ on ${\rm A}^{\rm cris}_{\rm
log}\bigl(\widetilde{R}\bigr)$ are $p$-adically
quasi-nilpotent;\smallskip

(3) the connections $\nabla_{\widetilde{R}/\WW(k)}$ and
$\nabla_{\widetilde{R}/\cO}$ are compatible with the derivation
$d\colon \widetilde{R} \lra
\widehat{\omega}^1_{\widetilde{R}/\WW(k)}$ (resp.~$d\colon
\widetilde{R} \lra
\widehat{\omega}^1_{\widetilde{R}/\cO}$);\smallskip

(4) we have  ${\rm A}_{\rm max}^\nabla(R)={\rm A}^{\rm max}_{\rm
log}(\widetilde{R})^{\nabla_{\widetilde{R}/\WW(k)}=0}$ and ${\rm
A}_{\rm log}^{\rm max,\nabla}(R)={\rm A}^{\rm max}_{\rm
log}(\widetilde{R})^{\nabla_{\widetilde{R}/\cO}=0}$.
\end{corollary}
\begin{proof}
Claims (2) and (4) and the claims that the connections are
integrable and that the filtration satisfies Griffiths'
transversality follows from the construction and
\ref{lemma:structAlog}. The $\cG_R$-equivariance is checked as in
\ref{lemma:nabla=0}. Claim (3) is proven arguing as in the proof
of \ref{cor:conenctioninducestandardconnection}.
\end{proof}

\subsubsection{Relation with ${\rm B}_{\rm
dR}$}\label{sec:reltoBdR} Note that the ideal $\Ker\bigl(\Theta_{\log}\bigr)$ admits
divided powers in $ {\rm B}_{\rm
dR}^{\nabla,+}\bigl(\widetilde{R}\bigr)/{\rm Fil}^n {\rm B}_{\rm
dR}^{\nabla,+}\bigl(\widetilde{R}\bigr)$ for every $n\in\N$ since
$p$ is invertible in the latter. Thus, the map
$\WW\bigl(\widetilde{\bf E}^+\bigr)\tensor_{\WW(k)} \cO \lra  {\rm
B}_{\rm dR}^{\nabla,+}\bigl(\widetilde{R}\bigr)/{\rm Fil}^n {\rm
B}_{\rm dR}^{\nabla,+}\bigl(\widetilde{R}\bigr)$ extends to a map
$$ \bigl(\WW\bigl(\widetilde{\bf E}^+\bigr)\tensor_{\WW(k)} \cO
\bigr)^{\rm logDP} \lra {\rm B}_{\rm
dR}^{\nabla,+}\bigl(\widetilde{R}\bigr)/{\rm Fil}^n {\rm B}_{\rm
dR}^{\nabla,+}\bigl(\widetilde{R}\bigr) .$$This provides with a
morphism ${\rm A}_{\rm log}^{\rm cris,\nabla}(R) \lra {\rm B}_{\rm
dR}^{\nabla,+}\bigl(\widetilde{R}\bigr)$. Similarly we get natural
morphisms $${\rm A}_{\rm log}^{\rm cris,\nabla}(R) \lra {\rm
A}_{\rm log}^{\rm max,\nabla}(R) \lra  {\rm B}_{\rm
dR}^{\nabla,+}\bigl(\widetilde{R}\bigr),\qquad {\rm A}^{\rm
cris}_{\rm log}(\widetilde{R}) \lra {\rm A}^{\rm max}_{\rm
log}(\widetilde{R}) \lra  {\rm B}_{\rm
dR}^{+}\bigl(\widetilde{R}\bigr).$$

\begin{proposition}\label{prop:BcrissubsetBdR} The given morphisms
have the following properties:\smallskip

(1) they are injective. In particular,  ${\rm A}^{\rm cris}_{\rm log}(\widetilde{R})$ and ${\rm A}^{\rm max}_{\rm log}(\widetilde{R})$, with and without $\nabla$,
are $t$-torsion free;\smallskip

(2)  they are compatible with respect to the
connections;\smallskip

(3) they are strictly compatible with respect to the filtrations.
In particular,
$${\rm Gr}^\bullet {\rm A}^{\rm cris}_{\rm log}(\widetilde{R})\cong
\oplus_{\underline{n}\in \N^{d+1}} \widehat{\overline{R}}
\xi^{[n_0]} (u-1)^{[n_1]} (v_2-1)^{[n_2]}\cdots
(v_a-1)^{[n_a]}(w_1-1)^{[n_{a+1}]}\cdots (w_b-1)^{[n_d]}$$and
$${\rm Gr}^\bullet {\rm A}^{\rm max}_{\rm log}(\widetilde{R})\cong
\oplus_{\underline{n}\in \N^{d+1}} \widehat{\overline{R}}
\left(\frac{\xi}{p}\right)^{n_0}
\left(\frac{u-1}{p}\right)^{n_1}
\left(\frac{v_2-1}{p}\right)^{n_2}\cdots
\left(\frac{w_b-1}{p}\right)^{n_d}.$$

(4) the maps ${\rm B}_{\rm log}^{\rm cris,\nabla}(R) \lra {\rm
B}_{\rm log}^{\rm max,\nabla}(R) \lra  {\rm B}_{\rm
dR}^{\nabla}\bigl(\widetilde{R}\bigr)$ and ${\rm B}^{\rm
cris}_{\rm log}(\widetilde{R}) \lra {\rm B}^{\rm max}_{\rm
log}(\widetilde{R}) \lra  {\rm B}_{\rm
dR}\bigl(\widetilde{R}\bigr)$ are injective, compatible with
connections, strictly compatible with the filtrations and $${\rm
Gr}^\bullet {\rm B}^{\rm cris}_{\rm log}(\widetilde{R})\cong {\rm
Gr}^\bullet {\rm B}^{\rm max}_{\rm log}(\widetilde{R}) \cong {\rm
Gr}^\bullet {\rm B}_{\rm dR}\bigl(\widetilde{R}\bigr).$$
\end{proposition}

\begin{proof} The compatibilities with the filtrations and
connections are clear from the construction. If the morphisms are
injective, since ${\rm B}_{\rm dR}^{+}\bigl(\widetilde{R}\bigr)$
is $t$-torsion free by  \ref{cor:BdRstr},  also  ${\rm A}^{\rm
cris}_{\rm log}(\widetilde{R})$ and ${\rm A}^{\rm max}_{\rm
log}(\widetilde{R})$ are $t$-torsion free. Then, also the
morphisms in (4) are injective and compatible with the
connections. They are also compatible with respect to the
filtrations and if (3) holds, they are strictly compatible with
respect to the filtrations and induce isomorphisms on graded rings
by \ref{cor:BdRstr}.

We are left to prove that the given morphism are injective and
that the filtration on ${\rm B}_{\rm dR}^+$ induce the filtrations
on ${\rm A}_{\rm log}^\nabla$ and ${\rm A}_{\rm log}$, using the
conventions of \S\ref{GaloisfiltFrob}. Due to \ref{cor:BdRstr},
\ref{lemma:structureAlog} and \ref{lemma:structAlog} it suffices
to prove that the maps ${\rm A}_{\rm cris}^{\nabla}(R) \lra {\rm
A}_{\rm max}^{\nabla}(R) \lra {\rm B}_{\rm
dR}^{\nabla,+}\bigl(R\bigr)$ are injective and that ${\rm Fil}^r
{\rm A}_{\rm cris}^{\nabla}(R)={\rm A}_{\rm cris}^{\nabla}(R)\cap
{\rm Fil}^r {\rm B}_{\rm dR}^{\nabla,+}\bigl(R\bigr)$ (and
similarly for ${\rm A}_{\rm max}^{\nabla}(R)$). For this we refer
to the proof of \cite[Prop. 6.2.1]{brinon}. The last statement
follows from the strict compatibility of the filtrations, the
explicit description of the filtrations in ${\rm A}^{\rm
cris}_{\rm log}(\widetilde{R})$ and ${\rm A}^{\rm max}_{\rm
log}(\widetilde{R})$ in \ref{lemma:structureAlog} and
\ref{lemma:structAlog} and the description of ${\rm
Gr}^\bullet{\rm B}_{\rm dR}^{\nabla,+}\bigl(R\bigr)$ in
\ref{cor:BdRstr}.
\end{proof}

\subsubsection{Descent from ${\rm B}^{\rm max}_{\rm log}$}
\label{sec:descentBlogmax}

Let $\widehat{\widetilde{R}}$ be the $(p,Z)$-adic completion of
$\widetilde{R}$.

\begin{definition}\label{def:RlogRmax}
Define $\widetilde{R}_{\rm cris}$  as the $p$-adic completion of the logarithmic divided power envelope of $\widehat{\widetilde{R}}$ with respect to the kernel
$\bigl(P_\pi(Z)\bigr)$ of the morphism from $\widehat{\widetilde{R}}$ to the $p$-adic completion of $R$, compatible with the canonical divided power structure on $p
\widehat{\widetilde{R}}$. Put
$$\widetilde{R}_{\rm max}:=\widehat{\widetilde{R}}\left\{\frac{P_\pi(Z)}{p}\right\}$$to
be the $p$-adic completion of the subring
$\widehat{\widetilde{R}}\left[\frac{P_\pi(Z)}{p}\right]$ of\/
$\widehat{\widetilde{R}}\bigl[p^{-1}\bigr]$.
\end{definition}

Consider the inclusion $\widetilde{R}_{\rm max}\bigl[p^{-1}\bigr] \subset {\rm B}^{\rm max}_{\rm log}\bigl(\widetilde{R}\bigr)$. We have the following fundamental
result:

\begin{theorem}\label{thm:Blogmaxff} (1) If a sequence of $\widetilde{R}_{\rm max}\bigl[p^{-1}\bigr]$-modules is exact after
base change to ${\rm B}^{\rm max}_{\rm log}\bigl(\widetilde{R}\bigr)$, then it is exact.

If an $\widetilde{R}_{\rm max}\bigl[p^{-1}\bigr]$-module becomes finite and projective as ${\rm B}^{\rm max}_{\rm log}\bigl(\widetilde{R}\bigr)$-module after base
change to ${\rm B}^{\rm max}_{\rm log}\bigl(\widetilde{R}\bigr)$, then it is finite and projective as $\widetilde{R}_{\rm max}\bigl[p^{-1}\bigr]$-module.

\smallskip

(2) If $\alpha=1$ then $\widetilde{R}_{\rm max}\bigl[p^{-1}\bigr]\subset {\rm B}^{\rm max}_{\rm log}\bigl(\widetilde{R}\bigr)$
is a faithfully flat extension.
\end{theorem}

Write ${\bf A}_{\widetilde{R},{\rm max}}^{+,\rm log,\nabla}$ (resp.~${\bf A}_{\widetilde{R}^o,{\rm max}}^{+,\rm log,\nabla}$, resp.~$\WW\bigl(\widetilde{\bf
E}^+\bigr)^{\rm max}_{\rm log}$) for the $p$-adic completion of the subring ${\bf A}_{\widetilde{R}}^+\left[\frac{P_\pi\bigl([\overline{\pi}]\bigr)}{p}\right]$
(resp.~of ${\bf A}_{\widetilde{R}^o}^+\left[\frac{P_\pi\bigl([\overline{\pi}]\bigr)}{p}\right]$, resp.~of $\WW\bigl(\widetilde{\bf
E}^+\bigr)\left[\frac{P_\pi\bigl([\overline{\pi}]\bigr)}{p}\right]$) of\/ $\WW\bigl(\widetilde{\bf E}^+\bigr)\bigl[p^{-1}\bigr]$. Then, ${\bf A}_{\widetilde{R},{\rm
max}}^{+,\rm log,\nabla}$ is isomorphic to $\widetilde{R}_{\rm max}$ by \ref{def_AR+}. It follows from \ref{lemma:KerTheta} that $\WW\bigl(\widetilde{\bf
E}^+\bigr)^{\rm max,\nabla}_{\rm log}\cong {\rm A}_{\rm max}^\nabla(R)$.

Consider the morphism of  rings with log structures $\theta\colon {\bf A}_{\widetilde{R}}^+\otimes_{\WW(k)} \widetilde{R}\lra \widehat{\overline{R}}$ induced by
$\Theta_{\widetilde{R},\log}$. Let $\bigl({\bf A}_{\widetilde{R}}^+\otimes_{\WW(k)} \widetilde{R}\bigr)^{\rm log}:={\bf A}_{\widetilde{R}}^+\tensor_{\Z[P'\times
P']} \Z\bigl[Q\bigr]$ and let $\theta_{\rm log}$ be the extension of $\theta$ to $\bigl({\bf A}_{\widetilde{R}}^+\otimes_{\WW(k)} \widetilde{R}\bigr)^{\rm log}$. We
write ${\bf A}_{\widetilde{R},{\rm max}}^{+,\rm log}$ for the $p$-adic completion of $\bigl({\bf A}_{\widetilde{R}}^+\otimes_{\WW(k)} \widetilde{R}\bigr)^{\rm
log}\left[p^{-1} \Ker(\theta^{\rm log})\right]$. We define ${\bf A}_{\widetilde{R}^o,{\rm max}}^{+,\rm log}$ similarly using ${\bf A}_{\widetilde{R}^o}^+$ instead
of ${\bf A}_{\widetilde{R}}^+$. We start with the following:

\begin{lemma}\label{lemma:AtildeRlogff}
(1) The extension ${\bf A}_{\widetilde{R}^o,{\rm max}}^{+,\rm log,\nabla} \lra {\rm A}_{\rm max}^\nabla(R)$ is ${\cal I}^{27}$-flat.\smallskip

(2) We have an isomorphism $$\widetilde{R}_{\rm max}\left\{\frac{u-1}{p}, \frac{v_2-1}{p},\ldots,\frac{v_a-1}{p},\frac{w_1-1}{p},\ldots,\frac{w_b-1}{p}\right\}\cong
{\bf A}_{\widetilde{R},{\rm max}}^{+,\rm log}$$of  $\widetilde{R}_{\rm max}$-algebras. They are faithfully flat as $\widetilde{R}_{\rm max}$-algebras.\smallskip

(3) ${\bf A}_{\widetilde{R},{\rm max}}^{+,\rm log}$ is a direct summand in ${\bf A}_{\widetilde{R}^o,{\rm max}}^{+,\rm log}$ as ${\bf A}_{\widetilde{R},{\rm
max}}^{+,\rm log}$-module and ${\bf A}_{\widetilde{R}^o,{\rm max}}^{+,\rm log}$ is a $Z^\alpha$-flat ${\bf A}_{\widetilde{R},{\rm max}}^{+,\rm
log}$-module.\smallskip

(4) The extension ${\bf A}_{\widetilde{R}^o,{\rm max}}^{+,\rm log} \lra {\rm A}^{\rm max}_{\rm log}(\widetilde{R}) $ is ${\cal I}^{81}$-flat. Thus the extension
${\bf A}_{\widetilde{R}^o,{\rm max}}^{+,\rm log} \lra {\rm B}^{\rm max}_{\rm log}\bigl(\widetilde{R}\bigr)$ is flat.

\smallskip

\noindent In particular the extension $\widetilde{R}_{\rm max}\bigl[(pZ)^{-1}\bigr]\subset {\rm B}^{\rm max}_{\rm log}\bigl(\widetilde{R}\bigr)\bigl[Z^{-1}\big]$ is flat.

\end{lemma}
\begin{proof}
(1) Since the extension ${\bf A}_{\widetilde{R}^o}^+\left[\frac{P_\pi\bigl([\overline{\pi}]\bigr)}{p}\right] \lra \WW\bigl(\widetilde{\bf
E}^+\bigr)\left[\frac{P_\pi\bigl([\overline{\pi}]\bigr)}{p}\right]$ is obtained from ${\bf A}_{\widetilde{R}^o}^+ \to\WW\bigl(\widetilde{\bf E}^+\bigr)$ by base
change via the extension ${\bf A}_{\widetilde{R}^o}^+\to {\bf A}_{\widetilde{R}^o}^+\left[\frac{P_\pi\bigl([\overline{\pi}]\bigr)}{p}\right]$, it is  ${\cal
I}^9$-flat due to \ref{prop:A+Rtildeff}. The extension obtained taking $p$-adic completions is the extension ${\bf A}_{\widetilde{R}^o,{\rm max}}^{+,\rm log,\nabla}
\lra \WW\bigl(\widetilde{\bf E}^+\bigr)^{\rm max}_{\rm log}$. Since ${\bf A}_{\widetilde{R}^o}^+\left[\frac{P_\pi\bigl([\overline{\pi}]\bigr)}{p}\right]$ is
noetherian and $p$-torsion free, the extension of the lemma is ${\cal I}^{27}$-flat by \cite[Thm 9.2.6]{brinon}.

(2)--(3)  Recall that ${\bf A}_{\widetilde{R},{\rm max}}^{+,\rm log}$ is the $p$-adic completion of $\bigl({\bf A}_{\widetilde{R}}^+\otimes_{\WW(k)}
\widetilde{R}\bigr)^{\rm log}\left[p^{-1} \Ker(\theta^{\rm log})\right]$ (resp.~${\bf A}_{\widetilde{R}^o,{\rm max}}^{+,\rm log}$ for ${\bf A}_{\widetilde{R}^o}^+$
instead of ${\bf A}_{\widetilde{R}}^+$). In both cases $\Ker(\theta_{\rm log})=\bigl(P_\pi([\overline{\pi}])\otimes
1,u-1,v_2-1,\ldots,v_a-1,w_1-1,\ldots,w_b-1\bigr)$; this ideal coincides also with $\bigl(1\otimes P_\pi(Z),u-1,v_2-1,\ldots,v_a-1,w_1-1,\ldots,w_b-1\bigr)$. It
follows as in \ref{lemma:structAlog}  that ${\bf A}_{\widetilde{R},{\rm max}}^{+,\rm log}$ is isomorphic to
$${\bf A}_{\widetilde{R},{\rm max}}^{+,\rm log}\cong\widetilde{R}_{\rm max}\left\{\frac{u-1}{p},
\frac{v_2-1}{p},\ldots,\frac{v_a-1}{p},\frac{w_1-1}{p},\ldots,\frac{w_b-1}{p}\right\}\cong $$ $$ \cong {\bf A}_{\widetilde{R},{\rm max}}^{+,\rm
log,\nabla}\left\{\frac{u-1}{p}, \frac{v_2-1}{p},\ldots,\frac{v_a-1}{p},\frac{w_1-1}{p},\ldots,\frac{w_b-1}{p}\right\}$$and, in particular, it is a faithfully flat
$\widetilde{R}_{\rm max}$-algebra. This proves (2). Similarly, we have $${\bf A}_{\widetilde{R}^o,{\rm max}}^{+,\rm log}\cong {\bf A}_{\widetilde{R}^o,{\rm
max}}^{+,\rm log,\nabla}\left\{\frac{u-1}{p}, \frac{v_2-1}{p},\ldots,\frac{v_a-1}{p},\frac{w_1-1}{p},\ldots,\frac{w_b-1}{p}\right\}.$$Note that ${\bf
A}_{\widetilde{R},{\rm max}}^{+,\rm log,\nabla}$ is a direct summand in ${\bf A}_{\widetilde{R}^o,{\rm max}}^{+,\rm log,\nabla}$ and the latter is a
$[\overline{\pi}]^\alpha$-flat ${\bf A}_{\widetilde{R},{\rm max}}^{+,\rm log,\nabla}$ thanks to \ref{prop:A+Rtildeff}. As $[\overline{\pi}]= Z u$ and $u$ is
invertible, Claim (3) follows.

(4)  As in (1) we deduce that $${\bf A}_{\widetilde{R}^o,{\rm max}}^{+,\rm log} \lra \WW\bigl(\widetilde{\bf E}^+\bigr)^{\rm max}_{\rm log}\left\{\frac{u-1}{p},
\frac{v_2-1}{p},\ldots,\frac{v_a-1}{p},\frac{w_1-1}{p},\ldots,\frac{w_b-1}{p}\right\}$$is $\bigl({\cal I}^{27}\bigr)^3$-flat. The latter is isomorphic to ${\rm
A}^{\rm max}_{\rm log}(\widetilde{R})$ due to \ref{lemma:structAlog}. Since ${\cal I} {\rm B}^{\rm max}_{\rm log}(\widetilde{R})={\rm B}^{\rm max}_{\rm
log}(\widetilde{R})$ cf.~\cite[Pf. Thm 6.3.8]{brinon}, the last claim follows.

\end{proof}

In order to prove Theorem \ref{thm:Blogmaxff} we show:

\begin{lemma}\label{lemma:imageofg}  The image of the map $g\colon  {\rm Spec}\left({\rm B}^{\rm max}_{\rm log}\bigl(\widetilde{R}\bigr)\right) \lra {\rm Spec}
\left(\widetilde{R}_{\rm max}\bigl[p^{-1}\bigr]\right)$ contains all maximal ideals not containing $Z$.
\end{lemma}
\begin{proof}
We first prove that the image of $g$  contains all prime ideals containing $P_\pi(Z)$. Consider the commutative diagram
$$\begin{array}{ccc}
\widetilde{R}_{\rm max}\bigl[p^{-1}\bigr]  & \lra & {\rm B}^{\rm max}_{\rm log}\bigl(\widetilde{R}\bigr) \cr \big\downarrow & & \big\downarrow \cr
\widehat{\widetilde{R}\bigl[p^{-1}\bigr]} & \lra &  {\rm B}_{\rm dR}\bigl(\widetilde{R}\bigr).
\end{array}$$
Recall that $\widehat{\widetilde{R}\bigl[p^{-1}\bigr]}$ is the $P_\pi(Z)$-adic completion of $\widetilde{R}_{\rm max}\bigl[p^{-1}\bigr]$. Since the latter is
noetherian, the set $\Spec\bigl(\widehat{\widetilde{R}\bigl[p^{-1}\bigr]} \bigr)$ is identified with the set of prime ideals of $\widetilde{R}_{\rm
max}\bigl[p^{-1}\bigr]$ containing $P_\pi(Z)$. Due to \ref{prop:BdRff} the last row is a faithfully flat extension and, in particular, the induced map on spectra is
surjective. We conclude that the image of $g$ contains all prime ideals of $\widetilde{R}_{\rm max}\bigl[p^{-1}\bigr]$ containing $P_\pi(Z)$.

\medskip

The maximal ideals of $\widetilde{R}_{\rm max}\bigl[p^{-1}\bigr]$ are defined by the $L$-valued points $h\colon \widetilde{R}_{\rm max}\bigl[p^{-1}\bigr]\lra L$ for
$L$ varying among the finite extensions of $K$.  Fix one and let us call it $h$.  We characterize  the images under the Frobenius morphism $\varphi\colon
\widetilde{R}_{\rm max}\bigl[p^{-1}\bigr]\lra \widetilde{R}_{\rm max}\bigl[p^{-1}\bigr]$ of the maximal ideals containing $P_\pi(Z)$. As $g$ is compatible with the
Frobenius morphism $\varphi\colon {\rm B}^{\rm max}_{\rm log}\bigl(\widetilde{R}\bigr) \lra {\rm B}^{\rm max}_{\rm log}\bigl(\widetilde{R}\bigr)$, we conclude from
the argument above that they also  lie in the image $g$. Assume that $P_\pi(Z)\in \Ker h$. Then, $h(Z)=\pi'$ for some root $\pi'$ of $P_\pi(Z)$. The Frobenius
morphism $\varphi$ on $\widetilde{R}_{\rm max}\bigl[p^{-1}\bigr]$ maps $P_\pi(Z)$ to $P_\pi^\sigma\bigl(Z^p\bigr)$ where, if $P_\pi(Z)=Z^e + \sum_i a_i Z^i\in
\WW(k)[Z]$, then $P_\pi^\sigma(Z)=Z^e + \sum_i \sigma(a_i) Z^i$ is the polynomial with coefficients twisted by Frobenius $\sigma$ on $\WW(k)$. Thus $h \circ
\varphi^n$ sends $P_\pi(Z)$ to $P_\pi^{\sigma^n}\bigl(\pi^{'p^n}\bigr)$ for every $n\in\N$. More generally take a  maximal ideal  of $\widetilde{R}_{\rm
max}\bigl[p^{-1}\bigr]$ corresponding to a homomorphism $f$ to $\widehat{\Kbar}$ sending $Z$ to $\pi^{'p^n}$ for some $n\in \Z$. As Frobenius $\varphi\colon
\widetilde{R}\to \widetilde{R}$ is finite and flat by construction, Frobenius induces a surjective morphism $\Spec(\widetilde{R})\to \Spec(\widetilde{R})$. Thus $f$
is obtained by pre-composing an homomorphism $h\colon \widetilde{R}\bigl[p^{-1}\bigr]\lra \Kbar$, sending $Z$ to $\pi'$,  with $\varphi^n$. Note that $h$ extends to
$\widetilde{R}_{\rm max}\bigl[p^{-1}\bigr]$. We conclude that $f$ is in the image of $g$ as $h$ is and Frobenius on $\widetilde{R}$ is the restriction of Frobenius
on  ${\rm B}^{\rm max}_{\rm log}\bigl(\widetilde{R}\bigr)$.

\medskip

We are left to consider homomorphisms  $h\colon\widetilde{R}_{\rm max}\bigl[p^{-1}\bigr]\lra L$ which {\it do not}  send $Z$ to $\pi^{'p^n}$ for some root $\pi'$ of
$P_\pi(Z)$. Let $\varrho$ be $h\bigl(P_\pi(Z)/p\bigr)$. It is non-zero and, since $\widetilde{R}_{\rm max}$ is $p$-adically complete, $h$ induces a map
$h\colon\widetilde{R}_{\rm max}\lra \cO_L$.  Consider the map
$$s\colon {\rm A}^{\rm max}_{\rm log}\bigl(\widetilde{R}\bigr)\lra
{\rm A}_{\rm max}^\nabla(R)$$sending $u-1$, $v_2-1,\ldots,v_a-1$ and $w_1-1,\ldots,w_b-1$ to $0$; see \ref{lemma:structAlog} for the notation. Recall from
\ref{def_AR+} that the $\bigl(p,P_\pi(Z)\bigr)$-adic completion of $\widetilde{R}$ is identified with the subring ${\bf A}_{\widetilde{R}}^+\subset {\rm A}_{\rm
max}^\nabla(R)$. In particular, $h$ defines a morphism $\tilde{h}\colon {\bf A}_{\widetilde{R}}^+\lra \cO_L$ and  ${\rm A}_{\rm max}^\nabla(R)$ is endowed with a
structure of $\widetilde{R}$-algebra via these identifications, which is the same as the $\widetilde{R}$-algebra structure induced by $s$ composed with the
structural morphism of ${\rm A}^{\rm max}_{\rm log}\bigl(\widetilde{R}\bigr)$ as $\widetilde{R}$-algebra. To prove that $h$ is in the image of  ${\rm
Spec}\left({\rm B}^{\rm max}_{\rm log}\bigl(\widetilde{R}\bigr)\right)$ it suffices to prove that there exists a morphism
$$r\colon {\rm A}_{\rm max}^\nabla(R) \lra \widehat{\cO}_{\Kbar}$$extending $\widetilde{h}$ and such that the image of $t$ is non-zero. Due to
\ref{lemma:structureAlog} we have ${\rm A}_{\rm max}^\nabla(R)\cong \WW\bigl(\widetilde{\bf E}^+\bigr)\left\{\frac{P_\pi\bigl([\overline{\pi}]\bigr)}{p} \right\}$.
It follows from \cite[\S5.2.4\&\S5.2.8(ii)]{Fontaineperiodes} that $t=v_0 \bigl([\varepsilon]-1\bigr)$ with $v_0$ a unit of Fontaine's $A_{\rm cris}$ so that
$r(t)\neq 0$ if and only if $r\bigl([\varepsilon]-1\bigr)\neq 0$.  Note also that $\widetilde{h}\bigl(P_\pi\bigl([\overline{\pi}]\bigr)/p\bigr)=\varrho\in \cO_L$ is
already determined. It then suffices to prove that there exists a morphism
$$q\colon \WW\bigl(\widetilde{\bf E}^+\bigr)\lra \widehat{\cO}_{\Kbar}$$(I) extending $\widetilde{h}$ and such that \enspace (II) $q\bigl([\varepsilon]-1\bigr)$ is non
zero.

We start with (I). It follows from \ref{lemma:A+R} that $\WW\bigl(\widetilde{\bf E}^+_{R_\infty}\bigr)$ is the
$\bigl(p,P_\pi\bigl([\overline{\pi}]\bigr)\bigr)$-adic completion of the ${\bf A}_{\widetilde{R}}^+$-algebra obtained by adjoining all roots of
$\bigl[\overline{\pi}\bigr]$, $\bigl[\overline{X}_i\bigr]$ for $i=1,\ldots,a$ and of $\bigl[\overline{Y}_j\bigr]$ for $j=1,\ldots,b$. We deduce that, once chosen
compatible roots of $\widetilde{h}\bigl(\bigl[\overline{\pi}\bigr]\bigr)$, $\widetilde{h}\bigl(\bigl[\overline{X}_i\bigr]\bigr)$ for $i=1,\ldots,a$ and of
$\widetilde{h}\bigl(\bigl[\overline{Y}_j\bigr]\bigr)$ for $j=1,\ldots,b$, the morphism $\widetilde{h}$ can be extended to a morphism $\widetilde{h}_\infty \colon
\WW\bigl(\widetilde{\bf E}^+_{R_\infty}\bigr)\lra \widehat{\cO}_{\Kbar}$.  By assumption $h(Z)\neq 0$ so that $\widetilde{h}\bigl([\overline{\pi}]\bigr)\neq 0$.
Since the image of $\widetilde{h}_\infty$ contains all $p$-th power roots of $\widetilde{h}\bigl([\overline{\pi}]\bigr)$, it contains elements of
$\widehat{\cO}_{\Kbar}$ of arbitrarily small valuation. Note that $\WW\bigl(\widetilde{\bf E}^+\bigr) $ is the $\bigl(p,P_\pi(Z)\bigr)$-completion of the union of
all extensions $\WW\bigl(\widetilde{\bf E}^+_{R_\infty}\bigr)\subset \WW\bigl(\widetilde{\bf E}^+_{S_\infty}\bigr)$ for $R_\infty\subset S_\infty (\subset \Omega)$
normal and union of finite and \'etale extensions of $R_\infty[p^{-1}]$ after inverting $p$. Since $\widehat{\cO}_{\Kbar}$ is $p$-adically complete and separated,
to achieve (I) it suffices to prove that $\widetilde{h}_\infty$ extends to compatible morphisms $\widetilde{h}_{S_\infty}$ on $\WW\bigl(\widetilde{\bf
E}^+_{S_\infty}\bigr)$. Using Zorn's lemma we are left to show that, given extensions $S_\infty\to T_\infty$ as above which are finite and \'etale after inverting
$p$ and a map $\widetilde{h}_{S_\infty}$ extending $\widetilde{h}_\infty$, the morphism $\widetilde{h}_{S_\infty}$ can be extended to a morphism
$\widetilde{h}_{T_\infty}$. Write $\mathcal{A}$ for the base change
$$\iota\colon \widehat{\cO}_{\Kbar} \lra
\mathcal{A}:=\WW\bigl(\widetilde{\bf E}^+_{T_\infty}\bigr)\otimes_{\WW\bigl(\widetilde{\bf E}^+_{S_\infty}\bigr)}^{\widetilde{h}_{S_\infty}}
\widehat{\cO}_{\Kbar}.$$

The existence of the ring homomorphism $\widetilde{h}_{T_\infty}\colon \WW\bigl({\bf E}^+_{T_\infty}\bigr) \lra \widehat{\cO}_{\Kbar}$ extending $h_{S_\infty}$ is
implied by the existence of a ring homomorphism $s\colon \mathcal{A}\lra \widehat{\cO}_{\Kbar}$ which is a section to $\iota$. Indeed if $s$ exists, we define
$\widetilde{h}_{T_\infty}$ as the composition $\displaystyle \WW\bigl({\bf E}^+_{T_\infty}\bigr) \stackrel{a}{\lra} \mathcal{A}\stackrel{s}{\lra}
\widehat{\cO}_{\Kbar}$, where $a$ is defined by $a(x)=x\otimes 1$.

We have the following properties of the $\widehat{\cO}_{\Kbar}$-algebra $\mathcal{A}$. Let us denote by $\mathcal{A}_{\rm tors}$ the ideal of $\mathcal{A}$ of
torsion elements and by $\mathcal{A}_0:=\mathcal{A}/\mathcal{A}_{\rm tors}$. The $\widehat{\cO}_{\Kbar}$-algebra $\mathcal{A}_0$ defined above  is flat since it is
torsion free. Let $\widehat{\mathcal{A}}:=\ds \lim_{\infty\leftarrow n} \mathcal{A}/p^n\mathcal{A}$ and similarly for $\widehat{\mathcal{A}}_0$.

\medskip \noindent 1) $\mathfrak{m}_{\Kbar}\mathcal{A}_{\rm tors}=0$.\smallskip

Due to \ref{prop:AE} the extension $\WW\bigl(\widetilde{\bf E}^+_{R_\infty}\bigr)\subset \WW\bigl(\widetilde{\bf E}^+_{S_\infty}\bigr)$ is almost \'etale so that
$\iota$ is almost \'etale and, in particular, $\mathfrak{m}_{\Kbar}$-flat. Here $\mathfrak{m}_{\Kbar}$ is the maximal ideal of $\widehat{\cO}_{\Kbar}$. In
particular, base changing to $\mathcal{A}$ the exact sequence $$0 \to \widehat{\cO}_{\Kbar} \stackrel{\cdot p^n}{\longrightarrow} \widehat{\cO}_{\Kbar} \to
\cO_{\Kbar}/p^n \cO_{\Kbar}  \to 0,$$we get that the kernel $\mathcal{A}[p^n]$ of multiplication by $p^n$ on $\mathcal{A}$ is annihilated by $\mathfrak{m}_{\Kbar}$
for every $n$ i.e., $\mathcal{A}_{\rm tors}=\cup \mathcal{A}[p^n]$ is annihilated by $\mathfrak{m}_{\Kbar}$.

\medskip \noindent 2) The $\widehat{\cO}_{\Kbar}$-algebra  $\widehat{\mathcal{A}}_0$ is torsion free. \smallskip

For every $n\in \N$ the kernel of multiplication by $p$ on $\mathcal{A}_0/ p^n\mathcal{A}_0$ is $p^{n-1}\mathcal{A}_0/p^n \mathcal{A}_0$ so that the kernel of
multiplication by $p$ on $\widehat{\mathcal{A}}_0$ is $\ds\lim_{\infty\leftarrow n}  p^{n-1}\mathcal{A}_0/p^n \mathcal{A}_0$ which is $0$.

\medskip \noindent 3) $\widehat{\mathcal{A}}_0$ is non-zero. In particular, $\widehat{\mathcal{A}}_0[1/p]\neq 0$ by (2).\smallskip

To prove this we describe the map induced by $\iota$ by taking quotients $\cO_{\Kbar}/p\varrho \cO_{\Kbar}\to \mathcal{A}/p\varrho \mathcal{A}$ as follows. The
quotient $\WW\bigl(\widetilde{\bf E}^+_{S_\infty}\bigr)\otimes_{\WW(k)} \cO_L$ modulo $\bigl(P_\pi\bigl([\overline{\pi}]\bigr)\otimes 1,1\otimes p \varrho\bigr)$
coincides by \ref{lemma:KerTheta} with $S_\infty \otimes_{\WW(k)} \cO_L /p\varrho \cO_L$ and similarly for $\WW\bigl(\widetilde{\bf
E}^+_{T_\infty}\bigr)\otimes_{\WW(k)} \cO_L$. Then, the map $\overline{h}_{S_\infty}:=\widetilde{h}_{S_\infty}$ modulo $p\varrho $ factors via $S_\infty
\otimes_{\WW(k)} \cO_L /p\varrho \cO_L$ and $\iota$ modulo $p\varrho$ is the base change via $\overline{h}_{S_\infty}$ of the extension
$$r\colon S_\infty \otimes_{\WW(k)} \cO_L /p\varrho \cO_L\lra T_\infty \otimes_{\WW(k)} \cO_L /p\varrho \cO_L.$$Since $T_\infty$
is the normalization of $S_\infty$ in a finite and \'etale extension of $S_\infty[p^{-1}]$, we conclude that the map induced by $r$ on spectra   is surjective on
generic points and, being an inductive limit of finite and finitely presented $S_\infty$-algebras,  it has closed image. Hence, it is surjective. In particular,
there exist prime ideals of $T_\infty \otimes_{\WW(k)} \cO_L /p\varrho \cO_L$ over the  prime ideal of $S_\infty \otimes_{\WW(k)} \cO_L /p\varrho\cO_L$ defined by
the kernel of $S_\infty \otimes_{\WW(k)} \cO_L /p\varrho \cO_L \to \cO_{\Kbar}/\mathfrak{m}_{\Kbar} \cO_{\Kbar} $ induced by $\overline{h}_{S_\infty}$. The set of
such ideals is $\Spec\bigl(\mathcal{A}/\mathfrak{m}_{\Kbar} \mathcal{A}\bigr)$. We conclude that $\mathcal{A}/\mathfrak{m}_{\Kbar} \mathcal{A}$ is non trivial. Due
to Faltings' almost purity theorem, see \ref{prop:AE}, the extension $S_\infty\subset T_\infty$ is almost \'etale so that the trace map ${\rm Tr}\colon T_\infty \to
S_\infty$ has $\mathfrak{m}_{\Kbar} S_\infty$ in its image. Its base change via $\overline{h}_{S_\infty}$ provides a map $\psi\colon \mathcal{A}/p\varrho
\mathcal{A} \lra \cO_{\Kbar}/p\varrho \cO_{\Kbar}$ of $\cO_{\Kbar}$-modules having $\mathfrak{m}_{\Kbar}$ in its image. Since any element of $\mathcal{A}_{\rm
tors}$ has image via $\psi$ annihilated by $\mathfrak{m}_{\Kbar}$ by 1) and since the only such element in $\cO_{\Kbar}/p\varrho \cO_{\Kbar}$ is $0$, we conclude
that $\psi(\mathcal{A}_{\rm tors})=0$. We conclude that $\mathcal{A}_{\rm tors}\subset \mathcal{A}/p\varrho \mathcal{A}$ is not surjective, i.e., the quotient which
is $\mathcal{A}_0/p\varrho \mathcal{A}_0$ is non trivial. In particular $p$ is not a unit in $\mathcal{A}_0$. Therefore for all $n\ge 0$ the ring
$\mathcal{A}_0/p^n\mathcal{A}_0$ is non-zero which implies that $\widehat{\mathcal{A}}_0$ is non-zero.

\medskip \noindent 4) $\widehat{\mathcal{A}}_0[1/p]$ is a finite dimensional $\widehat{\Kbar}$-vector space and coincides
with $\mathcal{A}[1/p]$.\smallskip

Since $S_\infty \subset T_\infty$ is almost \'etale,  $\pi^{\frac{1}{p}} T_\infty$ is finitely generated as $S_\infty$-module by \ref{cor:Frobonto}. Hence, there
exist $e_1,\ldots, e_n$ in $\mathcal{A}$ such that if $\mathcal{B}$ is the $\widehat{\cO}_{\Kbar}$-submodule of $\widehat{\mathcal{A}}$ generated by $e_1,\ldots,
e_n$ we have $\pi^{\frac{1}{p}} \widehat{\mathcal{A}}\subset \mathcal{B}+ p \widehat{\mathcal{A}}$. As $\mathcal{B}$ is a finitely  generated
$\widehat{\cO}_{\Kbar}$-module, it is $p$-adically complete. We claim this implies that we have:
$$\ds \pi^{\frac{1}{p}}\widehat{\mathcal{A}}\subset \mathcal{B}\subset \widehat{\mathcal{A}}.$$
Indeed, let us denote by $\ds p^\upsilon:=\pi^{\frac{1}{p}}$ with $0< \upsilon< 1$ and let $x\in \widehat{\mathcal{A}}$. Then $p^\upsilon x=b_0+px_1$, with $b_0\in
\mathcal{B}$ and $x_1\in \widehat{\mathcal{A}}$. Then $p^\upsilon x=b_0+p^{1-\upsilon}(b_1+px_2)$, with $b_1\in \mathcal{B}, x_2\in \widehat{\mathcal{A}}$.
Iterating this process and using the completeness of $\mathcal{B}$ we obtain that
$$
p^\upsilon x=b_0+p^{1-\upsilon}b_1+p^{2(1-\upsilon)}b_2+\ldots\in \mathcal{B}.
$$

Since multiplication by $p^n$ annihilates $\mathcal{A}_{\rm tors}$ and has trivial kernel on $\mathcal{A}_0$, we have for every $n$ that the map $\mathcal{A}_{\rm
tors}\to \mathcal{A}/p^n \mathcal{A}$ is injective with quotient $\mathcal{A}_0/p^n \mathcal{A}_0$. Taking projective limits we get the exact sequence $0 \to
\mathcal{A}_{\rm tors} \to \widehat{\mathcal{A}} \to \widehat{\mathcal{A}}_0\to 0$. Therefore
$\widehat{\mathcal{A}}[1/p]=\widehat{\mathcal{A}}_0[1/p]=\mathcal{B}[1/p]$, which proves the claim.

\medskip \noindent 5) There is a section $s\colon \mathcal{A}\lra \widehat{\cO}_{\Kbar}$ of $\iota$.

We have that $\widehat{\mathcal{A}}[1/p]$ is a finite \'etale $\widehat{\Kbar}$-algebra by (4). Therefore $\widehat{\mathcal{A}}[1/p]$ is a finite product of copies
of $\widehat{\Kbar}$ as $\widehat{\Kbar}$ is an algebraically closed field. Therefore there exists a section $s_{K}\colon\widehat{\mathcal{A}}[1/p]\lra
\widehat{\Kbar}$ to the structure morphism $\iota_{K}\colon \widehat{\Kbar}\lra \widehat{\mathcal{A}}[1/p]$.

As $\widehat{\mathcal{A}}_0\subset \widehat{\mathcal{A}}[1/p]$ is  $p$-adically complete and separated, we have $s_{K}(\widehat{\mathcal{A}}_0)\subset
\widehat{\cO}_{\Kbar}$. Denote by $s$ the following composition
$$
\mathcal{A}\lra \mathcal{A}_0\lra \widehat{\mathcal{A}}_0\lra \widehat{\mathcal{A}}_0[1/p] \stackrel{s_{K}}{\lra} \widehat{\Kbar}.
$$
It is clearly a section of $\iota$ as required.

\medskip

We now prove (II). First of all we consider the particular case that $h$ sends $Z$ to a root of $P^{\sigma^m}_\pi(Z)$, for some $m\in\Z$, assuming that
$P_\pi^{\sigma^m}(Z)\neq P_\pi(Z)$. Then, $P^{\sigma^m}_\pi(Z)$ is an Eisenstein polynomial so that $\cO/\bigl(P^{\sigma^m}_\pi(Z)\bigr)=\cO_{L}$ is a discrete
valuation ring with uniformizer $\pi'$, image of $Z$. Identifying $\WW\bigl(\widetilde{\bf E}^+_{\cO_{K_\infty'}}\bigr)$ with the $(p,
P^{\sigma^m}_\pi(Z))$-completion of $\cO\bigl[Z^{\frac{1}{n!}}\bigr]_{n\in\N}$ sending $Z$ to $[\overline{\pi}]$, we get that $\WW\bigl(\widetilde{\bf
E}^+_{\cO_{K_\infty'}}\bigr)/\bigl(P^{\sigma^m}_\pi([\overline{\pi}])\bigr)\cong \widehat{\cO}_{L_\infty}\subset \widehat{\cO}_\Kbar$ with $\cO_{L_\infty}\subset
\cO_\Kbar$ the direct limit of the discrete valuation rings $\cO_{L}\bigl[\pi^{'\frac{1}{n!}}\bigr]$ for ${n\in\N}$. Since $\bigl([\overline{\pi}]^e,p\bigr)$ is a
regular sequence in $\WW\bigl(\widetilde{\bf E}^+_{\cO_\Kbar}\bigr)$ by \ref{lemma:A+R}, we deduce that $A:=\WW\bigl(\widetilde{\bf
E}^+_{\cO_\Kbar}\bigr)/\bigl(P^{\sigma^m}_\pi([\overline{\pi}])\bigr)$ is $p$-torsion free. By construction it is the $p$-adic completion of almost \'etale
extensions of $\widehat{\cO}_{L_\infty}$. Hence, we can extend the inclusion $\cO_{L_\infty}\subset \cO_\Kbar$ to an injection $\eta\colon A \to
\widehat{\cO}_{\Kbar}$. As $P_\pi^{\sigma^m}(Z)$ and $P_\pi(Z)$ are monic Eisenstein polynomials of degree $e$, we have $P_\pi^{\sigma^m}(Z)\equiv P_\pi(Z)$ modulo
$p$ and $A/pA=\WW\bigl(\widetilde{\bf E}^+_{\cO_\Kbar}\bigr)/\bigl(P_\pi([\overline{\pi}]),p \bigr)=\cO_\Kbar/p \cO_\Kbar$ by \ref{lemma:KerTheta}. Hence $\eta$ is
an isomorphism and $\eta$ modulo $p$ factors via the canonical map $\Theta\colon \WW\bigl(\widetilde{\bf E}^+_{\cO_\Kbar}\bigr) \to \widehat{\cO}_\Kbar$. In
particular, $\eta\bigl([\varepsilon]^{\frac{1}{p}}\bigr)\neq 1$ as $\Theta\bigl([\varepsilon]^{\frac{1}{p}}\bigr)=\epsilon_p \not\equiv 1$ modulo $p$. Recall from
\ref{lemma:KerTheta} that $1+[\varepsilon]^{\frac{1}{p}}+\cdots + [\varepsilon]^{\frac{p-1}{p}}$ is $P_\pi\bigl([\overline{\pi}]\bigr)$ up to unit since they both
generate the kernel of $\Theta$. Thus $\eta\bigl([\varepsilon]^{\frac{1}{p}}\bigr)$ is not a primitive $p$-th root of $1$ as else
$\eta\bigl(P_\pi\bigl([\overline{\pi}]\bigr)\bigr)=0$ but we assumed that $P_\pi^{\sigma^m}(Z)$ and $ P_\pi(Z)$ are coprime. We conclude that
$\eta\bigl([\varepsilon]\bigr)=\eta\bigl([\varepsilon]^{\frac{1}{p}}\bigr)^p\neq 1$. As $q$ constructed in (I) is compatible with $\eta$, we conclude that in this
case $q$ satisfies (II) as wanted.

\medskip

We prove (II) in the general case. Thanks to the particular case just discussed and the argument with Frobenius at the beginning of the proof, we may assume that
there {\it do not exist} $m\in\Z$ and $n\in\N$ and roots $\pi'$ of $P_\pi^{\sigma^{m}}(Z)$ such that $h(Z)=\pi^{'p^n}$. Take any $q$ as in (I). Recall that
$\widetilde{h}\bigl(P_\pi\bigl([\overline{\pi}]\bigr)\bigr)=h\bigl(P_\pi(Z)\bigr)=p\varrho$ is non zero in $L$ by hypothesis.

\medskip \noindent a) There exists $n$ such that $\widetilde{h}\bigl(\big[\varepsilon^{\frac{1}{p^n}}\big]\big)\neq 1$.\smallskip

Recall from \ref{lemma:KerTheta} that $1+[\varepsilon]^{\frac{1}{p}}+\cdots + [\varepsilon]^{\frac{p-1}{p}}$ is $P_\pi\bigl([\overline{\pi}]\bigr)$ up to unit since
they both generate the kernel of $\Theta$.   In particular, applying $\varphi^{1-n}  $ we get that $1+\bigl[\varepsilon^{\frac{1}{p^n}}\bigr]+\cdots +
\bigl[\varepsilon^{\frac{p-1}{p^n}}\bigr]$ is $P_\pi^{\sigma^{1-n}}\bigl(\bigl[\overline{\pi}^{\frac{1}{p^{n-1}}}\bigr]\bigr)$ up to a unit for every $n\in\N$.
Thus, if $\widetilde{h}\bigl(\big[\varepsilon^{\frac{1}{p^n}}\big]\big)= 1$ for every $n$ then
$\widetilde{h}\left(P_\pi^{\sigma^{1-n}}\bigl(\bigl[\overline{\pi}^{\frac{1}{p^{n-1}}}\bigr]\bigr)\right)=p$ times a unit of $\cO_{\widehat{\Kbar}}$ for every $n$.
As $\widetilde{h}([\overline{\pi}])=\gamma \in \cO_L$ is not a unit and it is not zero and $P_\pi(Z)$ is an Eisenstein polynomial of the form $Z^e+p g(Z)$, we
deduce that for $n$ large enough $\widetilde{h}\left(P_\pi^{\sigma^{1-n}}\bigl(\bigl[\overline{\pi}^{\frac{1}{p^{n}}}\bigr]\bigr)\right)= \gamma^{\frac{e}{p^n}}+p
g^{\sigma^{1-n}}\bigl(\gamma^{\frac{1}{p^n}}\bigr)$ has valuation strictly smaller than the one of $p$, leading to a contradiction.

\medskip \noindent b) We have  $\widetilde{h}\bigl([\varepsilon]\bigr)\neq 1$ which proves (II).

\smallskip

Assume on the contrary that $\widetilde{h}\bigl([\varepsilon]\bigr)= 1$. By a) there exists $n$ such that
$\widetilde{h}\bigl(\bigl[\varepsilon^{\frac{1}{p^n}}\bigr]\bigr)\neq 1$. Take the smallest such $n$. Then,
$\widetilde{h}\bigl(\bigl[\varepsilon^{\frac{1}{p^n}}\bigr]\bigr)$ is a primitive $p$-th root of unity. Thus $\widetilde{h}$ maps
$1+\bigl[\varepsilon^{\frac{1}{p^n}}\bigr]+\cdots + \bigl[\varepsilon^{\frac{p-1}{p^n}}\bigr]$ to $0$ and, arguing as in (a), we conclude that
$\widetilde{h}\left(P_\pi^{\sigma^{1-n}}\bigl(\bigl[\overline{\pi}^{\frac{1}{p^{n-1}}}\bigr]\big)\right)=0$. Thus,
$\pi':=\widetilde{h}\bigl(\bigl[\overline{\pi}^{\frac{1}{p^{n-1}}}\bigr]\bigr)$ is a root of $P_\pi^{\sigma^{1-n}}(Z)$ and  $\widetilde{h}$ sends $Z$ to
$\pi^{'p^{n-1}}$. This contradicts our assumptions on $h$.
\end{proof}

In order to prove \ref{thm:Blogmaxff} we have the following lemma whose proof we leave to the reader:

\begin{lemma}\label{lemma:ABCD} Consider rings  $A \to B \to C \to D$ such that $A\to B$ is faithfully flat, $B$ is a direct summand of $C$ as $B$-module and $C\to D$ is faithfully
flat. Then,

(1) A sequence of $A$-modules which is exact after tensoring with $D$ over $A$ is exact;

(2) An $A$-module $M$, such that $M\otimes_A D$ is finite and projective as $D$-module, is finite and projective as $A$-module.

\end{lemma}

\begin{proof}(of Theorem \ref{thm:Blogmaxff})\enspace Thanks to \ref{lemma:AtildeRlogff} and \ref{lemma:imageofg} the inclusion
$\widetilde{R}_{\rm max}\bigl[(pZ)^{-1}\bigr] \subset {\rm B}^{\rm max}_{\rm log}\bigl(\widetilde{R}\bigr)\bigl[Z^{-1}\bigr]$ is faithfully flat. If $\alpha=1$, as
$\widetilde{R}=\widetilde{R}^o$ in this case (see \ref{lemmma:Rtildeinftyflat}),  the inclusion $\widetilde{R}_{\rm max}\bigl[p^{-1}\bigr] \subset {\rm B}^{\rm
max}_{\rm log}\bigl(\widetilde{R}\bigr)$ is  flat.

Due to \ref{lemma:AtildeRlogff} and \ref{lemma:ABCD} to conclude the proof of the theorem we are left to show that the map ${\bf A}_{\widetilde{R}^o,{\rm
max}}^{+,\rm log}\bigl[p^{-1}\bigr] \lra {\rm B}^{\rm max}_{\rm log}\bigl(\widetilde{R}\bigr)$ is faithfully flat if we localize at maximal ideals of ${\bf
A}_{\widetilde{R}^o,{\rm max}}^{+,\rm log}\bigl[p^{-1}\bigr]$ containing $Z$.  Equivalently we need to show that the map on spectra contains all closed points
associated to  $L$-valued points $h\colon {\bf A}_{\widetilde{R}^o,{\rm max}}^{+,\rm log}\bigl[p^{-1}\bigr]\lra L$, for some extension $K\subset L$, such that
$h(Z)=0$.

First of all the map $h$ defines the map $h_0\colon \cO_{\rm max} \to \WW(k)$ sending $Z$ to $0$. We claim that one can extend $h_0$ to a $ \widehat{\Kbar}$-point
$h_\Kbar$ of ${\rm B}^{\rm max}_{\rm log}\bigl(\cO\bigr)$. For this it suffices to show that $Z$ is not invertible in ${\rm B}^{\rm max}_{\rm log}\bigl(\cO\bigr)$.
As  $\varphi(Z)=Z^p$ and $\varphi\big({\rm B}^{\rm max}_{\rm log}\bigl(\cO\bigr)\big)$ is a subring of Kato's period ring $B_{\rm log}$ introduced in
\S\ref{sec:classical} by \ref{rmk:AmaxfrobaAcris}, it suffices to show that $Z$ is not invertible in $B_{\rm log}$. It follows from \cite[Cor. 4.1.3 \& Prop.
5.1.1(ii)]{breuil} that $\varphi^2\bigl(B_{\rm log}^{G_K}\bigr)\subset \cO_{\rm cris}[p^{-1}]$ which is contained in $\cO_{\rm max}[p^{-1}]$. Thus, if $Z$ were
invertible in $B_{\rm log}$, then $Z^{p^2}$ and thus $Z$ itself would be invertible in $\cO_{\rm max}[p^{-1}]$ which is not the case.

Since ${\bf A}_{\widetilde{R}^o,{\rm max}}^{+,\rm log,\nabla}$ is $p$-adically complete, $h$ defines a morphism $\widetilde{h}\colon {\bf A}_{\widetilde{R}^o,{\rm
max}}^{+,\rm log,\nabla} \lra \cO_L$. As the images of $u-1$, $v_2-1,\ldots,v_a-1$ and $w_1-1,\ldots,w_b-1$ are determined thanks to \ref{lemma:structureAlog} and
\ref{lemma:structAlog}, it suffices  to show that there exists a morphism $r\colon {\rm A}_{\rm max}^\nabla(R) \lra \widehat{\cO}_{\Kbar}$, extending
$\widetilde{h}$ and such that the image of $t$ is non zero. As in the proof of \ref{lemma:imageofg} we are left to construct a morphism $q\colon
\WW\bigl(\widetilde{\bf E}^+\bigr)\lra \widehat{\cO}_{\Kbar}$ extending $\widetilde{h}$ and such that $q([\varepsilon])\neq 1$. First of all we extend
$\widetilde{h}$ using the map  $h_\Kbar\colon \WW\bigl(\widetilde{\bf E}_{\cO_\Kbar}^+\bigr) \to \widehat{\cO}_\Kbar$ defined above.  Note that the image of
$[\varepsilon]-1$ is non zero as  $h_\Kbar(t)$ is non zero.  The map $q$, extending $\widetilde{h}$ and $h_\Kbar$, is then constructed as in the proof of
\ref{lemma:imageofg}.  We leave the details to the reader. \end{proof}

\subsubsection{Localizations}\label{sec:AlogforRsmooth}

Assume first that $R$ is $p$-adically complete and separated and that the log structure coincides with the log structure defined by the ideal $\pi$. This amounts to
require that $Y_1,\ldots, Y_b$ are invertible in $R$ and that there exists $1\le i\le a$ such that $X_1,\ldots, X_{i-1},X_{i+1},\ldots,X_a$ are invertible in $R$.
Up to renumbering the variables we assume that $X_i=X_a$. In particular, $R$ is obtained from $\cO_K\bigl[X_1^{\pm 1},\ldots,X_{a-1}^{\pm 1},Y_1^{\pm
1},\ldots,Y_b^{\pm 1}\bigr]$ by iterating the following operations: taking the $p$-adic completion of an \'etale extension, taking the $p$-adic completion of a
localization and taking the completion with respect to an ideal containing $p$. Put $R_0:=\widetilde{R}/Z \widetilde{R}$. It is $p$-adically complete and separated
and $R_0/p R_0\cong R/\pi R$.

\begin{lemma}\label{RR0Zincodim1} There exists a unique isomorphism $\widetilde{R}\cong R_0[\![Z]\!]$
of $\cO\bigl[X_1^{\pm 1},\ldots,X_{a-1}^{\pm 1},Y_1^{\pm
1},\ldots,Y_b^{\pm 1}\bigr]$-algebras lifting $R_0/p R_0\cong
R/\pi R$. In particular,
$$\widetilde{R}_{\rm cris}\cong R_0[\![Z]\!]\left\{\langle
P_\pi(Z) \rangle \right\},\qquad \widetilde{R}_{\rm max}\cong
R_0[\![Z]\!]\left\{\frac{P_\pi(Z)}{p}\right\}.$$
\end{lemma}
\begin{proof} Both $\widetilde{R}$ and $R_0[\![Z]\!]$
are $(p,Z)$-adically complete and separated. By definition of
$\widetilde{R}$ in \ref{lemma:tildeRn}, they are both obtained
from $\cO\bigl[X_1^{\pm 1},\ldots,X_{a-1}^{\pm 1},Y_1^{\pm
1},\ldots,Y_b^{\pm 1}\bigr]$ by iterating finitely many times the
following operations: taking the $(p,Z)$-adic completion of
\'etale extensions, the $(p,Z)$-adic completion of localizations
and completion with respect to some ideal containing $(p,Z)$. One
proceeds by induction on the number of iterations to show that
the algebras we obtain are isomorphic modulo $(p,Z)$ and, hence,
they are  isomorphic, cf.~\ref{lemma:tildeRn}.
\end{proof}
Following \cite[Def. 6.1.3]{brinon} we let ${\rm B}_{\rm
cris}(R_0):={\rm A}_{\rm cris}(R_0)\bigl[t^{-1}\bigr]$ where ${\rm
A}_{\rm cris}(R_0)$ is the $p$-adic completion of the DP envelope
of $\WW\bigl(\widetilde{\bf E}^+\bigr)\otimes_{\WW(k)} R_0$ with
respect to the kernel of the morphism $\Theta\colon
\WW\bigl(\widetilde{\bf E}^+\bigr)\otimes_{\WW(k)} R_0\lra
\widehat{\overline{R}}$. Similarly one defines ${\rm A}_{\rm
max}(R_0)$ and ${\rm B}_{\rm max}(R_0) :={\rm A}_{\rm
max}(R_0)\bigl[t^{-1}\bigr]$ where ${\rm A}_{\rm max}(R_0)$ is the
$p$-adic completion of the subalgebra of $\WW\bigl(\widetilde{\bf
E}^+\bigr)\otimes_{\WW(k)} R_0\bigl[p^{-1}\bigr]$ generated by
$p^{-1}\Ker\bigl(\Theta\bigr)$.

\begin{corollary}\label{cor:BR} We have ${\rm A}_{\rm log}^{\rm cris,\nabla}(R)\cong {\rm A}_{\rm
cris}^\nabla(R_0)[\![Z]\!]\left\{\langle P_\pi(Z) \rangle\right\}$ and ${\rm A}^{\rm cris}_{\rm log}(\widetilde{R})\cong {\rm A}_{\rm
cris}(R_0)\widehat{\otimes}_{R_0}\widetilde{R}_{\rm cris}$.

\

Similarly, ${\rm A}_{\rm log}^{\rm max,\nabla}(R)\cong {\rm A}_{\rm max}^\nabla(R_0)[\![Z]\!]\left\{\frac{P_\pi(Z)}{p}\right\}$ and ${\rm A}^{\rm max}_{\rm
log}(\widetilde{R})\cong {\rm A}_{\rm max}(R_0)\widehat{\otimes}_{R_0}\widetilde{R}_{\rm max}.$
\end{corollary}
\begin{proof} This follows since $\widetilde{R} \cong R_0[\![Z]\!]$ by \ref{RR0Zincodim1}.
\end{proof}

\smallskip

We now return to a general $R$, i.e.,  assume that $R$ satisfies the assumptions in \S\ref{section:localdescription}. Let $T$ be the set of minimal prime ideals of
$R$ over the ideal $(\pi)$ of $R$. For any such $\cP$ let $\overline{T}_\cP$ be the set of minimal prime ideals of $\overline{R}$ over the ideal $\cP$. For any
$\cP\in T$ denote by $\widehat{R}_\cP$ the $p$-adic completion of the localization of $R$ at $\cP\cap R$. It is a dvr. Let $\widetilde{R}(\cP)$ be the $(p,Z)$-adic
completion of the localization of $\widetilde{R}$ at the inverse image of $\cP$ and let $R_{\cP,0}:=\widetilde{R}(\cP)/Z \widetilde{R}(\cP)$. Then,
$\widetilde{R}(\cP)\cong R_{\cP,0}[\![Z]\!]$ by \ref{RR0Zincodim1}. For $\cQ\in \overline{T}_\cP$ let $\overline{R}(\cQ)$ be the  normalization of $R_{\cP,0}$ in an
algebraic closure of ${\rm Frac}\bigl(\overline{R}_\cQ\bigr)$.

\begin{lemma}\label{lemma:Acrisonlocisinjective} The maps
$${\rm A}^{\rm cris}_{\rm log}(\widetilde{R}) \lra \prod_{\cP\in T,\cQ\in \overline{T}_\cP} {\rm
A}^{\rm cris}_{\rm log}(\widetilde{R}(\cP))\cong \prod_{\cP\in T,\cQ\in \overline{T}_\cP} {\rm A}_{\rm cris}(R_{\cP,0})[\![Z]\!]\left\{\langle
P_\pi(Z)\rangle\right\}$$obtained from the functoriality of the construction of ${\rm A}^{\rm cris}_{\rm log}$ are injective, $\cG_R$-equivariant and compatible
with filtrations and Frobenii. Similarly, the maps $${\rm A}^{\rm max}_{\rm log}(\widetilde{R}) \lra \prod_{\cP\in T,\cQ\in \overline{T}_\cP} {\rm A}^{\rm max}_{\rm
log}(\widetilde{R}(\cP))\cong \prod_{\cP\in T,\cQ\in \overline{T}_\cP} {\rm A}_{\rm max}(R_{\cP,0})[\![Z]\!]\left\{\frac{P_\pi(Z)}{p}\right\}$$are injective,
$\cG_R$-invariant and compatible with filtrations and Frobenii. In particular, the same holds if we take ${\rm B}^{\rm cris}_{\rm log}$ instead of ${\rm A}^{\rm
cris}_{\rm log}$ and if we take ${\rm B}^{\rm max}_{\rm log}$ instead of ${\rm A}^{\rm max}_{\rm log}$.
\end{lemma}
\begin{proof} The compatibilities with filtrations and
Frobenii follow from the construction of ${\rm A}^{\rm cris}_{\rm log}$ and ${\rm A}^{\rm max}_{\rm log}$ and their functoriality. As remarked in the proof of
\ref{prop:GrDdR} the group $\cG_R$ acts transitively on $\overline{T}_\cP$ for every $\cP\in T$ and, by the normality of $\overline{R}$, we have an injective
$\cG_R$-equivariant homomorphism $\overline{R}/p \overline{R}\subset \prod_{\cP\in T, \cQ\in \overline{T}_\cP} \overline{R}(\cQ)/p\overline{R}(\cQ)$. This implies
the claimed $\cG_R$-equivariance. It follows from \ref{display:Alognabla} that the displayed map of the Lemma is injective modulo $p$ and, hence, it is injective.
One argues similarly in the case of ${\rm A}^{\rm max}_{\rm log}$.
\end{proof}

\begin{corollary}\label{cor:Frobeniusishorizontal}
Frobenius on ${\rm B}^{\rm cris}_{\rm log}(\widetilde{R})$ and on
${\rm B}^{\rm max}_{\rm log}(\widetilde{R})$ is horizontal with
respect to the connections $\nabla_{\widetilde{R}/\WW(k)}$ and
$\nabla_{\widetilde{R}/\cO}$.
\end{corollary}
\begin{proof} We need to prove that $\varphi \circ \nabla= \nabla \circ \varphi$
where Frobenius on the differentials is defined by sending $ d  x
\mapsto d \varphi(x)$. Due to \ref{lemma:Acrisonlocisinjective} it
suffices to prove it  in the case that $R$ is a complete dvr.
Thanks to \ref{cor:BR} one is reduced to prove the horizontality
for ${\rm B}_{\rm cris}(R_0)$ and ${\rm B}_{\rm max}(R_0)$. This
is the content of \cite[Prop. 6.2.5]{brinon}.
\end{proof}

\subsection{The geometric cohomology of ${\rm B}^{\rm cris}_{\rm log}$}
\label{sec:geomtericBlogcris}

Fix an embedding $\overline{K}\subset \Omega$ with $\Omega$ an algebraically closed field containing $R$. We call $$G_R:={\rm
Gal}\left(\overline{R}\bigl[p^{-1}\bigr]/R \overline{K}\right),\qquad \cG_R:={\rm Gal}\left(\overline{R}\bigl[p^{-1}\bigr]/R \bigl[p^{-1}\bigr]\right)$$the {\em
geometric} (resp.~the {\em arithmetic}) Galois group of $R\bigl[p^{-1}\bigr]$.

\medskip

In \ref{def:RlogRmax} we defined  $\widetilde{R}_{\rm cris}$ (resp.~$\widetilde{R}_{\rm max}$) as  the $p$-adic completions of the logarithmic DP envelope of
$\widetilde{R}$  with respect to the kernel $\Ker$ of the morphism $\widetilde{R}\lra R$ (resp.~of the subring $\widetilde{R}\left[\frac{\Ker}{p}\right]$ of
$\widetilde{R}[p^{-1}]$). Similarly, one defines the geometric counterparts $\widetilde{R}_{\rm log}^{\rm geo, cris}$ and $\widetilde{R}_{\rm log}^{\rm geo, max}$
using the subring $$\widetilde{R}^{\rm geo}:=\WW\bigl(\widetilde{\bf E}^+_{\cO_{\overline{K}}}\bigr)\otimes_{\WW(k)} \widetilde{R} \subset \WW\bigl(\widetilde{\bf
E}^+_{\overline{R}}\bigr)\otimes_{\WW(k)} \widetilde{R}$$instead of $\widetilde{R}$ and the kernel of the natural morphism $\widetilde{R}^{\rm geo} \lra \widehat{\Rbar}$ induced by $\Theta_{\widetilde{R},\log}\colon \WW\bigl(\widetilde{\bf
E}^+_{\overline{R}}\bigr)\otimes_{\WW(k)} \widetilde{R}\to \widehat{\Rbar}$. As in \S\ref{GaloisfiltFrob} one endows $\widetilde{R}_{\rm log}^{\rm geo, cris}$ and
$\widetilde{R}_{\rm log}^{\rm geo, max}$ with filtrations. There are morphisms $\widetilde{R}_{\rm log}^{\rm geo, cris} \to  {\rm A}^{\rm cris}_{\rm
log}(\widetilde{R})$ and $ \widetilde{R}_{\rm log}^{\rm geo, max} \to {\rm A}^{\rm max}_{\rm log}(\widetilde{R})$ preserving the filtrations. For $m\in \Z$ we set
$$\Fil^m \left( \widetilde{R}_{\rm log}^{\rm geo, cris}[t^{-1}]\right)=\sum_s {1\over t^s} \Fil^{s+m} \widetilde{R}_{\rm log}^{\rm geo, cris}, \quad \Fil^m
\left(\widetilde{R}_{\rm log}^{\rm geo, max}[t^{-1}]\right)=\sum_s {1\over t^s} \Fil^{s+m} \widetilde{R}_{\rm log}^{\rm geo, max}$$in $\widetilde{R}_{\rm log}^{\rm
geo, cris}[t^{-1}]$ (resp.~$\widetilde{R}_{\rm log}^{\rm geo, max}[t^{-1}]$). For $m=-\infty$ we put $\Fil^m \widetilde{R}_{\rm log}^{\rm geo, cris}[t^{-1}]=
\widetilde{R}_{\rm log}^{\rm geo, cris}[t^{-1}]$ and  $\Fil^m \widetilde{R}_{\rm log}^{\rm geo, max}[t^{-1}]= \widetilde{R}_{\rm log}^{\rm geo, max}[t^{-1}]$. The
main result of this section is the following

\begin{theorem}\label{thm:geometricacyclicity} (i) For $i\geq 1$ the cohomology groups

$${\rm H}^i\left(G_R, {\rm A}^{\rm cris}_{\rm log}(\widetilde{R})\right)$$are annihilated by
$\bigl([\varepsilon]-1\bigr)^{2d}\bigl([\varepsilon]^{\frac{1}{p}}-1\bigr)^{8}{\cal I}^2$ for $i\geq 1$, with $d=a+b$, and they are zero if we invert $t$. For $i=0$
we have injective morphisms
$$\widetilde{R}_{\rm log}^{\rm geo, cris} \lra   {\rm
H}^0\left(G_R, {\rm A}^{\rm cris}_{\rm log}(\widetilde{R})\right),\qquad \widetilde{R}_{\rm log}^{\rm geo, max}\lra {\rm H}^0\left(G_R, {\rm A}^{\rm max}_{\rm
log}(\widetilde{R})\right)$$with cokernel annihilated by a power of $t$ and which are strict with respect to the filtrations. \smallskip

(ii) We have an injective morphism
$$R\widehat{\otimes}_{\cO_K} {\rm Gr}^\bullet A_{\rm cris}\lra {\rm H}^0\left(G_R, {\rm
Gr}^\bullet {\rm A}^{\rm cris}_{\rm log}(\widetilde{R})\right)$$with cokernel annihilated by a power of $p$.\smallskip

(iii) For every $m\in \Z \cup \{-\infty\}$ and every $i\geq 1$ we have
$${\rm H}^i\left(G_R, \Fil^m {\rm B}^{\rm cris}_{\rm log}(\widetilde{R})\right)=0.$$For $i=0$ we have isomorphisms
$$\Fil^m \widetilde{R}_{\rm log}^{\rm geo, cris}[t^{-1}] \lra   {\rm
H}^0\left(G_R, \Fil^m {\rm B}^{\rm cris}_{\rm log}(\widetilde{R})\right),\quad \Fil^m \widetilde{R}_{\rm log}^{\rm geo, max}[t^{-1}]\lra {\rm H}^0\left(G_R, \Fil^m
{\rm B}^{\rm max}_{\rm log}(\widetilde{R})\right).$$

(iii') Statement (iii) holds replacing ${\rm B}^{\rm cris}_{\rm log}(\widetilde{R})$ with ${\rm B}^{\rm cris}_{\rm log}(\widetilde{R})\otimes_{B_{\rm log}}
\overline{B}_{\rm log}$ and replacing $\widetilde{R}_{\rm log}^{\rm geo, cris}$ with  $\widetilde{R}_{\rm log}^{\rm geo, cris}\otimes_{B_{\rm log}}
\overline{B}_{\rm log}$. See \S\ref{sec:notation} for the notation.

\end{theorem}

Using \ref{thm:geometricacyclicity}, we also prove the following analogue of \cite[Prop. 5.1.1(ii)]{breuil}:

\begin{proposition}\label{prop:arithemticinvariantsB} There exists $s\in\N$,
equal to $2$ if $p\geq 3$ and equal to $3$ if $p=2$, such that
$\varphi^s\left({\rm B}_{\rm log}^{\rm
cris}(\widetilde{R})^{\cG_R}\right) \subset \widetilde{R}_{\rm
cris}\bigl[p^{-1}\bigr]$ and $\varphi^s\left({\rm B}_{\rm
log}^{\rm max}(\widetilde{R})^{\cG_R}\right)\subset
\widetilde{R}_{\rm max}\bigl[p^{-1}\bigr]$.
\end{proposition}

Let $H_R\subset G_R$ be the Galois group of $\overline{R}\bigl[p^{-1}\bigr]$ over $R_\infty\Kbar$. Then, $\widetilde{\Gamma}_R:=G_R/H_R$ is the Galois group of
$R_\infty \Kbar$ over $R\overline{K}$. Due to Assumptions (3)\&(4) in \S\ref{section:localdescription} the monoid $\cO_\Kbar^\ast \cdot \psi_R\bigl(\{0\}\times
\N^{a-1}\times \N^{b}\bigr)\subset R \Kbar$ is saturated which implies by Kummer theory that $R\overline{K} \otimes_{R^{(0)}} R^{(0)}_\infty$ is an integral domain,
i.e., it coincides with $R_\infty \Kbar$. Thus $\widetilde{\Gamma}_R$ coincides with the Galois group $\widetilde{\Gamma}_{R^{(0)}}=\oplus_{i=2}^a \Z_p \gamma_i
\oplus \oplus_{j=1}^b \Z_p \delta_j$ of the extension
$$R^{(0)}\otimes_{\cO_{K}}\Kbar=\frac{\Kbar\bigl[X_1,\ldots,X_a,Y_1,\ldots,Y_b\bigr]}{\bigl(X_1\cdots
X_a-\pi^\alpha\bigr)}\lra \cup_{n\in\N} \frac{\Kbar\bigl[X_1^{\frac{1}{n!}},\ldots,X_a^{\frac{1}{n!}},Y_1^{\frac{1}{n!}},\ldots,Y_b^{\frac{1}{n!}}\bigr]}
{\bigl(X_1^{\frac{1}{n!}}\cdots X_a^{\frac{1}{n!}}-\pi^{\frac{\alpha}{n!}}\bigr)}=R^{(0)}_\infty\otimes_{\cO_{K_\infty'}}\Kbar,$$where for every $i=1,\ldots,a$ we
let $\gamma_i$ be the automorphism characterized by the property that for every $n\in\N$ we have
$$\gamma_i\bigl(X_h^{\frac{1}{n!}}\bigr)=\begin{cases}
\epsilon_{n!} X_i^{\frac{1}{n!}} & \hbox{{\rm if }}h=i \cr X_h^{\frac{1}{n!}} & \forall 1\leq h\leq a, \, h\neq i \cr
\end{cases}$$and $\gamma_i\bigl(Y_j^{\frac{1}{n!}}\bigr)=
Y_j^{\frac{1}{n!}}$ for every $j=1,\ldots,b$. Here, $\epsilon_{n!}$ is the primitive $n!$-root of unity chosen in \ref{sec:notation}. Similarly, for every
$i=j,\ldots,b$ we let $\delta_j$ be defined by the property that for every $n\in\N$ we have $\delta_j\bigl(X_i^{\frac{1}{n!}}\bigr)= X_i^{\frac{1}{n!}}$ for every
$i=1,\ldots,a$ and
$$\delta_j\bigl(Y_h^{\frac{1}{n!}}\bigr)=\begin{cases}
\epsilon_{n!} Y_j^{\frac{1}{n!}} & \hbox{{\rm if }}h=j \cr  Y_h^{\frac{1}{n!}} & \forall h=1,\ldots,b,\, h\neq j. \end{cases}$$The proof of
\ref{thm:geometricacyclicity} is in three steps:\smallskip

1) First of all, using Faltings' theory of almost \'etale extensions, we prove that $${\rm H}^i\left(H_R, {\rm A}^{\rm cris}_{\rm log}(\widetilde{R})/p^m {\rm
A}^{\rm cris}_{\rm log}(\widetilde{R})\right)$$and $${\rm H}^i\left(H_R, {\rm A}^{\rm max}_{\rm log}(\widetilde{R})/p^m {\rm A}^{\rm max}_{\rm
log}(\widetilde{R})\right)$$are annihilated by the ideal ${\cal I}$ for $i\geq 1$. We also  construct rings ${\rm A}_{\rm log,\infty}^{\rm cris}$ and ${\rm A}_{\rm
log,\infty}^{\rm max}$ with maps to ${\rm A}^{\rm cris}_{\rm log}(\widetilde{R})^{H_R}$ and ${\rm A}^{\rm max}_{\rm log}(\widetilde{R})^{H_R}$ respectively, such
that modulo $p^m$ kernel and cokernel are annihilated by ${\cal I}$ for every $m\in\N$; see \ref{prop:HiHvanishes}.

 \smallskip

2) We define subrings ${\bf A}_{{\rm log}}^{\rm geo, max}(\widetilde{R})$ and  ${\bf A}_{{\rm log}}^{\rm geo, cris}(\widetilde{R})$ of ${\rm A}_{\rm
log,\infty}^{\rm max}$ and ${\rm A}_{\rm log,\infty}^{\rm cris}$ respectively such that these inclusions modulo
$\left(p^m,\sum_{i=0}^{p-1}\bigl[\varepsilon\bigr]^{ip^{m-1}}\right)$ induce a morphism between the cohomology groups with respect to the group
$\widetilde{\Gamma}_R$ with kernel and cokernel annihilated by $([\varepsilon]^{\frac{1}{p}}-1)^2$. See \ref{prop:noetherianize} and
\ref{cor:cohoXvanishes}.
\smallskip

3) We prove that the cohomology groups $${\rm H}^i\left(\widetilde{\Gamma}_R, {\bf A}_{{\rm log}}^{\rm geo, cris}(\widetilde{R})/(p^m)\right)$$vanish for $i\geq
d+1$, are annihilated by the ideal $([\varepsilon]-1)^d$ for $i\geq 1$ and coincide with $\widetilde{R}_{\rm log}^{\rm geo, cris}/p^m \widetilde{R}_{\rm log}^{\rm
geo, cris}$ up to $\bigl([\varepsilon]-1\bigr)^d$-torsion for $i=0$. We also prove that
$${\rm H}^0\left(\widetilde{\Gamma}_R, {\bf A}_{{\rm log}}^{\rm geo,
max}(\widetilde{R})/(p^m)\right)$$ coincides with $\widetilde{R}_{\rm log}^{\rm geo, max}/p^m \widetilde{R}_{\rm log}^{\rm geo, max}$ up to multiplication by
$\bigl([\varepsilon]-1\bigr)^d$. See \ref{prop:cohA+Rloggeo}.\smallskip

{\it Proof of \ref{thm:geometricacyclicity}} \enspace We start by showing how Claim (i) follows from (1)--(3). First of all using the limit argument of \cite[lemma
23 \& Cor.~24]{andreatta_brinon} one proves that (2) holds modulo $p^m$ up to $([\varepsilon]^{\frac{1}{p}}-1)^{4}$-torsion for every $m\in \N$. Using the
Hochschild-Serre spectral sequence applied to $H_R\subset G_R$ giving $${\rm H}^r\bigl(\widetilde{\Gamma}_R, {\rm H}^s\bigl(H_R,\_ \bigr)\bigr)\Longrightarrow {\rm
H}^{r+s}\bigl(G_R,\_\bigr),
$$the first claim in \ref{thm:geometricacyclicity}(i) follows, considering the rings modulo $p^m$, up
to $\bigl([\varepsilon]-1\bigr)^d([\varepsilon]^{\frac{1}{p}}-1)^{4}{\cal I}$-torsion. Using once more using the limit argument of \cite[lemma 23 \&
Cor.~24]{andreatta_brinon} the first claim follows. As $[\varepsilon]^{p^2}-1$ belongs to ${\cal I}$ and it is invertible in $B_{\rm cris}$ by~\cite[Pf. Thm
6.3.8]{brinon}, the ideal $\bigl([\varepsilon]-1\bigr)^d([\varepsilon]^{\frac{1}{p}}-1)^{4}{\cal I}$ becomes a unit if we invert $t$ proving the vanishing in Claim
(i).

The injectivity for $i=0$ in \ref{thm:geometricacyclicity}(i) and the fact that the maps are strict with respect to the filtrations is proven in
\ref{cor:injstrictinvariants}.

Claim \ref{thm:geometricacyclicity}(ii) concerning the graded rings are proven according to similar lines. The analogue of (1) is contained in
\ref{cor:cohoRbarmodp}. The analogue of (2) is the content of \ref{cor:cohoBHTinftyvanishes}.
The analogue of (3) is also proven in \ref{prop:cohA+Rloggeo}.

Claim \ref{thm:geometricacyclicity}(iii) is discussed in \S\ref{section:invariantsFilBlog}.

Claim \ref{thm:geometricacyclicity}(iii'), concerning the vanishing of the cohomology groups,  is a variant of the
strategy described above and is discussed in \S\ref{section:variantBbarlog}. For the computations of the invariants,
see \ref{cor:injstrictinvariants}.

\

For the reader's convenience we summarize in the following diagram the various rings appearing in this section in the crystalline setting. The horizontal rows
should be thought of as analogues of the inclusions $R\cO_\Kbar \subset R_\infty \cO_\Kbar \subset \Rbar$. The top row is the  $\nabla=0$ analogue of the lower row:

$$\begin{array}{ccccccc}  &  & {\bf A}_{\widetilde{R},\rm cris}^{+,\rm geo}  & \subset & {\rm A}_{\rm log,\infty}^{\rm cris,\nabla} & \subset & {\rm
A}^{\rm cris.\nabla}_{\rm log}(\widetilde{R}) \cr
 & & \downarrow & & \downarrow & & \downarrow \cr
\widetilde{R}_{\log}^{\rm geo, cris} & \subset & {\bf A}_{{\rm log}}^{\rm geo, cris}(\widetilde{R}) & \subset & {\rm A}_{\rm log,\infty}^{\rm cris} & \subset & {\rm
A}^{\rm cris}_{\rm log}(\widetilde{R}),\cr\end{array}$$where:

i) ${\rm A}_{\rm log,\infty}^{\rm cris}$ (resp.~${\rm A}_{\rm log,\infty}^{\rm cris,\nabla}$) is the $p$--adic completion of the log DP envelope of
$\WW\bigl(\widetilde{\bf E}^+_{R_{\infty,\cO_{\Kbar}}}\bigr) \tensor_{\WW(k)} \widetilde{R}$ (resp.~$\WW\bigl(\widetilde{\bf
E}^+_{R_{\infty,\cO_{\Kbar}}}\bigr)\tensor_{\WW(k)} \cO$) with respect to the natural morphism to $\widehat{\overline{R}}$. It is the analogue of the inclusion
${R_\infty \cO}_\Kbar \subset \Rbar$ and by almost \'etale descent it reduces the computation of the $G_R$-cohomology of ${\rm A}^{\rm cris}_{\rm
log}(\widetilde{R})$ to the computation of the $\widetilde{\Gamma}_R$-cohomology of $ {\rm A}_{\rm log,\infty}^{\rm cris}$; see \ref{prop:HiHvanishes};

\smallskip

ii) ${\bf A}_{\widetilde{R},\rm cris}^{+,\rm geo}$ is the $p$-adic completion of the log divided power envelope of the image ${\bf A}^{+,\rm
geo}_{\widetilde{R}}$ of ${\bf A}^{+}_{\widetilde{R}} \tensor_{\WW(k)} \WW\bigl(\widetilde{\bf E}^+_{\cO_{\overline{K}}}\bigr)\to \WW(\widetilde{\bf E})$ with
respect to the morphism to $\widehat{\overline{R}}$; see \S\ref{subsec:deperf}. It is the de-perfectization of ${\rm A}_{\rm log,\infty}^{\rm cris,\nabla}$.

iii) ${\bf A}_{{\rm log}}^{\rm geo, cris}(\widetilde{R})$ is the $p$-adic completion of the log DP envelope of ${\bf A}^{+,\rm geo}_{\widetilde{R}} \tensor_{\WW(k)}
\widetilde{R} $ with respect to the morphism to $\widehat{\overline{R}}$. See \S\ref{subsec:deperf}. It is the de-perfectization of ${\rm A}_{\rm log,\infty}^{\rm
cris}$.
\smallskip

\subsubsection{Almost \'etale descent} Denote by
$R_{\infty,\cO_{\Kbar}}$ the composite $R_\infty \cO_\Kbar\subset \overline{R}$.

\begin{lemma}\label{lemma:HiHR} (1) For every $n\in\N$ the subring $R_n \cO_{\overline{K}}\subset
\overline{R}$ is a direct factor of  $R_n\otimes_{\cO_{K_n'}} \cO_{\overline{K}}$ and is a normal ring.\smallskip

(2) Let $S\subset \Omega$ be a normal
$R_{\infty,\cO_{\Kbar}}$-algebra, finite \'etale and Galois with
group $H_S$ after inverting $p$. Then, for every $i\geq 1$ the
group ${\rm H}^i\bigl(H_S,S\bigr)$ is annihilated by the maximal
ideal of\/~$\cO_{\overline{K}}$. For $i=0$ it coincides with
$R_{\infty,\cO_{\Kbar}}$.
\smallskip

In particular, ${\rm H}^i\bigl(H_R,\overline{R}\bigr)$ is annihilated by the maximal ideal of\/~$\cO_{\overline{K}}$ for $i\geq 1$.
For $i=0$ it coincides with
$R_{\infty,\cO_{\Kbar}}$ and the latter is a normal ring.

\end{lemma}
\begin{proof} (1) It follows as in \S\ref{lemma:Rinftyflat} that $R_n \tensor_{\cO_K} \cO_L$ is a normal ring for every
$n\in\N$ and every finite extension $K\subset L
\subset \Kbar$. As it is noetherian, it is the product of normal domains one of which is its image $R_n\cO_L\subset \Rbar$.
Thus, $R_n\cO_{\Kbar}$ is a direct
factor in $R_n \tensor_{\cO_K} \cO_\Kbar$ and it is a normal domain.

Statement (2) follows from \ref{prop:AE}; cf.~\cite[\S 2c]{faltingsAsterisque}. The claim concerning the invariants is clear if we invert $p$.
Since
$R_{\infty,\cO_{\Kbar}}$ is normal by (1), it follows that $\overline{R}^{H_R}=R_{\infty,\cO_{\Kbar}}$.

The  last statement follows from (2).
\end{proof}

\begin{corollary}\label{cor:injstrictinvariants} (1) The image of $\widetilde{R}^{\rm geo}$ via $\Theta_{\widetilde{R},\log}$ is
$\widehat{R\cO}_\Kbar$. \smallskip

(2) The map $\widetilde{R}_{\log}^{\rm geo, cris} \to  {\rm A}^{\rm cris}_{\rm log}(\widetilde{R})$ (resp.~$\widetilde{R}_{\log}^{\rm geo, max}\to {\rm A}^{\rm
max}_{\rm log}(\widetilde{R})$) is injective and strict with respect to the filtrations.  \smallskip

(3) We have a surjective, $G_K$-equivariant map  $\widetilde{R}\widehat{\otimes}_{\cO} A_{\rm log}\to \widetilde{R}_{\log}^{\rm geo, cris}$, where
$\widehat{\otimes}$ stands for the $p$-adically completed tensor product, which is compatible with the filtrations and admits a splitting compatible with the
filtrations. It is an isomorphism if the map $R\otimes_{\cO_K}\cO_{\Kbar}\to R\cO_\Kbar (\subset \Rbar)$ is an isomorphism.\smallskip

(4) Statements (2) and (3) hold after taking the $p$-adically completed tensor product $\widehat{\otimes}_{A_{\rm log}} \overline{A}_{\rm log}$ and
$\widetilde{R}\widehat{\otimes}_{\cO} \overline{A}_{\rm log}\cong R\widehat{\otimes}_{\cO_K} \overline{A}_{\rm log}$.

\end{corollary}
\begin{proof}
It follows as in \ref{lemma:structurBdR+}(1) that the kernel of the extension of $\Theta_{\widetilde{R},\log}\colon \widetilde{R}\WW\bigl(\widetilde{\bf
E}^+_{\cO_{\overline{K}}}\bigr) \lra \widehat{\Rbar}$ to $\left(\widetilde{R}_n\WW\bigl(\widetilde{\bf E}^+_{\cO_{\overline{K}}} \bigr)\right)^{\log}$ is generated
by a regular sequence consisting of  $2$ elements, given by  $(\xi,u-1)$ (or~$\left(P_\pi\bigl([\overline{\pi}]\bigr),u-1\right)$). The graded pieces are isomorphic
to the image of $\Theta_{\log}$.  By \ref{lemma:HiHR} the ring $R\cO_\Kbar$ is normal so that the map $R\cO_\Kbar/p^m R\cO_\Kbar\to \Rbar/p^m \Rbar$ is injective
for every $m\in\N$. We conclude that $\widehat{R\cO_\Kbar}\to \widehat{\Rbar}$ is injective. Thus, the image of $\widetilde{R}\WW\bigl(\widetilde{\bf
E}^+_{\cO_{\overline{K}}}\bigr)$ via  $\Theta_{\widetilde{R},\log}$ is $\widehat{R\cO}_\Kbar$ proving (1). Due to \ref{lemma:structurBdR+}(2) we conclude that the
maps in the statement (2) induce injective maps on the associated graded rings. It follows by induction on $m\in \N$ that they are injective modulo the $m$-th step
of the filtrations on the two sides of the given maps. As the filtration on ${\rm A}^{\rm cris}_{\rm log}(\widetilde{R})$ and on ${\rm A}^{\rm max}_{\rm
log}(\widetilde{R})$ is exhaustive, claim (2) follows.

Recall from \ref{lemma:HiHR} that $R \tensor_{\cO_K} \cO_{\overline{K}}$ is the product of integral normal domains one of which is $B:=R \cO_{\overline{K}}$. As the
latter is normal, the map $B/p B \to \overline{R}/p \overline{R}$ is injective so that, since $p$ is not a zero divisor in $B$, we deduce that the map on $p$-adic
completions $\widehat{B}\to \widehat{\overline{R}}$ is injective. As $\widetilde{R}/(P_\pi(Z))=R$, the reduction of the $\WW\bigl(\widetilde{\bf
E}^+_{\cO_{\overline{K}}}\bigr) \tensor_{\WW(k)} {\widetilde{R}} $ modulo $\bigl(p,\xi,Z\bigr)$ is $\cO_{\overline{K}}  \tensor_{\cO_K} (R/\pi R)$. By Hensel's
lemma the direct factor $B/\pi B$ of $\cO_{\Kbar}  \tensor_{\cO_K} (R/\pi R)$ lifts uniquely to a direct factor $\widetilde{B}$ of the
$\bigl(p,\xi,Z\bigr)$-adically completed tensor product ${\widetilde{R}}\widehat{\tensor}_{\WW(k)} \WW\bigl(\widetilde{\bf E}^+_{\cO_{\overline{K}}}\bigr)$. By
construction the map $\widetilde{B} \to \WW\bigl(\widetilde{\bf E}^+_{\cO_{\overline{K}}}\bigr) \widehat{\tensor}_{\WW(k)} {\widetilde{R}} $ modulo
$\bigl(\xi,P_\pi(Z)\bigr)$  coincides with the inclusion $\widehat{B} \subset R\widehat{\otimes}_{\cO_K}\cO_{\Kbar} $. This implies that the $p$-adic completion
$\widetilde{B}_{\rm log}^{\rm cris}$ of the logarithmic divided power envelope of the map $\Theta_B\colon \widetilde{B}\to \widehat{B}$ has $\widehat{B}$ as graded
piece for the DP filtration. As $\Theta_B$ is compatible with $\Theta_{\widetilde{R},\log}$ we get a natural map  $\widetilde{B}_{\rm log}^{\rm cris} \to
\widetilde{R}_{\log}^{\rm geo, cris}$ which is an isomorphism on graded pieces. The DP filtration being exhaustive,  it is an isomorphism. As
${\widetilde{R}}\widehat{\tensor}_{\WW(k)} \WW\bigl(\widetilde{\bf E}^+_{\cO_{\overline{K}}}\bigr)$ maps to $\widetilde{R}\widehat{\otimes}_{\cO} A_{\rm log}$,
which is $\bigl(p,\xi,P_\pi(Z)\bigr)$-adically complete, we get a map $\widetilde{B}\to \widetilde{R}\widehat{\otimes}_{\cO} A_{\rm log}$. Arguing  that the kernel
of $\Theta_B$ is generated by the regular sequence $\bigl(\xi,u-1\bigr)$, see \ref{lemma:structurBdR+}, and using  that $\bigl(\xi,u-1\bigr)$ admits DP powers in
$A_{\rm log}$, we obtain a natural  map $\widetilde{B}_{\rm log}^{\rm cris}\to \widetilde{R}\widehat{\otimes}_{\cO} A_{\rm log}$ inducing the inclusion
$\widehat{B}\to R\otimes_{\cO_K} \widehat{\cO}_\Kbar$ on the graded pieces for the DP filtration. It provides a splitting of the map
$\widetilde{R}\widehat{\otimes}_{\cO} A_{\rm log}\to \widetilde{R}_{\log}^{\rm geo, cris}$  as required in Claim (3).

We prove (4). As $Z=\pi$ in $\overline{A}_{\rm log}$, see \S\ref{sec:notation}, we have $P_\pi(Z)=0$  and $\widetilde{R}\widehat{\otimes}_{\cO} \overline{A}_{\rm
log}\cong R\widehat{\otimes}_{\cO_K} \overline{A}_{\rm log}$. The analogue of (3) is then clear. The filtration on $\overline{A}_{\rm log}[p^{-1}]$ is the one
induced from $B_{\rm dR}^+$ and the two rings have the same graded pieces, each isomorphic to $\widehat{\Kbar}$.  Consider the composite map  $$\tau\colon
\widetilde{R}_{\log}^{\rm geo, cris} \widehat{\otimes}_{A_{\rm log}} \overline{A}_{\rm log} \lra {\rm A}^{\rm cris}_{\rm
log}(\widetilde{R})\widehat{\otimes}_{A_{\rm log}} \overline{A}_{\rm log}  \lra {\rm B}_{\rm dR}^+(R),$$ where the second map is provided by
\ref{prop:BcrissubsetBdR}. The morphism $\tau$ is compatible with filtrations as both maps are. The ring $\widetilde{R}_{\log}^{\rm geo, cris}
\widehat{\otimes}_{A_{\rm log}} \overline{A}_{\rm log}$ has $\widehat{B}$ as graded pieces, using the analogue of (3). As $\widehat{B}$ injects in
$\widehat{\Rbar}$, the map $\tau$  is injective on graded pieces by \ref{cor:BdRstr}(6). Thus, it is injective and strict on filtrations. Therefore also
$\widetilde{R}_{\log}^{\rm geo, cris} \widehat{\otimes}_{A_{\rm log}} \overline{A}_{\rm log} \lra {\rm A}^{\rm cris}_{\rm
log}(\widetilde{R})\widehat{\otimes}_{A_{\rm log}} \overline{A}_{\rm log}$ must be injective and strict on filtrations. The claim follows.

\end{proof}

\begin{corollary}\label{cor:cohoRbarmodp} Consider the following situations:\smallskip

(i) $A=\overline{R}/p \overline{R}$ and $A_\infty$ equal to the
image of
$R_{\infty,\cO_{\Kbar}}/pR_{\infty,\cO_{\Kbar}}$;\smallskip

(ii) $A=\widehat{\overline{R}}$ and $A_\infty$ the $p$-adic
completion $\widehat{R}_{\infty,\cO_{\Kbar}}$ of
$R_{\infty,\cO_{\Kbar}}$;\smallskip

(iii) $A= {\rm Gr}^\bullet {\rm A}^{\rm cris}_{\rm log}$ and $\ds
A_\infty:= \oplus_{\underline{n}\in \N^{d+1}}
\widehat{R}_{\infty,\cO_{\Kbar}} \xi^{[n_0]} (u-1)^{[n_1]}
(v_2-1)^{[n_2]}\cdots (v_a-1)^{[n_a]}(w_1-1)^{[n_{a+1}]}\cdots
(w_b-1)^{[n_d]}$.

\smallskip

The groups ${\rm H}^i\bigl(H_R,A\bigr)$ are annihilated by the maximal ideal of\/~$\cO_{\overline{K}}$ for every $i\geq 1$. For $i=0$  the natural map $$A_\infty
\lra {\rm H}^i\bigl(H_R,A\bigr)$$has kernel and cokernel annihilated by the maximal ideal of\/~$\cO_{\overline{K}}$.

\end{corollary}
\begin{proof} The first two statements are clear. For (iii)
we use that ${\rm Gr}^\bullet {\rm A}^{\rm cris}_{\rm
log}(\widetilde{R})=\oplus_{\underline{n}\in \N^{d+1}}
\widehat{\overline{R}} \xi^{[n_0]} (u-1)^{[n_1]}
(v_2-1)^{[n_2]}\cdots (v_a-1)^{[n_a]}(w_1-1)^{[n_{a+1}]}\cdots
(w_b-1)^{n_d}$ proven in  \ref{prop:BcrissubsetBdR}.  The claim
follows then from (ii) noting that $H_R$ acts trivially on $\xi,
u,v_2,\ldots,v_a,w_1,\ldots,w_b$.

\end{proof}

Define ${\rm A}_{\rm log,\infty}^{{\rm cris},\nabla}$ and ${\rm A}_{\rm log,\infty}^{{\rm max},\nabla}$ to be the $p$-adic completion of  the log DP
envelopes of $\WW\bigl(\widetilde{\bf E}^+_{R_{\infty,\cO_{\Kbar}}}\bigr)\tensor_{\WW(k)} \cO$ with respect to the natural morphism to $\widehat{\overline{R}}$
induced by $\Theta$ and, respectively,  the $p$--adic completion of the $\bigl(\WW\bigl(\widetilde{\bf E}^+_{R_{\infty,\cO_{\Kbar}}}\bigr)\tensor_{\WW(k)}
\cO\bigr)^{\rm log}$-subalgebra of $\bigl(\WW\bigl(\widetilde{\bf E}^+_{R_{\infty,\cO_{\Kbar}}}\bigr)\tensor_{\WW(k)} \cO\bigr)^{\rm log}[p^{-1}]$ generated by
$p^{-1}\Ker\bigl(\Theta_{\log}'\bigr)$. As in \ref{lemma:structureAlog} and in \ref{display:Alognabla} one proves the following results:

\begin{lemma}\label{lemma:structureAloginfty}
The ring ${\rm A}_{\rm log,\infty}^{{\rm cris},\nabla}$ is the
$p$-adic completion of the DP envelope
of\/~$\WW\bigl(\widetilde{\bf
E}^+_{R_{\infty,\cO_{\Kbar}}}\bigr)[u]$ with respect to the ideal
$\bigl(\xi,u-1\bigr)$.
\smallskip

Similarly,  ${\rm A}_{\rm log,\infty}^{{\rm max},\nabla} \cong \WW\bigl(\widetilde{\bf
E}^+_{R_{\infty,\cO_{\Kbar}}}\bigr)\left\{\frac{\xi}{p},\frac{u-1}{p}\right\}$, the $p$-adic completion of the ring in  the variables $V$ and $W=\frac{u-1}{p}$
modulo the relation $p V=\xi $. \end{lemma}

\begin{corollary}\label{cor:Alogcrysnablamodp} We have
$${\rm A}_{\rm log}^{\rm cris,\nabla}(R)/p
{\rm A}_{\rm log}^{\rm cris,\nabla}(R)\cong \widetilde{\bf
E}^+[u]\bigl\{\delta_0,\delta_1,\ldots,\rho_0,\rho_1,\ldots\bigr\}/
\bigl(\xi^p,\delta_m^p, u^p-1,\rho_m^p\bigr)_{m\in\N}$$and
similarly for~${\rm A}_{\rm log,\infty}^{{\rm cris},\nabla}$
instead of ${\rm A}_{\rm log}^{\rm cris,\nabla}(R)$
and\/~$\widetilde{\bf E}^+_{R_{\infty,\cO_{\Kbar}}}$ instead of
$\widetilde{\bf E}^+$. On the other hand,
$${\rm A}_{\rm log}^{\rm max,\nabla}(R)/p
{\rm A}_{\rm log}^{\rm max,\nabla}(R) \cong \widetilde{\bf
E}^+/(\xi)\bigl[\delta,\rho\bigr]$$the polynomial ring in the
variables $\delta$ and $\rho$ where $\delta$ corresponds to the
class of $\frac{\xi}{p}$ and  $\rho$ corresponds to the class of
$\frac{u-1}{p}$. One has the same description for~${\rm A}_{\rm
log,\infty}^{{\rm max},\nabla}$ instead of\/~${\rm A}_{\rm
log}^{\rm max,\nabla}(R)$ and \/~$\widetilde{\bf
E}^+_{R_{\infty,\cO_{\Kbar}}}$ instead of $\widetilde{\bf E}$.
\end{corollary}

Define~${\rm A}_{\rm log,\infty}^{\rm cris}$ as the $p$--adic
completion of the logarithmic DP envelope of
$\WW\bigl(\widetilde{\bf E}^+_{R_{\infty,\cO_{\Kbar}}}\bigr)
\tensor_{\WW(k)} \widetilde{R}$ with respect to the natural
morphism to $\widehat{\overline{R}}$ induced by $\Theta$. Let
${\rm A}_{\rm log,\infty}^{\rm max}$ be the $p$--adic completion
of the $\bigl(\WW\bigl(\widetilde{\bf
E}^+_{R_{\infty,\cO_{\Kbar}}}\bigr) \tensor_{\WW(k)}
\widetilde{R}\bigr)^{\rm log}$-subalgebra of
$\bigl(\WW\bigl(\widetilde{\bf E}^+_{R_{\infty,\cO_{\Kbar}}}\bigr)
\tensor_{\WW(k)} \widetilde{R}\bigr)^{\rm log}[p^{-1}]$ generated
by $p^{-1}\Ker\bigl(\Theta_{\widetilde{R},\log}'\bigr)$. As in
\ref{lemma:structAlog} one proves:

\begin{lemma}\label{lemma:structAloginfty} The natural maps
$${\rm A}_{\rm log,\infty}^{{\rm cris},\nabla}\left\{\langle
v_2-1,\ldots,v_a-1,w_1-1,\ldots,w_b-1\rangle \right\}  \lra {\rm
A}_{\rm log,\infty}^{\rm cris}$$and $${\rm A}_{\rm
log,\infty}^{\rm max,\nabla}\left\{
\frac{v_2-1}{p},\ldots,\frac{v_a-1}{p},\frac{w_1-1}{p},\ldots,\frac{w_b-1}{p}\right\}
\lra  {\rm A}_{\rm log,\infty}^{\rm max}$$are isomorphisms.
\end{lemma}

\begin{corollary}\label{cor:Alogcrysmodp} We have
$${\rm A}^{\rm cris}_{\rm log}(\widetilde{R})/p {\rm A}^{\rm cris}_{\rm log}(\widetilde{R})\cong
\frac{{\rm A}_{\rm log}^{\rm cris,\nabla}(R)}{p {\rm A}_{\rm
log}^{\rm cris,\nabla}(R)} \frac{
\bigl[v_i,w_j,h_{i,0},h_{i,1},\cdots,\ell_{j,0},\ell_{j,1}\bigr]_{
i=2,\ldots,a,\, j=1,\ldots,b}}{\bigl(
(v_i-1)^p,h_{i,m}^p,(w_j-1)^p,\ell_{j,m}^p\bigr)_{i=2,\ldots,a,j=1,\ldots,b,m\in\N}}$$
and similarly  for ${\rm A}_{\rm log,\infty}^{\rm cris}$ instead
of ${\rm A}^{\rm cris}_{\rm log}(\widetilde{R})$. We also have

$${\rm A}^{\rm max}_{\rm log}(\widetilde{R})/p {\rm A}^{\rm max}_{\rm log}(\widetilde{R})
\cong {\rm A}_{\rm log}^{\rm max,\nabla}(R)/p {\rm A}_{\rm
log}^{\rm max,\nabla}(R)\bigl[h_i,\ell_j\bigr]_{ i=2,\ldots,a,\,
j=1,\ldots,b}$$and similarly for ~${\rm A}_{\rm log,\infty}^{\rm
max}$ instead of ${\rm A}^{\rm max}_{\rm log}(\widetilde{R})$.
\end{corollary}

Write ${\rm A}$ for ${\rm A}^{\rm cris}_{\rm log}(\widetilde{R})$ or ${\rm A}^{\rm max}_{\rm log}(\widetilde{R})$. Write ${\rm A}_{\infty}$ for ${\rm A}_{\rm
log,\infty}^{{\rm cris}}$ or ${\rm A}_{\rm log,\infty}^{{\rm max}}$. We deduce from \ref{cor:Alogcrysmodp}, the following:

\begin{proposition}\label{prop:HiHvanishes}
For every $i\geq 1$ and every $n\in\N$ the group ${\rm
H}^i\left(H_R,{\rm A}/p^n {\rm A}\right)$ is annihilated by the
ideal ${\cal I}$. The morphism ${\rm A}_{\infty}/p^n {\rm
A}_{\infty} \to \bigl({\rm A}/p^n{\rm A} \bigr)^{H_R}$ has kernel
and cokernel annihilated by ${\cal I}$.
\end{proposition}
\begin{proof} Since ${\rm A}$ is $p$-torsion free, proceeding by
induction on $n$ it suffices to show the claim for $n=1$. It follows from \ref{cor:Alogcrysnablamodp} and \ref{cor:Alogcrysmodp}  that ${\rm A}^{\rm cris}_{\rm
log}/p{\rm A}^{\rm cris}_{\rm log}$ (resp.~${\rm A}^{\rm max}_{\rm log}/p{\rm A}^{\rm max}_{\rm log}$) is a free $\widetilde{\bf E}^+/(\xi^p)$-module
(resp.~$\widetilde{\bf E}^+/(\xi)$-module) with a basis fixed by the action of $H_R$. Since $\widetilde{\bf E}^+/(\xi^p)\cong \overline{R}/p\overline{R}$ and
$\widetilde{\bf E}^+/(\xi)\cong \overline{R}/p\overline{R}$, the statement follows from \ref{cor:cohoRbarmodp}.
\end{proof}

\subsubsection{De-perfectization}\label{subsec:deperf}

Recall that  we have introduced in \ref{def_AR+} subrings ${\bf A}^+_{\widetilde{R}_n}$ of $\WW\bigl(\widetilde{\bf E}^+\bigr)$ isomorphic to the $(p,Z)$-adic
completion of $\widetilde{R}$. We have
$${\bf
A}^+_{\widetilde{R}^{(0)}_n}=\cO_n\left\{\big[\overline{X}_1\big]^{\frac{1}{n!}},\ldots,\big[\overline{X}_a\big]^{\frac{1}{n!}},
\big[\overline{Y}_1\big]^{\frac{1}{n!}},\ldots,\big[\overline{Y}_b\big]^{\frac{1}{n!}}\right\}/\bigl( \big[\overline{X}_1\big]^{\frac{1}{n!}}\cdots
\big[\overline{X}_a\big]^{\frac{1}{n!}}-Z^{\frac{\alpha}{n!}}\bigr) ,$$where $\cO_n:=\WW[\![Z^{1\over n!}]\!]$ for every $n\in\N$. Since $\WW\bigl(\widetilde{\bf
E}^+_{\cO_{\overline{K}}}\bigr)$ is a $\WW(k)$-algebra, we can make it into an $\cO_n$-algebra by sending $Z^{1\over n!}$ to $\bigl[\overline{\pi}\bigr]^{1\over
n!}$. Set ${\bf A}^{+,\rm geo}_{\widetilde{R}_n}$ to be the image of
$${\bf A}^+_{\widetilde{R}_n}\widehat{\tensor}_{\cO_n}
\WW\bigl(\widetilde{\bf E}^+_{\cO_{\overline{K}}}\bigr)\lra \WW\bigl(\widetilde{\bf E}^+\bigr) ,$$ where the completion is taken with respect to the ideal $(p,Z)$.
Identifying ${\bf A}^+_{\widetilde{R}_n}$ with $\widetilde{R}_n$ we let $\widetilde{R}^{\rm geo}_n$ to be the quotient of $\widetilde{R}_n\widehat{\tensor}_{\cO_n}
\WW\bigl(\widetilde{\bf E}^+_{\cO_{\overline{K}}}\bigr)$ isomorphic to ${\bf A}^{+,\rm geo}_{\widetilde{R}_n}$.  For every $m\in\N$ set
$${\bf A}_{m}\bigl(\widetilde{R}_n\bigr):=\varphi^m\left({\bf A}^{+,\rm geo}_{\widetilde{R}_n}\right)/
\left(p^m,\sum_{i=0}^{p-1}\bigl[\varepsilon\bigr]^{i p^{m-1}}\right) .$$The action of the group $\widetilde{\Gamma}_{R^{(0)}}=\oplus_{i=2}^a \Z_p \gamma_i \oplus
\oplus_{j=1}^b \Z_p \delta_j$ on $\WW\bigl(\widetilde{\bf E}^+\bigr)$, for $R=R^{(0)}$, stabilizes ${\bf A}^{+,\rm geo}_{\widetilde{R}^{(0)}_n}$ for every $n$. More
explicitly, it acts trivially on $\WW\bigl(\widetilde{\bf E}^+_{\cO_{\overline{K}}}\bigr)$ and it acts by
$$\gamma_i\bigl(\big[\overline{X}_h\big]^{\frac{1}{n!}}\bigr)=\begin{cases} [\varepsilon]^{-\frac{1}{n!}} \big[\overline{X}_1\big]^{\frac{1}{n!}} & \hbox{{\rm if }} h=1 \cr
[\varepsilon]^{\frac{1}{n!}} \big[\overline{X}_i\big]^{\frac{1}{n!}} & \hbox{{\rm if }} h=i \cr \big[\overline{X}_h\big]^{\frac{1}{n!}} & \forall 2\leq h\leq a, \,
h\neq i \cr
\end{cases}$$and $\gamma_i\bigl(\big[\overline{Y}_j\big]^{\frac{1}{n!}}\bigr)=
\big[\overline{Y}_j\big]^{\frac{1}{n!}}$ for every $j=1,\ldots,b$. Similarly, for every $i=j,\ldots,b$ we let $\delta_j$ act via
$\delta_j\bigl(\big[\overline{X}_i\big]^{\frac{1}{n!}}\bigr)= \big[\overline{X}_i\big]^{\frac{1}{n!}}$ for every $i=1,\ldots,a$ and
$$\delta_j\bigl(\big[\overline{Y}_h\big]^{\frac{1}{n!}}\bigr)=\begin{cases}
[\varepsilon]^{\frac{1}{n!}} \big[\overline{Y}_j\big]^{\frac{1}{n!}} & \hbox{{\rm if }} h=j\cr \big[\overline{Y}_h\big]^{\frac{1}{n!}} & \forall h=1,\ldots,b,\,
h\neq j.
\end{cases}$$

\begin{lemma}\label{lemma:Aminvariant} (1) The ring ${\bf A}^{+,\rm
geo}_{\widetilde{R}_n}$  is a direct factor of ${\bf A}^+_{\widetilde{R}_n}\widehat{\tensor}_{\cO_n} \WW\bigl(\widetilde{\bf E}^+_{\cO_{\overline{K}}}\bigr)$. They
are equal for $\widetilde{R}=\widetilde{R}^{(0)}$. \smallskip

(2) The maps ${\bf A}^{+,\rm
geo}_{\widetilde{R}_n}/\bigl(p,P_\pi(Z)\bigr)\to \widetilde{\bf
E}^+_{\overline{R}}/\bigl(P_\pi(Z)\bigr)$ and ${\bf
A}_m\bigl(\widetilde{R}_n\bigr) \lra \WW_m\bigl(\widetilde{\bf
E}^+_{\overline{R}}\bigr)/\left(p^m,\sum_{i=0}^{p-1}\bigl[\varepsilon\bigr]^{i
p^{m-1}}\right)$ for $m\in\N$ are injective. Moreover, $\varphi^m$
induces an isomorphism ${\bf A}^{+,\rm geo}_{\widetilde{R}_n}/
\left(p^m,P_\pi(Z)\right) \lra {\bf
A}_{m}\bigl(\widetilde{R}_n\bigr)$.\smallskip

(3) The subring ${\bf A}^{+,\rm geo}_{\widetilde{R}_n}$  of
$\WW\bigl(\widetilde{\bf E}^+_{\overline{R}}\bigr)$ is stable
under the action of the group $\widetilde{\Gamma}_R$ for every
$n$. Moreover, for $n=1$ the induced action
of\/~$\widetilde{\Gamma}_R$ on ${\bf
A}_m\bigl(\widetilde{R}\bigr)$ is trivial.\smallskip

(4) The ring and $\widetilde{\Gamma}_R$-module ${\bf A}^{+,\rm geo}_{\widetilde{R}_n}/\left(p^m,\sum_{i=0}^{p-1}\bigl[\varepsilon\bigr]^{ip^{m-1}}\right)$ is a
direct factor in  ${\bf A}_m\bigl(\widetilde{R}\bigr)\otimes_{{\bf A}_m\bigl(\widetilde{R}^{(0)}\bigr)} {\bf A}^{+,\rm
geo}_{\widetilde{R}_n^{(0)}}/\left(p^m,\sum_{i=0}^{p-1}\bigl[\varepsilon\bigr]^{ip^{m-1}}\right)$.

\end{lemma}
\begin{proof} Without loss of generality in proving  (1), (2) and the first part of (3) it
suffices to consider the case $\widetilde{R}_n=\widetilde{R}$.

(1) The argument is as in \ref{cor:injstrictinvariants}(3).

(2) Since the map ${\bf A}^{+,\rm geo}_{\widetilde{R}}\to \WW\bigl(\widetilde{\bf E}^+_{\overline{R}}\bigr)$ modulo $(p,Z)$ is the map $B/p B \to \overline{R}/p
\overline{R}$, it is injective as proven above. Since $(p,Z)$ is a regular sequence in ${\bf A}^{+,\rm geo}_{\widetilde{R}}$ and $\WW\bigl(\widetilde{\bf
E}^+_{\overline{R}}\bigr)$, also $\bigl(p,P_\pi(Z)\bigr)$ is a regular sequence. Note that $P_\pi(Z)$ and $q^\prime=\sum_{i=0}^{p-1} [\varepsilon]^{\frac{i}{p}}$
generate the same ideal in $\WW\bigl(\widetilde{\bf E}^+_{\cO_{\overline{K}}}\bigr)$ by \ref{lemma:KerTheta} so that also $(p,q^\prime)$ is a regular sequence. We
conclude that the map ${\bf A}^{+,\rm geo}_{\widetilde{R}}\to \WW\bigl(\widetilde{\bf E}^+_{\overline{R}}\bigr)$ is injective modulo $\bigl(p^m,q^\prime\bigr)$.
Frobenius to the $m$-th power defines an isomorphism $\WW_m\bigl(\widetilde{\bf E}^+\bigr)/\left(q^\prime\right) \cong \WW_m\bigl(\widetilde{\bf
E}^+\bigr)/\left(\varphi^m\bigl(q^\prime\bigr)\right)$ and  an isomorphism ${\bf A}^{+,\rm geo}_{\widetilde{R}}/\left(p^m,q^\prime\right)\cong \varphi^m\left({\bf
A}^{+,\rm geo}_{\widetilde{R}}\right)/ \left(p^m,\varphi^m\bigl(q^\prime\bigr)\right)$. The second claim follows.\smallskip

(3.a) Recall from \ref{lemma:tildeRn} that there exists a unique chain of $\cO$-algebras
$${\bf A}^+_{\widetilde{R}^{(0)}}\subset {\bf
A}^+_{\widetilde{R}^{(1)}}\subset \ldots \subset {\bf A}^+_{\widetilde{R}^{(n)}}={\bf A}^+_{\widetilde{R}}$$lifting $R^{(0)}\subset R^{(1)}\subset \cdots \subset
R^{(n)}=R$ modulo $P_\pi(Z)$. Since the subgroup $\widetilde{\Gamma}_R\subset \widetilde{\Gamma}_{R^{(0)}}$ stabilizes ${\bf A}^+_{\widetilde{R}^{(0)}}$ and acts
trivially on the chain $R^{(0)}\subset R^{(1)}\subset \cdots \subset  R^{(n)}=R$, one proves by induction on $i$ that it stabilizes $ {\bf
A}^+_{\widetilde{R}^{(i)}}$ for every $i$ by uniqueness. Hence, it stabilizes ${\bf A}^+_{\widetilde{R}}$.\smallskip

(3.b) We prove the second part of claim (3)  by induction on $i$ in $\widetilde{R}^{(i)}$. Since ${\bf A}^+_{\widetilde{R}^{(0)}}$ is the $(Z,p)$-adic completion of
$\cO\bigl[P'\bigr]$, then ${\bf A}_m\bigl(\widetilde{R}^{(0)}\bigr)$ satisfies
$${\bf A}_m\bigl(\widetilde{R}^{(0)}\bigr)\cong
\frac{\WW\bigl(\widetilde{\bf E}^+_{\cO_{\overline{K}}}\bigr)}{\left(p^m,\varphi^m(q^\prime)\right)} \bigl[\big[\overline{X}_1\big]^{p^m
},\ldots,\big[\overline{X}_a\big]^{p^m}, \big[\overline{Y}_1\big]^{p^m},\ldots,\big[\overline{Y}_b\big]^{p^m}\bigr]/\bigl( \big[\overline{X}_1\big]^{p^m}\cdots
\big[\overline{X}_a\big]^{p^m}-Z^{\alpha p^m}\bigr).$$Since $[\varepsilon]^{p^m}-1=\varphi^m(q^\prime) \bigl([\varepsilon]^{p^{m-1}}-1\bigr) $, it follows from the
definition of the action of $\widetilde{\Gamma}_{R^{(0)}}$ that the latter acts trivially on ${\bf A}_m\bigl(\widetilde{R}^{(0)}\bigr)$. Assume that
$\widetilde{\Gamma}_{R^{(i)}}$ acts trivially on ${\bf A}_m\bigl(\widetilde{R}^{(i)}\bigr)$. By construction and the argument in (3.a) we have that ${\bf
A}_m\bigl(\widetilde{R}^{(i+1)}\bigr)$ is obtained from ${\bf A}_m\bigl(\widetilde{R}^{(i)}\bigr)$ taking  a localization, the completion with respect to an ideal
or an \'etale extension. In the first two cases $\widetilde{\Gamma}_{R^{(i+1)}}$ acts  trivially on ${\bf A}_m\bigl(\widetilde{R}^{(i+1)}\bigr)$. In the last case
we remark that $\widetilde{\Gamma}_{R^{(i+1)}}$, acting on ${\bf A}_m\bigl(\widetilde{R}^{(i+1)}\bigr)$, acts trivially on ${\bf
A}_m\bigl(\widetilde{R}^{(i)}\bigr)$ by assumption. Moreover, the action on ${\bf A}^+_{\widetilde{R}^{(i+1)}}/(p,Z)\cong R^{(i+1)}/(\pi)$ is trivial and, hence, it
is trivial on ${\bf A}_m\bigl(\widetilde{R}^{(i+1)}\bigr)/(p,Z)$. Since ${\bf A}_m\bigl(\widetilde{R}^{(i+1)}\bigr)$ is $(p,Z)$-adically complete and separated, we
conclude that $\widetilde{\Gamma}_{R^{(i+1)}}$ acts trivially on ${\bf A}_m\bigl(\widetilde{R}^{(i+1)}\bigr)$ as well. This concludes the proof of (3).\smallskip

(4) Consider the map $${\bf A}_m\bigl(\widetilde{R}\bigr)\otimes_{{\bf A}_m\bigl(\widetilde{R}^{(0)}\bigr)} {\bf A}^{+,\rm
geo}_{\widetilde{R}_n^{(0)}}/\left(p^m,\sum_{i=0}^{p-1}\bigl[\varepsilon\bigr]^{ip^{m-1}}\right) \lra {\bf A}^{+,\rm
geo}_{\widetilde{R}_n}/\left(p^m,\sum_{i=0}^{p-1}\bigl[\varepsilon\bigr]^{ip^{m-1}}\right).$$Due to (1) and (2) it suffices to prove that ${\bf
A}^+_{\widetilde{R}_n}$ is a direct factor of $\varphi^m\bigl({\bf A}^+_{\widetilde{R}}\bigr) \otimes_{\varphi^m\bigl({\bf A}^+_{\widetilde{R}^{(0)}}\bigr)} {\bf
A}^+_{\widetilde{R}^{(0)}_n}$. Due to \ref{lemma:A+R} it suffices to show that $\widetilde{R}_n$ is isomorphic to $\varphi^m\bigl(\widetilde{R}\bigr)
\otimes_{\varphi^m\bigl(\widetilde{R}^{(0)}\bigr)} \widetilde{R}^{(0)}_n$. Arguing as in (3.a) it suffices to show this for $\widetilde{R}_n=\widetilde{R}_n^{(0)}$
and this is clear.
\end{proof}

\

Define ${\bf A}_{\widetilde{R}_n,\rm cris}^{+,\rm geo}$ and ${\bf A}_{\widetilde{R}_{n,\rm max}}^{+,\rm geo}$ as in \ref{def:RlogRmax} using ${\bf A}^{+,\rm
geo}_{\widetilde{R}_n}$  instead of $\widetilde{R}$. Define~${\bf A}^{\rm cris}_{\rm log}\bigl(\widetilde{R}_n\bigr)$ as the $p$--adic completion of the logarithmic
divided power envelope  $\left({\bf A}^+_{\widetilde{R}_n} \tensor_{\WW(k)} \widetilde{R}\right)^{\rm logDP}$. Let ${\bf A}^{\rm max}_{\rm
log}\bigl(\widetilde{R}_n\bigr)$ be the $p$--adic completion of the $\bigl({\bf A}^+_{\widetilde{R}_n} \tensor_{\WW(k)} \widetilde{R}\bigr)^{\rm log}$-subalgebra of
$\bigl({\bf A}^+_{\widetilde{R}_n} \tensor_{\WW(k)} \widetilde{R}\bigr)^{\rm log}\bigl[p^{-1}\bigr]$ generated by
$p^{-1}\Ker\bigl(\Theta_{\widetilde{R},\log}'\bigr)$. Similarly, define~${\bf A}^{\rm geo, cris}_{\rm log}\bigl(\widetilde{R}_n\bigr)$ and~${\bf A}^{\rm geo,
max}_{\rm log}\bigl(\widetilde{R}_n\bigr)$ using ${\bf A}^{+,\rm geo}_{\widetilde{R}_n} \tensor_{\WW(k)} \widetilde{R}$ instead.

Recall that we also have rings $\widetilde{R}_{\log}^{\rm geo, cris}$ and $\widetilde{R}_{\log}^{\rm geo, max}$ defined before \ref{thm:geometricacyclicity}. Then,

\begin{lemma}\label{lemma:AlogcrysAlogmax} We
have isomorphisms of $\cG_R$-modules $${\bf A}_{\widetilde{R}_n,{\rm cris}}^{+}\left\{\langle u-1, v_2-1,\ldots,v_a-1,w_1-1,\ldots,w_b-1\rangle \right\} \lra {\bf
A}^{\rm cris}_{\rm log}\bigl(\widetilde{R}_n\bigr)$$ and $$  {\bf A}_{\widetilde{R}_n,{\rm max}}^{+}\left\{\frac{u-1}{p},
\frac{v_2-1}{p},\ldots,\frac{v_a-1}{p},\frac{w_1-1}{p},\ldots,\frac{w_b-1}{p} \right\}  \lra {\bf A}^{\rm max}_{\rm log}\bigl(\widetilde{R}_n\bigr)$$ and similarly
for the geometric counterparts $$ {\bf A}^{+,\rm geo}_{\widetilde{R}_n,{\rm cris}}\left\{\langle u-1, v_2-1,\ldots,v_a-1,w_1-1,\ldots,w_b-1\rangle \right\}  \lra
{\bf A}^{\rm geo, cris}_{\rm log}\bigl(\widetilde{R}_n\bigr)$$and $$ {\bf A}^{+,\rm geo}_{\widetilde{R}_n,{\rm max}}\left\{\frac{u-1}{p},\frac{v_2-1}{p},\ldots,
\frac{v_a-1}{p},\frac{w_1-1}{p},\ldots,\frac{w_b-1}{p} \right\}  \lra {\bf A}^{\rm geo, max}_{\rm log}\bigl(\widetilde{R}_n\bigr).$$For $n=0$, the natural morphsims
$$\widetilde{R}_{\rm cris}\left\{\langle u-1, v_2-1,\ldots,v_a-1,w_1-1,\ldots,w_b-1\rangle \right\}  \lra {\bf A}^{\rm cris}_{\rm log}\bigl(\widetilde{R}\bigr)$$
and $$  \widetilde{R}_{\rm max}\left\{\frac{u-1}{p}, \frac{v_2-1}{p},\ldots,\frac{v_a-1}{p},\frac{w_1-1}{p},\ldots,\frac{w_b-1}{p} \right\}  \lra {\bf A}^{\rm
max}_{\rm log}\bigl(\widetilde{R}\bigr)$$are isomorphisms and similarly for the geometric counterparts $$ \widetilde{R}_{\log}^{\rm geo, cris}\left\{\langle
v_2-1,\ldots,v_a-1,w_1-1,\ldots,w_b-1\rangle \right\}  \lra {\bf A}^{\rm geo, cris}_{\rm log}\bigl(\widetilde{R}\bigr)$$and $$ \widetilde{R}_{\log}^{\rm geo, cris}
\left\{\frac{v_2-1}{p},\ldots, \frac{v_a-1}{p},\frac{w_1-1}{p},\ldots,\frac{w_b-1}{p} \right\}  \lra {\bf A}^{\rm geo, max}_{\rm log}\bigl(\widetilde{R}\bigr).$$
\end{lemma}
\begin{proof} The first claims are proven as in \ref{lemma:structureAlog} and in \ref{lemma:structAlog}.

For  $n=0$ we certainly have natural maps as stated. To prove that they are isomorphisms we remark that the images of $\widetilde{R}_{\log}^{\rm
geo}$ and of ${\bf A}^{+,\rm geo}_{\widetilde{R}}$ in $\widehat{\Rbar}$ via $\Theta_{\widetilde{R},\log}$  coincide with $\widehat{R\cO}_{\Kbar}$ by
\ref{cor:injstrictinvariants}(1) and \ref{lemma:Aminvariant}(1). Thus, the given maps define isomorphisms on the graded rings and, hence, are isomorphisms.

\end{proof}

Set  ${\bf A}_m(\cO):=\WW\bigl(\widetilde{\bf E}_{\cO_\Kbar}^+\bigr)/\left(p^m,\sum_{i=0}^{p-1}\bigl[\varepsilon\bigr]^{ip^{m-1}}\right)$. For every $m$ and
$n\in\N$ define $$E_{n,m}=\left\{(\alpha_1,\ldots,\alpha_a,\beta_1,\ldots,\beta_b)\in  \frac{1}{n!} \N^{a+b} \cap [0,p^m)^{a+b} \vert \alpha_1 \cdots
\alpha_a=0\right\},$$i.e., at least one of the $\alpha_i$'s is $0$. Set $E_m:=\cup_n E_{n,m}$. For
$\bigl(\underline{\alpha},\underline{\beta}\bigr)=\bigl(\alpha_1,\ldots,\alpha_a,\beta_1,\ldots,\beta_b\bigr)\in E_m$ write
$$\big[\overline{X}\big]^{\underline{\alpha}}\big[\overline{Y}\big]^{\underline{\beta}}:=\prod_{i=1}^a
\big[\overline{X}_i\big]^{\alpha_i}\prod_{j=1}^b \big[\overline{Y}_j\big]^{\beta_j}.$$Define
$${\bf X}_{n,m}:=\oplus_{(\underline{\alpha},\underline{\beta})\in E_{n,m}}
{\bf A}_m(\cO) \big[\overline{X}\big]^{\underline{\alpha}}\big[\overline{Y}\big]^{\underline{\beta}}.$$They are endowed with an action of
$\widetilde{\Gamma}_{R^{(0)}}$ where the action on ${\bf A}_m(\cO)$ is trivial and the action on $
\big[\overline{X}\big]^{\underline{\alpha}}\big[\overline{Y}\big]^{\underline{\beta}}$ has been described above. For $h$ and $i\in\{1,\ldots,a\}$ with $i\neq h$,
consider the $\widetilde{\Gamma}_{R^{(0)}}$-submodules
$${\bf X}^{(h,i)}_{n,m}:=\oplus_{(\underline{\alpha},\underline{\beta})\in E_{n,m}, \alpha_h=0,
\alpha_1,\ldots,\alpha_{i-1}\in\N , \alpha_i\not\in\N} {\bf A}_m(\cO) \big[\overline{X}\big]^{\underline{\alpha}}\big[\overline{Y}\big]^{\underline{\beta}}
$$and for $i\in \{a+1,\ldots,a+b\}$ set $${\bf X}^{(h,i)}_{n,m}:=\oplus_{(\underline{\alpha},\underline{\beta}),\alpha_h=0,\alpha_1,\ldots,\alpha_a,
\beta_1,\ldots,\beta_{i-a-1}\in \N, \beta_{i-a}\not\in\N } {\bf A}_m(\cO) \big[\overline{X}\big]^{\underline{\alpha}}\big[\overline{Y}\big]^{\underline{\beta}}.$$In
particular we have  ${\bf X}_{n,m}=\oplus_{h\in \{1,\ldots,a\}, i\in \{1,\ldots,d\}, i\neq h}{\bf X}^{(h,i)}_{n,m}$.

\begin{lemma}\label{lemma:vanishingkercokergammai-1} For every $m\in\N$,  every $1\leq i,h\leq a$ with $h\neq i$ and every $1\leq j \leq
b$ the kernel and the cokernel of the following maps are  annihilated by $\bigl[\varepsilon\bigr]^{\frac{1}{p}}-1$:\smallskip

1) $\gamma_i-1$ on ${\bf X}_{n,m}^{(h,i)}$ for $i>1$; \smallskip

2) $\gamma_h-1$  on ${\bf X}_{n,m}^{(h,1)}$; \smallskip

3) $\delta_j-1$  on $ {\bf X}_{n,m}^{(h,a+j)}$.

\end{lemma}
\begin{proof} Notice that
$\bigl(\gamma_i-1\bigr) \big[\overline{X}\big]^{\underline{\alpha}}\big[\overline{Y}\big]^{\underline{\beta}}=
\bigl([\varepsilon]^{\alpha_i-\alpha_1}-1\bigr)\big[\overline{X}\big]^{\underline{\alpha}}\big[\overline{Y}\big]^{\underline{\beta}}$ and $\bigl(\delta_j-1\bigr)
\big[\overline{X}\big]^{\underline{\alpha}}\big[\overline{Y}\big]^{\underline{\beta}}=
\bigl([\varepsilon]^{\beta_j}-1\bigr)\big[\overline{X}\big]^{\underline{\alpha}}\big[\overline{Y}\big]^{\underline{\beta}}$. The assumption (1)  (resp.~(2),
resp.~(3)) amounts to require that $\alpha_i-\alpha_1$ (resp.~$-\alpha_h$, $\beta_j$)  are rational numbers of the form $c:=\frac{r}{s}$ for some $r$ and $s\in \Z$
with $r$ and $s$ coprime and $s>1$. If $s$ is not a power of $p$, then $[\varepsilon]^{c}$  is a primitive $s$-th root of unity $\zeta_s$ in
$\WW\bigl(\widetilde{\bf E}^+_{\cO_{\overline{K}}}\bigr)/\bigl(p,[\overline{\pi}]\bigr)\cong \cO_{\overline{K}}/\pi \cO_{\overline{K}}$. Since $\zeta_s-1$ is a unit
in $\overline{\bf F}_p$, we conclude that $[\varepsilon]^{c}-1$ is a unit. In this case $\gamma_i-1$ (resp.~$\gamma_h-1$,  resp.~$\delta_j-1$) is a bijection on
${\bf A}_m(\cO) \big[\overline{X}\big]^{\underline{\alpha}}\big[\overline{Y}\big]^{\underline{\beta}}$.

If on the other hand  $s$ is a power of $p$ it follows from \cite[lemme 12]{andreatta_brinon} that $[\varepsilon]^{c}-1$ divides
$[\varepsilon]^{\frac{1}{p}}-1$. In particular, the cokernel and the kernel of $\gamma_i-1$ (resp.~$\gamma_h-1$, resp.~$\delta_j-1$)  on ${\bf A}_m(\cO)
\big[\overline{X}\big]^{\underline{\alpha}}\big[\overline{Y}\big]^{\underline{\beta}}$ is annihilated by $[\varepsilon]^{c}-1$ and, hence, by
$[\varepsilon]^{\frac{1}{p}}-1$ as well.
\end{proof}

Let us recall that we denoted by $A_{\rm cris}(\cO_K)$ and $A_{\rm max}(\cO_K)$ the classical period rings.

\begin{lemma}\label{lemma_:gammaionvi} For every $i\in\{1,\ldots,a\}$
and $N\in\N$, we have
$$(\gamma_i-1)\big((v_i-1)^{[N]}\big)\in (1-[\varepsilon]^{\frac{1}{p}})\sum\limits_{m=0}^{N-1} A_{\rm cris}(\cO_K)
(v_i-1)^{[m]}$$and $$ (\gamma_i-1)\left(\frac{(v_i-1)^{N}}{p^N}\right)\in (1-[\varepsilon]^{\frac{1}{p}})
\sum\limits_{m=0}^{N-1} A_{\rm max}(\cO_K)
\frac{(v_i-1)^m}{p^m}.$$Similarly, for every $j\in\{1,\ldots,b\}$ and $N\in\N$, we have
$$(\delta_j-1)\big(w_j^{[N]}\big)\in (1-[\varepsilon]^{\frac{1}{p}})\sum\limits_{m=0}^{N-1} A_{\rm cris}(\cO_K)
(w_j-1)^{[m]}$$ and $$ (\delta_j-1)\left(\frac{(w_j-1)^{N}}{p^N}\right)\in (1-[\varepsilon]^{\frac{1}{p}})
\sum\limits_{m=0}^{N-1} A_{\rm max}(\cO_K)
\frac{(w_j-1)^m}{p^m}.$$
\end{lemma}

\begin{proof}
We prove the first statement. The second one is  similar. We show how to deal with  $v_i$ and $\gamma_i$,
the computations for $w_j$ and $\delta_j$ are the
same.  For every $i=1,\ldots,a$ recall that $v_i:=\frac{\bigl[\overline{X}_i\bigr]}{\widetilde{X}_i}$ so that
$\gamma_i(v_i)=v_i+ \bigl([\varepsilon]-1\bigr) v_i$.
In particular, $\gamma_i(v_i-1)=(v_i-1)+ \bigl([\varepsilon]-1\bigr) v_i$. Recall that $[\varepsilon]-1=\bigl([\varepsilon]^{\frac{1}{p}}-1\bigr) \lambda$ for
$\lambda\in \WW(\widetilde{\bf E}^+_{\overline{K}})$ mapping to zero via $\Theta$; cf.~\cite[5.1.1]{Fontaineperiodes}. Thus, for every $N\in\N$, we have
\begin{align*}
\gamma_i\big((v_i-1)^{[N]}\big) &= \big((v_i-1)+
\bigl([\varepsilon]-1\bigr)v_i\big)^{[N]} \\
 &= \sum_{m=0}^N\bigl(v_i-1\bigr)^{[N-m]} \big([\varepsilon]-1\big)^{[m]}
v_i^{m}\\ &=
(v_i-1)^{[N]}+\sum_{m=1}^N([\varepsilon]^{\frac{1}{p}}-1)^m
v_i^m\lambda^{[m]}(v-i)^{[N-m]}.
\end{align*}

Similarly
$$\gamma_i\big(\frac{(v_i-1)^{N}}{p^N}\big)=\big(\frac{(v_i-1)^{N}}{p^N}\big)+
\sum_{m=1}^N \left(\begin{array}{c} N \cr m \end{array}
\right)([\varepsilon]^{\frac{1}{p}}-1)^m v_i^m
\frac{\lambda^m}{p^m}\frac{(v-i)^{N-m}}{p^{N-m}}.$$
\end{proof}

In particular, it follows that the rings ${\bf A}_m\left\{\langle u-1, v_2-1,\ldots,v_a-1,w_1-1,\ldots,w_b-1\rangle \right\}$ and ${\bf A}_m \left\{\frac{u-1}{p},
\frac{v_2-1}{p},\ldots,\frac{v_a-1}{p},\frac{w_1-1}{p},\ldots,\frac{w_b-1}{p}\right\}$ are endowed with an action of $\widetilde{\Gamma}_R$.

Note that ${\bf X}^{(h,i)}_{n,m}$ and ${\bf X}_{n,m}$ are modules over $${\bf A}_m(\widetilde{R}^{(0)})={\bf
A}_m(\cO)\left[\big[\overline{X}_1\big]^{\frac{1}{n!}},\ldots,\big[\overline{X}_a\big]^{\frac{1}{n!}},
\big[\overline{Y}_1\big]^{\frac{1}{n!}},\ldots,\big[\overline{Y}_b\big]^{\frac{1}{n!}}\right]/\bigl( \big[\overline{X}_1\big]^{\frac{1}{n!}}\cdots
\big[\overline{X}_a\big]^{\frac{1}{n!}}-Z^{\frac{\alpha}{n!}}\bigr)  ,$$where the equality follows from \ref{lemma:Aminvariant}(1). Let ${\bf
X}^{(h,i)}_{n,m}\bigl(\widetilde{R}\bigr):={\bf X}^{(h,i)}_{n,m}\otimes_{{\bf A}_m(\widetilde{R}^{(0)})} {\bf A}_m\bigl(\widetilde{R}\bigr)$ (resp.~${\bf
X}_{n,m}\bigl(\widetilde{R}\bigr):={\bf X}_{n,m} \otimes_{{\bf A}_m(\widetilde{R}^{(0)})} {\bf A}_m\bigl(\widetilde{R}\bigr)$). The next proposition will allow us
to reduce the computation of the Galois cohomology of ${\rm A}_{\rm log,\infty}^{\rm cris}$ to the cohomology of ${\bf A}_{\rm log}^{\rm geo,
cris}\bigl(\widetilde{R}\bigr)$ and the cohomology of another module that will be computed in \ref{cor:cohoXvanishes}.

\begin{proposition}\label{prop:noetherianize} For every $m$ the ${\bf A}_m\bigl(\widetilde{R}\bigr)$-module
${\rm A}_{\rm log,\infty}^{\rm cris}/\left(p^m,\sum_{i=0}^{p-1}\bigl[\varepsilon\bigr]^{ip^{m-1}}\right)$ is a direct summand, as $\widetilde{\Gamma}_R$-module, of
$${\bf
A}_{\rm log}^{\rm geo, cris}\bigl(\widetilde{R}\bigr)/\left(p^m,\sum_{i=0}^{p-1}\bigl[\varepsilon\bigr]^{ip^{m-1}}\right) \oplus \lim_{n\to \infty} {\bf
X}_{n,m}\bigl(\widetilde{R}\bigr)\left\{\langle u-1, v_2-1,\ldots,v_a-1,w_1-1,\ldots,w_b-1\rangle \right\}.$$Similarly,  $ {\rm A}_{\rm log,\infty}^{\rm
max}/\left(p^m,\sum_{i=0}^{p-1}\bigl[\varepsilon\bigr]^{ip^{m-1}}\right)$ is a direct summand of
$${\bf A}_{\rm log}^{\rm geo, max}\bigl(\widetilde{R}\bigr)/\left(p^m,\sum_{i=0}^{p-1}\bigl[\varepsilon\bigr]^{ip^{m-1}}\right)
\oplus \lim_{n\to \infty}{\bf X}_{n,m}\bigl(\widetilde{R}\bigr)\left\{\frac{u-1}{p}, \frac{v_2-1}{p},\ldots,\frac{v_a-1}{p},\frac{w_1-1}{p},\ldots,\frac{w_b-1}{p}
\right\}.$$
\end{proposition}
\begin{proof}
First of all we claim that for every $n$ and $m\in\N$ the natural maps ${\bf A}_{\rm log}^{\rm geo, cris}\bigl(\widetilde{R}_n\bigr)\lra {\rm A}_{\rm
log,\infty}^{\rm cris} $ and ${\bf A}_{\rm log}^{\rm geo, max}\bigl(\widetilde{R}_n\bigr) \lra {\rm A}_{\rm log,\infty}^{\rm max}$ are injective modulo $p^m$ and
induce an isomorphism onto passing to the direct limit $\ds\lim_{n\to\infty}$. Since in both rings $p$ is not a zero divisor, it suffices to prove the claim for
$m=1$. Due to \ref{cor:Alogcrysmodp} and \ref{lemma:AlogcrysAlogmax} it suffices to show that the maps ${\bf A}_{\widetilde{R}_n,\rm cris}^{+,\rm geo}\{\langle u-1
\rangle \} \lra {\rm A}_{\rm log,\infty}^{\rm cris,\nabla}$ and ${\bf A}_{\widetilde{R}_n,\rm max}^{+,\rm geo} \{\frac{ u-1}{p}  \}\lra {\rm A}_{\rm
log,\infty}^{\rm max,\nabla}$ are injective modulo $p$ and induce an isomorphism onto passing to the direct limit $\ds\lim_{n\to\infty}$. Due to
\ref{cor:Alogcrysnablamodp} it suffices to show that ${\bf A}^{+,\rm geo}_{\widetilde{R}_n}/(p,\xi)\to \widetilde{\bf
E}^+_{R_\infty\cO_{\overline{K}}}/(\overline{p})$ is injective and induces an isomorphism onto passing to the direct limit over all $n\in\N$. The injectivity
follows from \ref{lemma:Aminvariant}(2). The image in $\widetilde{\bf E}^+_{R_\infty\cO_{\overline{K}}}/(\overline{p})=R_\infty\cO_{\overline{K}}/(p)$ is the image
of $R_\infty\cO_{\overline{K}}$ by construction. Due to \ref{lemma:A+R}, the union of such images over all $n\in\N$ is the whole $R_\infty\cO_{\overline{K}}/(p)$.

We are then left to prove that for every $m$ and $n\in\N$ the quotient of  ${\bf A}_{\rm log}^{\rm geo, cris}\bigl(\widetilde{R}_n\bigr)$ modulo
$\left(p^m,\sum_{i=0}^{p-1}\bigl[\varepsilon\bigr]^{ip^{m-1}}\right)$ is, as $\widetilde{\Gamma}_R$-modules, a direct summand in
$${\bf
A}_{\rm log}^{\rm geo, cris}\bigl(\widetilde{R}\bigr)/\left(p^m,\sum_{i=0}^{p-1}\bigl[\varepsilon\bigr]^{ip^{m-1}}\right) \oplus  {\bf
X}_{n,m}\bigl(\widetilde{R}\bigr)\left\{\langle u-1, v_2-1,\ldots,v_a-1,w_1-1,\ldots,w_b-1\rangle \right\}$$and similarly for ${\bf A}_{\rm log}^{\rm geo,
max}\bigl(\widetilde{R}_n\bigr)$. Due to \ref{lemma:AlogcrysAlogmax} it suffices to show that ${\bf A}^{+,\rm geo}_{\widetilde{R}_n}/
\left(p^m,\sum_{i=0}^{p-1}\bigl[\varepsilon\bigr]^{i p^{m-1}}\right)$ is a direct summand of ${\bf A}^{+,\rm geo}_{\widetilde{R}}/
\left(p^m,\sum_{i=0}^{p-1}\bigl[\varepsilon\bigr]^{ip^{m-1}}\right) \oplus  {\bf X}_{n,m}\bigl(\widetilde{R}\bigr)$. Thanks to \ref{lemma:Aminvariant}(4) we may
replace $\widetilde{R}_n$ with   $\widetilde{R}_n^{(0)}$ and $\widetilde{R}$ with $\widetilde{R}^{(0)}$. The claim follows then from \ref{lemma:Aminvariant}(1).
\end{proof}

\begin{corollary}\label{cor:cohoXvanishes} For every $i\in\N$ and every $m$ and $n\in\N$ the cohomology
groups $${\rm H}^i\left(\widetilde{\Gamma}_R, {\bf X}_{n,m}\bigl(\widetilde{R}\bigr)\left\{\langle u-1, v_2-1,\ldots,v_a-1,w_1-1,\ldots,w_b-1\rangle
\right\}\right)$$and ${\rm H}^i\left(\widetilde{\Gamma}_R, {\bf X}_{n,m}\left\{\frac{ u-1}{p},
\frac{v_2-1}{p},\ldots,\frac{v_a-1}{p},\frac{w_1-1}{p},\ldots,\frac{w_b-1}{p} \right\}\right)$ are annihilated by
$\bigl(\bigl[\varepsilon\bigr]^{\frac{1}{p}}-1\bigr)^2$. The same holds if we take the direct limit over all $n\in\N$.
\end{corollary}
\begin{proof}
We prove the first statement. The second one is similar and left to the reader. Using the direct sum decomposition ${\bf X}_{n,m}=\oplus_{h\in \{1,\ldots,a\}, i\in
\{1,\ldots,d\}, i\neq h}{\bf X}^{(h,i)}_{n,m}$ it suffices to prove the statement for ${\bf X}^{(h,i)}_{n,m}$ instead of ${\bf X}_{n,m}$. Apply the Hochschild-Serre
spectral sequence associated to the exact sequence of groups:
$$0\to\widehat{\Z}\gamma_h\to\widetilde{\Gamma}_{R}\to\widetilde{\Gamma}_{R}/\widehat{\Z}\gamma_h\to
0$$for $2\leq i\leq a$ (resp. $0\to\widehat{\Z}\delta_j\to\widetilde{\Gamma}_{R}\to\widetilde{\Gamma}_{R}/\widehat{\Z}\delta_j\to 0$ for $a+1\leq h\leq d$ with
$j=h-a$) with coefficients in  $\mathbf{X}_n^{(h)}\left\{\langle u-1, v_2-1,\ldots,v_a-1,w_1-1,\ldots,w_b-1\rangle \right\}$. The cohomology of $\widehat{\Z}
\gamma_i$ (resp.~$\widehat{\Z} \delta_j$) is zero in degrees $\geq 2$ and is computed as the kernel of $\gamma_i-1$ (resp.~$\delta_j-1$) in degree $0$ and as the
cokernel of $\gamma_i-1$ (resp.~$\delta_j-1$) in degree $1$. It follows from \ref{lemma:vanishingkercokergammai-1} and \ref{lemma_:gammaionvi} arguing as in
\cite[Lemme 15]{andreatta_brinon} that kernel and cokernel of $\gamma_i-1$ on ${\bf X}_{n,m}^{(i)}\left\{\langle u-1, v_2-1,\ldots,v_a-1,w_1-1,\ldots,w_b-1\rangle
\right\}$ for $2\leq i\leq a$ and of $\delta_j-1$  for $1\leq j\leq b$ are annihilated by $\bigl[\varepsilon\bigr]^{\frac{1}{p}}-1$. The result follows.
\end{proof}

We also have the following analogue on graded rings. In \S\ref{lemma:Rinftyflat} we have proven that the $R$-subalgebra $R_\infty^o$ of $R_\infty$, generated by the
elements $ X_\alpha Y_\beta:=\prod_{i=2}^a X_i^{{\alpha_i}} \prod_{j=1}^b Y_j^{{\beta_j}}$ for non negative rational numbers $\alpha_i$, $\beta_j$, is free as
$R$-module and it has the property that $\pi^\alpha R_\infty\subset R_\infty^o$. Write ${\bf X}:=\sum_i {\bf X}^{(i)}$ where ${\bf X}^{(i)}$ is the $\widehat{R
\cO}_{\Kbar}$-submodule of $\widehat{R_\infty \cO}_{\Kbar}$ generated by $X_\alpha Y_\beta$ with $\alpha_2,\ldots,\alpha_{i-1}\in \N , \alpha_i\not\in \N$ if $1\leq
i\leq a$ and  with $\alpha_2,\ldots,\alpha_a, \beta_1,\ldots,\beta_{i-a-1}\in\N,\beta_{i-a}\not\in\N$ if $i\in \{a+1,\ldots,a+b\}$.  We have

\begin{corollary}\label{cor:cohoBHTinftyvanishes}
For every $i\in\N$ the cohomology groups $${\rm H}^i\left(\widetilde{\Gamma}_R, \widehat{\oplus}_{\underline{n}\in \N^{d+1}} {\bf X} \xi^{[n_0]}
(v_1-1)^{[n_1]}\cdots (v_a-1)^{[n_a]}(w_1-1)^{[n_{a+1}]}\cdots (w_b-1)^{[n_d]}\right)$$are annihilated by $\bigl(\epsilon_p-1\bigr)^2$. The morphism
$${\rm H}^i\left(\widetilde{\Gamma}_R, \widehat{\oplus}_{\underline{n}\in \N^{d+1}}
\widehat{R  \cO}_{\Kbar}  \xi^{[n_0]} (u-1)^{[n_1]} (v_2-1)^{[n_2]} \cdots (v_a-1)^{[n_a]}\cdots (w_b-1)^{[n_d]}\right)\lra {\rm H}^i\left(G_R, {\rm Gr}^\bullet
{\rm A}^{\rm cris}_{\rm log}\right)$$has kernel and cokernel annihilated by ${m}_{\overline{R}}\pi^{\alpha}\bigl(\epsilon_p-1\bigr)^{2}$ for every $i\in\N$.
\end{corollary}
\begin{proof} The first statement follows as in \ref{cor:cohoXvanishes}.

For the second statement note that multiplication by $\pi^\alpha$ on $\widehat{R_\infty \cO}_{\Kbar}$ factors via $\widehat{R \cO}_{\Kbar} \oplus {\bf X}$ thanks to
the results of \S\ref{lemma:Rinftyflat} and \ref{lemma:HiHR}. We conclude using \ref{cor:cohoRbarmodp}(iii) and the Hochschild-Serre spectral sequence for the
subgroup $H_R\subset G_R$.

\end{proof}

\subsubsection{The cohomology of ${\bf
A}^{\rm geo, cris}_{\rm log}\bigl(\widetilde{R}\bigr)$ and of ${\bf A}^{\rm geo, max}_{\rm log}\bigl(\widetilde{R}\bigr)$}

In view of \ref{cor:cohoXvanishes} and \ref{prop:noetherianize}, to conclude the proof of Claims
\ref{thm:geometricacyclicity}(i)\&(ii) we are left to show that

\begin{proposition}\label{prop:cohA+Rloggeo} (1) For every $i$ and
$n,m\in\N$ the groups ${\rm H}^i\left(\widetilde{\Gamma}_R, {\bf A}^{\rm geo, cris}_{\rm log}\bigl(\widetilde{R}\bigr)/ p^m {\bf A}^{\rm geo, cris}_{\rm
log}\bigl(\widetilde{R}\bigr)\right)$ vanish if $i\geq d+1$, are annihilated by $\bigl([\varepsilon]-1\bigr)^d$ for $i\geq 1$ and contains $\widetilde{R}_{\rm
log}^{\rm geo, cris}/p^m \widetilde{R}_{\rm log}^{\rm geo, cris}$ for $i=0$ with cokernel annihilated by $\bigl([\varepsilon]-1\bigr)^d$. The map
$$
\widetilde{R}_{\rm log}^{\rm geo, max}\lra {\rm H}^0\left(\widetilde{\Gamma}_R, {\bf A}^{\rm geo, max}_{\rm log}\bigl(\widetilde{R}\bigr)\right)$$ is injective with
cokernel annihilated by $\bigl([\varepsilon]-1\bigr)^d$.\smallskip

(2)  For every $i$  the group
$${\rm H}^i\left(\widetilde{\Gamma}_R,
\widehat{\oplus}_{\underline{n}\in \N^{d+1}} \widehat{R \cO}_{\Kbar} \xi^{[n_0]} (u-1)^{[n_1]}(v_2-1)^{[n_2]} \cdots (v_a-1)^{[n_a]}(w_1-1)^{[n_{a+1}]}\cdots
(w_b-1)^{[n_d]} \right)$$vanishes if $i\geq d+1$, is annihilated by $\pi\bigl(\epsilon_p-1\bigr)^d$ for $i\geq 1$ and contains $\bigl(R\cO_{\Kbar}\bigr)
\widehat{\otimes}_{\cO_\Kbar} {\rm Gr}^\bullet\bigl(A_{\rm cris}\bigr) $ for $i=0$ with cokernel annihilated by $\bigl(\epsilon_p-1\bigr)^d$.

\end{proposition}
\begin{proof} We prove  statement (1) for ${\bf
A}^{\rm geo, cris}_{\rm log}\bigl(\widetilde{R}\bigr)$. The proof of statement (2) is similar and easier and is left to the reader. Define
$$K_m^{(i)}:=
\begin{cases} \widetilde{R}_{\rm log}^{\rm geo, cris}/p^m
\widetilde{R}_{\rm log}^{\rm geo, cris}\left\{\langle u-1,
 v_2-1,\ldots, v_{i-1}-1\rangle
 \right\}  & \forall 1\leq i\leq a \cr
\widetilde{R}_{\rm log}^{\rm geo, cris}/p^m \widetilde{R}_{\rm
log}^{\rm geo, cris}\left\{\langle u-1,
 v_2-1,\ldots,  v_a-1,w_1-1,\ldots,w_{i-a}-1\rangle
 \right\}  & \forall a+1\leq i\leq b .\cr
\end{cases}$$
Note that
$$\gamma_i(v_i-1)=\big([\varepsilon]-1\big)v_i+ (v_i-1), \qquad  \delta_j(w_j-1)=\big([\varepsilon]-1\big) w_j +(w_j-1)$$
and hence
\begin{align*}
(\gamma_i-1)\big((v_i-1)^{[m]}\big) &= \big(\big([\varepsilon]-1\big)v_i+ (v_i-1)\big)^{[m]}-(v_i-1)^{[m]}\\
 &= \sum\limits_{j=1}^m([\varepsilon]-1)^{[j]}v_i^j \big(v_i-1\big)^{[m-j]}\\
&= \big([\varepsilon]-1\big)\Big(v_i (v_i-1)^{[m-1]}+\sum\limits_{j=2}^m\beta_jv_i^j(v_i-1)^{[m-j]}\Big)
\end{align*}
with $\beta_j=\frac{\big([\varepsilon]-1\big)^{[j]}}{[\varepsilon]-1}\in\Ker(\theta)$ of $A_{\rm cris}(\cO_K)$ by \cite[Lemme 17]{andreatta_brinon}. Then,
$\gamma_i-1$ defines a $K_m^{(i)}$-linear homorphism on $K_m^{(i)}\langle v_i-1\rangle $ whose matrix with respect to
$\big(v_i-1,(v_i-1)^{[2]},\ldots,(v_i-1)^{[N]}\big)$ and $\big(1,(v_i-1),\ldots,(v_i-1)^{[N-1]}\big)$ is given by $\big([\varepsilon]-1\big)G_{n,N}^{(i)}$ with
$$G_{n,i}^{(N)}=\begin{pmatrix}
  v_i & \beta_2 v_i^2 & v_i^3\beta_3 & \cdots & \cdots & v_i^{N-1}\beta_{N-1} & \cdots  & v_i^N\beta_N\\
 0 & v_i & v_i^2\beta_2 & \ddots & \ddots & v_i^{N-2}\beta_{N-2} & \cdots  & v_i^{N-1}\beta_{N-1}\\
 \vdots & \ddots & v_i & \ddots  & \ddots &  \ddots & \ddots & v_i^{N-2}\beta_{N-2}\\
 \vdots & & \ddots & \ddots & \ddots & \ddots & \ddots & \vdots \\
 \vdots & & & \ddots & \ddots & \ddots & v_i^3\beta_3 & v_i^4\beta_4\\
 \vdots & & & & \ddots & \ddots & v_i^2 \beta_2  & v_i^3\beta_3\\
 \vdots & & & & & \ddots & v_i &  v_i^2 \beta_2 \\
 0 & \cdots & \cdots & \cdots & \cdots & \cdots & 0 & v_i
\end{pmatrix}.$$Since $v_i$ is invertible, it follows that
$G_{n,N}^{(i)}$ is an invertible matrix. This implies that the
cokernel of $\gamma_i-1$ on $K_m^{(i)}\langle v_i-1\rangle $ is
annihilated by $[\varepsilon]-1$ and that the kernel coincides with
$K_m^{(i)}$ up to a direct summand which is also annihilated by
$[\varepsilon]-1$.

A similar argument shows that $(\delta_j-1)$ defines a
$K_m^{(a+j)}$-linear homorphism on $K_m^{(a+j)}\langle
w_j-1\rangle $ whose matrix with respect to
$\big(w_j-1,(w_j-1)^{[2]},\ldots,(w_j-1)^{[N]}\big)$ and
$\big(1,(w_j-1),\ldots,(w_j-1)^{[N-1]}\big)$ is given by
$\big([\varepsilon]-1\big)$ times an invertible matrix. This
implies that also in this case the cokernel of $\gamma_i-1$ on
$K_m^{(j+a)}\langle w_j-1\rangle $ is annihilated by
$[\varepsilon]-1$ and that the kernel coincides with $K_m^{(j+a)}$ up
to a direct summand killed by $[\varepsilon]-1$. The conclusion
follows  proceeding by descending induction on $2\leq h\leq d$.

Note that $K_m^{(d)}\langle w_b-1\rangle={\bf A}^{\rm geo, cris}_{\rm log}\bigl(\widetilde{R}\bigr)/p^m {\bf A}^{\rm geo, cris}_{\rm log}\bigl(\widetilde{R}\bigr)$
due to \ref{lemma:AlogcrysAlogmax}. The cohomology with respect to $\widehat{\Z}\delta_b$ is zero in degrees $\geq 2$, is annihilated by $[\varepsilon]-1$ in degree
$1$ and is $K_m^{(d)}=K_m^{(d-1)}\langle w_{b-1}-1\rangle$ up to a direct summand annihilated by $[\varepsilon]-1$. Applying the Hochschild-Serre spectral sequence
associated to the subgroup $\widehat{\Z}\delta_b\subset \widetilde{\Gamma}_R $, we conclude that, up to $([\varepsilon]-1)$-torsion, the cohomology groups of the
proposition coincide with the cohomology of $K_m^{(d-1)}\langle w_{b-1}-1\rangle$ with respect to $\widetilde{\Gamma}_R /\widehat{\Z}\delta_b$. For general $h<d$
one assumes that, up to $([\varepsilon]-1)^{d-h+1}$-torsion, the cohomology groups of the proposition coincide with the cohomology of $K_m^{(h)}$ with respect to
the groups $\Gamma_h=\oplus_{j=h-a}^b \widehat{\Z}\delta_j \widetilde{\Gamma}_R $ if $h\geq a+1$ and $\Gamma_h=\oplus_{i=h}^a \widehat{\Z} \gamma_i \oplus_{j=1}^b
\widehat{\Z}\delta_j$ if $h\leq a$. Applying the Hochschild-Serre spectral sequence associated to the subgroup $\widehat{\Z}\delta_h \subset \Gamma_h$ if $h\geq
a+1$ and $\widehat{\Z} \gamma_h\subset \Gamma_h$ if $h\leq a$, one concludes that, up to $([\varepsilon]-1)^{d-h+2}$-torsion, the cohomology groups of the
proposition coincide with the cohomology of $K_m^{(h-1)}$ with respect to the group $\Gamma_{h-1}$.

The statement concerning   ${\rm H}^0\left(\widetilde{\Gamma}_R,{\bf A}^{\rm geo, max}_{\rm log}\bigl(\widetilde{R}\bigr)\right)$ follows as before, using that

\begin{align*}
(\gamma_i-1)\left(\frac{(v_i-1)^m}{p^m}\right) &= \frac{\big(\big([\varepsilon]-1\big)v_i+ (v_i-1)\big)^{m}}{p^m}-\frac{(v_i-1)^{m}}{p^m}\\
 &= m v_i \frac{([\varepsilon]-1)}{p} \frac{\big(v_i-1\big)^{m-1}}{p^{m-1}} + \sum\limits_{j=2}^m \begin{pmatrix} m \cr j \end{pmatrix}  v_i^j
\frac{([\varepsilon]-1)^j}{p^j} \frac{\big(v_i-1\big)^{m-j}}{p^{m-j}}\\
\end{align*}and similarly

$$(\delta_j-1)\left(\frac{(w_j-1)^m}{p^m}\right) = m w_j \frac{([\varepsilon]-1)}{p} \frac{\big(w_j-1\big)^{m-1}}{p^{m-1}} +
\sum\limits_{h=2}^m \begin{pmatrix} m \cr h \end{pmatrix}  w_j^h
\frac{([\varepsilon]-1)^h}{p^h}
\frac{\big(w_j-1\big)^{m-h}}{p^{m-h}}.$$
\end{proof}

We remark that the same strategy to prove the vanishing of ${\rm H}^i\left(\widetilde{\Gamma}_R, {\bf A}^{\rm geo, max}_{\rm
log}\bigl(\widetilde{R}\bigr)/(p^m)\right)$, for $1\leq i\leq d$, fails due to the presence of binomial coefficients. For example, coming back to the proof of
\ref{prop:cohA+Rloggeo}, the matrix of $(\gamma_i-1)$ with respect to the bases $\big(v_i-1,(v_i-1)^2/p^2,\ldots,(v_i-1)^{N}/p^N\big)$ and
$\big(1,(v_i-1)/p,\ldots,(v_i-1)^{N-1}/p^{N-1}\big)$ is upper triangular with the elements $(v_i, 2 v_i, \ldots, N v_i)$ on the diagonal.

\subsubsection{The cohomology of the filtration of ${\rm B}^{\rm cris}_{\rm log}(\widetilde{R})$}\label{section:invariantsFilBlog}

In order to conclude the proof of \ref{thm:geometricacyclicity}(iii), we are left to show the vanishing of $${\rm H}^i\left(G_R, \Fil^r{\rm B}^{\rm cris}_{\rm
log}(\widetilde{R})\right)=\ds \lim_{s\to\infty} {\rm H}^i\left(G_R, t^{-s} \Fil^{r+s}{\rm A}^{\rm cris}_{\rm log}(\widetilde{R})\right)$$for $i\geq 1$. The
strategy is the same as in \cite[\S 5]{andreatta_brinon}. Due to \ref{thm:geometricacyclicity}(i) and the fact that ${\rm Gr} {\rm A}^{\rm cris}_{\rm
log}(\widetilde{R})$ is annihilated by $t$, we know that ${\rm H}^i\left(G_R, \Fil^r{\rm A}^{\rm cris}_{\rm log}(\widetilde{R})\right)$ is annihilated by a power
$t^N$ of $t$ depending only on $i$ and $r$. In particular, the composite map ${\rm H}^i\left(G_R, \Fil^r{\rm A}^{\rm cris}_{\rm log}(\widetilde{R})\right)\to {\rm
H}^i\left(G_R, t^{-N} \Fil^r{\rm A}^{\rm cris}_{\rm log}(\widetilde{R})\right)$ is zero. One proves as in \cite[Lemme 33]{andreatta_brinon} that ${\rm
H}^i\left(G_R, \Fil^r{\rm B}^{\rm cris}_{\rm log}(\widetilde{R})\right)$ is a $\Q_p$-vector space. One is reduced to prove that the kernel of the map ${\rm
H}^i\left(G_R, t^{-N} \Fil^{r+N}{\rm A}^{\rm cris}_{\rm log}(\widetilde{R})\right) \to {\rm H}^i\left(G_R, t^{-N} \Fil^r{\rm A}^{\rm cris}_{\rm
log}(\widetilde{R})\right)$ is $p$-torsion; compare with the proof of \cite[Prop.~34]{andreatta_brinon}. Arguing by induction on $N$ we may assume that $N=1$ and we
are reduced to prove the following:

\begin{lemma}\label{lemma:invariantsFilBlog}
The cokernel of ${\rm H}^{i-1}\left(G_R, \Fil^r{\rm A}^{\rm cris}_{\rm log}(\widetilde{R})\right)\to {\rm H}^{i-1}\left(G_R, {\rm Gr}^r{\rm A}^{\rm cris}_{\rm
log}(\widetilde{R})\right)$ is $p$-torsion for every $i\geq 1$ and every $r\in\N$. \end{lemma}\begin{proof}  Thanks to \ref{cor:cohoBHTinftyvanishes} the
$G_R$-cohomology of ${\rm Gr}^r{\rm A}^{\rm cris}_{\rm log}(\widetilde{R})$ is  the $\widetilde{\Gamma}_R$-cohomology of the $\widehat{R \cO}_{\Kbar}$-module
$M_r:=\widehat{\oplus}_{\sum n_i=r} \widehat{R \cO}_{\Kbar} (\epsilon_p-1)^{n_0} \xi^{[n_0]} (u-1)^{[n_1]}(v_2-1)^{[n_2]}\cdots (v_a-1)^{[n_a]}\cdots
(w_b-1)^{[n_d]}$ up to $p$-torsion. On the other hand, define the $\widetilde{R}^{\rm geo}$-submodule $\widetilde{M}_r $ of $\Fil^r{\rm A}^{\rm cris}_{\rm
log}(\widetilde{R})$ spanned by $t^{[n_0]} \log(u)^{[n_1]}\log(v_2)^{[n-2]}\cdots \log(v_a)^{[n_a]}\cdots \log(w_b)^{[n_d]}$ for $\sum_i n_i=r$. Arguing as in
\cite[Lemme 36]{andreatta_brinon} one shows that it is $\widetilde{\Gamma}_R$-stable, it maps surjectively onto $M_r$ and the induced map on
$\widetilde{\Gamma}_R$-cohomology is surjective. This concludes the proof.\end{proof}

\subsubsection{The cohomology of  $\overline{\rm B}^{\rm cris}_{\rm log}(\widetilde{R})$}\label{section:variantBbarlog}
We are left to prove \ref{thm:geometricacyclicity}(iii'). Let $\overline{A}_{\rm log}$ be the image of $A_{\rm log}$ in $\overline{B}_{\rm log}$. Set $$\overline{\rm A}_{\rm log}^{\rm cris,\nabla}(R):={\rm A}_{\rm log}^{\rm cris,\nabla}(R)\widehat{\otimes}_{A_{\rm cris}} \overline{A}_{\rm log},\quad \overline{\rm A}_{\rm log}^{\rm cris}(R):={\rm A}_{\rm log}^{\rm cris}(R)\widehat{\otimes}_{A_{\rm cris}} \overline{A}_{\rm log}.$$Then,\smallskip

(i) $\overline{\rm A}_{\rm log}^{\rm cris,\nabla}(R)\cong  {\rm A}_{\rm
cris}^\nabla(R)\widehat{\otimes}_{A_{\rm cris}} \overline{A}_{\rm log}$;\smallskip

(ii) $\overline{\rm A}_{\rm log}^{\rm cris,\nabla}(R)\left\{\langle v_2-1,\ldots,v_a-1,w_1-1,\ldots,w_b-1\rangle \right\} \lra  \overline{\rm A}^{\rm cris}_{\rm
log}(\widetilde{R})$ is an isomorphism.\smallskip

(iii) $\overline{R}_{\rm log}^{\rm geo}:=\widetilde{R}_{\rm log}^{\rm geo, cris}\widehat{\otimes}_{A_{\rm log}} \overline{A}_{\rm log}$ is the image of ${R}\widehat{\otimes}_{\cO_K} \overline{A}_{\rm log}$ in $\overline{\rm
A}^{\rm cris}_{\rm log}(\widetilde{R})$. \smallskip

The first statement follows as ${\rm A}_{\rm log}^{\rm cris,\nabla}(R)\cong {\rm A}_{\rm cris}^\nabla(R)\widehat{\otimes}_{A_{\rm cris}} A_{\rm log}$ by
\ref{lemma:structureAlog}. The second statement is a consequence of \ref{lemma:structAlog}. The third statement follows from \ref{cor:injstrictinvariants}(4). One
proves the analogues of \ref{prop:HiHvanishes} with $A=\overline{\rm A}^{\rm cris}_{\rm log}(\widetilde{R})$ and ${A}_\infty$ the image of ${\rm A}^{\rm cris}_{\rm
log,\infty}\otimes_{A_{\rm log}}\overline{A}_{\rm log}$ in $\overline{\rm A}^{\rm cris}_{\rm log}(\widetilde{R})$. One  defines $\overline{{\bf A}}^{\rm cris}_{\rm
log}\bigl(\widetilde{R}_n\bigr)$, resp.~$\overline{\bf A}_{\widetilde{R}_n}$, resp.~$\overline{\bf A}^{\rm geo, cris}_{\rm log}\bigl(\widetilde{R}\bigr)$  as the
image of ${\bf A}^{\rm cris}_{\rm log}\bigl(\widetilde{R}_n\bigr)\otimes_{A_{\rm log}}\overline{A}_{\rm log}$ in $\overline{\rm A}^{\rm cris}_{\rm
log}(\widetilde{R})$, resp.~${\bf A}_{\widetilde{R}_n,{\rm cris}}^{+}\left\{\langle u-1\rangle\right\}\otimes_{A_{\rm log}}\overline{A}_{\rm log}$, resp.~${\bf
A}^{\rm geo, cris}_{\rm log}\bigl(\widetilde{R}\bigr)\otimes_{A_{\rm log}}\overline{A}_{\rm log}$; see \ref{lemma:AlogcrysAlogmax} for the notation. Due to
loc.~cit., one gets isomorphisms  $$\overline{\bf A}_{\widetilde{R}_n}\left\{ \langle v_2-1,\ldots,v_a-1,w_1-1,\ldots,w_b-1\rangle \right\} \cong \overline{{\bf
A}}^{\rm cris}_{\rm log}\bigl(\widetilde{R}_n\bigr) $$and $$\overline{R}_{\rm log}^{\rm geo}\left\{\langle v_2-1,\ldots,v_a-1,w_1-1,\ldots,w_b-1\rangle \right\}
\cong \overline{\bf A}^{\rm geo, cris}_{\rm log}\bigl(\widetilde{R}\bigr). $$Define $\overline{\bf A}_m$ as the image of ${\bf A}_m \left\{\langle
u-1\rangle\right\}\otimes_{A_{\rm log}} \overline{A}_{\rm log}$ in $\overline{\rm A}^{\rm cris}_{\rm
log}(\widetilde{R})/\left(p^m,\sum_{i=0}^{p-1}\bigl[\varepsilon\bigr]^{ip^{m-1}}\right)$ and $\overline{{\bf
X}}_{n,m}:=\oplus_{(\underline{\alpha},\underline{\beta})\in E_{n,m}} \overline{\bf
A}_m\big[\overline{X}\big]^{\underline{\alpha}}\big[\overline{Y}\big]^{\underline{\beta}}$. One shows that the analogues of \ref{prop:noetherianize},
\ref{cor:cohoXvanishes}, \ref{prop:cohA+Rloggeo}(1) and \S\ref{section:invariantsFilBlog} hold  proving \ref{thm:geometricacyclicity}(iii').

\subsubsection{The arithmetic invariants}\label{section:arithminvariants}
It follows from \ref{thm:geometricacyclicity} that $\bigl({\rm B}_{\rm
log}^{{\rm cris}}\bigr)^{G_R}=\widetilde{R}_{\rm log}^{\rm geo,
cris}\bigl[t^{-1}\bigr]$ and $\bigl({\rm B}_{\rm max}^{{\rm log}}\bigr)^{G_R}=
\widetilde{R}_{\rm log}^{\rm geo, max}\bigl[t^{-1}\bigr]$.

\begin{lemma}\label{lemma:arithinvawgt0} We have $\widetilde{R}^{\rm cris}_{\rm log}=\left(\widetilde{R}_{\rm log}^{\rm geo, cris}\right)^{G_K}$
and $\widetilde{R}_{\rm max}=\left(\widetilde{R}_{\rm log}^{\rm
geo, max}\right)^{G_K}$.
\end{lemma}
\begin{proof}
Let $A_{\rm cris}(\cO)$ (resp.~$A_{\rm cris,\infty}(\cO)$) be the
$p$-adic completion of the DP envelope of $
\WW\bigl(\widetilde{\bf
E}^+_{\cO_{\overline{K}}}\bigr)\otimes_{\WW(k)} \cO$ (resp.~of $
\WW\bigl(\widetilde{\bf
E}^+_{\cO_{K_\infty'}}\bigr)\otimes_{\WW(k)} \cO$) with respect to
the morphism $\theta$ to $\widehat{\cO}_{\overline{K}}$. Let
$A_{\rm max}(\cO)$  be the
$p$-adic completion of the $ \WW\bigl(\widetilde{\bf
E}^+_{\cO_{\overline{K}}}\bigr)\otimes_{\WW(k)} \cO$-subalgebra of
$\WW\bigl(\widetilde{\bf
E}^+_{\cO_{\overline{K}}}\bigr)\otimes_{\WW(k)} \cO[p^{-1}]$
generated by $p^{-1}\Ker(\theta)$; using the notations of \S\ref{sec:classical} we have inclusions $A_{\rm cris}\subset A_{\rm max} \subset A_{\rm log}$. Analogously, define $ A_{\rm
max,\infty}(\cO)$ using $\WW\bigl(\widetilde{\bf
E}^+_{\cO_{K_\infty'}}\bigr)$ instead of $\WW\bigl(\widetilde{\bf
E}^+_{\cO_{\overline{K}}}\bigr)$. By \ref{lemma:Aminvariant} the
ring $\widetilde{R}_{\rm log}^{\rm geo, cris}$ is a direct factor
of $\widetilde{R}\widehat{\tensor}_{\cO} A_{\rm
cris}(\cO)\left\{\langle u-1\rangle\right\}$ and
$\widetilde{R}_{\rm log}^{\rm geo, max}$ is a direct factor of
$\widetilde{R}\widehat{\tensor}_{\cO} A_{\rm max}(\cO)\left\{\frac{u-1}{p}\right\}$, where $\widehat{\tensor}$
is the $p$-adically completed tensor product.

Let $H$ be the Galois group of $K_\infty' \subset \overline{K}$.
Since every finite field extension of $K_\infty'$ is almost
\'etale, arguing as in \ref{prop:HiHvanishes} one proves that the
invariants of $\widetilde{R}_{\rm log}^{\rm geo, cris}$ (resp.~of
$\widetilde{R}_{\rm log}^{\rm geo, cris}$) with respect to $H$ are
contained in $\widetilde{R}\widehat{\tensor}_{\cO} A_{\rm cris,\infty}(\cO)$ (resp.~in
$\widetilde{R}\widehat{\tensor}_{\cO} A_{\rm
max,\infty}(\cO)$).

Recall that $\WW\bigl(\widetilde{\bf E}^+_{\cO_{K_\infty'}}\bigr)$ contains a subring ${\bf A}^+=\WW(k)\bigl[\!\bigl[ [\overline{\pi}] \bigr]\!\bigr]$ isomorphic to
$\cO$, where the isomorphism is defined by sending $Z$ to $ [\overline{\pi} ]$. Moreover, $\WW\bigl(\widetilde{\bf E}^+_{\cO_{K_\infty'}}\bigr)$ is the
$[\overline{\pi} ]$-completion of $\cup_{m\in \N} {\bf A}^+\bigl[ [\overline{\pi} ]^{\frac{1}{p^m}}\bigr]$; cf.~\ref{lemma:A+R}. In particular, we may write
$\WW\bigl(\widetilde{\bf E}^+_{\cO_{K_\infty'}}\bigr)$ as a direct sum ${\bf A}^+ \oplus {\bf X}$ where ${\bf X}$ is the $\bigl(p,[\overline{\pi} ]\bigr)$-adic
completion of $\sum_{m,a} {\bf A}^+ \cdot [\overline{\pi}]^{\frac{a}{p^m}}$, where the sum is taken over all integers $m\geq 1$  and $1\leq a< p^m$.  Note that the
$p$-adic completion ${\bf A}^{+,\rm cris}_{\rm log}$ of the DP envelope of ${\bf A}^+\otimes_{\WW(k)} \cO$ with respect to $\Ker(\theta)$ is isomorphic to
$\cO\left\{\langle u-1,P_\pi(Z) \rangle\right\}$. Similarly the $p$-adic completion ${\bf A}^{+,\rm max}_{\rm log}$ of the $\bigl({\bf A}^+\otimes_{\WW(k)}
\cO\bigr)^{\rm log}$-subalgebra of $\bigl({\bf A}^+\otimes_{\WW(k)} \cO\bigr)^{\rm log}\bigl[p^{-1}\bigr] $ generated by $p^{-1} \Ker(\theta)$ is isomorphic to
$\cO\left\{\frac{u-1}{p},\frac{P_\pi(Z)}{p}\right\}$. In particular,
$$A_{\rm cris,\infty}(\cO) \cong
{\bf A}^{+,\rm cris}_{\rm log} \oplus {\bf X}\otimes_{{\bf A}^+} {\bf A}^{+,\rm cris}_{\rm log},
\qquad A_{\rm max,\infty}(\cO) \cong
{\bf A}^{+,\rm max}_{\rm log} \oplus {\bf X}\otimes_{{\bf A}^+} {\bf A}^{+,\rm
max}_{\rm log}.$$ Let $\gamma\in G_K$ be an element such that
$\gamma\bigl([\overline{\pi}]\bigr)=[\varepsilon] [\overline{\pi}]$;
it is a topological generator of the coset $G_K/H_K$.  As in
\ref{lemma:vanishingkercokergammai-1} one proves that  the kernel
of $\gamma-1$ on
$\widetilde{R}\widehat{\tensor}_{\cO}\WW\bigl(\widetilde{\bf
E}^+_{\cO_{\overline{K}}}\bigr) $ intersected with
$\widetilde{R}\widehat{\tensor}_{\cO}{\bf X}$ is annihilated by
$\bigl[\varepsilon\bigr]^{\frac{1}{p}}-1$. As in
\ref{cor:cohoXvanishes} one deduces that the kernel of $\gamma-1$
on $\widetilde{R}\widehat{\tensor}_{\cO}A_{\rm cris}(\cO)$ intersected with
$\widetilde{R}\widehat{\tensor}_{\cO}{\bf X}\otimes_{{\bf A}^+}
{\bf A}^{+,\rm cris}_{\rm log}$ (resp.~on
$\widetilde{R}\widehat{\tensor}_{\cO}A_{\rm max}(\cO)$
intersected with $\widetilde{R}\widehat{\tensor}_{\cO}{\bf
X}\otimes_{{\bf A}^+} {\bf A}^{+,\rm max}_{\rm log}$) is annihilated by
$\bigl(\bigl[\varepsilon\bigr]^{\frac{1}{p}}-1\bigr)^2$. In
particular,  it is zero since
$\widetilde{R}\widehat{\tensor}_{\cO} A_{\rm cris}(\cO)$ (resp.~$\widetilde{R}\widehat{\tensor}_{\cO}A_{\rm
max}(\cO)$) is $([\varepsilon]-1)$-torsion free. We
conclude that the invariants we want to compute are the elements
of $\widetilde{R}\left\{\langle u-1,P_\pi(Z) \rangle\right\}$
(resp.~$\widetilde{R}\left\{\frac{u-1}{p},\frac{P_\pi(Z)}{p}\right\}$)
which are invariant under $\gamma-1$ acting on
$\widetilde{R}\widehat{\tensor}_{\cO}A_{\rm log}(\cO) $ (resp.~on
$\widetilde{R}\widehat{\tensor}_{\cO}A_{\rm max}(\cO)$). Arguing as in \ref{prop:cohA+Rloggeo} we conclude
that such invariants coincide with
$\widetilde{R}\left\{\langle P_\pi(Z) \rangle\right\}$
(resp.~$\widetilde{R}\left\{\frac{P_\pi(Z)}{p}\right\}$) as
wanted.
\end{proof}

\begin{remark}\label{rmk:AmaxfrobaAcris} Let $A$ be a ring which is
$p$-adically complete and has no $p$-torsion. Assume that it is
endowed with an operator $\varphi$ lifting Frobenius modulo $p$.
Let $x\in A$ be such that $x-1$ is a regular element and
$\varphi(x)=x^p$. Write $A_{\rm cris}:=A\left\{\langle x-1
\rangle\right\}$ (resp.~$A_{\rm max}:=A\left\{\frac{x-1}{p}
\right\}$) for the $p$-adic completion of the DP envelope of $A$
with respect to $x-1$ (resp.~of the subring $A[(x-1)/p]$ of
$A[p^{-1}]$. Note that Frobenius extends to $A_{\rm cris}$ and
$A_{\rm max}$.

For every  $m\in \N$ we have  $\bigl(x-1\bigr)^{[m]}=
\frac{p^m}{m!}  \bigl(x-1\bigr)^{m} p^{-m} $. In particular, we
have a morphism $A_{\rm cris} \lra A_{\rm max}$. Since
$\varphi(x-1)=x^p-1=(x-1)^p+ p y$, then $\varphi\bigl((x-1)
p^{-1}\bigr)=(p-1)! (x-1)^{[p]}+y$ so that $\varphi$ on $A_{\rm
cris}$ factors via a morphism $A_{\rm cris} \lra A_{\rm max}$.
\end{remark}

It follows from \ref{rmk:AmaxfrobaAcris} and from
\ref{lemma:structureAlog} and \ref{lemma:structAlog} that we have
ring homomorphisms $${\rm B}_{\rm log}^{\rm
cris}\bigl(\widetilde{R}\bigr)\lra {\rm B}^{\rm max}_{\rm
log}\bigl(\widetilde{R}\bigr),\quad \widetilde{R}_{\rm log}^{\rm
cris}\lra \widetilde{R}_{\rm max}.$$Furthermore, as shown in
\ref{rmk:AmaxfrobaAcris} Frobenius on ${\rm B}^{\rm max}_{\rm
log}\bigl(\widetilde{R}\bigr)$ factors via ${\rm B}^{\rm
cris}_{\rm log}\bigl(\widetilde{R}\bigr)$ inducing a ring
homomorphism $\widetilde{R}_{\rm max}[p^{-1}\bigr]\lra
\widetilde{R}_{\rm cris}[p^{-1}\bigr] $. In particular, it
suffices to prove \ref{prop:arithemticinvariantsB} for ${\rm
B}^{\rm cris}_{\rm log}\bigl(\widetilde{R}\bigr)$. Define $$A_{\rm
cris}(r):={\rm A}^{\rm cris}_{\rm log}(\widetilde{R}) \cdot
t^{-r}.$$
Using \ref{lemma:arithinvawgt0} and since $\varphi$ is
Galois equivariant,  to prove \ref{prop:arithemticinvariantsB} we
are reduced to show that for every $r\in\N$ we have
$$
\varphi^s\left( \bigl(A_{\rm cris}(r)\bigr)^{\cG_R}\right)\subset \frac{1}{p^r}\widetilde{R}_{\rm cris}.
$$
This is proven in
\cite[Prop. 4.11.2]{tsuji}. We sketch the argument.

Take $x\in {\rm A}^{\rm cris}_{\rm log}(\widetilde{R})$ such that
$x t^{-r}$ is Galois invariant. Then,  $\varphi^m\bigl(x
t^{-r}\bigr)$ is also Galois invariant. Its image in ${\rm B}_{\rm
dR}^+\bigl(\widetilde{R}\bigr)$ is then also invariant under
$\cG_R$ and those invariants coincide with $
\widehat{\widetilde{R}[p^{-1}]}$ by \ref{prop:BdRff}. In
particular, $\varphi^m\bigl(x t^{-r}\bigr)\in {\rm B}_{\rm
dR}^+\bigl(\widetilde{R}\bigr)$. Since $\varphi^m(t^r)=p^{rm}
t^r$, this implies that $\varphi^m(x)\in t^r {\rm
B}_{\rm dR}^+\bigl(\widetilde{R}\bigr)$ for every $m\in\N$.

Using \ref{lemma:structAlog} write  $x$ as $\sum_{\nu\in \N^{a+b}}
\beta_\nu (v_1-1)^{[\nu_1]}\cdots (v_a-1)^{[\nu_a]}
(w_1-1)^{[\nu_{a+1}]} \cdots (w_b-1)^{[\nu_{a+b}]}$ with
$\nu=(\nu_1,\ldots,\nu_{a+b})$ and $\beta_\nu\in {\rm A}_{\rm
cris}^\nabla(R)$.  Write $\varphi^m(x)= \sum_{\nu\in \N^{a+b}}
\beta_{m,\nu} (v_1-1)^{[\nu_1]}\cdots (v_a-1)^{[\nu_a]}
(w_1-1)^{[\nu_{a+1}]} \cdots (w_b-1)^{[\nu_{a+b}]}$. Since
$\varphi^m(X-1) = \bigl((X-1) + 1\bigr)^{p^m} -1= p^m(X-1) + $ higher
order terms in $(X-1)$ for $X=v_1,\ldots,v_a, w_1,\ldots w_b$, one
argues  that $\beta_{m,\nu}=p^{m \sum_i \nu_i}
\varphi^m(\beta_\nu)+  $ a $\Z$-linear combination of the
$\varphi^m(\beta_{\nu'})$ for $\nu'$ such that $\nu'_i\leq \nu_i$
for every $1\leq i \leq a+b$ and there exists $i$ such that
$\nu_i'<\nu_i$. Since  ${\rm B}_{\rm
dR}^+\bigl(\widetilde{R}\bigr)={\rm B}_{\rm
dR}^{+,\nabla}\bigl(\widetilde{R}\bigr)\bigl[\!\bigl[
v_1-1,\ldots,v_a-1,w_1-1,\ldots,w_b-1\bigr]\!\bigr]$ by
\ref{cor:BdRstr} one concludes by induction that
$\varphi^m(\beta_\nu)\in t^r {\rm B}_{\rm
dR}^{+,\nabla}\bigl(\widetilde{R}\bigr)$ for every
$\nu\in\N^{a+b}$.

Let $I^{[r]}{\rm A}_{\rm log}^{\rm cris,\nabla}(R)$ be the subset
of elements $y$ such that $\varphi^m\bigl(y\bigr)\in \Fil^r {\rm
B}_{\rm dR}^{+,\nabla}\bigl(\widetilde{R}\bigr)$ for every $m\in
\N$. Then, $\beta_\nu\in I^{[r]}{\rm A}_{\rm log}^{\rm
cris,\nabla}(R)$ for every $\nu$.  We are left to prove that
$\varphi^s\bigl(I^{[r]}{\rm A}_{\rm log}^{\rm
cris,\nabla}(R)\bigr)\subset t^r {\rm A}_{\rm log}^{\rm
cris,\nabla}(R)$ with $s=1$ if $p\geq 3$ and $s=2$ if $p=2$.  This follows from \cite[Lemma 4.11.4]{tsuji} or \cite[Prop.
A.3.20]{tsujiinventiones}.

\subsection{The functors ${\rm D}^{\rm cris}_{\rm log}$ and
${\rm D}^{\rm max}_{\rm log}$. Semistable
representations.}\label{sec:Dcrislogmaxlog}

Let $V$ be a finite dimensional $\Q_p$--vector space endowed with
a continuous action of $\cG_R$.  Due to
\ref{prop:arithemticinvariantsB} there exists  $s\in \N$ such that
$\varphi^s\left({\rm B}_{\rm cris}^{{\rm log},\cG_R}\right)\subset
\widetilde{R}_{\rm cris}\bigl[p^{-1}\bigr]$ and
$\varphi^{s}\left({\rm B}_{\rm log}^{{\rm
max},\cG_R}\right)\subset \widetilde{R}_{\rm
max}\bigl[p^{-1}\bigr]$. Write $${\rm D}^{\rm cris}_{\rm
log}(V):=\left(V\otimes_{\Q_p} {\rm B}^{\rm cris}_{\rm
log}\bigl(\widetilde{R}\bigr)\right)^{\cG_R} \otimes_{{\rm B}_{\rm
cris}^{{\rm log},\cG_R}}^{\varphi^{s}} \widetilde{R}_{\rm
cris}\bigl[p^{-1}\bigr].$$It is a $\widetilde{R}_{\rm
cris}\bigl[p^{-1}\bigr]$--module. The  connection and Frobenius on
${\rm B}^{\rm cris}_{\rm log}\bigl(\widetilde{R}\bigr)$ induce  a
connection and a Frobenius on the $\cG_R$-invariants of
$V\otimes_{\Q_p} {\rm B}^{\rm cris}_{\rm
log}\bigl(\widetilde{R}\bigr)$ and, hence, by base change via
$\varphi^{s}\colon {\rm B}_{\rm cris}^{{\rm log},\cG_R}\lra
\widetilde{R}_{\rm cris}\bigl[p^{-1}\bigr]$,  an integrable
connection
$$\nabla_{V,\WW(k)}\colon {\rm D}^{\rm cris}_{\rm log}(V)
\lra {\rm D}^{\rm cris}_{\rm
log}(V)\otimes_{\widehat{\widetilde{R}}}
\widehat{\omega}^1_{\widetilde{R}/\WW(k)}$$and a Frobenius
$\varphi$ horizontal with respect to $\nabla_{V,\WW(k)}$. The
morphism $\varphi^{s}$ on ${\rm B}^{\rm cris}_{\rm
log}\bigl(\widetilde{R}\bigr)$ induces a natural map $${\rm
D}^{\rm cris}_{\rm log}(V) \lra V\otimes_{\Q_p} {\rm B}^{\rm
cris}_{\rm log}\bigl(\widetilde{R}\bigr) . $$We define a
decreasing filtration $\Fil^\bullet {\rm D}^{\rm cris}_{\rm
log}(V)$ as the inverse image of $V\otimes_{\Q_p}\Fil^\bullet {\rm
B}^{\rm cris}_{\rm log}\bigl(\widetilde{R}\bigr)$. Since Frobenius
on ${\rm B}^{\rm cris}_{\rm log}\bigl(\widetilde{R}\bigr)$ is
horizontal with respect to the connection and the filtration on
${\rm B}^{\rm cris}_{\rm log}\bigl(\widetilde{R}\bigr)$ satisfies
Griffiths' transversality, also $\Fil^\bullet {\rm D}^{\rm
cris}_{\rm log}(V)$ satisfies Griffiths' transversality.

Similarly, let
$${\rm D}^{\rm max}_{\rm log}(V):=\left(V\otimes_{\Z_p}
{\rm B}^{\rm max}_{\rm
log}\bigl(\widetilde{R}\bigr)\right)^{\cG_R}\otimes_{{\rm B}_{\rm
log}^{{\rm max},\cG_R}}^{\varphi^{s}} \widetilde{R}_{\rm
max}\bigl[p^{-1}\bigr].$$It is a $\widetilde{R}_{\rm
max}\bigl[p^{-1}\bigr]$--module endowed with  an integrable
connection $\nabla_{V,\WW(k)}$ and a Frobenius $\varphi$. It is
also endowed with an exhaustive decreasing filtration ${\rm
Fil}^n{\rm D}^{\rm max}_{\rm log}(V)$, for $n\in\Z$, given by the
inverse image of $V\otimes_{\Q_p}\Fil^\bullet {\rm B}^{\rm
max}_{\rm log}\bigl(\widetilde{R}\bigr)$ via the morphism $${\rm
D}^{\rm max}_{\rm log}(V) \lra V\otimes_{\Q_p} {\rm B}^{\rm
max}_{\rm log}\bigl(\widetilde{R}\bigr)$$induced by
$\varphi^{s}$ on ${\rm B}^{\rm max}_{\rm
log}\bigl(\widetilde{R}\bigr)$.
\smallskip

It follows from \ref{rmk:AmaxfrobaAcris} and from
\ref{lemma:structureAlog} and \ref{lemma:structAlog} that we have
ring homomorphisms $${\rm B}_{\rm log}^{\rm
cris}\bigl(\widetilde{R}\bigr)\lra {\rm B}^{\rm max}_{\rm
log}\bigl(\widetilde{R}\bigr),\quad \widetilde{R}_{\rm cris}\lra
\widetilde{R}_{\rm max}.$$In particular, we get a map
$$f_V\colon {\rm D}^{\rm cris}_{\rm log}(V)\lra
{\rm D}^{\rm max}_{\rm log}(V).$$It sends ${\rm Fil}^n {\rm
D}^{\rm cris}_{\rm log}(V)$ to ${\rm Fil}^n {\rm D}^{\rm max}_{\rm
log}(V)$ and it is compatible with Frobenius and connections.
Furthermore, as shown in \ref{rmk:AmaxfrobaAcris}, Frobenius on
${\rm B}^{\rm max}_{\rm log}\bigl(\widetilde{R}\bigr)$ factors via
${\rm B}^{\rm cris}_{\rm log}\bigl(\widetilde{R}\bigr)$ inducing a
ring homomorphism $\widetilde{R}_{\rm max}[p^{-1}\bigr]\lra
\widetilde{R}_{\rm cris}[p^{-1}\bigr] $. In particular, Frobenius
on ${\rm D}^{\rm cris}_{\rm log}(V)$ factors as a morphism
$$g_V\colon {\rm D}^{\rm max}_{\rm log}(V)\lra {\rm D}^{\rm cris}_{\rm log}(V).$$We then get the property that $g_V\circ f_V$
and $f_V\circ g_V$ define Frobenius on ${\rm D}^{\rm cris}_{\rm
log}(V)$ (resp.~on ${\rm D}^{\rm max}_{\rm log}(V)$).

\begin{proposition}\label{prop:equivsemistable} Let $V$ be a finite dimensional
$\Q_p$--vector space of dimension $n$ endowed with a
continuous action of $\cG_R$. The following are
equivalent:\smallskip

1) the map of ${\rm B}^{\rm cris}_{\rm
log}\bigl(\widetilde{R}\bigr)$--modules
$$ \left(V\otimes_{\Q_p}
{\rm B}^{\rm cris}_{\rm
log}\bigl(\widetilde{R}\bigr)\right)^{\cG_R} \otimes_{{\rm B}_{\rm
log}^{{\rm cris},\cG_R}} {\rm B}^{\rm cris}_{\rm
log}\bigl(\widetilde{R}\bigr) \lra V\otimes_{\Z_p} {\rm B}^{\rm
cris}_{\rm log}\bigl(\widetilde{R}\bigr)$$is an
isomorphism;\smallskip

2) the map of ${\rm B}^{\rm cris}_{\rm
log}\bigl(\widetilde{R}\bigr)$--modules
$$\alpha_{{\rm cris},V}\colon{\rm D}^{\rm cris}_{\rm log}(V)\otimes_{{\widetilde{R}_{\rm cris}}}
{\rm B}^{\rm cris}_{\rm log}\bigl(\widetilde{R}\bigr) \lra
V\otimes_{\Z_p} {\rm B}^{\rm cris}_{\rm
log}\bigl(\widetilde{R}\bigr)$$is an isomorphism;\smallskip

3) the map ${\rm B}^{\rm max}_{\rm
log}\bigl(\widetilde{R}\bigr)$--modules
$$\left(V\otimes_{\Z_p} {\rm B}^{\rm max}_{\rm log}
\bigl(\widetilde{R}\bigr)\right)^{\cG_R}\otimes_{{\rm B}_{\rm
log}^{{\rm max},\cG_R}} {\rm B}^{\rm max}_{\rm
log}\bigl(\widetilde{R}\bigr) \lra V\otimes_{\Z_p} {\rm B}^{\rm
max}_{\rm log}\bigl(\widetilde{R}\bigr)$$is an
isomorphism;\smallskip

4) the map ${\rm B}^{\rm max}_{\rm
log}\bigl(\widetilde{R}\bigr)$--modules
$$\alpha_{{\rm max},V}\colon
{\rm D}^{\rm max}_{\rm log}(V)\otimes_{\widetilde{R}_{\rm max}}
{\rm B}^{\rm max}_{\rm log}\bigl(\widetilde{R}\bigr) \lra
V\otimes_{\Z_p} {\rm B}^{\rm max}_{\rm
log}\bigl(\widetilde{R}\bigr)$$is an isomorphism.\smallskip

If one of these conditions holds then ${\rm D}^{\rm cris}_{\rm
log}(V)$ is a projective and finitely generated
$\widetilde{R}_{\rm cris}\bigl[p^{-1}\bigr]$-module of rank $n$
and the natural morphisms $${\rm D}^{\rm cris}_{\rm log}(V)
\otimes_{\widetilde{R}_{\rm cris}} {\widetilde{R}_{\rm max}}\lra
{\rm D}^{\rm max}_{\rm log}(V),$$induced by $f_V$, and $${\rm
D}^{\rm cris}_{\rm log}(V) \otimes_{\widetilde{R}_{\rm cris}} {\rm
B}_{\rm log}^{{\rm cris},\cG_R} \lra \left(V\otimes_{\Q_p} {\rm
B}^{\rm cris}_{\rm
log}\bigl(\widetilde{R}\bigr)\right)^{\cG_R}$$and
$${\rm D}^{\rm max}_{\rm log}(V) \otimes_{\widetilde{R}_{\rm max}}
{\rm B}_{\rm log}^{{\rm max},\cG_R} \lra \left(V\otimes_{\Q_p}
{\rm B}^{\rm max}_{\rm
log}\bigl(\widetilde{R}\bigr)\right)^{\cG_R}$$ are all
isomorphisms compatible with Frobenii and connections. Similarly
the morphism
$${\rm D}^{\rm max}_{\rm log}(V) \otimes_{\widetilde{R}_{\rm max}}
{\widetilde{R}_{\rm cris}}\lra {\rm D}^{\rm cris}_{\rm
log}(V),$$induced by $g_V$, is an isomorphism.
\end{proposition}
\begin{proof}
We write  ${\rm B}_{\rm cris}$ for ${\rm B}_{\rm log}^{{\rm
cris}}\bigl(\widetilde{R}\bigr)$ and  ${\rm B}_{\rm max}$ for
${\rm B}_{\rm log}^{{\rm max}}\bigl(\widetilde{R}\bigr)$. We also
let ${\rm D}_{\rm cris}$ be ${\rm D}^{\rm cris}_{\rm log}(V)$ and
${\rm D}_{\rm max}$ be ${\rm D}^{\rm max}_{\rm log}(V)$.
Eventually, we write ${\rm E}_{\rm cris}$ for
$\left(V\otimes_{\Q_p} {\rm B}^{\rm cris}_{\rm
log}\bigl(\widetilde{R}\bigr)\right)^{\cG_R}$ and ${\rm E}_{\rm
max}$ for $\left(V\otimes_{\Q_p} {\rm B}^{\rm max}_{\rm
log}\bigl(\widetilde{R}\bigr)\right)^{\cG_R}$.

(1) $\Longrightarrow$ (2).  We have
$${\rm D}_{\rm cris}\otimes_{\widetilde{R}_{\rm cris}} {\rm B}_{\rm
cris}\cong {\rm E}_{\rm cris}\otimes_{{\rm B}_{\rm
cris}^{\cG_R}}^{\varphi^{s}} \widetilde{R}_{\rm cris}
\otimes_{\widetilde{R}_{\rm cris}} {\rm B}_{\rm cris}\cong {\rm
E}_{\rm cris}\otimes_{{\rm B}_{\rm cris}^{\cG_R}}^{\varphi^{s}}
{\rm B}_{\rm cris} \cong {\rm E}_{\rm cris}\otimes_{{\rm B}_{\rm
cris}^{\cG_R}} {\rm B}_{\rm cris} \otimes^{\varphi^{s}}_{{\rm
B}_{\rm cris}} {\rm B}_{\rm cris}.$$Since (1) holds the latter is
isomorphic to $V\otimes_{\Q_p} {\rm B}_{\rm cris}
\otimes^{\varphi^{s}}_{{\rm B}_{\rm cris}} {\rm B}_{\rm
cris}\cong V\otimes_{\Q_p} {\rm B}_{\rm cris}$. This implies (2).

One proves similarly that (3) $\Longrightarrow$ (4).
\smallskip

(4) $\Longrightarrow$ (1). As ${\rm D}_{\rm max}:={\rm E}_{\rm max} \otimes_{{\rm B}_{\rm log}^{{\rm max},\cG_R}}^{\varphi^s} \widetilde{R}_{\rm max}$, we conclude from \ref{thm:Blogmaxff} and (4) that ${\rm D}_{\rm max}$ is a projective
$\widetilde{R}_{\rm max}[p^{-1}]$-module of rank $n$ i.e., it is a direct summand in a free $\widetilde{R}_{\rm max}[p^{-1}]$-module. In particular, the
$\cG_R$-invariants of ${\rm D}_{\rm max}\otimes_{\widetilde{R}_{\rm max}[p^{-1}]}^{\varphi} {\rm B}_{\rm cris}$ are ${\rm D}_{\rm max}\otimes_{\widetilde{R}_{\rm
max}[p^{-1}]}^{\varphi} {\rm B}_{\rm cris}^{\cG_R}$ and its base change via ${\rm B}_{\rm cris}^{\cG_R}\to {\rm B}_{\rm cris}$ is $ {\rm D}_{\rm
max}\otimes_{\widetilde{R}_{\rm max}[p^{-1}]} {\rm B}_{\rm max} \otimes_{{\rm B}_{\rm max}}^{\varphi} {\rm B}_{\rm cris}$ which is $V\otimes_{\Q_p} {\rm B}_{\rm
cris}$ by (4). This proves (1).\smallskip

We have also proved that if (4) holds then ${\rm D}_{\rm max}$ is a projective $\widetilde{R}_{\rm max}[p^{-1}]$-module of rank $n$ and ${\rm D}_{\rm
max}\otimes_{\widetilde{R}_{\rm max}[p^{-1}]}^{\varphi} {\rm B}_{\rm cris}^{\cG_R} \cong {\rm E}_{\rm cris}$. This implies that the map ${\rm D}_{\rm max}(V)
\otimes_{\widetilde{R}_{\rm max}} {\widetilde{R}_{\rm cris}}\lra {\rm D}_{\rm cris}(V)$, induced by $g_V$, is an isomorphism. Using the projectivity one proves
similarly that ${\rm D}_{\rm max}\otimes_{\widetilde{R}_{\rm max}} {\rm B}_{\rm max}^{\cG_R}\cong {\rm E}_{\rm max}$ and ${\rm D}_{\rm
cris}\otimes_{\widetilde{R}_{\rm cris}} {\rm B}_{\rm cris}^{\cG_R}\cong {\rm E}_{\rm cris}$ compatibly with Frobenius, filtrations and connections so that the last
statements of the proposition hold.\smallskip

(2) $\Longrightarrow$ (3).  Since by (2) we have that ${\rm D}_{\rm cris} \otimes_{\widetilde{R}_{\rm cris}} {\rm B}_{\rm max}$ is isomorphic to $V\otimes_{\Q_p}
{\rm B}_{\rm max}$-module,  it follows from \ref{thm:Blogmaxff} that ${\rm D}_{\rm cris} \otimes_{\widetilde{R}_{\rm cris}} \widetilde{R}_{\rm max}$ is a
projective $\widetilde{R}_{\rm max}[p^{-1}]$-module of rank $n$ and one argues that ${\rm D}_{\rm cris} \otimes_{\widetilde{R}_{\rm cris}} {\rm B}_{\rm
max}^{\cG_R}$ is ${\rm E}_{\rm max}$ and (3) holds.

\end{proof}

If one of the conditions of the proposition holds, we say that $V$
is a {\em semistable} representation of $\cG_R$. For any such the restriction of the
filtration on ${\rm B}_{\rm log}^{\rm cris}$ (resp.~${\rm B}_{\rm
log}^{\rm max}$) via the inclusion ${\rm D}^{\rm cris}_{\rm
log}(V)\subset V\otimes_{\Q_p} {\rm B}_{\rm log}^{\rm cris}$
(resp.~${\rm D}^{\rm max}_{\rm log}(V)\subset V\otimes_{\Q_p} {\rm
B}_{\rm log}^{\rm max}$) define an exhaustive decreasing
filtration ${\rm Fil}^n{\rm D}^{\rm cris}_{\rm log}(V)$, for
$n\in\Z$ (resp.~${\rm Fil}^n{\rm D}^{\rm max}_{\rm log}(V)$).

\begin{proposition}\label{prop:propsemistable} Assume that $V$ is a semistable representation. Then,

\smallskip
(1) Frobenius is horizontal with respect to the connections and it
is \'etale on ${\rm D}^{\rm max}_{\rm log}(V)$ and on ${\rm
D}^{\rm cris}_{\rm log}(V)$ i.e., the maps
$$\varphi\otimes 1\colon {\rm D}^{\rm max}_{\rm log}(V)\otimes_{\widetilde{R}_{\rm max}}^\varphi
\widetilde{R}_{\rm max}\lra {\rm D}^{\rm max}_{\rm log}(V),\qquad
\varphi\otimes 1\colon {\rm D}^{\rm max}_{\rm
log}(V)\otimes_{\widetilde{R}_{\rm cris}}^\varphi
\widetilde{R}_{\rm max}\lra {\rm D}^{\rm cris}_{\rm log}(V)$$are
isomorphisms.\smallskip

(2) The connection is integrable and  topologically nilpotent on
${\rm D}^{\rm cris}_{\rm log}(V)$ and it is integrable and
convergent  on  ${\rm D}^{\rm max}_{\rm log}(V)$.\smallskip

(3) The representation $V$ is de Rham and the natural morphisms $$
{\rm D}^{\rm cris}_{\rm log}(V)\otimes_{\widetilde{R}_{\rm cris}}
\widehat{\widetilde{R}[p^{-1}]}\cong {\rm D}^{\rm max}_{\rm
log}(V)\otimes_{\widetilde{R}_{\rm max}}
\widehat{\widetilde{R}[p^{-1}]}\cong \widetilde{{\rm D}}_{\rm
dR}(V)$$are isomorphisms as
$\widehat{\widetilde{R}[p^{-1}]}$-modules  with
connections. \smallskip

(4) The filtrations ${\rm Fil}^\bullet {\rm D}^{\rm cris}_{\rm
log}(V)$ and ${\rm Fil}^\bullet {\rm D}^{\rm max}_{\rm log}(V)$
satisfy Griffiths' transversality with respect to the given
connection. The morphisms ${\rm D}^{\rm cris}_{\rm log}(V) \lra
{\rm D}^{\rm max}_{\rm log}(V) \lra \widetilde{{\rm D}}_{\rm
dR}(V) $ are strict with respect to the filtrations and for every $r\in\N$ we have isomorphisms

$${\rm Gr}^r {\rm D}^{\rm cris}_{\rm log}(V) \cong {\rm Gr}^r {\rm D}^{\rm max}_{\rm log}(V) \cong {\rm Gr}^r\widetilde{{\rm D}}_{\rm
dR}(V).$$In particular, via the natural maps  $${\rm D}^{\rm cris}_{\rm log}(V) \longrightarrow {\rm D}^{\rm max}_{\rm log}(V)\longrightarrow \widetilde{{\rm
D}}_{\rm dR}(V) \longrightarrow \widetilde{{\rm D}}_{\rm dR}(V)/(Z-\pi)\cong {\rm D}_{\rm dR}(V)$$the filtration on  ${\rm D}_{\rm dR}(V)$ is the $R[p^{-1}]$-span
of the image of the filtration on ${\rm D}^{\rm cris}_{\rm log}(V)$ or on ${\rm D}^{\rm max}_{\rm log}(V)$. Moreover $\Fil^n {\rm D}^{\rm cris}_{\rm log}(V)$  and $\Fil^n
{\rm D}^{\rm max}_{\rm log}(V)$ are uniquely characterized, as filtrations, by the fact that their images span  $\Fil^n {\rm D}_{\rm dR}(V)$ and they satisfy Griffiths'
transversality.

\end{proposition}
\begin{proof} (1) The horizontality of Frobenius follows from
\ref{cor:conenctioninducestandardconnection}. The assertions regarding the \'etalness of ${\rm D}^{\rm cris}_{\rm log}(V)$ follows from the one about ${\rm D}^{\rm
max}_{\rm log}(V)$ and \ref{prop:equivsemistable}. We use the notation of the proof of loc.~cit. We know that ${\rm E}_{\rm max}\otimes_{{\rm B}_{\rm
max}^{\cG_R}}^\varphi {\rm B}_{\rm max}^{\cG_R}$ is a projective ${\rm B}_{\rm max}^{\cG_R}$ module and its base change via ${\rm B}_{\rm max}^{\cG_R} \to {\rm
B}_{\rm max}$ is $V\otimes_{\Q_p} {\rm B}_{\rm max}$. In particular, Frobenius defines an isomorphism ${\rm E}_{\rm max}\otimes_{{\rm B}_{\rm max}^{\cG_R}}^\varphi
{\rm B}_{\rm max}^{\cG_R}\cong {\rm E}_{\rm max}$ thanks to \ref{thm:Blogmaxff}. Taking the base change via ${\rm B}_{\rm max}^{\cG_R}\otimes^{\varphi^{s}} \widetilde{R}_{\rm
max}\bigl[p^{-1}\bigr]$ we deduce the claimed \'etalness for ${\rm D}^{\rm max}_{\rm log}(V)$.
\smallskip

(2) Let $N_i$ be the derivation $\widetilde{R}$ defined by
$\widetilde{X}_i \partial / \partial \widetilde{X}_i $ for $1\leq
i \leq a$ and by  $\widetilde{Y}_j \partial / \partial
\widetilde{Y}_j $ for $a+1\leq i \leq a+b$ with $j=i-a$. Since
${\rm D}^{\rm cris}_{\rm log}(V)$ is \'etale, it suffices to show
that it is generated as $\widetilde{R}_{\rm
cris}\bigl[p^{-1}\bigr]$ by a finite $\widetilde{R}_{\rm
cris}$-module $E$ stable under the connection and such that
$N_i^p(E)\subset p E$ for every $1\leq i\leq a+b$. It suffices to
show that $D:={\rm D}^{\rm max}_{\rm log}(V)$ is generated as
$\widetilde{R}_{\rm max}\bigl[p^{-1}\bigr]$ by a finite
$\widetilde{R}_{\rm max}$-module $D_0$ stable under the connection
and such that $N_i^p(D_0)\subset p D_0$ for every $1\leq i\leq
a+b$. Indeed, in this case $E:=D_0\otimes_{\widetilde{R}_{\rm
max}} \widetilde{R}_{\rm cris} \to {\rm D}^{\rm cris}_{\rm
log}(V)$ is a finite $\widetilde{R}_{\rm cris}$-module with the
required properties.

We may assume that $V$ is in fact a $\Z_p$-representation. Since ${\rm D}^{\rm max}_{\rm log}(V) $ is a projective and finitely generated $\widetilde{R}_{\rm max}
\bigl[p^{-1}\bigr]$-module, it is a direct summand in a finite and free $\widetilde{R}_{\rm max} \bigl[p^{-1}\bigr]$-module $T$. Let $T_0$ be a free
$\widetilde{R}_{\rm max} $-submodule of $T$ such that $T_0\bigl[p^{-1}\bigr]=T$. Let $n\in\N$ be large enough so that the image of $V$ in $${\rm D}^{\rm max}_{\rm
log}(V)\otimes_{\widetilde{R}_{\rm max}} {\rm B}^{\rm max}_{\rm log}\bigl(\widetilde{R}\bigr) \subset T \otimes_{\widetilde{R}_{\rm max}} {\rm B}^{\rm max}_{\rm
log}\bigl(\widetilde{R}\bigr) $$is contained in  $T_0
 \otimes_{\widetilde{R}_{\rm max}} \frac{1}{t^n}
{\rm A}^{\rm max}_{\rm log}\bigl(\widetilde{R}\bigr)$. Then,
$$D_0':=\left(V\otimes_{\Z_p}  \frac{1}{t^n}{\rm A}^{\rm max}_{\rm log}\bigl(\widetilde{R}\bigr)\right)^{\cG_R}\subset T_0
 \otimes_{\widetilde{R}_{\rm max}} \left( \frac{1}{t^n}
{\rm A}^{\rm max}_{\rm
log}\bigl(\widetilde{R}\bigr)\right)^{\cG_R}.$$It follows from
\S\ref{section:arithminvariants} that $\varphi^{s} \colon
\left(\frac{1}{t^n}{\rm A}^{\rm max}_{\rm
log}\bigl(\widetilde{R}\bigr)\right)^{\cG_R} \lra
\widetilde{B}^{\rm max}_{\rm log}$ factors via
$\frac{1}{p^n}\widetilde{R}_{\rm max}$. Write $D_0$ for the $\widetilde{R}_{\rm max}$-span of the image
in $T_0 \otimes_{\widetilde{R}_{\rm max}} \left( \frac{1}{p^n}
\widetilde{R}_{\rm max}\right)$ of the base change of $D_0'$ via
$\varphi^{s}$. It is stable under the connection and
$N_i^p(D_0)\subset p D_0$ for every $1\leq i\leq a+b$ since this
holds for ${\rm A}^{\rm max}_{\rm log}\bigl(\widetilde{R}\bigr)$.
Since $\widetilde{R}_{\rm max}$ is a noetherian ring and $D_0$ is
contained in $T_0
 \otimes_{\widetilde{R}_{\rm max}} \left( \frac{1}{p^n}
\widetilde{R}_{\rm max}\right)$, then $D_0$ is a finite
$\widetilde{R}_{\rm max}$-module. Consider
$D_0\bigl[p^{-1}\bigr]$. It is contained in $D$ and after base
changing via the extension $\widetilde{R}_{\rm max}
\bigl[p^{-1}\bigr]\lra {\rm B}^{\rm max}_{\rm
log}\bigl(\widetilde{R}\bigr)$ it contains $V$ so that it
surjects onto  $D\otimes_{\widetilde{R}_{\rm max}} {\rm B}^{\rm max}_{\rm
log}\bigl(\widetilde{R}\bigr)\cong V\otimes_{\Z_p} {\rm B}^{\rm max}_{\rm
log}\bigl(\widetilde{R}\bigr)$. In particular, the inclusion
$D_0\bigl[p^{-1}\bigr]\subset D$ is surjective after base
changing via $\widetilde{R}_{\rm max} \bigl[p^{-1}\bigr]\lra {\rm
B}^{\rm max}_{\rm log}\bigl(\widetilde{R}\bigr)$. Due to \ref{thm:Blogmaxff} this implies that
$D_0\bigl[p^{-1}\bigr]=D$.
\smallskip

(3) We prove the claim for ${\rm D}^{\rm cris}_{\rm log}(V)$. The
one for ${\rm D}^{\rm max}_{\rm log}(V)$ follows similarly. The
natural $\cG_R$-equivariant morphism ${\rm B}^{\rm cris}_{\rm
log}\bigl(\widetilde{R}\bigr)\lra {\rm B}_{\rm
dR}\bigl(\widetilde{R}\bigr) $ induces a morphism of
$\widetilde{R}_{\rm cris}$-modules ${\rm D}^{\rm cris}_{\rm
log}(V)\lra \widetilde{{\rm D}}_{\rm dR}(V) $. Write $D:={\rm
D}^{\rm cris}_{\rm log}(V)\otimes_{\widetilde{R}_{\rm cris}}
\widehat{\widetilde{R}[p^{-1}]}$. It  is a projective
$\widehat{\widetilde{R}[p^{-1}]}$-module with a natural map
$\alpha\colon D\lra \widetilde{{\rm D}}_{\rm dR}(V)$ of
$\widehat{\widetilde{R}[p^{-1}]}$-modules. Note that
$D\otimes_{\widehat{\widetilde{R}[p^{-1}]}} {\rm B}_{\rm
dR}\bigl(\widetilde{R}\bigr) \cong V\otimes {\rm B}_{\rm
dR}\bigl(\widetilde{R}\bigr) $. Thus, to prove that $\alpha$ is an
isomorphism it suffices to show that ${\rm B}_{\rm
dR}\bigl(\widetilde{R}\bigr)^{\cG_R}=
\widehat{\widetilde{R}[p^{-1}]}$. This is proven in
\ref{prop:BdRff}.

(4) The morphisms for ${\rm B}_{\rm log}^{\rm
cris}\bigl(\widetilde{R}\bigr)\subset {\rm B}_{\rm log}^{\rm
max}\bigl(\widetilde{R}\bigr) \subset {\rm B}_{\rm
dR}\bigl(\widetilde{R}\bigr)$ are strict with respect to the
filtrations by \ref{prop:BcrissubsetBdR}. This implies that the
morphisms ${\rm D}^{\rm cris}_{\rm log}(V)\lra {\rm D}^{\rm
max}_{\rm log}(V) \lra \widetilde{{\rm D}}_{\rm dR} (V)$ are
strict. Since the filtration on ${\rm B}_{\rm log}^{\rm
cris}\bigl(\widetilde{R}\bigr)$ and on $ {\rm B}_{\rm log}^{\rm
max}\bigl(\widetilde{R}\bigr)$ satisfy Griffiths' transversality,
the same holds for ${\rm D}^{\rm cris}_{\rm log}(V)$ and ${\rm
D}^{\rm max}_{\rm log}(V)$. The rest of the claim follows from this,
(3) and \ref{lemma:eqdeRham}.
\end{proof}

\subsubsection{Localizations}\label{sec:DlogforRsmooth}  We assume that we are in the setting
of \S\ref{sec:AlogforRsmooth} and, in particular,  $\widetilde{R}\cong R_0[\![Z]\!]$ by \ref{RR0Zincodim1} and ${\rm B}^{\rm cris}_{\rm log}(\widetilde{R})\cong
{\rm B}_{\rm cris}(R_0)\widehat{\otimes}_{R_0}\widetilde{R}_{\rm cris}$ and ${\rm B}^{\rm max}_{\rm log}(\widetilde{R})\cong {\rm B}_{\rm
max}(R_0)\widehat{\otimes}_{R_0}\widetilde{R}_{\rm max}$ due to \ref{cor:BR}. Here,  ${\rm B}_{\rm cris}(R_0)$ and  ${\rm B}_{\rm max}(R_0)$ are the period rings
introduced in \cite[Def. 6.1.3]{brinon}. Let $V$ be a  representation of $\cG_R$. Define  ${\rm D}_{\rm cris}(V):=\bigl(V\otimes  {\rm B}_{\rm
cris}(R_0)\bigr)^{\cG_R}$ and  ${\rm D}_{\rm max}(V):=\bigl(V\otimes  {\rm B}_{\rm max}(R_0)\bigr)^{\cG_R}$. They are projective $R_0[p^{-1}]$-modules  endowed with
Frobenius, an integrable connection and an exhaustive and decreasing filtrations satisfying Griffiths' transversality; see \cite[\S 8.3]{brinon}.

\begin{proposition} Let $V$ be a  representation of $\cG_R$. Then,\smallskip

(i) $V$ is a crystalline representation of
$\cG_R$ in the sense of \cite[\S 8.2]{brinon} if and only if $V$ is semistable in the sense of \ref{prop:equivsemistable}.

(ii) if (i) holds, then the morphisms  ${\rm D}_{\rm cris}(V)\widehat{\otimes}_{R_0}\widetilde{R}_{\rm cris}\to {\rm D}^{\rm cris}_{\rm log}(V)$ and  ${\rm D}_{\rm
max}(V)\widehat{\otimes}_{R_0}\widetilde{R}_{\rm max}\to {\rm D}^{\rm max}_{\rm log}(V)$ are isomorphisms of $\widetilde{R}_{\rm cris}$-modules (resp.
$\widetilde{R}_{\rm max}$-modules), compatibly with Frobenius and connections and strictly compatible with the filtrations.
\end{proposition}
\begin{proof} (i) Due to \cite[Prop. 8.2.6]{brinon} the morphism
$$\alpha_{{\rm cris},V}\colon {\rm D}_{\rm cris}(V)\otimes_{R_0} {\rm B}_{\rm cris}\bigl(R_0\bigr)
\lra V\otimes_{\Q_p} {\rm B}_{\rm cris}\bigl(R_0\bigr)$$is injective so that $V$ is crystalline if and only if the image of $\alpha_{{\rm cris},V}$ contains $V$. We
have compatible maps  $\widetilde{R}\to R_0 $ and ${\rm B}^{\rm cris}_{\rm log}(\widetilde{R})\to {\rm B}_{\rm cris}(R_0)$ given by $Z\mapsto 0$.  This induces  a
section  ${\rm D}^{\rm cris}_{\rm log}(V)\to {\rm D}_{\rm cris}(V)$ to the morphism given in (i). In particular, if $V$ is semistable then $V$ is in the image of
${\rm D}^{\rm cris}_{\rm log}(V) \otimes {\rm B}_{\rm log}^{\rm cris}\bigl(\widetilde{R}\bigr) \to V\otimes_{\Q_p} {\rm B}_{\rm cris}\bigl(R_0\bigr)$ induced  by
$Z\mapsto 0$. Thus, it is in the image of $\alpha_{{\rm cris},V}$ and $V$ is crystalline.

Viceversa, if $V$ is crystalline then $\alpha_{{\rm cris},V}\otimes_{{\rm B}_{\rm cris}(R_0)} {\rm B}_{\rm cris}\bigl(\widetilde{R}\bigr)$ is an isomorphism,
strictly compatible with the filtrations. As ${\rm D}_{\rm cris}(V)$ is a projective $R_0[p^{-1}]$-module by \cite[Prop. 8.3.1]{brinon}, taking the
$\cG_R$-invariants we get that $\bigl(V\otimes {\rm B}_{\rm cris}\bigl(\widetilde{R}\bigr)\bigr)^{\cG_R}\cong {\rm D}_{\rm cris}(V)\otimes_{R_0} {\rm B}_{\rm
cris}\bigl(\widetilde{R}\bigr)^{\cG_R}$, compatibly with Frobenius and connections and strictly compatible with the filtrations. Moreover,  condition
\ref{prop:equivsemistable}(1) holds. In particular, $V$ is semistable. As ${\rm D}_{\rm cris}(V)$ is an \'etale  $R_0[p^{-1}]$-module by \cite[Prop. 8.3.4]{brinon}
the map in (ii) is an isomorphism.

\end{proof}

We go back to the general ring $R$. Let $T$ be the set of minimal prime ideals of
$R$ over the ideal $(\pi)$ of $R$. For any such $\cP$ let
$\overline{T}_\cP$ be the set of minimal prime ideals of
$\overline{R}$ over the ideal $\cP$. For any $\cP\in T$ denote by
$\widehat{R}_\cP$ the $p$-adic completion of the localization of
$R$ at $\cP\cap R$. It is a dvr. Let $\widetilde{R}(\cP)$ be the
$(p,Z)$-adic completion of the localization of $\widetilde{R}$ at
the inverse image of $\cP$ and let
$R_{\cP,0}:=\widetilde{R}(\cP)/Z \widetilde{R}(\cP)$. For $\cQ\in
\overline{T}_\cP$ let $\overline{R}(\cQ)$ be the  normalization of
$R_{\cP,0}$ in an algebraic closure of ${\rm
Frac}\bigl(\overline{R}_\cQ\bigr)$ and let $G_{R,\cQ}$ be the
Galois group of $R_{\cP,0}\subset \overline{R}(\cQ)$. If $V$ is a
representation of $\cG_R$, we can consider it as a representation
of $G_{R,\cQ}$ and form ${\rm D}_{\rm
cris}\bigl(V\vert_{G_{R,\cQ}}\bigr)$ as in \cite[\S 8.2]{brinon}.
Using \ref{lemma:Acrisonlocisinjective} we get injective maps
$${\rm D}^{\rm cris}_{\rm log}(V) \lra \prod_{\cP\in T,\cQ\in \overline{T}_\cP}
{\rm D}_{\rm cris}\bigl(V\vert_{G_{R,\cQ}}\bigr)\otimes_{R_{\cP,0}}
\widetilde{R}(\cP).$$

\begin{proposition}\label{prop:tensorss}
(1) Let $V$ be a semistable representation of $\cG_R$. Then,
$V\vert_{G_{R,\cQ}}$ is a crystalline representation of
$G_{R,\cQ}$ and ${\rm D}_{\rm
cris}\bigl(V\vert_{G_{R,\cQ}}\bigr)\cong {\rm D}^{\rm cris}_{\rm
log}(V)\otimes_{\widetilde{R}_{\rm cris}} R_{\cP,0}$ compatibly with connections and Frobenius and strictly compatibly with the filtrations.\smallskip

(2) If $V$ and $V'$ are semistable representations of $\cG_R$ then
$V\otimes_{\Q_p} V'$ is a semistable representation of $\cG_R$ and
$ {\rm D}^{\rm cris}_{\rm log}\bigl(V\otimes_{\Q_p} V'\bigr)\cong
{\rm D}^{\rm cris}_{\rm log}(V)\otimes_{\widetilde{R}_{\rm cris}}
{\rm D}^{\rm cris}_{\rm log}(V')$ compatibly with Frobenius and
connections and strictly compatibly with the
filtrations.\smallskip

(3) Let $V$ be a semistable representation of $\cG_R$. Then, the
$\Q_p$-dual $V^\vee$ is a semsistable representation and ${\rm
D}^{\rm cris}_{\rm log}(V^\vee) $ is the $\widetilde{R}_{\rm
cris}\bigl[p^{-1}\bigr]$-dual ${\rm D}^{\rm cris}_{\rm
log}(V)^\vee$ of ${\rm D}^{\rm cris}_{\rm log}(V)$, compatibly
with connections, and Frobenius and strictly compatibly with the
filtrations.
\end{proposition}
\begin{proof} (1) Due to \cite[Prop. 8.2.6]{brinon} the morphism
$$\alpha_{{\rm cris},V\vert_{G_{R,\cQ}}}\colon {\rm D}^{\rm cris}_{\rm log}(V)\otimes_{R_{\cP,0}} {\rm B}_{\rm cris}\bigl(R_{\cP,0}\bigr)
\lra V\vert_{G_{R,\cQ}}\otimes_{\Q_p} {\rm B}_{\rm
cris}\bigl(R_{\cP,0}\bigr)$$is injective so that
$V\vert_{G_{R,\cQ}}$ is crystalline if and only if the image of
$\alpha_{{\rm cris},V\vert_{G_{R,\cQ}}}$ contains $V\vert_{G_{R,\cQ}}$. Due to our
assumption, $V$ is contained in the image of $\alpha_{{\rm
cris},V}$. Since $\alpha_{{\rm cris},V}$ and $\alpha_{{\rm
cris},V\vert_{G_{R,\cQ}}}$ are compatible, we deduce that the
image of $\alpha_{{\rm cris},V\vert_{G_{R,\cQ}}}$ contains
$V\vert_{G_{R,\cQ}}$ as well. This proves that $V\vert_{G_{R,\cQ}}$ is a crystalline representation of
$G_{R,\cQ}$. We certainly have a morphism
$f\colon {\rm D}^{\rm cris}_{\rm
log}(V)\otimes_{\widetilde{R}} R_{\cP,0}\lra {\rm
D}_{\rm cris}\bigl(V\vert_{G_{R,\cQ}}\bigr)$. They are both
projective $R_{\cP,0}\bigl[p^{-1}\bigr]$-modules of rank equal to
the dimension of $V$ as $\Q_p$-vector space. After base change via
$R_{\cP,0}\bigl[p^{-1}\bigr] \subset {\rm B}_{\rm
cris}\bigl(R_{\cP,0}\bigr)$ the map $f$ is an isomorphism. Since
such extension is faithfully flat by \cite[6.3.8]{brinon} the
morphism $f$ is an isomorphism as claimed.\smallskip

(2) By assumption we have an isomorphism $${\rm D}^{\rm cris}_{\rm
log}(V)\otimes_{\widetilde{R}_{\rm cris}} {\rm D}^{\rm cris}_{\rm
log}(V') \otimes_{{\widetilde{R}_{\rm cris}}} {\rm B}^{\rm
cris}_{\rm log}\bigl(\widetilde{R}\bigr) \lra V\otimes_{\Z_p} V'
\otimes_{\Z_p} {\rm B}^{\rm cris}_{\rm
log}\bigl(\widetilde{R}\bigr).$$Since ${\rm D}^{\rm cris}_{\rm
log}(V)$ and ${\rm D}^{\rm cris}_{\rm log}(V')$ are projective
${\widetilde{R}_{\rm cris}}\bigl[p^{-1}\bigr]$-modules, the base
change of the $\cG_R$-invariants of the LHS via ${\rm B}_{\rm
cris}^{{\rm log},\cG_R}\lra {\widetilde{R}_{\rm
cris}}\bigl[p^{-1}\bigr] $ coincide with ${\rm D}^{\rm cris}_{\rm
log}(V)\otimes_{\widetilde{R}_{\rm cris}} {\rm D}^{\rm cris}_{\rm
log}(V')$ due to \ref{prop:arithemticinvariantsB}. It also
coincides with ${\rm D}^{\rm cris}_{\rm log}\bigl(V\otimes_{\Q_p}
V'\bigr)$ by definition, compatibly with connections, filtrations
and Frobenius. The claim follows.

(3) By assumption we have an isomorphism $${\rm D}^{\rm cris}_{\rm
log}(V)^\vee\otimes_{{\widetilde{R}_{\rm cris}}} {\rm B}^{\rm
cris}_{\rm log}\bigl(\widetilde{R}\bigr)\cong {\rm
Hom}_{\Q_p}\bigl(V,\Q_p\bigr) \otimes_{\Z_p} {\rm B}^{\rm
cris}_{\rm log}\bigl(\widetilde{R}\bigr).$$Since ${\rm D}^{\rm
cris}_{\rm log}(V)^\vee$ is a projective ${\widetilde{R}_{\rm
cris}}\bigl[p^{-1}\bigr]$-module and thanks to
\ref{prop:arithemticinvariantsB}, the base change of the
$\cG_R$-invariants of the LHS via ${\rm B}_{\rm cris}^{{\rm
log},\cG_R}\lra {\widetilde{R}_{\rm cris}}\bigl[p^{-1}\bigr] $
coincide with ${\rm D}^{\rm cris}_{\rm log}(V)^\vee$. It also
coincide with ${\rm D}^{\rm cris}_{\rm log}\bigl(V^\vee\bigr)$
compatibly with connections, filtrations and Frobenius. The claim
follows.

We are left to prove the isomorphisms ${\rm D}^{\rm cris}_{\rm
log}(V)\otimes_{\widetilde{R}_{\rm cris}} {\rm D}^{\rm cris}_{\rm
log}(V')\to {\rm D}^{\rm cris}_{\rm
log}(V\otimes V')$ and  ${\rm D}^{\rm cris}_{\rm
log}(V)^\vee \to {\rm D}^{\rm cris}_{\rm
log}\bigl(V^\vee\bigr)$ constructed in (2) and (3) are strictly
compatible with the filtrations. It suffices to prove that they are injective on the associated graded modules. As the maps
$${\rm D}^{\rm cris}_{\rm log}(V) \lra \prod_{\cP\in T,\cQ\in \overline{T}_\cP}
{\rm D}_{\rm cris}\bigl(V\vert_{G_{R,\cQ}}\bigr)\otimes_{R_{\cP,0}} \widetilde{R}(\cP)$$are injective and induce injective maps on ${\rm Gr}^\bullet$, we may reduce
to the smooth case. The claim is then the content of \cite[Prop. 8.4.3]{brinon}.

\end{proof}

\subsubsection{Relation with isocrystals}\label{sec:relcrystals}
Assume that $V$ is a semistable representation. It follows from
\ref{prop:propsemistable}, see the proof of (2), that there exists
a coherent $\widetilde{R}_{\rm cris}$-submodule $D(V)$ of ${\rm
D}^{\rm cris}_{\rm log}(V)$ such that $D(V)\tensor_{\Z_p}
\Q_p={\rm D}^{\rm cris}_{\rm log}(V)$ and\smallskip

(i)  $D(V)$ is stable under the connection $\nabla_{V,\WW(k)}$ and  the
induced logarithmic connection $\nabla_{D(V)}$ is integrable and
topologically nilpotent;
\smallskip

(ii) due to \ref{prop:propsemistable}, choosing suitable integers
$m$ and $n\in\N$ the map $\varphi_{D(V)}:=p^h \varphi$ sends
$D(V)$ to $D(V)$, the morphism $\varphi_{D(V)}$ is horizontal with
respect to $\nabla_{D(V)}$ and multiplication by $p^n$ on $D(V)$
factors via $p^h \varphi_{D(V)}$.\smallskip

We deduce from \cite[Thm.~6.2]{katolog} that
$\bigl(D(V),\nabla_{V,\WW(k)}\bigr)$ defines a crystal ${\cal
D}(V)$ of $\cO_{X_k/\cO_{\rm cris}}$-modules on the site
$\bigl(X_k/\cO_{\rm cris}\bigr)^{\rm cris}_{\rm log}$; see
\ref{sec:isocrystals} for the notation. Moreover, the absolute
Frobenius on $X_k$ and the given Frobenius $\varphi_\cO$ define a
morphism of sites $F\colon \bigl(X_k/\cO_{\rm cris}\bigr)^{\rm
cris}_{\rm log}\lra \bigl(X_k/\cO_{\rm cris}\bigr)^{\rm cris}_{\rm
log}$. Then, $\varphi_{D(V)}$ defines a morphism $\varphi \colon
F^\ast\bigl({\cal D}(V)\bigr)\lra {\cal D}(V)$ of crystals of
$\cO_{X_k/\cO_{\rm cris}}$-modules. Due to (ii) this is well
defined up to multiplication by $p$.

Given two charts on $\widetilde{R}$, inducing two choices of
Frobenius $\varphi_1$ and $\varphi_2$ on $\widetilde{R}$, we get
two Frobenii $\varphi_1$ and $\varphi_2$ on ${\cal D}(V)$. Then,

\begin{corollary}\label{cor:changefrobenius}
Assume that $V$ is a semistable representation. Then, the two
Frobenii $\varphi_1$ and $\varphi_2$ on the crystal ${\cal D}(V)$
differ by multiplication by a power of $p$.
\end{corollary}
\begin{proof} Choose in (ii) above $h$ large enough so that it
works both for $\varphi_1$ and for $\varphi_2$. We then prove that
$\varphi_1$ and $\varphi_2$ on the crystal ${\cal D}(V)$ coincide.

Let $T$ be the set of minimal prime ideals of $R$ over the ideal
$(\pi)$ of $R$. For any such $\cP$ let $\overline{T}_\cP$ be the set
of minimal prime ideals of $\overline{R}$ over the ideal $\cP$.
Using the injective maps
$${\rm D}^{\rm cris}_{\rm log}(V) \lra \prod_{\cP\in T,\cQ\in \overline{T}_\cP}
{\rm D}_{\rm cris}\bigl(V\vert_{G_{R,\cQ}}\bigr)\otimes_{R_{\cP,0}}
\widetilde{R}(\cP),$$it suffices to prove the claim for ${\rm
D}_{\rm cris}\bigl(V\vert_{G_{R,\cQ}}\bigr)$ for every $\cP\in T$
and $\cQ\in \overline{T}_\cP$. Since the log structure on
$R_{\cP,0}$ is trivial, our claim is  the content of \cite[Prop.
7.2.3]{brinon}.
\end{proof}

\subsection{The functors ${\rm D}_{\rm cris}^{\rm log, geo}$ and
${\rm D}_{\rm max}^{\rm log, geo}$. Geometrically semistable
representations.}\label{sec:Vgeosemistable}

Let $V$ be a finite dimensional $\Q_p$--vector space endowed with
a continuous action of the geometric Galois group $G_R$. Define
$${\rm D}_{\rm log}^{\rm geo, cris}(V):=\left(V\otimes_{\Q_p} {\rm
B}^{\rm cris}_{\rm log}\bigl(\widetilde{R}\bigr)\right)^{G_R},
\qquad {\rm D}_{\rm log}^{\rm geo, max}(V):=\left(V\otimes_{\Q_p}
{\rm B}^{\rm max}_{\rm
log}\bigl(\widetilde{R}\bigr)\right)^{G_R}.$$They are
$\widetilde{R}^{\rm geo,cris}_{\rm log}$-modules (resp.
$\widetilde{R}^{\rm geo,max}_{\rm log}$-modules)
endowed with filtrations, connections $\nabla_{V,\WW(k)}$ and
$\nabla_{V,B_{\rm log}}$ and semilinear Frobenius $\varphi_V$. We
have

\begin{proposition}\label{prop:equivgeosemistable} The following are equivalent:\smallskip

(1) ${\rm D}_{\rm log}^{\rm geo, cris}(V)$ is a finite and
projective $\widetilde{R}_{\rm log}^{\rm geo,
cris}\bigl[t^{-1}\bigr]$-module and the map $${\rm D}_{\rm
log}^{\rm geo, cris}(V)\otimes_{\widetilde{R}_{\rm log}^{\rm geo,
cris}} {\rm B}^{\rm cris}_{\rm log}\bigl(\widetilde{R}\bigr)\lra
V\otimes_{\Q_p} {\rm B}^{\rm cris}_{\rm
log}\bigl(\widetilde{R}\bigr)$$is an isomorphism;\smallskip

(2) ${\rm D}_{\rm log}^{\rm geo, max}(V)$ is a finite and
projective $\widetilde{R}_{\rm log}^{\rm geo,
max}\bigl[t^{-1}\bigr]$-module and the map $${\rm D}_{\rm
log}^{\rm geo, max}(V)\otimes_{\widetilde{R}_{\rm log}^{\rm geo,
max}} {\rm B}^{\rm max}_{\rm log}\bigl(\widetilde{R}\bigr)\lra
V\otimes_{\Q_p} {\rm B}^{\rm cris}_{\rm
log}\bigl(\widetilde{R}\bigr)$$is an isomorphism.\smallskip

Moreover, in this case ${\rm D}_{\rm log}^{\rm geo,
cris}(V)\otimes_{\widetilde{R}_{\rm log}^{\rm geo, cris}}
\widetilde{R}_{\rm log}^{\rm geo, max}\cong {\rm D}_{\rm log}^{\rm
geo, max}(V)$ compatibly with filtrations, connections and
Frobenius.
\end{proposition}
\begin{proof} This is a consequence of the projectivity assumptions and the fact that ${\rm B}_{\rm
cris}^{{\rm log},G_R}=\widetilde{R}_{\rm log}^{\rm geo,
cris}\bigl[t^{-1}\bigr]$ and ${\rm B}_{\rm max}^{{\rm log},G_R}=
\widetilde{R}_{\rm log}^{\rm geo, cris}\bigl[t^{-1}\bigr]$ proven
in \ref{thm:geometricacyclicity}.
\end{proof}

\begin{definition}\label{def:geosemionR} We say that a representation $V$ is {\em geometrically semistable}
if one of the two conditions above hold and if furthermore there exists a coherent $\widetilde{R}\widehat{\otimes}_{\cO} A_{\rm log}$-submodule $D$ of ${\rm D}_{\rm
log}^{\rm geo, cris}(V)$ such that:\smallskip

(a) it is stable under the connection $\nabla_{V,\WW(k)}$ and $\nabla_{V,\WW(k)}\vert_{D}$ is  topologically nilpotent;\smallskip

(b) $D[t^{-1}]={\rm D}_{\rm log}^{\rm geo, cris}(V) $;\smallskip

(c) there exist integers $h$ and $n\in\N$ such that the map $t^h \varphi$ sends $D$ to $D$ and its image contains $t^n D$.
\smallskip

\end{definition}

The following corollary provides examples of geometrically semistable representations:

\begin{corollary}\label{cor:arithmimpliesgeo} If $V$ is a
semistable representation of $\cG_R$, then it is geometrically
semistable and we have natural isomorphisms
$${\rm D}^{\rm cris}_{\rm log}(V)\otimes_{\widetilde{R}^{\rm cris}_{\rm log}} \widetilde{R}_{\rm log}^{\rm geo, cris}\cong {\rm
D}_{\rm log}^{\rm geo, cris}(V), \qquad  {\rm D}^{\rm max}_{\rm
log}(V)\otimes_{\widetilde{R}_{\rm max}} \widetilde{R}_{\rm
log}^{\rm geo, max}\cong   {\rm D}_{\rm log}^{\rm geo,
max}(V)$$compatible with connections and Frobenius and strictly compatible with the filtrations.
\end{corollary}
\begin{proof} The displayed isomorphisms follow from the isomorphisms in \ref{prop:equivsemistable}, the fact that ${\rm D}^{\rm cris}_{\rm log}(V)$ and ${\rm D}^{\rm max}_{\rm log}(V)$ are projective modules and the computation ${\rm B}_{\rm
log}^{{\rm cris},G_R}=\widetilde{R}_{\rm log}^{\rm geo,
cris}\bigl[t^{-1}\bigr]$ and ${\rm B}_{\rm max}^{{\rm log},G_R}=
\widetilde{R}_{\rm log}^{\rm geo, max}\bigl[t^{-1}\bigr]$ provided in \ref{thm:geometricacyclicity}.

Such isomorphisms are clearly compatible with connections, Frobenius and filtrations. The conditions in \ref{def:geosemionR} are satisfied due to \S\ref{sec:relcrystals}. The strict compatibility with the filtrations follows from \ref{prop:BcrissubsetBdR}(4), \ref{prop:propsemistable}(4) and \ref{prop:GrDdR}.

\end{proof}

We state our main result:

\begin{proposition}
(i) The category of geometrically semistable representations is closed
under duals, tensor products and extensions. \smallskip

(ii) The functors ${\rm D}_{\rm log}^{\rm geo, max}$
and  ${\rm D}_{\rm log}^{\rm geo, cris}$, from the category of
geometrically semistable representations to the category of $\widetilde{R}^{\rm geo,cris}_{\rm log}$-modules (resp.
$\widetilde{R}^{\rm geo, max}_{\rm log}$-modules)
endowed with connections and Frobenius, commute with duals and tensor products and are exact.
\end{proposition}
\begin{proof} The claims concerning duals and tensor products follow proceeding as in \ref{prop:tensorss}(2)\&(3).
Let $0 \to V_1 \to V_2 \to V_3 \to 0$ be an exact sequence of $\Q_p$-vector spaces endowed with an action of $G_R$ with $V_1$ and $V_3$ geometrically semistable.
First of all we claim that the sequence $0\to {\rm D}_{\rm log}^{\rm geo, max}(V_1)\to {\rm D}_{\rm log}^{\rm geo, max}(V_2)\to {\rm D}_{\rm log}^{\rm geo,
max}(V_3)\to 0$ is exact. This follows if we prove that ${\rm H}^1\bigl(G_R, V_1\otimes_{\Q_p} {\rm B}_{\rm log}^{\rm max}\bigr)=0$. This group coincides with ${\rm
H}^1\bigl(G_R, {\rm D}_{\rm log}^{\rm geo, max}(V)\otimes_{\widetilde{R}_{\rm log}^{\rm geo, max}} {\rm B}^{\rm max}_{\rm log}\bigl(\widetilde{R}\bigr) \bigr)$
since $V_1$ is geometrically semistable. Since ${\rm D}_{\rm log}^{\rm geo, max}(V)$ is a projective $\widetilde{R}_{\rm log}^{\rm geo,
max}\bigl[t^{-1}\bigr]$-module of finite rank, it suffices to prove the vanishing of ${\rm H}^1\bigl(G_R,  {\rm B}^{\rm max}_{\rm
log}\bigl(\widetilde{R}\bigr)\bigr)$. This follows from \ref{thm:geometricacyclicity}. In particular, ${\rm D}_{\rm log}^{\rm geo, max}(V_2)$ is a finite and
projective $\widetilde{R}_{\rm log}^{\rm geo, max}\bigl[t^{-1}\bigr]$-module. Consider the commutative diagram with exact rows:

$$\begin{array}{ccccccc}
& {\rm D}_{\rm log}^{\rm geo, max}(V_1)\otimes {\rm B}^{\rm
max}_{\rm log} & \lra & {\rm D}_{\rm log}^{\rm geo,
max}(V_2)\otimes {\rm B}^{\rm max}_{\rm log} & \lra & {\rm D}_{\rm
log}^{\rm geo, max}(V_3)\otimes {\rm B}^{\rm max}_{\rm log} & \lra
0\cr &\big\downarrow & & \big\downarrow & & \big\downarrow\cr 0
\lra & V_1\otimes_{\Q_p} {\rm B}^{\rm max}_{\rm
log}\bigl(\widetilde{R}\bigr) & \lra & V_2\otimes_{\Q_p} {\rm
B}^{\rm max}_{\rm log}\bigl(\widetilde{R}\bigr) &\lra &
V_3\otimes_{\Q_p} {\rm B}^{\rm max}_{\rm
log}\bigl(\widetilde{R}\bigr) & \lra 0,\cr
\end{array}$$where the tensor product in the first row is taken over ${\widetilde{R}_{\rm log}^{\rm geo, max}}$ and ${\rm B}^{\rm max}_{\rm log}$ stands for ${\rm
B}^{\rm max}_{\rm log}\bigl(\widetilde{R}\bigr)$. The right and left vertical arrows are isomorphisms by assumption. The snake lemma implies that also the vertical
arrow in the middle is an isomorphism as wanted. In particular, $V_2$ satisfies \ref{prop:equivgeosemistable}(2).

We are left to show that the other conditions of \ref{def:geosemionR} are satisfied. Let $D_1\subset {\rm D}_{\rm log}^{\rm geo, cris}(V_1)$ and $D_3\subset  {\rm
D}_{\rm log}^{\rm geo, cris}(V_3)$ be the submodules as in loc.~cit. We have just proven that ${\rm D}_{\rm log}^{\rm geo, cris}(V_2)$ is an extension of the projective
$\widetilde{R}^{\rm geo,cris}_{\rm log}$-modules  ${\rm D}_{\rm log}^{\rm geo, cris}(V_1)$
and ${\rm D}_{\rm log}^{\rm geo, cris}(V_3)$. In particular, it is isomorphic to their  direct sum. We view
$D_2':=D_1\oplus D_3\subset {\rm D}_{\rm log}^{\rm geo, cris}(V_2)$ as a submodule. Note that $D_2'[t^{-1}]={\rm D}_{\rm log}^{\rm geo, cris}(V_2)$. The connection
$\nabla_{V_2,\WW(k)}$ is compatible with the connections $\nabla_{V_1,\WW(k)}$ and $\nabla_{V_3,\WW(k)}$ so that it preserves $D_1$ and sends $D_2'$ to
$\bigl(t^{-N} D_1\oplus D_3\bigr) \widehat{\otimes} \omega^1_{\widetilde{R}/\WW(k)}$ for some $N\in\N$. Set $D_2:=t^{-N} D_1\oplus D_3$. Then, $D_2$ is a coherent
$\widetilde{R}\widehat{\otimes}_{\cO} A_{\rm log}$-module, it is stable under $\nabla_{V_2,\WW(k)}$ and $\nabla_{V_2,\WW(k)}\vert_{D_2}$ is topologically nilpotent
as $\nabla_{V_1,\WW(k)}\vert_{D_1}$ and $\nabla_{V_3,\WW(k)}\vert_{D_3}$ are. Thus  conditions \ref{def:geosemionR}(a)\&(b) hold. If we take $n\in\N$ and $h\leq n$
so that $t^h \varphi$ satisfies condition \ref{def:geosemionR}(c) for $D_1$ and $D_3$, then $t^{h} \varphi$ sends $D_1$ to $D_1$ and $D_2$ to $t^{-m} D_1\oplus
D_3$ for some $m\in\N$.  Then, $t^{2m+2n} D_2$ is contained in $t^{h+m}\varphi(D_2)$ so that condition
\ref{def:geosemionR}(c) holds.

\end{proof}

\section{List of Symbols}

\noindent $\widetilde{\bf E}_{\cO_\Kbar}^+$ classical Fontaine ring, \S \ref{sec:classical}

\noindent $A_{\rm inf}\left(\cO_{\Kbar}\right)$ classical Fontaine ring, \S \ref{sec:classical}

\noindent $A_{\rm cris}$,  $B_{\rm cris}$ classical Fontaine rings, \S \ref{sec:classical}

\noindent  $A_{\rm log}$, $B_{\rm log}$  classical Fontaine rings, \S \ref{sec:classical}

\noindent $B_{\rm dR}^+$, $B_{\rm dR}^+(\cO)$  classical Fontaine rings, \S \ref{sec:classical}

\noindent $D_{\rm cris}$, $D_{\rm log}$, $D_{\rm dR}$  classical Fontaine functors, \S \ref{sec:classical}

\noindent $\fX_L$, ${\rm T}_{X_L}$ Faltings' site and respectively Faltings' topos associated to $X$ and $L$, \S \ref{sec:faltingssite}

\noindent $\cO_{\fX}$, $\widehat{\cO}_{\fX}$ Fontaine sheaves, \S \ref{sec:fontainesheaves}

\noindent $\bA_{\rm inf,L}^+$ Fontaine sheaf, \S \ref{sec:fontainesheaves}

\noindent $\bA_{\rm cris}^\nabla$  Fontaine sheaf, \S \ref{sec:fontainesheaves}

\noindent $\bA_{\rm log,L}^\nabla$, Fontaine sheaf, \S \ref{sec:fontainesheaves}

\noindent $\bA_{\rm log}$ Fontaine sheaf, \S \ref{sec:Alog}

\noindent $\bB_{\rm log}^\nabla$, $\bB_{\rm log}$ Fontaine sheaves \S \ref{sec:BlogBlognabla}

\noindent $\overline{\bB}_{\rm log,\Kbar}^\nabla$, $\overline{\bB}_{\rm log,\Kbar}$ Fontaine sheaves \S \ref{sec:defBbarlog}

\noindent $\bDcrisgeo$, Fontaine functor \S \ref{sec:Dcrisgeo}

\noindent $\bDcrisar$ Fontaine functor \S \ref{sec:Dcrisar}

\noindent $R_n$, $R^o$ rings, \S \ref{section:localdescription}

\noindent $\widehat{\overline{R}}$ relative Fontaine ring, \S \ref{sec:widehatR}

\noindent $\widetilde{R}_\infty$ relative Fontaine ring \S \ref{sec:widetildeRinfty}

\noindent $\widetilde{\bf E}^+_S$ relative Fontaine ring, \S \ref{sec:Thetalocalized}

\noindent ${\bf A}_{\widetilde{R}_n}^+$, relative Fontaine ring, \S \ref{sec:AwidetildeRn}

\noindent ${\rm A}_{\rm inf}\bigl(R/\cO\bigr)$ , ${\rm A}_{\rm inf}\bigl(R/R\bigr)$, ${\rm A}_{\rm inf}\bigl(R/\widetilde{R}\bigr)$ relative Fontaine rings, \S
\ref{sec:BdR(R)}

\noindent ${\rm B}_{{\rm dR},n}^{\nabla,+}\bigl(R\bigr)$, ${\rm B}_{{\rm dR},n}^{\nabla,+}\bigl(\widetilde{R}\bigr)$, ${\rm B}_{{\rm dR},n}^+\bigl(R\bigr)$, ${\rm
B}_{{\rm dR},n}^+\bigl(\widetilde{R}\bigr)$ relative Fontaine rings, \S \ref{sec:BdR(R)}

\noindent  ${\rm B}^{\rm cris}_{\rm log}$, ${\rm B}^{\rm max}_{\rm log}$ relative Fontaine rings, \S \ref{sec:AcrisnablaRbar}

\noindent ${\rm A}_{\rm log}^{\rm cris,\nabla}(R)$, ${\rm A}_{\rm log}^{\rm cris,\nabla}$, relative Fontaine rings, \S \ref{sec:AcrisnablaRbar}

\noindent  $\widetilde{R}_{\rm max}$ ring, \S \ref{sec:descentBlogmax}

\noindent ${\bf A}_{\widetilde{R},{\rm max}}^{+,\rm log,\nabla}$ , ${\bf A}_{\widetilde{R}^o,{\rm max}}^{+,\rm log,\nabla}$ relative Fontaine rings \S
\ref{sec:descentBlogmax}

\noindent ${\bf A}_{\widetilde{R},\rm cris}^{+,\rm geo} $, ${\rm A}_{\rm log,\infty}^{\rm cris,\nabla}$,  $ {\rm A}^{\rm cris.\nabla}_{\rm log}(\widetilde{R})$,
${\bf A}_{{\rm log}}^{\rm geo, cris}(\widetilde{R}) $, ${\rm A}_{\rm log,\infty}^{\rm cris} $ relative Fontaine rings, \S \ref{sec:geomtericBlogcris}

\noindent ${\rm D}_{\rm dR}$, $\widetilde{{\rm D}}_{\rm dR}$ relative de Rham functors, \S \ref{sec:DdeRam}

\noindent  ${\rm D}^{\rm cris}_{\rm log}$, ${\rm D}^{\rm max}_{\rm log}$ relative Fontaine functors, \S \ref{sec:Dcrislogmaxlog}

\noindent ${\rm D}_{\rm cris}^{\rm log, geo}$, ${\rm D}_{\rm max}^{\rm log, geo}$ geometric, relative  Fontaine functors, \S \ref{sec:Vgeosemistable}


\begin{thebibliography}{9}

\bibitem[AB]{andreatta_brinon}
F.~Andreatta, O.~Brinon: {\em Acyclicit\'e g\'eom\'etrique de $B_{\rm cris}$}. To appear in Commentarii Mathematici Helvetici.



\bibitem[AI1]{andreatta_iovita}
F.~Andreatta, A.~Iovita: {\em Global applications of relative $(\varphi,\Gamma)$-modules}. In ``Repr\'esentations $p$-adiques de groupes $p$-adiques I :
repr\'esentations galoisiennes et $(\varphi,\Gamma)$-modules", Ast\'erisque {\bf 319} (2008), 339-419.

\bibitem[Err]{erratum} F.~Andreatta, A.~Iovita, {\em Erratum to the
article ``Global applications to relative $(\varphi,\Gamma)$-modules,
I''}, Ast\'erisque 330, Repr\'esentations $p$-adiques de groupes
$p$-adiques II: Repr\'esentations de $\GL_2(\Q_p)$ et $(\varphi,
\Gamma)$-modules, (2010), 543-554.



\bibitem[AI2]{andreatta_iovita_comparison}
F.~Andreatta, A.~Iovita: {\em Comparison Isomorphisms for Smooth Formal Schemes}. To appear in the  Journal de l'Institut de Math\'ematiques de Jussieu.





\bibitem[Be]{berthelot1} P. Berthelot, \emph{Cohomologie cristalline des sch\'emas
de caract\'eristique $p>0$}. LNM \textbf{407},  Springer, Berlin-Heidelberg-New York, 1974.


\bibitem[BO]{berthelot_ogus}
P.~Berthelot, A.~Ogus: {\em Notes on crystalline cohomology}. Princeton University Press, 1978.


\bibitem[Bre]{breuil} C.~Breuil: {\em Repr\'esentations $p$-adiques
semi-stables et transversalit\'e de Griffiths}. Math. Ann. {\bf
307} (1997),  191--224.

\bibitem[Bri]{brinon}
O.~Brinon: {\em Repr\'esentations $p$-adiques cristallines et de de Rham dans le cas relatif}. M\'emoires de la SMF {\bf 112} (2008), pp.~158.


\bibitem[CF]{colmez_fontaine}
P.~Colmez, J.-M.~Fontaine: {\em Construction des repr\'esentations $p$-adiques semi-stables}.  Invent. Math. {\bf 140}, 1--43, (2000).

\bibitem[El]{Elkik} R.~Elkik: {\em Solutions d'\'equations \`a
coefficients dans un anneau hens\'elien}. Ann.~Sci. E.N.S. {\bf 6} (1973), 553--604.

\bibitem[F3]{faltingsAsterisque} G.~Faltings: {\em Almost \'etale
extensions}, In ``Cohomologies $p$-adiques et applications
arithm\'etiques,'' vol.~II. P.~Berthelot, J.-M.~Fontaine,
L.~Illusie, K.~Kato, M.~Rapoport eds. Ast\'erisque {\bf 279}
(2002), 185--270.

\bibitem[Fo]{Fontaineperiodes} J.-M.~Fontaine: {\em Le corps des p\'eriodes
$p$-adiques. With an appendix by Pierre Colmez}. In ``P\'eriodes
$p$-adiques (Bures-sur-Yvette, 1988)".  Ast\'erisque  {\bf 223}
(1994), 59--111.

\bibitem[FdP]{Fresnel_VanderPut} J.~Fresnel, M.~van der Put: {\em  Rigid analytic geometry and its
applications}. Progress in Mathematics, {\bf 218},  pp.~296.


\bibitem[SGAI]{SGAI} {\em Rev\^etements \'etales et groupe fondamental}.
S\'eminaire de G\'eom\'etrie Alg\'ebrique du Bois-Marie 1960--1961
(SGA 1). Dirig\'e par Alexandre Grothendieck. Augment\'e de deux
expos\'es de M. Raynaud. LNM {\bf 224} (1971).

\bibitem[SGAIV]{SGAIV} {\it Th\'eorie des topos et cohomologie \'etale des sch\'emas.
Tome 1: Th\'eorie des topos.}  S\'eminaire de G\'eom\'etrie
Alg\'ebrique du Bois-Marie 1963--1964 (SGA 4). Dirig\'e par M.
Artin, A. Grothendieck, et J. L. Verdier. Avec la collaboration de
N. Bourbaki, P. Deligne et B. Saint-Donat. LNM {\bf 269} (1972).



\bibitem[Ga]{Gabber} O.~Gabber; {\em Affine analog of the
proper base change theorem}. Israel J.~of Math.~{\bf 87} (1994), 325--335.

\bibitem[GR]{GabberRamero} O.~Gabber, L.~Ramero: {\em Foundations of p-adic Hodge theory}. Available at http://math.univ-lille1.fr/~ramero/.

\bibitem[Il]{IllusieKummer} L.~Illusie: {\em An overview of the work of K. Fujiwara, K. Kato, and C.
Nakayama on logarithmic \'etale cohomology}. In `` Cohomologies $p$-adiques et applications arithm�tiques, II". Ast\'erisque  {\bf 279} (2002), 271--322.




\bibitem[K1]{katoperiodes} K.~Kato: {\em Semi-stable reduction and $p$-adic \'etale cohomology}. In
{\rm P\'eriodes $p$-adiques (Bures-sur-Yvette, 1988)}.
Ast\'erisque {\bf 223} (1994), 269--293.

\bibitem[K2]{katolog} K.~Kato: {\em Logarithmic structures of
Fontaine-Illusie}, in {\rm Proceedings of JAMI conference},
(1988), 191-224.

\bibitem[K3]{katotoric} K.~Kato: {\em Toric singularities}, Amer. J. Math. {\bf 116} (1994), 1073--1099.

\bibitem[Ka]{katz} N.~Katz: {\em Nilpotent connexions and the monodromy theorem: applications of a result of
Turrittin}. Publ. Math. IHES {\bf 39} (1970), 176--232.

\bibitem[Ni]{niziol} W.~Niziol: {\em K-Theory of Log-Schemes I}. Documenta Mathematica {\bf 13} (2008) 505-�551.

\bibitem[O]{ogus} A.~Ogus: {$F$-Isocrystals and de Rham cohomology II. Convergent isocrystals}. Duke Math.~J.~{\bf 51} (1984),
 765--850.



\bibitem[Ol]{Olsson} M.~Olsson: {\em On Faltings' method of almost \'etale
extensions}. In ``Algebraic geometry---Seattle 2005. Part 2".
Proc. Sympos. Pure Math. {\bf 80} (2009), 811--936.

\bibitem[S]{MSaito} M.~Saito: {\it Modules de Hodge polarisables}.
Publ. Res. Inst. Math. Sci. {\bf 24} (1988),  849--995.


\bibitem[Sc]{Scholze} P.~Scholze: {\it $p$-Adic Hodge theory for rigid analytic varieties}, preprint (2012).


\bibitem[T1]{tsujiinventiones} T.~Tsuji: {\em $p$-adic \'etale cohomology and crystalline
cohomology in the semi-stable reduction case}.  Invent. Math. {\bf
137} (1999), 233--411.

\bibitem[T2]{tsuji} T.~Tsuji: {\em Crystalline sheaves and filtered convergent
$F$-isocrystals on log schemes.} Preprint.
\end{thebibliography}
\end{document}